\documentclass[a4paper,11pt]{amsart}
\usepackage{amsmath,amsthm,amssymb,mathtools}
\usepackage[mathscr]{eucal}
\usepackage{cite}
\usepackage{upgreek}
\usepackage{comment}
\usepackage[bookmarks,bookmarksnumbered]{hyperref}
\usepackage{enumerate}
\usepackage{mathrsfs}
\usepackage{blindtext}
\usepackage{scrextend}
\usepackage{enumitem}
\usepackage{bbm}
\addtokomafont{labelinglabel}{\sffamily}
\usepackage{color}
\usepackage[at]{easylist}

\numberwithin{equation}{section}

\theoremstyle{plain}
\newtheorem{main}{Theorem}
\newtheorem{mcor}[main]{Corollary}

\newtheorem{theorem}{Theorem}[section]
\newtheorem{claim}[theorem]{Claim}
\newtheorem{lemma}[theorem]{Lemma}
\newtheorem{proposition}[theorem]{Proposition}
\newtheorem{corollary}[theorem]{Corollary}
\theoremstyle{definition}
\newtheorem{definition}[theorem]{Definition}
\newtheorem*{definition*}{Definition}
\newtheorem{example}[theorem]{Example}
\newtheorem{notation}[theorem]{Notation}
\newtheorem{remark}[theorem]{Remark}
\newtheorem{fact}[theorem]{Fact}

\theoremstyle{plain}

\usepackage[all]{xy}
\usepackage{bbm}

\newcommand{\calp}{\mathcal{P}}

\newcommand{\st}{\mathrm{st}}
\newcommand{\lk}{\mathrm{lk}}

\title[Rigidity for graph product von Neumann algebras]{Rigidity  for graph product von Neumann algebras}
\author{\ Camille Horbez and Adrian Ioana}
\date{\today}

\begin{document}

\maketitle

\begin{abstract}
We establish rigidity theorems for graph product von Neumann algebras $M_\Gamma=*_{v,\Gamma}M_v$ associated to finite simple graphs $\Gamma$ and families of tracial von Neumann algebras $(M_v)_{v\in\Gamma}$. We consider the following three broad classes of vertex algebras: diffuse, diffuse amenable, and II$_1$ factors. In each of these three regimes, we exhibit a large class of graphs $\Gamma,\Lambda$ for which the following holds: any isomorphism $\theta$ between $M_\Gamma$ and $N_\Lambda$ ensures the existence of a graph isomorphism $\alpha:\Gamma\to\Lambda$, and tight relations between $\theta(M_v)$ and $N_{\alpha(v)}$ for every vertex $v\in\Gamma$, ranging from strong intertwining in both directions (in the sense of Popa), to unitary conjugacy in some cases.  

Our results lead to a wide range of applications to the classification of graph product von Neumann algebras and the calculation of their symmetry groups. First, we obtain general classification theorems for von Neumann algebras of right-angled Artin groups and of graph products of ICC groups. We also provide a new family of II$_1$ factors with trivial fundamental group, including all graph products of II$_1$ factors over graphs with girth at least $5$ and no vertices of degree $0$ or $1$. Finally, we compute the outer automorphism group of certain graph products of II$_1$ factors.
\end{abstract}

\setcounter{tocdepth}{1}
\tableofcontents

\section{Introduction and statement of the main results}

\subsection{Introduction}
Von Neumann algebras arise naturally from a variety of dynamical, group theoretic, geometric, and combinatorial data via several constructions, including crossed products (e.g.,  group and group measure space von Neumann algebras), tensor products, free products, and more generally graph products, which encompass the latter two. Von Neumann algebras often forget the data from which they are constructed. This lack-of-rigidity phenomenon is best illustrated by Connes' work \cite{Co76}  which implies that all II$_1$ factors arising from amenable data via the crossed product construction are isomorphic.

Over the past two decades, Popa's deformation/rigidity theory has led to major advances in the classification of nonamenable von Neumann algebras, see the surveys \cite{Po07,Va10,Io18}. 
Much of this progress has been focused on proving rigidity results  which show that the initial data can sometimes be reconstructed from the isomorphism class of the von Neumann algebras. These results are usually accompanied by computations of the symmetries, e.g., fundamental groups or outer automorphism groups, of the von Neumann algebras considered.

The goal of this article is to broaden the scope of rigidity for von Neumann algebras to cover those which arise as graph products.  
 Originally defined in the context of groups by Green in \cite{Gr90}, the graph product construction was extended to the setting of von Neumann algebras in \cite{Ml04,SW16,CF17}. At the group-theoretic level, given a finite simple graph $\Gamma$ and a group $G_v$ for every vertex $v\in\Gamma$, the \emph{graph product} $*_{v,\Gamma}G_v$ is the group obtained from the free product of the groups $G_v$ by adding as only extra relations that $G_v$ and $G_w$ commute whenever $v$ and $w$ are adjacent. The operator algebraic notion is similar:
 consider a finite simple graph $\Gamma$ and a von Neumann algebra $M_v$, for every vertex $v\in\Gamma$. Then the  {\it graph product von Neumann algebra} associated to this data, 
 denoted $*_{v,\Gamma}M_v$ or simply $M_\Gamma$, is generated by $(M_v)_{v\in\Gamma}$ in such a way any two  ``vertex"  algebras $M_v$ and $M_w$ are in a tensor product or free product position, depending on whether $v$ and $w$ are connected or not; see  Section~\ref{graph_vN} for more background.
 If  $(M_v)_{v\in\Gamma}$ is a family of tracial von Neumann algebras, as we assume in this article, then so is their graph product $M_\Gamma$. 
 Graph products are a common generalization of tensor products and free products: $M_\Gamma$ is  equal to $\overline{\otimes}_{v\in\Gamma}M_v$, if $\Gamma$ is a complete graph, and to $*_{v\in\Gamma}M_v$, if $\Gamma$ has no edges. 
  The graph product construction  
   for von Neumann algebras is compatible with that for groups: given a graph product group $G_\Gamma=*_{v,\Gamma}G_v$, we have an identification of group von Neumann algebras $\text{L}(G_\Gamma)=*_{v,\Gamma}\text{L}(G_v)$.

The structure and rigidity of graph products has a long history in group theory. Graph products generalize two emblematic classes of groups: right-angled Coxeter groups $\mathcal W_\Gamma$ (when all vertex groups are $\mathbb{Z}/2\mathbb{Z}$) and right-angled Artin groups $A_\Gamma$ (when all vertex groups are $\mathbb{Z}$). Geometric group theorists have extensively exploited their actions on $\mathrm{CAT}(0)$ cube complexes and negatively curved spaces \cite{Dav98} in order to derive algebraic information, such as the computation of their automorphism groups, 
see e.g.\ \cite{GM19,Gen24}. Classification and rigidity theorems for right-angled Artin groups from the viewpoint of quasi-isometry, including \cite{BKS08,Hua17,Hua18}, have later inspired similar classification results in measure equivalence, first for right-angled Artin groups \cite{HH22,HHI23,HH23,HH24}, then in the general context of graph products \cite{EH24}. 
Such measure equivalence classification results together with the uniqueness of Cartan subalgebras \cite{CKE24} allow to classify II$_1$ factors associated to probability measure preserving actions of graph products. 

   In recent years, following the 
 work of Caspers and Fima \cite{CF17}, there has been a surge of interest in the structural and rigidity properties of graph product von Neumann algebras. There is now a growing literature on the topic which includes  \cite{CdSS18,Ca20,CdSHJKEN23,Bo24,BC24,BCC24,CdSHJKEN24,CKE24, DKE24a,DKE24b,Oy24,Bo25,CC25,CDD25a,CDD25b,KEP25,DV25}. The unifying theme of these works is understanding to what extent properties of the graph $\Gamma$ and the vertex algebras $(M_v)_{v\in\Gamma}$ 
are retained by the
graph product algebra $M_\Gamma$. Ideally, one would like to show that, under minimal assumptions on the graph and the vertex algebras, $M_\Gamma$ remembers the isomorphism classes of both $\Gamma$ and $(M_v)_{v\in\Gamma}$. Several rigidity results in this direction have been obtained recently in \cite{BCC24,Bo25,CC25,CDD25a,CDD25b,DV25}. These results require that the vertex algebras $(M_v)_{v\in\Gamma}$ belong to specific classes of II$_1$ factors. As such, despite this progress, the rigidity problem for graph products associated to arbitrary vertex algebras remained wide open.

\subsection{Rigidity results}\label{rigidity_results} We make progress on this problem here by establishing rigidity theorems for graph products of arbitrary tracial von Neumann algebras, with vertex algebras in one of the following three classes: diffuse, diffuse amenable, and II$_1$ factors (see Theorems \ref{arbitrary}-\ref{factors'}).  
In novel fashion, these results cover non-factorial vertex algebras (Theorems \ref{arbitrary} and \ref{amenable}) and give significantly stronger conclusions than previously known in the case of II$_1$ factor vertex algebras (Theorems \ref{factors} and \ref{factors'}). 
A key finding is that the graph product von Neumann algebras associated to many graphs and factorial vertex algebras are more rigid than both tensor products and free products, see the comments after Theorem \ref{factors'}.

Our general goal is to obtain results of the following form, for each of the three aforementioned classes of vertex algebras: under suitable assumptions on the graphs $\Gamma,\Lambda$, an isomorphism $M_\Gamma\cong N_\Lambda$ forces the existence of a graph isomorphism $\alpha:\Gamma\to\Lambda$, and of a tight relationship (that will be specified) between the vertex subalgebras $M_v$ and $N_{\alpha(v)}$. In general, some basic assumptions on the graphs are required, as a von Neumann algebra can sometimes split as a graph product in two different ways. A basic example where this happens is if $M$ splits as 
a graph product over $\Gamma$, in such a way that the vertex algebra $M_v$ associated to some vertex $v\in\Gamma$ is itself a graph product over a graph $\Upsilon_v$. By ``replacing'' $v$ by $\Upsilon_v$ (i.e., connecting every vertex $w$ that is adjacent to $v$ in $\Gamma$, to every vertex in $\Upsilon_v$), we get a new graph $\Lambda$ such that $M$ splits as a graph product over $\Lambda$. In order to reach a canonical decomposition as a graph product, it is therefore natural to consider minimal such decompositions. To this end, we will often need to make assumptions that ensure that every connected component of $\Gamma$ is \emph{strongly reduced} in the sense of \cite[Definition~6.1]{EH24}, meaning that no proper subgraph can be collapsed to a vertex to give a new graph product decomposition (see Section~\ref{graph_notions} for the formal definition).

In order to state our results, we first need to recall some graph-theoretic 
terminology. 
For a finite simple graph $\Gamma$, we denote still by $\Gamma$ the set of its vertices. We say that $\Gamma$ contains a square (as a full subgraph) 
if it contains a  cycle  $v_1,v_2,v_3,v_4$ such that $v_1$ is not adjacent to $v_3$ and $v_2$ is not adjacent to $v_4$.
The {\it link} of a vertex $v\in\Gamma$, denoted by $\text{lk}(v)$, is the set of all vertices in $\Gamma$ that are adjacent to $v$. The {\it star} of $v$ is given by $\text{st}(v)=\{v\}\cup\text{lk}(v)$.
We say that  $\Gamma$ is {\it transvection-free}  if there are no distinct vertices $v,v'\in\Gamma$ such that $\text{lk}(v)\subset\text{st}(v')$. This terminology comes from the fact that if $\text{lk}(v)\subset\text{st}(v')$, for $v\not=v'$, then the right-angled Artin group $A_\Gamma$ admits an infinite-order automorphism, called a {\it transvection}, which sends $v$ to $vv'$ (after identifying $v,v'$ with generators of the associated vertex groups, isomorphic to $\mathbb {Z}$, of $A_\Gamma$), and is the identity on all other vertices.

\subsection*{Arbitrary vertex algebras}
The following is our first main result. 
Given two finite simple graphs $\Gamma$ and $\Lambda$, and families of tracial von Neumann algebras $(M_v,\tau_v)_{v\in\Gamma}$ and $(N_w,\tau_w)_{w\in\Lambda}$, we denote by
$M_\Gamma=*_{v,\Gamma}M_v$ and $N_\Lambda=*_{w,\Lambda}N_w$ the associated  graph product von Neumann algebras.

\begin{main}\label{arbitrary}
Let $\Gamma$, $\Lambda$ be  two  finite simple graphs which are transvection-free, do not contain a square,  and are not reduced to a vertex.
Let $(M_v,\tau_v)_{v\in\Gamma}$ and $(N_w,\tau_w)_{w\in\Lambda}$ be families of diffuse tracial von Neumann algebras.

If $\theta:M_\Gamma\rightarrow N_\Lambda$ is any $*$-isomorphism, then the graphs $\Gamma$ and $\Lambda$ are isomorphic and there exists a graph isomorphism $\alpha:\Gamma\rightarrow\Lambda$ such that $\theta(M_v)\prec^s_{N_{\Lambda}}N_{\alpha(v)}$ and $N_{\alpha(v)}\prec_{N_\Lambda}^s\theta(M_v)$, for every $v\in\Gamma$.
\end{main}

Here, we write $P\prec^s_M Q$, for von Neumann subalgebras $P,Q$ of a tracial von Neumann algebra ($M,\tau)$, to mean that a corner of $Pp'$ embeds into $Q$ inside $M$ in the sense of Popa \cite{Po06b}, for every nonzero projection $p'\in P'\cap M$ (see Theorem \ref{intertwine}, and Remark~\ref{strong_intertwining} below for some consequences).

Theorem \ref{arbitrary} covers a large class of graphs, which includes all finite simple graphs $\Gamma$ and $\Lambda$ having girth at least $5$ (i.e., contain no triangles or squares)
 and no vertices of valence $0$ or $1$ (i.e., contain no isolated vertices or leaves). Indeed, these assumptions imply that $\Gamma$ and $\Lambda$ are transvection-free, see Remark \ref{basic_graph_facts}(1),  in addition to containing no squares and not being reduced to a vertex.
 In particular, Theorem \ref{arbitrary} applies when $\Gamma=C_m$ and $\Lambda=C_n$ are the cycle graphs on $m,n\geq 5$ vertices.

We continue with three remarks on the statement of Theorem \ref{arbitrary}.

\begin{remark}\label{strong_intertwining}
Consider the following condition from Theorem \ref{arbitrary}:
 ($\star$) $\theta(M_v)\prec^s_{N_{\Lambda}}N_{\alpha(v)}$ and $N_{\alpha(v)}\prec_{N_\Lambda}^s\theta(M_v)$, for every $v\in\Gamma$. This condition implies
 that $M_v$ has certain properties (e.g., is amenable or of type I) if and only if $N_{\alpha(v)}$ does.
 In fact, much more is true, and $M_v$
 and $N_{\alpha(v)}$ are strongly related, as follows: there are nonzero projections $p_v\in M_v,q_v\in N_{\alpha(v)}$ and a unital embedding $p_vM_vp_v\subset q_vN_{\alpha(v)}q_v$ which has finite index 
 (see Lemma \ref{interwining_to_isomorphism}(2)).
     
     If all the vertex algebras in Theorem \ref{arbitrary} are II$_1$ factors, then $(\star)$ implies 
     that $M_v,N_{\alpha(v)}$ are stably isomorphic, for every $v\in\Gamma$ (cf. \cite{BCC24} 
     or Lemma~\ref{interwining_to_isomorphism}(1)). 
     Recall that two II$_1$ factors $M, N$ are {\it stably isomorphic} if $M$ is isomorphic to $N^t$, for some $t>0$. 
     Here,  the {\it amplification} $N^t$ of $N$ is defined as the isomorphism class of $p(N\overline{\otimes}\mathbb B(\ell^2))p$, for a projection $p\in N\overline{\otimes}\mathbb B(\ell^2)$ with $(\tau\otimes\text{Tr})(p)=t$, where $\tau$ and $\text{Tr}$ are the canonical traces of $N$ and $\mathbb B(\ell^2)$, respectively.
     As we will see in Theorem~\ref{factors'}, under an additional assumption on the graphs, we will prove that $M_v,N_{\alpha(v)}$ are in fact isomorphic, for every $v\in\Gamma$. 

\end{remark}

\begin{remark}\label{obstruct}
We discuss the necessity of the assumptions made on the graphs $\Gamma,\Lambda$ in Theorem~\ref{arbitrary}.
\begin{enumerate}
\item The assumptions on $\Gamma,\Lambda$ are sufficient to ensure that each connected component of $\Gamma,\Lambda$ is strongly reduced and not reduced to one vertex (see Remark~\ref{basic_graph_facts}(2)). Therefore, they cannot be dropped from the statement. For a simple example, let $C_4=\{v_1,v_2,v_3,v_4\}$ be a square with its vertices listed in order. Since $M_{C_4}$ is equal to $(M_{v_1}*M_{v_3})\overline{\otimes}(M_{v_2}*M_{v_4})$, it can be realized as a graph product algebra over the segment $C_2$ with vertex algebras  $M_{v_1}*M_{v_3}$ and $M_{v_2}*M_{v_4}$. We do not know, however, if the square-freeness assumption can be weakened to only assuming that every connected component of $\Gamma,\Lambda$ is strongly reduced; see Theorem~\ref{factors} below where we solve this question when all vertex algebras are II$_1$ factors.
\item We also emphasize the necessity of the transvection-freeness assumption in order to achieve the entire conclusion of the theorem (that is, not merely a graph isomorphism $\Gamma\cong\Lambda$, but one which is compatible with $\theta$ in the precise sense given by the statement). Indeed, suppose $v,v'\in\Gamma$ are distinct vertices with $\text{lk}(v)\subset\text{st}(v')$. If $M_v=M_{v'}=\text{L}(\mathbb Z)$, then there is an automorphism $\theta$ of $M_\Gamma$ satisfying $\theta(M_v)\nprec_{M_\Gamma}M_w$, for every  $w\in\Gamma$. For this fact and more  general conditions implying the existence of such an automorphism $\theta$, see Corollary \ref{obstruction2}. 
\end{enumerate}
\end{remark}

\begin{remark}
It is open to what extent Theorem \ref{arbitrary} holds if the vertex algebras are not diffuse, and notably if they are finite dimensional. 
To illustrate this point, in the context of Theorem \ref{arbitrary}, assume that $\Gamma=C_m$, $\Lambda=C_n$, for $m,n\geq 5$, and $M_v=N_w=\text{L}(\mathbb Z/2\mathbb Z)$, for every $v\in\Gamma,w\in\Lambda$. Then $M_\Gamma=\text{L}(\mathcal W_m)$ and $N_\Lambda=\text{L}(\mathcal W_n)$, where $\mathcal W_m,\mathcal W_n$ are the right-angled Coxeter groups with defining graphs $C_m,C_n$. Moreover, $\mathcal W_m$ and $\mathcal W_n$ are cocompact Fuchsian groups. 
The classification of von Neumann algebras of Fuchsian groups is a longstanding open problem, see \cite{dHV95}. 
\end{remark}
 
\subsection*{Amenable vertex algebras}
When the vertex algebras are further assumed to be amenable, we can remove the square-freeness assumption from Theorem~\ref{arbitrary} and obtain the following statement.

\begin{main}\label{amenable}
Let $\Gamma$ and $\Lambda$ be  two  finite simple graphs which are transvection-free. Let $(M_v,\tau_v)_{v\in\Gamma}$ and $(N_w,\tau_w)_{w\in\Lambda}$ be families of amenable diffuse  tracial von Neumann algebras. 

If $\theta:M_\Gamma\rightarrow N_\Lambda$ is any $*$-isomorphism, then the graphs $\Gamma$ and $\Lambda$ are isomorphic and there exists a graph isomorphism $\alpha:\Gamma\rightarrow\Lambda$ such that $\theta(M_v)\prec^s_{N_{\Lambda}}N_{\alpha(v)}$ and $N_{\alpha(v)}\prec_{N_\Lambda}^s\theta(M_v)$, for every $v\in\Gamma$.
\end{main}

Theorem \ref{amenable} in particular applies when all the vertex algebras are separable abelian and diffuse, i.e., isomorphic to $\text{L}(\mathbb Z)$. This case of Theorem \ref{amenable} leads to a classification result for von Neumann algebras of right-angled Artin groups, see Corollary~\ref{RAAGs}.  

Theorem \ref{amenable} also provides a partial classification of graph product II$_1$ factors, where all the vertex algebras are the hyperfinite II$_1$ factor. We state this result below as Corollary \ref{hyperfinite}, and compare it with a similar result obtained recently and independently in \cite{CC25}.  

Note that Theorem~\ref{amenable} does not require the graphs the $\Gamma,\Lambda$ to be strongly reduced, but imposes a strong assumption on the vertex algebras instead. The transvection-freeness assumption cannot be dropped from the statement, as the conclusion is false without this assumption by Remark \ref{obstruct}(3).

\begin{remark}
    We note that most finite graphs are transvection-free, in the following precise sense, and thus are covered by Theorem \ref{amenable}. Consider the Erd\"{o}s-R\'{e}nyi $G(n,p)$ model of random graphs on $n$ vertices, where each edge is included with probability $p=p(n)\in (0,1)$ independently of all other edges. If $p\in (0,1)$ is a constant independent of $n$, then almost surely as $n\rightarrow\infty$, any graph $\Gamma\in G(n,p)$ is transvection-free \cite{CF12}. More generally, it was shown in \cite{Da12} that the same holds  if $p(n)$ essentially belongs to the range $(q(n),1-q(n))$, where $q(n)=n^{-1}(\log(n)+\log(\log(n)))$.
\end{remark}

\subsection*{II$_1$ factor vertex algebras}
Our next goal is to strengthen Theorem \ref{arbitrary} when the vertex algebras are II$_1$ factors, as follows.

\begin{main}\label{factors}
Let $\Gamma,\Lambda$  be two finite simple graphs. 
Assume that every connected component of $\Gamma$ and $\Lambda$ is strongly reduced, transvection-free, and not reduced to a vertex. Let $(M_v)_{v\in\Gamma}$ and $(N_w)_{w\in\Lambda}$ be families of II$_1$ factors, and let $t>0$. 

 If $\theta:M_\Gamma\rightarrow N_\Lambda^t$ is any $*$-isomorphism, then the graphs $\Gamma$ and $\Lambda$ are isomorphic and there exists a graph isomorphism $\alpha:\Gamma\rightarrow\Lambda$ such that $\theta(M_v)\prec^s_{N_{\Lambda}^t}N_{\alpha(v)}^t$ and $N_{\alpha(v)}^t\prec_{N_\Lambda^t}^s\theta(M_v)$, for every $v\in\Gamma$.
 Moreover, the II$_1$ factors $M_v$ and $N_{\alpha(v)}$ are stably isomorphic, for every $v\in\Gamma$.
\end{main}
As explained in Remark~\ref{obstruct}(1), Theorem~\ref{factors} covers a larger class of graphs than Theorem~\ref{arbitrary}. Note that if every connected component of a graph $\Gamma$ is strongly reduced, then the only collapsible subgraphs of $\Gamma$ are unions of its connected components.

\begin{remark}
It is worth comparing Theorem \ref{factors} to some related results  obtained recently by Borst, Caspers and Chen in \cite{BCC24} and by Drimbe and Vaes in \cite{DV25}. 

Specifically, \cite[Theorem A]{BCC24} established the same conclusion as Theorem \ref{factors} when $t=1$, the graphs $\Gamma,\Lambda$ are \emph{rigid} (i.e., there are no distinct vertices $v,v'\in\Gamma$ with $\lk(v)\subset\lk(v')$), and the vertex II$_1$ factors $(M_v)_{v\in\Gamma},(N_w)_{w\in\Lambda}$ are nonamenable with strong property (AO) (see \cite[Definition 6.4]{BCC24}). 
The article \cite{DV25} studies a new equivalence notion, of being W$^*$-correlated, for II$_1$ factors. In \cite[Theorem 6.1]{DV25}, the authors prove a rigidity statement for W$^*$-correlated graph products $M_\Gamma, N_\Lambda$ when $\Gamma,\Lambda$ are rigid (possibly infinite) graphs and the vertex II$_1$ factors $(M_v)_{v\in\Gamma},(N_w)_{w\in\Lambda}$ are nonamenable and stably solid in the sense of \cite[Definition 3.1]{DV25}. 
 When $M_\Gamma, N_\Lambda$ are moreover isomorphic, a sharper result is obtained in \cite[Theorem C]{DV25}, which shows that the conclusion of Theorem \ref{factors} holds in this case. This generalizes \cite[Theorem A]{BCC24} to infinite graphs, for slightly different classes of II$_1$ factors (see the discussion after \cite[Theorem C]{DV25}); subsequently, in an updated version of \cite{BCC24}, Theorem A therein was also shown to hold for infinite graphs. 

Since transvection-free graphs are rigid, \cite[Theorem~A]{BCC24} and \cite[Theorem C]{DV25} apply to larger classes of graphs than Theorem \ref{factors}. Moreover, these results cover infinite graphs, which we do not treat in this work. We note that Theorem \ref{factors} is false for various rigid graphs, including (finite or infinite) complete graphs. As such, we have to require additional conditions on the graphs (see Remark \ref{assumptions Thm factors}).
On the other hand, unlike \cite{BCC24,DV25} which only cover certain  nonamenable vertex II$_1$ factors, we impose no restrictions on the vertex II$_1$ factors. Indeed, a central finding in the present paper is that, for large classes of graphs, rigidity holds with no assumption on the vertex II$_1$ factors. 

\end{remark}

\begin{remark}\label{assumptions Thm factors}
    The assumptions in Theorem \ref{factors} cannot be dropped.
\begin{enumerate}  
\item As explained at the beginning of Section~\ref{rigidity_results}, the assumption that the graphs are strongly reduced is needed.
\item The assumption that the graphs are transvection-free is also needed. In the context of Theorem \ref{factors}, assume that there exist distinct vertices $v,v'\in\Gamma$ such that $\text{lk}(v)\subset\text{lk}(v')$. Suppose additionally that $M_v$ is a free product of diffuse tracial von Neumann algebras (e.g., if $M_v=\text{L}(\mathbb F_n)$, for $n\geq 2$). 
Then there is an automorphism $\theta$ of $M_\Gamma$ such that $\theta(M_v)\nprec_{M_\Gamma}M_{w}$, for every $w\in\Gamma$, see Corollary \ref{obstruction2}, which shows that the conclusion of Theorem \ref{factors} fails in this case. 
\end{enumerate}
\end{remark}

Theorem \ref{factors} in particular applies to all finite simple graphs with girth at least $5$ and no vertices of valence $0$ or $1$. 
For graphs in this class, we improve considerably the conclusion of Theorem \ref{factors}. In the following statement, saying that $\Gamma$ contains \emph{no separating star} means that for every $v\in\Gamma$, the full subgraph spanned by the vertices in $\Gamma\setminus\st(v)$ is connected. 
This assumption is natural: when $\Gamma$ has a separating star, every graph product group over $\Gamma$ has \emph{partial conjugations}, i.e.\ automorphisms given by conjugating every vertex group from one connected component of $\Gamma\setminus\st(v)$ by an element in the vertex group associated to $v$.

\begin{main}\label{factors'}
Let $\Gamma,\Lambda$ be two finite simple graphs of girth at least $5$, which contain no vertices of valence $0$ or $1$. Let $(M_v)_{v\in\Gamma}$ and $(N_w)_{w\in\Lambda}$ be families of II$_1$ factors. Let $t>0$, and let $\theta:M_\Gamma\to N_\Lambda^t$ be a $*$-isomorphism. 

Then $t=1$, and there exist a graph isomorphism $\alpha:\Gamma\rightarrow\Lambda$ and   unitaries $(u_v)_{v\in\Gamma}\subset N_\Lambda$ such that $\theta(M_v)=u_vN_{\alpha(v)}u_v^*$, for every $v\in\Gamma$.

If moreover $\Gamma$ contains no separating star, then one can choose $u_v$ to be independent of $v\in\Gamma$.
\end{main}

Traditionally, the most studied classes of graph product von Neumann algebras have been those associated to complete graphs (i.e., tensor products) or edgeless graphs (i.e., free products). Starting with the pioneering works of Ozawa and Popa \cite{OP04} and Ozawa \cite{Oz06},  various classes of II$_1$ factors have been shown to admit unique decompositions as tensor products of prime II$_1$ factors (see, e.g., \cite{OP04,Pe09,CS13,HI17,CdSS18,DHI19,Is19,IM22}) or, respectively, as free products of II$_1$ factors (see, e.g., \cite{Oz06,IPP08,Pe09,CH10,Io15,HU16,Dr23,DKE24a}). However, these results always require strong assumptions on the II$_1$ factors involved.  
Again, Theorem~\ref{factors'} exhibits a stronger rigidity phenomenon for graph products over a large family of graphs, in that no assumption on the factorial vertex algebras is required. We note that a counterpart of this phenomenon for measure equivalence has been established recently in \cite{EH24}, 
which exhibits a stronger rigidity phenomenon for graph products in measure equivalence than those for free products \cite{AG12} or direct products \cite{MS06,Sa09}.

The unique prime factorization results mentioned above show that for certain  II$_1$ factors and any $*$-isomorphism  $\theta:\overline{\otimes}_{i=1}^mM_i\rightarrow N:=\overline{\otimes}_{j=1}^nN_j$, there exist a bijection $\alpha:\{1,\dots,m\}\rightarrow\{1,\dots,n\}$, a unitary $u\in N$   and $t_1,\dots,t_m>0$ such that  $t_1\dots t_m=1$ and $\theta(M_i)=uN_{\alpha(i)}^{t_i}u^*$,  for every $1\leq i\leq m$. In turn,  unique free product factorization results give instances  when any $*$-isomorphism $\theta:*_{i=1}^mM_i\rightarrow N:=*_{j=1}^nN_j$ forces the existence of a bijection $\alpha:\{1,\dots,m\}\rightarrow\{1,\dots,n\}$ and unitaries $(u_i)_{i=1}^m\subset N$ such that $\theta(M_i)=u_iN_{\alpha(i)}u_i^*$, for every $1\leq i\leq m$. In general, these results cannot be improved to show that $t_i=1$, for every $1\leq i\leq m$, and, respectively, that $u_i$ can be chosen independently of $1\leq i\leq m$. In contrast, the moreover assertion of Theorem \ref{factors'} shows that such an enhanced rigidity statement does hold for graph products over many graphs. This is a key fact, which allows to compute symmetry groups of certain graph product II$_1$ factors, see Corollary~\ref{symmetry_groups}.

\subsection{Applications} We now present several applications of our main results.

\subsection*{Classification results}

We begin with some classification results, starting with two consequences of Theorem \ref{amenable}.
Recall that the right-angled Artin group $A_\Gamma$ is the graph product over $\Gamma$ where all vertex groups are isomorphic to $\mathbb{Z}$. Recall also that two discrete groups $G$ and $H$ are said to be {\it W$^*$-equivalent} if they have isomorphic group von Neumann algebras, $\text{L}(G)\cong\text{L}(H)$.

\begin{mcor}\label{RAAGs}
Let $\Gamma$ and $\Lambda$ be  two finite simple graphs which are transvection-free. 

Then the right-angled Artin groups $A_\Gamma$ and $A_\Lambda$ are W$^*$-equivalent if and only if they are isomorphic.
\end{mcor}

In turn, this is equivalent to the graphs $\Gamma,\Lambda$ being isomorphic \cite{Dr87}. Corollary \ref{RAAGs} provides a von Neumann algebraic analogue of the main theorem of \cite{HH22} (see also \cite{HH24}) for measure equivalence. Specifically, it was shown in \cite[Theorem 1]{HH22} that if $\Gamma,\Lambda$
are transvection-free finite simple graphs which additionally have no separating stars 
(an additional assumption that cannot be dropped in the context of measure equivalence), then $A_\Gamma$ and $A_\Lambda$ are measure equivalent if and only if they are isomorphic. 

\begin{remark}
It is open to what extent Corollary~\ref{RAAGs} holds without the transvection-freeness assumption. Some assumptions are needed to account for the following examples of pairs of W$^*$-equivalent right-angled Artin groups, and the famous (still unsolved) free group factor isomorphism problem: 
\begin{enumerate}
   
    \item If $\Gamma=K_m$ and $\Lambda=K_n$ are complete graphs on $m,n\geq 1$ vertices, then $A_\Gamma\cong\mathbb Z^m$ and $A_\Lambda\cong\mathbb Z^n$ are W$^*$-equivalent. 
    \item If $\Gamma=\overline{K}_m$ and $\Lambda=\overline{K}_n$ are edgeless graphs on $m,n\geq 2$ vertices, then $A_\Gamma\cong\mathbb F_m$ 
     and $A_\Lambda\cong\mathbb F_n$  are W$^*$-equivalent if and only if the free group factors $\text{L}(\mathbb F_m)$ and $\text{L}(\mathbb F_n)$ are isomorphic.
    \item If 
    $\Gamma=K_{m,m'}$ is a complete bipartite graph, for some $m,m'\geq 2$, then $A_\Gamma\cong \mathbb F_m\times\mathbb F_{m'}$, and therefore $\text{L}(A_\Gamma)\cong\text{L}(\mathbb F_m)\overline{\otimes}\text{L}(\mathbb F_{m'})$. 
      On the other hand, as shown in \cite{Ra94}, the isomorphism class of $\text{L}(\mathbb F_m)\overline{\otimes}\text{L}(\mathbb F_{m'})$ only depends on $(m-1)(m'-1)$. Consequently, if $\Lambda=K_{n,n'}$, for some
      $n,n'\geq 2$ satisfying $(m-1)(m'-1)=(n-1)(n'-1)$, then $A_\Gamma$ and $A_\Lambda$ are W$^*$-equivalent. In particular $A_{K_{3,3}}$ and $A_{K_{2,5}}$ are W$^*$-equivalent (cf.\ \cite{CC25}).
\end{enumerate}

Considering these facts, the best one can hope for in Corollary \ref{RAAGs} is to relax the transvection-freeness assumption by assuming instead that the graphs have no collapsible subgraphs of at least $2$ vertices which are either complete or edgeless. While this remains open,
 we note that our techniques still enable us to cover a few cases beyond the transvection-free case (see Example~\ref{ex:raag}).
\end{remark}

Our next corollary deals with the case when all the vertex algebras are the hyperfinite II$_1$ factor.

\begin{mcor}\label{hyperfinite}
    Let $\Gamma$ and $\Lambda$ be two finite simple graphs which are transvection-free. For every $v\in\Gamma$ and $w\in\Lambda$, let $M_v=N_w=R$, where $R$ denotes the hyperfinite II$_1$ factor.

    Then $M_\Gamma$ and $N_\Lambda$ are isomorphic if and only if the graphs $\Gamma$ and $\Lambda$ are isomorphic.
\end{mcor}

If $\Gamma$ and $\Lambda$ are allowed to have transvections, 
but no collapsible clique on at least two vertices, and if $M_\Gamma\cong N_\Lambda$, then we can show that the subgraphs of $\Gamma$ and $\Lambda$ spanned by their untransvectable vertices are isomorphic, see Corollary~\ref{hyperfinite2}. This sharpens the main theorem of Caspers and Chen \cite{CC25} when the graphs $\Gamma$ and $\Lambda$ are finite, and removes their assumption of H-rigidity of the graphs. In fact, Corollary~\ref{hyperfinite2} highlights the \emph{untransvectable subgraph} as a new $W^*$-invariant that extends the notion of an \emph{internal subgraph} from \cite{CC25}. On the other hand, we note that the results of \cite{CC25} also apply to some infinite graphs, while ours do not. 

We say that two discrete ICC groups $G$ and $H$ are {\it stably W$^*$-equivalent} if their group II$_1$ factors are stably isomorphic, $\text{L}(G)\cong\text{L}(H)^t$, for some $t>0$. Recall that a group $G$ is  called \text{ICC} (standing for \emph{infinite conjugacy classes}) if the conjugacy class of every non-trivial element of $G$ is infinite. A non-trivial discrete group $G$ is ICC if and only if $\text{L}(G)$ is a II$_1$ factor.

Theorem \ref{factors'} leads to the following broad W$^*$-equivalence classification result for graph products.

\begin{mcor}\label{ICC}
Let $\Gamma,\Lambda$ be two finite simple graphs of girth at least $5$, which contain no vertices of valence $0$ or $1$. Let $(G_v)_{v\in\Gamma}$ and $(H_w)_{w\in\Lambda}$ be families of ICC groups. Let  $G_\Gamma=*_{v,\Gamma}G_v$ and $H_\Lambda=*_{w,\Lambda}H_w$ be the associated graph product groups.
Then the following conditions are equivalent:
\begin{enumerate}
    \item  $G_\Gamma$ and $H_\Lambda$  are W$^*$-equivalent.
    \item  $G_\Gamma$ and $H_\Lambda$  are stably W$^*$-equivalent.
    \item There exists a graph isomorphism $\alpha:\Gamma\rightarrow\Lambda$ such that $G_v$ is W$^*$-equivalent to $H_{\alpha(v)}$, for every $v\in\Gamma$.
\end{enumerate} 
\end{mcor}

This result echoes the main theorem of \cite{EH24}, where an analogous statement was proved in the setting of measure equivalence.

\subsection*{Calculations of symmetry groups} 

Let $M$ be a II$_1$ factor. Then $\mathcal F(M)=\{t>0\mid M^t\cong M\}$ is a multiplicative subgroup of $\mathbb R_+^*$, called the {\it fundamental group} of $M$ \cite{MvN43}.
We denote by $\text{Out}(M)=\text{Aut}(M)/\text{Inn}(M)$ the {\it outer automorphism group} of $M$, where $\text{Aut}(M)$  and $\text{Inn}(M)$ are the group of all and, respectively, inner automorphisms of $M$ (i.e.\ given by the conjugation by a unitary). We denote by $\varepsilon:\text{Aut}(M)\rightarrow\text{Out}(M)$ the quotient homomorphism.

In the group-theoretic setting, there has been a lot of work around the description of the automorphism group of a graph product, first for right-angled Artin groups \cite{Se89,La95}, then in a more general setting, see most notably \cite{GM19,Gen24}. In particular, in \cite[Corollary~C]{GM19}, Genevois and Martin give a complete description of the automorphism group of any graph product over a graph of girth at least $5$ with no separating star. We will now provide a version of their result in the context of graph products of II$_1$ factors.

Assume that $M=M_\Gamma$ is the graph product von Neumann algebra associated to a finite simple graph $\Gamma$ and a family of tracial von Neumann algebras $(M_v,\tau_v)_{v\in\Gamma}$. 
Then $\mathrm{Aut}(M_\Gamma)$ has a natural subgroup consisting of local automorphisms, that we now describe.
Denote by $\text{Aut}(\Gamma)$ the group of automorphisms of $\Gamma$, i.e., of all bijections $\sigma$ of $\Gamma$ such that both $\sigma$ and $\sigma^{-1}$ preserve adjacency of vertices.
We denote by $\text{Aut}(\Gamma; M_\Gamma)$ the subgroup of $\sigma\in\text{Aut}(\Gamma)$ such that $M_v\cong M_{\sigma(v)}$, for every $v\in\Gamma$. 

Following \cite{GM19,CDD25a}, an automorphism $\theta\in\text{Aut}(M_\Gamma)$ is called  {\it local} if there exists $\sigma\in\text{Aut}(\Gamma)$ such that $\theta(M_v)=M_{\sigma(v)}$, for every $v\in\Gamma$.
The set of local automorphisms of $M_\Gamma$ forms a subgroup of $\text{Aut}(M_\Gamma)$ which we denote by $\text{Aut}_{\text{loc}}(M_\Gamma)$.  
The map $\theta\mapsto\sigma$ then yields a natural onto homomorphism $\mathrm{Aut}_{\mathrm{loc}}(M_\Gamma)\to\mathrm{Aut}(\Gamma;M_\Gamma)$, whose kernel is isomorphic to $\oplus_{v\in\Gamma}\mathrm{Aut}(M_v)$. After fixing a vertex $v$ in each $\mathrm{Aut}(\Gamma;M_\Gamma)$-orbit, and an automorphism $\theta_{wv}:M_w\to M_v$ for every $w\in \mathrm{Aut}(\Gamma;M_\Gamma)\cdot v$ (with $\theta_{vv}=\mathrm{id}$), we get a section $\mathrm{Aut}(\Gamma;M_\Gamma)\to\mathrm{Aut}_{\mathrm{loc}}(M_\Gamma)$. Altogether this shows that 

\begin{equation}\label{aut_loc}\mathrm{Aut}_{\mathrm{loc}}(M_\Gamma)\cong\left(\bigoplus_{v\in\Gamma}\mathrm{Aut}(M_v)\right)\rtimes\mathrm{Aut}(\Gamma;M_\Gamma).\end{equation}

Theorem \ref{factors'} allows us to give the first calculations of symmetry groups of graph product von Neumann algebras with arbitrary II$_1$ factor vertex algebras.
Specifically, applying Theorem \ref{factors'} to $N_\Lambda=M_\Gamma$ gives that $\mathcal F(M_\Gamma)=\{1\}$ and moreover that $\text{Aut}(M_\Gamma)=\text{Aut}_{\text{loc}}(M_\Gamma)\text{Inn}(M_\Gamma)$, if $\Gamma$ has no separating stars. By combining these facts with the isomorphism in \eqref{aut_loc}, we derive the following corollary 
(see Section~\ref{applications} for the proof).

\begin{mcor}\label{symmetry_groups}

Let $\Gamma$ be a finite simple graph of girth at least $5$, which contains no vertices of valence $0$ or $1$.   
Let $(M_v)_{v\in\Gamma}$ be a family of II$_1$ factors.
Then $\mathcal F(M_\Gamma)=\{1\}$.

Moreover, if $\Gamma$ has no separating stars, 
then  $$\emph{Out}(M_\Gamma)\cong \left(\bigoplus_{v\in\Gamma}\emph{Aut}(M_v)\right)\rtimes\emph{Aut}(\Gamma;M_\Gamma).$$

\end{mcor}

The first class of II$_1$ factors with trivial fundamental group was given by Popa in his breakthrough work \cite{Po06a}. Subsequently, several additional such classes were found, including notably all crossed products  associated to free ergodic probability measure preserving actions of $\mathbb F_n, 2\leq n<\infty$ \cite{PV14}.  
Starting with \cite{IPP08}, the outer automorphism group has been calculated for a number of families of II$_1$ factors, see \cite{PV22,CIOS23,Va24} and the references therein for some more recent advances. 

Corollary \ref{symmetry_groups} provides a new family of II$_1$ factors with trivial fundamental group which have a remarkably simple algebraic description. In particular, Corollary \ref{symmetry_groups} implies that if one places arbitrary II$_1$ factors $M_1,\dots,M_n$, for any $n\geq 5$, on the vertices of the cycle graph $C_n$ on $n$ vertices, then the associated graph product II$_1$ factor has trivial fundamental group. 

Corollary \ref{symmetry_groups} also shows that for graphs $\Gamma$ as in its statement, the outer automorphism group of $M_\Gamma$ is completely determined by the automorphism groups of the vertex II$_1$ factors $(M_v)_{v\in\Gamma}$. This represents an analogue of a group-theoretic result of Genevois and Martin \cite[Corollary C]{GM19}. 

Prior to Corollary \ref{symmetry_groups}, the only known calculations of symmetry groups of graph product II$_1$ factors associated to non-complete graphs  were given by Chifan, Davis and Drimbe in \cite{CDD25a}. Specifically, the conclusion of Corollary~\ref{symmetry_groups} was shown to hold for specific graphs $\Gamma$  (called simple cycles of cliques) and certain  II$_1$ factors $(M_v)_{v\in\Gamma}$ arising from property (T) groups and satisfying unique prime factorization.
Corollary \ref{symmetry_groups} vastly extends this phenomenon to a fairly large class of graphs and arbitrary vertex II$_1$ factors.

 \begin{remark}
Assume that the vertex II$_1$ factors $(M_v)_{v\in\Gamma}$ are equal to a fixed II$_1$ factor $M$. Then Corollary \ref{symmetry_groups} gives that $\text{Out}(M_\Gamma)$ is isomorphic to the generalized wreath product group $\text{Aut}(M)\wr_{\Gamma}\text{Aut}(\Gamma)=(\bigoplus_{\Gamma}\text{Aut}(M))\rtimes\text{Aut}(\Gamma)$. 
If $\text{Aut}(M)$ is connected with respect to pointwise $\|\cdot\|_2$-convergence (e.g., if $\text{Out}(M)=\{1\}$), then the connected component, denoted $\text{Out}(M_\Gamma)^0$, of the identity in $\text{Out}(M_\Gamma)$ is equal to $\bigoplus_\Gamma \text{Aut}(M)$. This implies that the automorphism group $\text{Aut}(\Gamma)$ of $\Gamma$ can be recovered from $M_\Gamma$ as the quotient group $\text{Out}(M_\Gamma)/\text{Out}(M_\Gamma)^0$.
\end{remark}

\subsection*{Unique prime factorization}
A II$_1$ factor $M$ is called {\it prime} if it cannot be written as a tensor product of two II$_1$ factors. Assume that $M$ is not prime, but can be decomposed as a tensor product of prime II$_1$ factors, $M=M_1\overline{\otimes}\cdots\overline{\otimes}M_n$. We say that $M$ has a {\it unique prime factorization} if any tensor product decomposition $M=P_1\overline{\otimes}P_2$ is, modulo amplifications, unitary conjugacy and permutations of the indices, 
of the form $P_1=\overline{\otimes}_{i=1}^mM_i$ and $P_2=\overline{\otimes}_{i=m+1}^nM_i$, for some $1\leq m\leq n$. 

A finite simple graph $\Gamma$ is called \emph{irreducible} if it does not split as the join $\Gamma=\Gamma_1\circ\Gamma_2$ of two proper subgraphs, meaning that any two vertices $v_1\in\Gamma_1$ and $v_2\in\Gamma_2$ are adjacent.  
Every finite simple graph $\Gamma$ decomposes in a unique way (up to permuting the factors) as a join $\Gamma=\Gamma_1\circ\dots\circ\Gamma_n$ of irreducible subgraphs, called the irreducible components of $\Gamma$. Given a family of vertex II$_1$ factors $(M_v)_{v\in\Gamma}$, 
such a join decomposition induces a tensor product decomposition $M_\Gamma=M_{\Gamma_1}\overline\otimes\dots\overline\otimes M_{\Gamma_n}$.

As shown in \cite[Theorem B]{BCC24}, if $|\Gamma|\geq 2$  and $\Gamma$ is irreducible, then $M_\Gamma$ is a prime II$_1$ factor. Therefore, if every irreducible component $\Gamma_i$ of $\Gamma$  satisfies $|\Gamma_i|\geq 2$, 
then $M_\Gamma$ decomposes as a tensor product of prime II$_1$ factors, 
$M_\Gamma=M_{\Gamma_1}\overline\otimes\dots\overline\otimes M_{\Gamma_n}$.
By \cite[Theorem D]{CdSS18}, this prime factorization is unique, provided that $(M_v)_{v\in\Gamma}$ are group II$_1$ factors. The same conclusion, under some additional conditions, was shown to hold in \cite[Theorem C]{BCC24} whenever $(M_v)_{v\in\Gamma}$ are nonamenable and satisfy strong property (AO).
As a byproduct of our techniques (and as a crucial ingredient in the proof of our main theorems), we establish uniqueness of this prime factorization for arbitrary vertex II$_1$ factors.

\begin{main}\label{prime_factorization}
Let $\Gamma$ be a finite simple graph such that $|\Gamma|\geq 2$ and no vertex of $\Gamma$ is adjacent to every other vertex. Let $(M_v,\tau_v)_{v\in\Gamma}$ be a collection of diffuse tracial von Neumann algebras and $M_\Gamma=*_{v,\Gamma}M_v$ the associated graph product II$_1$ factor.  Let  $\Gamma_1,\dots,\Gamma_n$ be the irreducible components of $\Gamma$. 

Then $M_{\Gamma_i}$ is a prime II$_1$ factor, for every $1\leq i\leq n$, and the following hold:

\begin{enumerate}
\item If $M_\Gamma=P_1\overline{\otimes}P_2$, for some II$_1$ factors $P_1, P_2$, then there exist a partition $I_1\sqcup I_2=\{1,\dots,n\}$ and a decomposition $M_\Gamma=P_1^t\overline{\otimes}P_2^{1/t}$, for some $t>0$, such that $P_1^t=\overline{\otimes}_{i\in I_1}M_{\Gamma_i}$ and $P_2^{1/t}=\overline{\otimes}_{i\in I_2}M_{\Gamma_i}$, modulo conjugacy by a unitary element in $M_\Gamma$. 
\item If $M_\Gamma=P_1\overline{\otimes}\dots\overline{\otimes}P_m$, for some $m\geq n$ and II$_1$ factors $P_1,\dots,P_m$, then $m=n$ and there exists a decomposition $M_\Gamma=P_1^{t_1}\overline{\otimes}\dots\overline{\otimes}P_n^{t_n}$ for some $t_1,\dots,t_n>0$ with $t_1\dots t_n=1$ such that, after permutation of indices and unitary conjugacy, 
$M_{\Gamma_i}=P_i^{t_i}$, for every $1\leq i\leq n$.
\item In (2), the assumption that $m\geq n$ can be omitted if $P_i$ is prime, for every $1\leq i\leq n$. 
\end{enumerate}
\end{main}

\subsection{Comments on the proofs of the main results.}
We now describe, briefly and informally, some of the techniques and ideas used in the proofs of our main theorems.

\subsubsection*{Parabolic hulls of subalgebras}

Zimmer's cocycle superrigidity theorem \cite{Zi84} for higher-rank lattices crucially exploits the notion of the \emph{algebraic hull} of a measurable cocycle, which one can informally think of  
as the smallest algebraic subgroup into which the cocycle can be conjugated. Parabolic subgroups of graph product groups, i.e., subgroups associated to subgraphs of the underlying graph, play a similarly important role, first highlighted by Tits in the context of Coxeter and Artin groups, and studied specifically for graph products in \cite{AM}. For measure equivalence, the notion of parabolic support of a measured groupoid equipped with a cocycle taking values in a graph product is a central tool in \cite{HH22,EH24}. 
While finding an appropriate von Neumann algebraic analogue for the algebraic hull in the setting of higher-rank lattices remains a challenging task, we establish here a systematic framework, which further develops ideas from \cite{BCC24} 
(see also \cite{BC24,CKE24,CDD25a}), for studying parabolic hulls of subalgebras of graph product von Neumann algebras.

Consider a graph product von Neumann algebra $M_\Gamma$ associated to a finite simple graph $\Gamma$ and a collection of diffuse tracial von Neumann algebras $(M_v)_{v\in\Gamma}$.
For any full subgraph $\Sigma\subset\Gamma$, 
we have an embedding $M_\Sigma\subset M_\Gamma$. Moreover, the von Neumann algebra $\mathcal N_{M_\Gamma}(M_\Sigma)''$ generated by the normalizer of $M_{\Sigma}$ in $M_\Gamma$ is equal to $M_{\Sigma\cup\Sigma^\perp}$, 
where $\Sigma^{\perp}\subset\Gamma$ is the subgraph spanned by all vertices of $\Gamma\setminus\Sigma$ that are adjacent to all vertices of $\Sigma$. Any von Neumann subalgebra $Q\subset M_\Gamma$ whose relative commutant $Q'\cap M_\Gamma$ is a factor has a \emph{parabolic hull}, i.e., there is a smallest subgraph $\Sigma\subset\Gamma$ such that $Q\prec^s_{M_\Gamma}M_{\Sigma}$.  
The hull behaves well with respect to normalizers: if $M_\Sigma$ is the hull of $Q$, then  $\mathcal{N}_{M_\Gamma}(Q)''$ 
has a corner that embeds in $M_{\Sigma\cup\Sigma^{\perp}}$. 
In general, when $Q'\cap M_\Gamma$ is not a factor, there is a finite partition $\{p_i\}_{i\in I}$ of the identity in $Q'\cap M_\Gamma$ such that each $Qp_i$ has a hull.
These facts and several other (see Section~\ref{intertwining_results}) are crucial in order to implement the following general strategy. 

\subsubsection*{General strategy for proving Theorems~\ref{arbitrary}, \ref{amenable} and ~\ref{factors}} Let $P$ be a von Neumann algebra which splits into two ways as a graph product, $P=M_\Gamma=N_\Lambda$.
For von Neumann subalgebras $Q_1,Q_2\subset P$, we write for simplicity $Q_1\sim Q_2$ to indicate that $Q_1\prec_P^sQ_2$ and $Q_2\prec_P^sQ_1$.
Our ultimate goal is to prove that for every vertex $v\in\Gamma$, there exists a vertex $w\in\Lambda$ such that $M_v\sim N_w$ -- once this is done, a simple combinatorial argument ensures that the map $v\mapsto w$ yields a graph isomorphism. For certain well-chosen subgraphs $\Sigma\subset\Gamma$, we will find a subgraph $\Delta\subset\Lambda$ such that $M_\Sigma\sim N_\Delta$ until we narrow down to the case where $\Sigma$ is a vertex. 

The graph product structure of $P$ induces many amalgamated free product decompositions, e.g., $P=M_{\Gamma\setminus\{v\}}*_{M_{\text{lk}(v)}}M_{\text{st}(v)}$, for every $v\in\Gamma$. This fact allows us to use \cite{Io15,Va13} to derive strong information on commuting subalgebras:  commuting nonamenable subalgebras of $P$ are essentially located in a subalgebra associated to a join subgraph or an isolated vertex of $\Gamma$. This result, proved in Section~\ref{sec:commuting}, 
enables us to achieve the above goal when $\Sigma\subset\Gamma$ is a maximal join subgraph. More precisely, if all vertex algebras are factors (which entails that $M_\Sigma$ has a hull in $P=N_\Lambda$), 
it follows that $M_\Sigma\sim N_\Delta$, for some maximal join subgraph $\Delta\subset\Lambda$.  In general, 
for non-factorial vertex algebras (when $M_\Sigma$ does not always have a hull), the situation is necessarily more complicated.
More precisely, we find a partition of identity $\{p_{\Sigma,\Delta}\}_{\Delta}$ in the center of $M_\Sigma'\cap P$ indexed by the maximal join subgraphs $\Delta\subset\Lambda$ such that $M_\Sigma p_{\Sigma,\Delta}\sim N_\Delta q_{\Delta,\Sigma}$, for some projection $q_{\Delta,\Sigma}\in \mathcal Z(N_\Delta'\cap P)$.

In the context of Theorem~\ref{arbitrary} (proved in Section~\ref{sec:arbitrary}), the graphs $\Gamma,\Lambda$ are assumed transvection-free and square-free, so their maximal join subgraphs are exactly the stars of vertices. Starting with any vertex $v\in\Gamma$, we find a partition of identity $\{p_{v,w}\}_w$ in $\mathcal Z(M_{\text{st}(v)}'\cap P)$ indexed by vertices $w\in\Lambda$ such that $M_{\st(v)}p_{v,w}\sim N_{\st(w)}q_{w,v}$, for some projection $q_{w,v}\in \mathcal Z(N_{\text{st}(w)}'\cap P)$.
 Using the tensor product decomposition $M_{\st(v)}=M_v\overline\otimes M_{\lk(v)}$ we then deduce that $M_vp_{v,w}\sim N_wq_{w,v}$, for every $v\in\Gamma, w\in\Lambda$. 
 At this point, a general conjugacy criterion established in Section \ref{sec:tools} implies that for every $v\in\Gamma$, there is $w\in\Lambda$ such that $p_{v,w}=q_{w,v}=1$ and therefore $M_v\sim N_w$, as desired.

The ideas above are further exploited in the proofs of Theorems~\ref{amenable} and~\ref{factors}, where we take advantage of the additional assumptions imposed on the vertex algebras in order to make the conditions on the graphs less restrictive. For instance, in Theorem~\ref{amenable}, the graph $\Gamma$ is allowed to contain squares, so it is no longer true that all of its maximal join subgraphs $\Sigma\subset\Gamma$ are stars. But since the vertex algebras are assumed amenable, we can now distinguish subalgebras associated to stars from other subalgebras associated to maximal join subgraphs, using the fact that the former contain an amenable factor in their tensor product decomposition, while the latter do not. Further,
when $\Sigma$ is a maximal join subgraph which is not a star,
by using our unique prime factorization result (Theorem \ref{prime_factorization})
 we can locate the tensor factors of $M_\Sigma$ (which arise from the irreducible components of $\Sigma$). These ideas are exploited in Section~\ref{sec:amenable} via an induction argument on the cardinality of $\Gamma$ to establish Theorem~\ref{amenable}.
A variation of this strategy is then used in Section~\ref{sec:factors} to prove Theorem~\ref{factors}.

Our strategy (locating subalgebras of maximal join subgraphs $\Sigma$ of $\Gamma$, their factors, and finally vertices) parallels the one used for measure equivalence in \cite{EH24}. However, carrying the strategy in the context of von Neumann algebras is a completely different task. 
Indeed, as explained above, we work in the framework of Popa's intertwining-by-bimodules theory. Moreover, for 
 general diffuse (not necessarily factorial) vertex algebras,
 the relative commutant $M_\Sigma'\cap P$ is not always a factor and 
  we therefore have to analyze projections in its center.  Handling this issue is one of the main technical novelties of the paper, 
  and is precisely what made it possible to obtain the first rigidity results in the non-factorial regime.

\subsubsection*{A word on  the proof of Theorem~\ref{factors'}} 

The passage from Theorem~\ref{factors} to Theorem~\ref{factors'}, i.e., from a stable to an actual isomorphism of the vertex von Neumann algebras, is particularly delicate and requires additional ideas. 
An analogous passage from stable  
to usual orbit equivalence of the vertex groups was performed in \cite{EH24}. However, the proof therein exploits the action of the graph product on its right-angled building, for which there is no von Neumann algebraic analogue. 
In order to discuss the proof of Theorem \ref{factors'} in more detail, assume for simplicity that we are in the above setting. Theorem \ref{factors} then implies the existence of a graph isomorphism $\alpha:\Gamma\rightarrow\Lambda$ such that $M_v$ and $N_{\alpha(v)}$ are stably isomorphic, for every $v\in\Gamma$. Specifically, for every $v\in\Gamma$, there are a unitary $u_v\in P$ and a decomposition $M_{\text{st}(v)}=M_v^{t_v}\overline{\otimes}M_{\text{lk}(v)}^{1/t_v}$, for some $t_v>0$,  such that $u_vM_v^{t_v}u_v^*=N_{\alpha(v)}$.
By using this relation and  that $N_{\alpha(v)}$ and $N_{\alpha(v')}$ commute, we obtain that $u_v^*u_{v'}$ has a  special decomposition for adjacent vertices $v,v'\in\Gamma$. 
This decomposition and the identity $(u_{v_0}^*u_{v_1})(u_{v_1}^*u_{v_2})\cdots (u_{v_{n-1}}^*u_{v_n})=1$ along cycles $v_0,v_1,\dots,v_n=v_0$  in $\Gamma$
are then combined to show that $t_v=1$, and thus derive the main assertion of Theorem \ref{factors'}. 
Assuming moreover that $\Gamma$ has no separating stars, additional work allows us to prove that $u_v$ decomposes as $u_v=u\beta_v\gamma_v$, for a unitary $u\in P$ independent of $v$, and unitaries $\beta_v\in M_v,\gamma_v\in M_{\text{lk}(v)}$. This shows that $uM_vu^*=u_vM_vu_v^*=N_{\alpha(v)}$ and finishes the proof.

\subsection*{Acknowledgments.} 
The first-named author would like to thank Cyril Houdayer for useful discussions in early stages of this project. The authors are also grateful to Martijn Caspers and Stefaan Vaes for helpful comments.

The first-named author was funded by the European Union (ERC, Artin-Out-ME-OA, 101040507). Views and opinions expressed are however those of the authors only and do not necessarily reflect those of the European Union or the European Research Council. Neither the European Union nor the granting authority can be held responsible for them.
The second-named author was supported in part by the NSF grants FRG-DMS-1854074,  DMS-2153805 and DMS-2451697.
\section{Preliminaries}

\subsection{Tracial von Neumann algebras}\label{basic_facts}  We recall some basic terminology regarding tracial von Neumann algebras, and refer to \cite{AP} for additional information.
A {\it tracial von Neumann algebra} is a pair $(M,\tau)$ consisting of a von Neumann algebra $M$ and a faithful normal tracial state $\tau:M\rightarrow\mathbb C$.

Let $(M,\tau)$ be a tracial von Neumann algebra.
For $x\in M$, we denote by $\|x\|$ the {\it operator norm} of $x$ and by $\|x\|_2:=\sqrt{\tau(x^*x)}$ the {\it $2$-norm} of $x$. We denote by $\mathcal Z(M)=\{x\in M\mid xy=yx,\forall y\in M\}$ the {\it center} of $M$, by $\mathcal U(M)$ the {\it unitary group} of $M$ and by $(M)_1=\{x\in M\mid\;\|x\|\leq 1\}$ the {\it unit ball} of $M$. We denote by $\text{L}^2(M)$ the Hilbert space obtained by completing $M$ with respect to $\|\cdot\|_2$, or, equivalently, the scalar product $\langle x,y\rangle=\tau(y^*x)$. We consider the standard  embeddings $M\subset\mathbb B(\text{L}^2(M))$ and $M\subset\text{L}^2(M)$, and the anti-unitary $J:\text{L}^2(M)\rightarrow\text{L}^2(M)$ given by $J(x)=x^*$. 
We say that $M$ is {\it diffuse} if it admits no nonzero minimal projections.

Let $P\subset M$ be a von Neumann subalgebra, which we will always assume to be unital. We denote by $\text{E}_P:M\rightarrow P$ the {\it conditional expectation} onto $P$, by $P'\cap M$ 
the {\it relative commutant} 
of $P$ in $M$, and by $\mathcal N_M(P)=\{u\in\mathcal U(M)\mid uPu^*=P\}$ the {\it normalizer} of $P$ in $M$. The {\it quasi-normalizer} of $P$ in $M$, denoted $\mathcal Q\mathcal N_M(P)$, is the set of $x\in M$ for which there exist $x_1,\dots,x_n\in M$, for some $n\geq 1$, such that $xP\subset \sum_{i=1}^nPx_i$ and $Px\subset \sum_{i=1}^nx_iP$. We clearly have that $\mathcal N_M(P)\subset \mathcal Q\mathcal N_M(P)$.
For a subset $S\subset M$, we denote by $\text{sp}(S)$ the linear span of $S$ and by $\overline{\text{sp}}(S)$ the $\|\cdot\|_2$-closure of $\text{sp}(S)$.

Given $S\subset \mathbb B(\mathcal H)$, for a Hilbert space $\mathcal H$,  we denote by $S'=\{x\in\mathbb B(\mathcal H)\mid xy=yx,\forall y\in Y\}$ the {\it commutant} of $S$. If $1\in S$ and $S$ is closed under adjoint, then $S''=(S')'$ is the von Neumann algebra generated by $S$ by von Neumann's bicommutant theorem. If $P,Q\subset \mathbb B(\mathcal H)$ are von Neumann subalgebras, we denote by $P\vee Q:=(P\cup Q)''$ the von Neumann algebra generated by $P$ and $Q$. For subsets $S_1,\dots,S_n\subset\mathbb B(\mathcal H)$ we denote $S_1+\dots+S_n=\{x_1+\dots+x_n\mid x_1\in S_1,\dots,x_n\in S_n\}$ and $S_1\cdots S_n=\{x_1\cdots x_n\mid x_1\in S_1,\dots,x_n\in S_n\}$. Hereafter, all $*$-homomorphisms $\theta:P\rightarrow Q$ between von Neumann algebras $P$ and $Q$ are assumed unital and normal.

Let $(N,\tau)$ be a tracial von Neumann algebra. The \emph{opposite} von Neumann algebra $N^{\text{op}}$ is $N$ as a vector space, with the same involution and trace, but multiplication defined by $x\cdot y:=yx$. An $M$-$N$-{\it bimodule} is a Hilbert space $\mathcal H$ endowed with normal $*$-homomorphisms $\pi: M\rightarrow \mathbb B(\mathcal H)$ and $\rho:N^{\text{op}}\rightarrow \mathbb B(\mathcal H)$ such that $\rho(N^{\text{op}})\subset\pi(M)'$. For $x\in M,y\in N$ and $\xi\in\mathcal H$, we let $x\xi y:=\pi(x)\rho(y)\xi$. 

For a von Neumann subalgebra $Q\subset M$, we denote by $e_Q:\text{L}^2(M)\rightarrow\text{L}^2(Q)$ the orthogonal projection. {\it Jones' basic construction} $\langle M,e_Q\rangle$ is  defined as $JQ'J=(M\cup\{e_Q\})''\subset\mathbb B(\text{L}^2(M))$. It admits a faithful semi-finite trace given by $\text{Tr}(xe_Q y)=\tau(xy)$, for every $x,y\in M$. We denote by L$^2(\langle M,e_Q\rangle)$ the associated Hilbert space endowed with the natural $M$-bimodule structure. The inclusion $Q\subset M$ has {\it finite index} if $\text{Tr}(1)<\infty$. Note that this notion depends on the choice of $\tau$. 

Let $v\in M$, $\mathcal H\subset \text{L}^2(M)$ the $\|\cdot\|_2$-closure of $vQ$ and $e$ the orthogonal projection onto $\mathcal H$. Since $\mathcal H$ is a right $Q$-module, $e\in\mathbb B(\text{L}^2(M))$ commutes with $JQJ$ and thus belongs to $\langle M,e_Q\rangle$. Moreover, let $s\in Q$ be the support projection of $\text{E}_Q(v^*v)$. Consider the map $T:s\text{L}^2(Q)\rightarrow\mathcal H$ given by $T(\xi)=v\xi$. Then $T$ is a right $Q$-modular, injective, bounded operator with dense image. 
This implies that the right $Q$-modules $s\text{L}^2(Q)$ and $\mathcal H$ are isomorphic (with an isomorphism given by the partial isometry in the polar decomposition of $T$).
Hence, $\text{Tr}(e)=\dim\mathcal H_Q=\dim s\text{L}^2(Q)_Q=\tau(s)\leq 1$.

A tracial von Neumann algebra $(M,\tau)$ is called {\it amenable} if there exists a net $\xi_n\in \text{L}^2(M)\otimes \text{L}^2(M)$ such that $\langle x\xi_n,\xi_n\rangle\rightarrow\tau(x)$ and $\|x\xi_n-\xi_nx\|_2\rightarrow 0$, for every $x\in M$. Let $P\subset pMp$, where $p\in M$ is a projection, and $Q,R\subset M$ be von Neumann subalgebras. Following Ozawa and Popa \cite[Section~2.2]{OP10} we say that $P$ is  {\it amenable relative to $Q$ inside $M$} if there exists a net $\xi_n\in p\text{L}^2(\langle M,e_Q\rangle)p$ such that $\langle x\xi_n,\xi_n\rangle\rightarrow\tau(x)$, for every $x\in pM p$ and $\|y\xi_n-\xi_ny\|_2\rightarrow 0$, for every $y\in P$. If $P\subset Q$, then $P$ is clearly amenable relative to $Q$ inside $M$; see Proposition \ref{results}(4) for a strengthening of this fact. We recall that by \cite[Proposition 2.4(3)]{OP10} if $P$ is amenable relative to $Q$ inside $M$ and $Q$ is amenable relative to $R$ inside $M$, then $P$ is amenable relative to $R$ inside $M$. 
We also note that if $P$ is amenable relative to $Q$ inside $M$, then so is $P\oplus\mathbb C(1-p)$. 

\subsection{Popa's intertwining-by-bimodules theory}

We next recall from  \cite [Theorem 2.1 and Corollary 2.3]{Po06b} Popa's influential {\it intertwining-by-bimodules} theory.
\begin{theorem}[\!\!\cite{Po06b}]
\label{intertwine} Let $(M,\tau)$ be a tracial von Neumann algebra and $P\subset p_0M p_0, Q\subset q_0Mq_0$ be von Neumann subalgebras. Let $\mathcal U\subset\mathcal U(P)$ be a subgroup such that $\mathcal U''=P$.
Then the following conditions are equivalent:

\begin{enumerate}

\item There exist projections $p\in P, q\in Q$, a nonzero partial isometry $v\in M$ with $v^*v\leq p$ and $vv^*\leq q$, and a  
$*$-homomorphism $\theta:pP p\rightarrow qQ q$   such that $\theta(x)v=vx$, for every $x\in pP p$.

\item There is no net $(u_n)\subset \mathcal U$ satisfying $\|{\rm  E}_Q(x^*u_ny)\|_2\rightarrow 0$, for every $x,y\in p_0Mq_0$.

\end{enumerate}

\end{theorem}
If the equivalent conditions (1)-(2) in Theorem \ref{intertwine} hold true,  we write $P\prec_{M}Q$.
Moreover, if $P p'\prec_{M}Q$ for every nonzero projection $p'\in P'\cap p_0Mp_0$, we write $P\prec^s_{M}Q$.

\begin{remark}\label{elementary_facts}
We record here several elementary facts which will be used throughout the paper and whose proofs are left to the reader.  In the context of Theorem \ref{intertwine}, let $p\in P,p'\in P'\cap p_0Mp_0,q\in Q$ and $q'\in Q'\cap q_0Mq_0$ be nonzero projections. Let $(\widetilde M,\widetilde\tau)$ be a tracial von Neumann algebra which contains $M$ such that ${\widetilde\tau}_{|M}=\tau$. Then the following hold:

\begin{enumerate}
\item If $P\prec_MQ$ and $P\nprec_M Qq'$, then $P\prec_MQ(q_0-q')$.
\item If $pPpp'\prec_M Q$ or $P\prec_M qQqq'$, then $P\prec_MQ$ (see \cite[Lemma 3.4]{Va08}).
\item If $P\prec_M^sQ$, then $pPpp'\prec_M^s Q$.
\item If $P\prec_M^s Q$, then $P\prec_{\widetilde M}^sQ$ (see \cite[Remark 2.2]{DHI19}).
\item\label{central_support_right} If $Pp'\nprec_M Q$, then $Pe\nprec_MQ$, where $e\in\mathcal Z(P'\cap p_0Mp_0)$ is the central support of $p'$.
\item If $P\nprec_M qQq$, then $P\nprec_M Qf$, where $f\in\mathcal Z(Q)$ is the central support of $q$.
\item\label{inherit_to_subalg} If $P\prec_M^s Q$ and $P_0\subset P$ is a von Neumann subalgebra, then $P_0\prec_M^s Q$.  
(To see this, let $r\in P_0'\cap p_0Mp_0$ be a nonzero projection. Then by (3) we have $P_0r\prec_M P$ and since $P\prec_M^sQ$, \cite[Lemma 3.7]{Va08} implies that $P_0r\prec_MQ$.)
\item\label{sum_1} If $q'_1,q'_2\in \mathcal Z(Q'\cap q_0Mq_0)$ are nonzero projections, and if $P\prec_M Q(q_1'\vee q_2')$, then $P\prec_M Qq_1'$ or $P\prec_M Qq_2'$.
\item\label{sum_2} If $p'_1,p'_2\in \mathcal Z(P'\cap p_0Mp_0)$ are nonzero projections, and if $Pp'_1\prec^s_MQ$ and $Pp'_2\prec^s_MQ$, then $P(p'_1\vee p'_2)\prec^s_M Q$.
\end{enumerate}
\end{remark}

\begin{remark}\label{finite_index}
We also record a useful fact connecting intertwining and finite index inclusions.
    Let $(M,\tau)$ be a tracial von Neumann algebra and $P\subset M$ a von Neumann subalgebra with $M\prec_MP$. Then there is a nonzero projection $p'\in P'\cap M$ such that the inclusion $Pp'\subset p'Mp'$ has finite index. Indeed, since $M\prec_M P$,  \cite[Theorem 2.1]{Po06b} gives a nonzero projection $f\in M'\cap \langle M,e_P\rangle$ with $\text{Tr}(f)<\infty$. 
    Since $M'=JMJ$ and $\langle M,e_P\rangle=JP'J$, we get that $M'\cap \langle M,e_P\rangle=JMJ\cap JP'J=J(P'\cap M)J$.
    Writing $f=Jp'J$, for a projection $p'\in P'\cap M$, we get  $\text{dim}(\text{L}^2(M)p')_P=\text{Tr}(Jp'J)<\infty$. 
    After replacing $p'$ with $p'z$,  for a projection $z\in\mathcal Z(P)$,
    we may moreover assume that $\text{L}^2(M)p'$ is a finitely generated right $P$-module    (see \cite[Lemma A.1]{Va07}).
   Therefore, $\text{L}^2(p'Mp')$ is a finitely generated right $Pp'$-module, which by \cite[Lemma A.1]{Va08} gives that  $Pp'\subset p'Mp'$ has finite index.   

  Since $Pp'$ is isomorphic to $Pz$, where $z\in\mathcal Z(P)$ is the support projection of $\text{E}_P(p')$, we get a $*$-homomorphism $\alpha:Pz\rightarrow p'Mp'$ such that the inclusion $\alpha(Pz)\subset p'Mp'$ has finite index.
\end{remark}

For further reference, we recall several useful facts from the literature.

\begin{proposition}\label{results}
Let $(M,\tau)$ be a tracial von Neumann algebra. Let $P\subset p_0Mp_0, Q\subset q_0Mq_0, R\subset rMr$ and $S_1,\dots,S_k \subset q_0Mq_0$ be von Neumann subalgebras. Then the following hold:
\begin{enumerate}
\item If $P\prec_M Q$, then $Q'\cap q_0Mq_0\prec_M P'\cap p_0Mp_0$.
\item  If $P\prec_M Q$ and $Q\prec_M^s R$, then $P\prec_M R$.
\item If $P\prec_M Q$, then there is a nonzero projection $z\in\mathcal N_{p_0Mp_0}(P)'\cap p_0Mp_0$ such that $Pz\prec_M^sQ$.
More precisely, if $Pp'\prec^s_MQ$, for a nonzero projection $p'\in P'\cap p_0Mp_0$, then $Pz\prec_M^sQ$, where $z\in\mathcal N_{p_0Mp_0}(P)'\cap p_0Mp_0$ is the smallest projection such that $p'\leq z$.
\item  If $P\prec_M^sQ$ and $q_0=1$,
then $P$ is amenable relative to $Q$ inside $M$.
\item  Assume that $P\nprec_M S_i$, for every $1\leq i\leq k$. Then there exists a net $(u_n)\subset\mathcal U(P)$ satisfying $\|{\rm  E}_{S_i}(x^*u_ny)\|_2\rightarrow 0$, for every $x,y\in p_0Mq_0$ and $1\leq i\leq k$.
\item Assume that $P\prec_MQ$, $S_i\subset Q$ and $P\nprec_MS_i$, for every $1\leq i\leq k$. Then there exist projections $p\in P, q\in Q$, a nonzero partial isometry $v\in qM p$  and a $*$-homomorphism $\theta:pPp\rightarrow qQq$ such that $\theta(x)v=vx$ and $\theta(pPp)\nprec_QS_i$, for every $x\in pPp$ and $1\leq i\leq k$.
\end{enumerate}
\end{proposition}

\begin{proof}
For (1) and (2), see Lemmas 3.5 and 3.7 from \cite{Va08}, respectively.  For (3) and (4), see Lemmas 2.4(3) and 2.6(3) from \cite{DHI19}, respectively. For (5), see the proof of \cite[Theorem~4.3]{IPP08} and \cite[Remark~2.3]{DHI19}. For (6), see \cite[Remark~3.8]{Va08} and \cite[Lemma~5.4]{BCC24}.
\end{proof}

We will also need the following result, the second part of which is a variation of Proposition \ref{results}(6).
While this result is elementary, and likely known to experts, we provide a proof for completeness.

\begin{proposition}\label{cutting_down}
Let $(M,\tau)$ be a tracial von Neumann algebra. Let $P\subset p_0Mp_0, Q\subset q_0Mq_0$ and $S_1,\dots, S_k\subset Q$ be von Neumann subalgebras. Let $p'\in P'\cap p_0Mp_0$ and $q'\in Q'\cap q_0Mq_0$ be nonzero projections.
Then the following hold:
\begin{enumerate}
\item $Pp'\prec_M Qq'$ if and only if there exist projections $p\in P,q\in Q$, a nonzero partial isometry $v\in qq'Mpp'$ and a $*$-homomorphism $\theta:pPp\rightarrow qQq$ such that $\theta(x)v=vx$, for all $x\in pPp$.
\item Assume that $Pp'\prec_MQq'$ and $Pp'\nprec_M S_i$, for every $1\leq i\leq k$. Then there exist projections $p\in P,q\in Q$, a nonzero partial isometry $v\in qq'Mpp'$ and a $*$-homomorphism $\theta:pPp\rightarrow qQq$ such that $\theta(x)v=vx$ for every $x\in pPp$, and 
 $\theta(pPp)\nprec_QS_i$ for every $1\leq i\leq k$.
\end{enumerate}
\end{proposition}

\begin{proof} 
(1) Assume that there exist projections $p\in P,q\in Q$, a nonzero partial isometry $v\in qq'Mpp'$ and a $*$-homomorphism $\theta:pPp\rightarrow qQq$ such that $\theta(x)v=vx$, for all $x\in pPp$.
Since $p'\in P'$, we have $\text{E}_P(p')\in\mathcal Z(P)$, so the support projection $s$ of $E_P(p')$ also belongs to $\mathcal Z(P)$, 
and it satisfies $p'\le s$. 
Let $p_1=ps$, which belongs to $pPp$. Then the map $\alpha:p_1Pp_1\rightarrow  pPpp'$ given by $\alpha(x)=xp'$ is a $*$-isomorphism. 
Put $q_1=\theta(\alpha^{-1}(pp'))=\theta(p_1)\in qQq$ and let $\Psi:pPpp'\rightarrow q_1Qq_1q'$ be the $*$-homomorphism given by $\Psi(x)=\theta(\alpha^{-1}(x))q'$.
Then for every $y\in pPpp'$ we have $$\Psi(y)v=\theta(\alpha^{-1}(y))q'v=\theta(\alpha^{-1}(y))v=v\alpha^{-1}(y)=v(p'\alpha^{-1}(y))=v(p'y)=vy.$$
In particular, $(q_1q')v=\Psi(pp')v=v(pp')=v$.
Since $v\in  q_1q'Mpp'$, it follows that $Pp'\prec_MQq'$. 

Conversely, assume that $Pp'\prec_MQq'$. Then we can find projections $p\in P,q\in Q$, a nonzero partial isometry $v\in qq'Mpp'$ and $*$-homomorphism $\Psi:pPpp'\rightarrow qQqq'$ such that $\Psi(x)v=vx$, for every $x\in pPpp'$. Let $t\in\mathcal Z(Q)$ be the support projection of $\text{E}_Q(q')$. Let $q_1=qt$, which belongs to $qQq$. As before, the map $\beta:q_1Qq_1\rightarrow qQqq'$ given by $\beta(x)=xq'$ is a $*$-isomorphism. If $x\in q_1Qq_1$, then $xv=xq'v=\beta(x)v$. Thus, $\beta^{-1}(y)v=yv$, for every $y\in qQqq'$.
Define a $*$-homomorphism $\theta:pPp\rightarrow q_1Qq_1$ by letting $\theta(x)=\beta^{-1}(\Psi(xp'))$. Then for every $x\in pPp$, we have $$\theta(x)v=\beta^{-1}(\Psi(xp'))v=\Psi(xp')v=vxp'=vp'x=vx.$$ 
Since $q'\leq t$, we get $q_1q'v=qtq'v=qq'v=v$. Thus, $v\in q_1q'Mpp'$, which finishes the proof.

(2) Since $Pp'\prec_MQq'$, by (1) there exist projections $p\in P,q\in Q$, a nonzero partial isometry $v\in qq'Mpp'$ and a $*$-homomorphism $\theta_0:pPp\rightarrow qQq$ such that $\theta_0(x)v=vx$, for all $x\in pPp$.
Since $vv^*\in \theta_0(pPp)'\cap qMq$, the support projection $r$ of $\text{E}_Q(vv^*)$ belongs to $\theta_0(pPp)'\cap qQq$. Define $\theta:pPp\rightarrow rQr$ by letting $\theta(x)=\theta_0(x)r$. Since $rv=v$, we get $\theta(x)v=vx$, for every $x\in pPp$. Additionally, we have $v\in rq'Mpp'$.

In order to finish the proof, it  suffices to show that if  $S\subset Q$ is a von Neumann subalgebra with $Pp'\nprec_MS$, then $\theta(pPp)\nprec_QS$. We do this by adapting the argument from \cite[Remark~3.8]{Va08}. Assume by contradiction that $\theta(pPp)\prec_QS$. Then there exist projections $e\in pPp, f\in S$, a $*$-homomorphism $\Phi:\theta(ePe)\rightarrow fSf$ and a nonzero partial isometry $w\in fQ\theta(e)$ such that $\Phi(y)w=wy$, for every $y\in \theta(ePe)$. Define $\rho:ePe\rightarrow fSf$ by letting $\rho(x)=\Phi(\theta(x))$. Then $\rho$ is a $*$-homomorphism such that for every $x\in ePe$ we have
\begin{equation}\label{rho}\rho(x)(wv)=\Phi(\theta(x))wv=w\theta(x)v=(wv)x.\end{equation}
In particular, we get that $wve=\rho(e)wv=fwv=wv$. By combining this with the fact that $fw=w$ and $vp'=v$, we deduce that $f(wv)(ep')=f(wve)p'=f(wv)p'=(fw)(vp')=wv$. We claim that $wv\not=0$. Otherwise, if $wv=0$, then $wvv^*=0$, hence $w\text{E}_Q(vv^*)=\text{E}_Q(wvv^*)=0$, which implies that $wr=0$. Since $\theta_0(e)\leq r$, we would get that $w=w\theta(e)=wr\theta(e)=0$, which is false. Let $\xi\in M$ be the partial isometry in the polar decomposition of $wv$. Since $wv\not=0$ and $wv\in fM(ep')$, we also get that $\xi\not=0$ and $\xi\in fM(ep')$. Moreover, \eqref{rho} implies that $\rho(x)\xi=\xi x$, for every $x\in ePe$. Altogether, (1) implies that $Pp'\prec_M S$, which gives the desired contradiction.
\end{proof}

\subsection{Results on intertwiners} In the context of Theorem \ref{intertwine}, we have $P\prec_M Q$ if and only if $Pv\subset \sum_{i=1}^kv_iQ$, for some  $v\in p_0M$, $v_1,\dots,v_k\in M$ and $k\geq 1$ (see \cite[Proposition C.1]{Va07}). 
In this section, we record two results on such ``intertwiners" $v$. 

\begin{proposition}\label{approx}
 Let $(M,\tau)$ be a tracial von Neumann algebra and $P\subset pM p, Q\subset M$ be von Neumann subalgebras. Let $S$ and $S_0$ be the sets of $v\in pM$ for which there exist $v_1,\dots,v_k\in M$, for some $k\geq 1$, such that $Pv\subset \sum_{i=1}^kv_iQ$ and, respectively, $(P)_1v\subset \sum_{i=1}^kv_i(Q)_1$. 
 
Then $S_0$ is $\|\cdot\|_2$-dense in $S$.
\end{proposition}

\begin{proof} This result is extracted from \cite{Po06b,Va07}. We follow closely the proof of \cite[Lemma~D.3]{Va07}. 
We aim to show that  $S\subset S_1$, where $S_1\subset\text{L}^2(M)$ denotes the $\|\cdot\|_2$-closure of $S_0$.

To this end, let $v\in S$, $\varepsilon>0$ and $\mathcal H\subset\text{L}^2(M)$ be the $P$-$Q$-bimodule obtained as the $\|\cdot\|_2$-closure of the linear span of  $PvQ$. Let $v_1,\dots,v_k\in M$ such that  $Pv\subset \sum_{i=1}^kv_iQ$ and denote by $\mathcal H_i\subset\text{L}^2(M)$ the right $Q$-module obtained as the $\|\cdot\|_2$-closure of $v_iQ$, for every $1\leq i\leq k$. 
Let $e$ and $e_i$ be the orthogonal projections from $\text{L}^2(M)$ onto $\mathcal H$ and, respectively, $\mathcal H_i$, for every $1\leq i\leq k$.
Then $e,e_i\in\langle M,e_Q\rangle$ and since  $PvQ\subset \sum_{i=1}^k v_iQ$, we get that $e\leq\vee_{i=1}^ke_i$.  Thus, $\text{Tr}(e)\leq\sum_{i=1}^k\text{Tr}(e_i)$, where $\text{Tr}:\langle M,e_Q\rangle\rightarrow\mathbb C$ is the usual semi-finite trace. Since  $\mathcal H_i$ is isomorphic as a right $Q$-module to $s_i\text{L}^2(Q)$, where $s_i\in Q$ is the support projection of $\text{E}_Q(v_i^*v_i)$, we have $\text{Tr}(e_i)\leq 1$, for every $1\leq i\leq k$ (see Section \ref{basic_facts}). Hence, $\text{Tr}(e)\leq k<\infty$, and
\cite[Lemma A.1]{Va07}  gives a projection $z_\varepsilon\in\mathcal Z(Q)$ with $\tau(z_\varepsilon)\geq 1-\varepsilon$ such that $\mathcal Hz_\varepsilon$ is a finitely generated right $Q$-module. 

The proof of \cite[Lemma D.3]{Va07} implies the existence of a projection $q\in\mathbb M_m(\mathbb C)\otimes Q$, for some  $m\geq 1$, a $*$-homomorphism $\theta:P\rightarrow q(\mathbb M_m(\mathbb C)\otimes Q)q$ and $\xi=(\xi_1,\dots,\xi_m)\in p(\mathbb M_{1,m}(\mathbb C)\otimes\text{L}^2(M))q$ such that $x\xi=\xi\theta(x)$, for every $x\in P$, and $\mathcal Hz_\varepsilon$ is generated by the entries $\xi_1,\dots,\xi_m$ of $\xi$ as a right $Q$-module. Then $|\xi|\in \text{L}^2(q(\mathbb M_m(\mathbb C)\otimes M)q)$ commutes with $\theta(P)$. For $\ell\geq 1$, consider the spectral projection $r_\ell={\bf 1}_{[0,\ell]}(|\xi|)$ of $|\xi|$. Then $r_\ell\in\theta(P)'\cap q(\mathbb M_m(\mathbb C)\otimes Q)q$ and therefore $\xi_\ell=\xi r_\ell\in  p(\mathbb M_{1,m}(\mathbb C)\otimes M)q$ satisfies $x\xi_\ell=\xi_\ell\theta(x)$, for every $x\in P$. If we write $\xi_\ell=(\xi_{\ell,1},\dots,\xi_{\ell,m})$, then  $\xi_{\ell,1},\dots,\xi_{\ell,m}\in S_0$. Since $\|\xi_\ell-\xi\|_2\rightarrow 0$, we get that $\|\xi_{\ell,i}-\xi_i\|_2\rightarrow 0$ and thus $\xi_i\in S_1$, for every $1\leq i\leq m$.

Since $S_1$ is a right $Q$-module, we conclude that $\mathcal Hz_\varepsilon\subset S_1$. Since this holds for every $\varepsilon>0$ and $\tau(z_\varepsilon)\rightarrow 1$, as $\varepsilon\rightarrow 0$, it follows that $\mathcal H\subset S_1$. Hence, $v\in S_1$, as desired.
\end{proof}

\begin{corollary}\label{space_of_intertwiners}
 Let $(M,\tau)$ be a tracial von Neumann algebra and $P\subset pM p, Q\subset M$ be von Neumann subalgebras. Let $S$  be the set of $v\in pM$ such that $Pv\subset \sum_{i=1}^k v_iQ$, for some $k\geq 1$ and $v_1,\dots,v_k\in M$. Assume that $\mathcal H_0\subset pM$ is a vector subspace and $(u_n)\subset\mathcal U(P)$ is a net such that $\mathcal H=\overline{\mathcal H_0}^{\|\cdot\|_2}$ is a $P$-$Q$-sub-bimodule of $\emph{L}^2(M)$ and $\|\emph{E}_Q(a^*u_nb)\|_2\rightarrow 0$, for every $a\in \mathcal H_0$ and $b\in M$.
 
Then $S\subset\emph{L}^2(M)\ominus\mathcal H$.
\end{corollary}

\begin{proof}
By Proposition \ref{approx}, in order to prove the conclusion it suffices to show that if $v\in pM$ satisfies $(P)_1v\subset \sum_{i=1}^k v_i(Q)_1$, for some $k\geq 1$ and $v_1,\dots,v_k\in M$, then $v\in\text{L}^2(M)\ominus\mathcal H$. Let $e:\text{L}^2(M)\rightarrow\mathcal H$ be the orthogonal projection onto $\mathcal H$. Denote $w=e(v)$ and $w_i=e(v_i)$, for $1\leq i\leq k$. Since $\mathcal H$ is a $P$-$Q$-bimodule, $e$ is a $P$-$Q$-bimodular map. Using this fact we get that $(P)_1w\subset \sum_{i=1}^k w_i(Q)_1$.
Hence, for every $n$, we can find $x_{1,n},\dots,x_{k,n}\in (Q)_1$ such that $u_nw=\sum_{i=1}^kw_i x_{i,n}$.
Therefore,  \begin{equation}\label{sum}\text{$\|w\|_2^2=\langle u_nw,u_nw\rangle=\sum_{i=1}^k\langle u_nw,w_ix_{i,n}\rangle$, for every $n$}.\end{equation}
On the other hand, we claim that \begin{equation}\label{ortho} \text{$\langle u_nb,ax_{i,n}\rangle\rightarrow 0$, for every $a\in\mathcal H$, $b\in\text{L}^2(M)$ and $1\leq i\leq k$.}\end{equation}
To this end, we may assume that $a\in \mathcal H_0$ and $b\in M$. If $1\leq i\leq k$, then $x_{i,n}\in (Q)_1$ and thus 
$$|\langle u_nb,ax_{i,n}\rangle|=|\tau(a^*u_nbx_{i,n}^*)|=|\tau(\text{E}_Q(a^*u_nbx_{i,n}^*))|=|\tau(\text{E}_Q(a^*u_nb)x_{i,n}^*)|\leq\|\text{E}_Q(a^*u_nb)\|_2.$$ Since $\|\text{E}_Q(a^*u_nb)\|_2\rightarrow 0$, this  proves \eqref{ortho}.
Finally, combining \eqref{sum} and \eqref{ortho} implies that $w=0$ and thus $v\in\text{L}^2(M)\ominus\mathcal H$.
\end{proof}

\subsection{Results on tensor products}
We continue with two  results on tensor product II$_1$ factors. 

\begin{lemma}[\!\!\cite{OP04}]\label{tensor_product}
Let $M_1,M_2,N_1,N_2$ be II$_1$ factors such that $P=M_1\overline{\otimes}M_2=N_1\overline{\otimes}N_2$ and $M_1\prec_PN_1$. 
Then the following hold:
\begin{enumerate}
    \item  There exist a decomposition $P=M_1^s\overline{\otimes}M_2^{1/s}$, for some $s>0$, and a unitary $v\in P$ such that $vM_1^sv^*\subset N_1$. 
    \item Assume that  $N_1\prec_PM_1$ or $N_1$ is prime. 
Then there exist a decomposition $P=M_1^t\overline{\otimes}M_2^{1/t}$, for some $t>0$, and a unitary $u\in P$ such that $uM_1^tu^*=N_1$ 
and $uM_2^{1/t}u^*=N_2$.
\end{enumerate}

\end{lemma}

\begin{proof}
(1) 
This part follows from \cite[Proposition 12]{OP04}.

(2) By part (1), there exists a decomposition $P=M_1^s\overline{\otimes}M_2^{1/s}$, for some $s>0$, and a unitary $v\in P$ such that $vM_1^sv^*\subset N_1$.
Then 
$N_1=vM_1^sv^*\overline{\otimes}R$, where $R=(vM_1^sv^*)'\cap N_1$ is a factor, 
see, e.g., \cite[Lemma~3.3]{Is20}. If $N_1$ is prime or $N_1\prec_P M_1$, then $R$ is finite dimensional. Thus, $R\cong\mathbb M_\ell(\mathbb C)$, for some $\ell\geq 1$. Letting $t=\ell s$, we get that $uM_1^tu^*=N_1$, for some unitary $u\in P$. 
The fact that $uM_2^{1/t}u^*=N_2$ now follows by taking relative commutants.
\end{proof}

Later on, we will also need the following lemma.

\begin{lemma}\label{conjugacy_in_tensors}
Let $M_1,M_2$ be II$_1$ factors, $p,p'\in M_1,q,q'\in M_2$ nonzero projections and $v\in M_1\overline{\otimes}M_2$  a partial isometry. Assume that $v^*v=p\otimes q, vv^*=p'\otimes q'$ and $v(pM_1p\otimes q)v^*\subset p'M_1p'\otimes q'$.

Then $\tau(p)\leq \tau(p')$. In particular, if $v(pM_1p\otimes q)v^*=p'M_1p'\otimes q'$, then $\tau(p)=\tau(p')$.

Moreover, if $\tau(p)=\tau(p')$, then $v=\omega\otimes \eta$, for some  partial isometries $\omega\in M_1$ and $\eta\in M_2$ such that $\omega^*\omega=p,\omega\omega^*=p',\eta^*\eta=q$ and $\eta\eta^*=q'$.
\end{lemma}

\begin{proof}
Let $\theta:pM_1p\rightarrow p'M_1p'$ be a unital $*$-homomorphism such that $v(x\otimes q)v^*=\theta(x)\otimes q'$, for every $x\in pM_1p$.
Since $v\in (p'\otimes q')(M_1\overline{\otimes}M_2)(p\otimes q)$, we can find $\zeta\in q'M_2q$ such that $w:=\text{E}_{M_1\otimes 1}(v(1\otimes \zeta^*))\in p'M_1p$ is nonzero.
Since $v(x\otimes q)=(\theta(x)\otimes q')v$, we deduce that $wx=\theta(x)w$, for every $x\in pM_1p$.
Then  $w^*w\in \mathcal Z(pM_1p)=\mathbb Cp$.  Let $\alpha>0$ such that $w^*w=\alpha p$. Then $\omega=\alpha^{-\frac{1}{2}}w$ is a partial isometry such that $\omega^*\omega=p$ and $\omega\omega^*\leq p'$. This proves that $\tau(p)\leq\tau (p')$. 

To prove the second assertion, assume that $\tau(p)=\tau(p')$. Then $\tau(q)=\tau(q')$, $\omega\omega^*=p'$ and $\theta(x)=\omega x\omega^*$, for every $x\in pM_1p$. 
Let $\delta\in M_2$ be a partial isometry such that $q=\delta^*\delta$ and $q'=\delta\delta^*$. Then $v(x\otimes q)v^*=\theta(x)\otimes q'=(\omega\otimes\delta)(x\otimes q)(\omega\otimes\delta)^*$, for every $x\in pM_1p.$ Thus, we get that $(\omega\otimes\delta)^*v\in (pM_1p\otimes q)'\cap (pM_1p\overline{\otimes}qM_2q)=p\otimes qM_2q$. 
Let $\rho\in \mathcal U(qM_2q)$ such that $(\omega\otimes\delta)^*v=p\otimes\rho$. Then $v=\omega\otimes\delta\rho$, and letting $\eta=\delta\rho$ proves the second assertion.
\end{proof}

\subsection{Results on free products}
We continue with the following result which will be used in the proof of Theorem \ref{factors'}. 
\begin{lemma}\label{partial_isometry}
Let  $(M,\tau)$ be the free product of two tracial von Neumann algebras $(M_1,\tau_1)$ and $(M_2,\tau_2)$. Let $\zeta\in M$ be a partial isometry such that $\zeta\in\overline{\emph{sp}}(M_1M_2)$, $\zeta\zeta^*\in M_1$ and $\zeta^*\zeta\in M_2$.

Then $\zeta$ is a unitary and $\zeta=\zeta_1\zeta_2$, for some unitaries $\zeta_1\in M_1$ and $\zeta_2\in M_2$.
\end{lemma}

\begin{proof}
Let $(\eta_i)_{i\in I}\subset M_2$ be an orthonormal basis for $\text{L}^2(M_2)$. Since $\zeta\in \overline{\text{sp}}(M_1M_2)$, we have  $\zeta=\sum_{i\in I}\xi_i\eta_i$, where $\xi_i=\text{E}_{M_1}(\zeta\eta_i^*)$,  for every $i\in I$. Thus,  $\zeta^*\zeta=\sum_{i,j\in I}\eta_i^*\xi_i^*\xi_j\eta_j.$

Fix $i,j\in I$. We define $e_{i,j}$ to be the orthogonal projection from $\text{L}^2(M)$ onto the $\|\cdot\|_2$-closure of $\eta_i^*(M_1\ominus\mathbb C)\eta_j$. The free product structure implies that
 $\eta_i^*(M_1\ominus\mathbb C)\eta_j\perp \eta_{i'}^*(M_1\ominus\mathbb C)\eta_{j'}$ if $(i',j')\not=(i,j)$. From this we derive that
  $$\text{$e_{i,j}(\eta_{i'}^*\xi_{i'}^*\xi_{j'}\eta_{j'})=\delta_{((i',j'),(i,j))}\eta_i^*(\xi_i^*\xi_j-\tau(\xi_i^*\xi_j))\eta_j$, for every $i',j'\in I$}.$$
This implies that $e_{i,j}(\zeta^*\zeta)=\eta_i^*(\xi_i^*\xi_j-\tau(\xi_i^*\xi_j))\eta_j$. However, since $\zeta^*\zeta\in M_2$, we have $e_{i,j}(\zeta^*\zeta)=0$. Thus, we get that $\eta_i^*(\xi_i^*\xi_j-\tau(\xi_i^*\xi_j))\eta_j=0$.  Since $\eta_i,\eta_j$ are nonzero, and $\xi_i^*\xi_j-\tau(\xi_i^*\xi_j)\in M_1\ominus \mathbb C1$, it follows that 
\begin{equation}\label{xi_ixi_j}\text{$\xi_i^*\xi_j=\tau(\xi_i^*\xi_j)$, for every $i,j\in I$.}
\end{equation}
Since $\zeta\not=0$, there exists $i_0\in I$ such that $\xi_{i_0}\not=0$. Since $\xi_{i_0}^*\xi_{i_0}=\tau(\xi_{i_0}^*\xi_{i_0})\not=0$, we get that $\zeta_1:=\|\xi_{i_0}\|_2^{-1}\xi_{i_0}\in M_1$ is a unitary. Further, \eqref{xi_ixi_j} implies that for every $i\in I$, we can find $c_i\in\mathbb C$ such that $\xi_i=c_i\zeta_1$. Thus, $\zeta=\sum_{i\in I}\xi_i\eta_i=\sum_{i\in I}c_i\zeta_1\eta_i=\zeta_1\zeta_2,$ where $\zeta_2=\sum_{i\in I}c_i\eta_i\in M_2$. 

Finally, since $\zeta_1\in M_1, \zeta_2\in M_2$, $\zeta\zeta^*\in M_1$ and $\zeta\zeta^*=\zeta_1(\zeta_2\zeta_2^*)\zeta_1^*$, it follows that $\zeta_2\zeta_2^*\in\mathbb C1$.
Since $\zeta$ is a partial isometry, we also get that $\zeta\zeta^*$ and hence $\zeta_2\zeta_2^*$ is a projection. Thus, we derive that $\zeta_2\zeta_2^*=1$, and therefore $\zeta_2$ is a unitary. This finishes the proof of the lemma.
\end{proof}

\begin{corollary}\label{splitting}

    Let $M=(M_1*M_2)\overline{\otimes}M_3$, where $(M_k,\tau_k),1\leq k\leq 3,$ are tracial von Neumann algebras.  Let $u_1,u_2\in\mathcal U(M)$ such that $u_1\in M_1\overline{\otimes}M_3, u_2\in M_2\overline{\otimes}M_3$ and $u_1u_2\in (M_1*M_2)\overline{\otimes}\mathbb C1$. 

    Then there exist unitaries $\zeta_1\in M_1,\zeta_2\in M_2$ and $\zeta_3\in M_3$ such that $u_1=\zeta_1\otimes \zeta_3$ and $u_2=\zeta_2\otimes\zeta_3^*$. In particular, $u_1u_2\in(\mathcal U(M_1)\mathcal U(M_2))\otimes 1$.
\end{corollary}

\begin{proof}
   Since $\text{E}_{(M_1*M_2)\overline{\otimes}\mathbb C1}((M_1\overline{\otimes}M_3)(M_2\overline{\otimes}M_3))\subset\overline{\text{sp}}(M_1M_2)\otimes\mathbb C1$,
   we derive that 
   $$u_1u_2=\text{E}_{(M_1*M_2)\overline{\otimes}\mathbb C1}(u_1u_2)\in \overline{\text{sp}}(M_1M_2)\otimes\mathbb C1.$$
   Therefore, if $u\in\mathcal U(M_1*M_2)$ is such that $u_1u_2=u\otimes 1$, then $u\in\overline{\text{sp}}(M_1M_2)$.
   By applying Lemma \ref{partial_isometry}, we derive that $u=\zeta_1\zeta_2$, for some unitaries $\zeta_1\in M_1$ and $\zeta_2\in M_2$. Hence, we have  $u_1u_2=(\zeta_1\zeta_2)\otimes 1$ and thus $(\zeta_1^*\otimes 1)u_1=(\zeta_2\otimes 1)u_2^*$. Since $(\zeta_1^*\otimes 1)u_1\in M_1\overline{\otimes}M_3$, $(\zeta_2\otimes 1)u_2^*\in M_2\overline{\otimes}M_3$ and $(M_1\overline{\otimes}M_3)\cap (M_2\overline{\otimes}M_3)=\mathbb C1\overline{\otimes}M_3$, it follows that we can find $\zeta_3\in M_3$ such that $$(\zeta_1^*\otimes 1)u_1=(\zeta_2\otimes 1)u_2^*=1\otimes\zeta_3.$$
   In conclusion,  $u_1=\zeta_1\otimes\zeta_3$ and $u_2=\zeta_2\otimes\zeta_3^*$, which finishes the proof.
\end{proof}

\subsection{Graphs}\label{graph_notions} A graph $\Gamma$ is called {\it simple} if it is undirected and contains no loops and no multiple edges between any pair of vertices.
We write $v\in\Gamma$ to mean that $v$ is a {\it vertex} of $\Gamma$. We say that two vertices $v,w\in\Gamma$ are {\it adjacent} if they share an edge.
We write $\Lambda\subset\Gamma$ to mean that $\Lambda$ is a {\it full subgraph} of $\Gamma$: the vertices of $\Lambda$ are a subset of the vertices of $\Gamma$ and two vertices in $\Lambda$ are adjacent if and only if they are adjacent in $\Gamma$.

The {\it link} and {\it star} of a vertex $v\in\Gamma$ are defined by
$$\text{$\text{lk}(v)=\{w\in\Gamma\mid v$ and $w$ are adjacent$\}$\;\; and \;\; $\text{st}(v)=\{v\}\cup\text{lk}(v)$.}$$
For a full subgraph $\Lambda\subset\Gamma$, we define $$\text{$\Lambda^\perp=\bigcap_{v\in\Lambda}\text{lk}(v)=\{w\in\Gamma\mid w$ is adjacent to every vertex  $v\in\Lambda\}$}.$$ 
We make the convention that $\emptyset^\perp=\Gamma$ and note that $\text{lk}(v)=\{v\}^\perp$, for every vertex $v\in\Gamma$.

A {\it join decomposition} $\Gamma=\Gamma_1\circ\Gamma_2$ of a graph $\Gamma$ means that $\Gamma_1,\Gamma_2\subset\Gamma$ are disjoint full subgraphs such that $\Gamma=\Gamma_1\cup \Gamma_2$ and $\Gamma_2=\Gamma_1^\perp$.
We say that $\Gamma$ is a {\it join graph} if it admits a join decomposition $\Gamma=\Gamma_1\circ\Gamma_2$, with $\Gamma_1$ and $\Gamma_2$ nonempty. We say that $\Gamma$ is {\it irreducible} if it is not a join graph.
A graph $\Gamma$ is called a {\it clique} if it is a complete graph. 
The {\it maximal clique factor} of $\Gamma$  is the clique  $\Gamma_0\subset\Gamma$ whose vertices are 
those $v\in\Gamma$ such that $\Gamma=\text{st}(v)$. Note that  $\Gamma=\Gamma_0\circ \Gamma_0^\perp$ and $\Gamma_0$ is the maximal clique with this property (which may be empty). 
Moreover, $\Gamma_0^\perp$ has an empty maximal clique factor, and decomposes as $\Gamma_0^\perp=\Gamma_1\circ\cdots\circ\Gamma_n$, where $\Gamma_i$ is an irreducible graph with at least two vertices, for every $1\leq i\leq n$. 

Following \cite[Definition 6.1]{EH24}, we say that $\Lambda$ is a {\it collapsible} full subgraph of $\Gamma$ if $\text{st}(v)\cap (\Gamma\setminus\Lambda)=\Lambda^\perp$, for every $v\in\Lambda$.
A graph $\Gamma$ is called {\it strongly reduced} if it contains no proper collapsible full subgraph $\Lambda\subsetneq\Gamma$ on at least two vertices. It is \emph{clique-reduced} if it contains no proper collapsible complete full subgraph on at least two vertices \cite[Section~9.1]{EH24}.

A vertex $v\in\Gamma$ is \emph{untransvectable} if there is no vertex $w\neq v$ in $\Gamma$ such that $\lk(v)\subset\st(w)$. A finite simple graph is \emph{transvection-free} if all its vertices are untransvectable. In general, we denote by $\Gamma^u$ the full subgraph of $\Gamma$ whose vertices are the untransvectable vertices of $\Gamma$.

\begin{remark}\label{basic_graph_facts} In this remark, we justify two basic graph-theoretic claims made in the introduction.

\begin{enumerate}
    \item Let $\Gamma$ be a finite simple graph with girth at least $5$ and no vertices of valence less than $2$. 
    Then $\Gamma$ is transvection-free. Assume by contradiction that $\text{lk}(v)\subset\text{st}(v')$, for $v,v'\in\Gamma$ with $v\not=v'$.
Since $|\text{lk}(v)|\geq 2$, we can choose distinct vertices $v_1,v_2\in \text{lk}(v)$. Then $v_1,v_2\in\text{st}(v')$.
If $v'$ is equal to $v_1$ or $v_2$, 
 then $v,v_1,v_2$ is a triangle. Otherwise, the distinct vertices $v,v_1,v',v_2$ are connected in this order. Either case contradicts the fact that $\Gamma$ has girth at least $5$.

 \item Let $\Gamma$ be a finite simple graph which is transvection-free, contains no squares.  
 Then every connected component of $\Gamma$ is strongly reduced. To see this, consider a component $\Gamma_0$ of $\Gamma$ and a collapsible full subgraph $\Lambda\subset\Gamma_0$ with $|\Lambda|\geq 2$.
Note that $\Lambda$ is not a clique, for otherwise we would have that $\text{st}(v)=\Lambda\cup\Lambda^\perp$, for every $v\in\Lambda$, which contradicts the assumption that $\Gamma$ is transvection-free since $|\Lambda|\geq 2$.

 If $\Lambda^\perp=\emptyset$, then since $\Lambda$ is collapsible, we get that $\Lambda$ is disconnected from $\Gamma_0\setminus\Lambda$. Since $\Gamma_0$ is connected, we deduce that $\Lambda=\Gamma_0$, as desired. 
 Thus, it remains to argue that we cannot have $\Lambda^\perp\not=\emptyset$. If $\Lambda^\perp\not=\emptyset$, then $\Lambda^\perp$ is not a clique, for otherwise we could have that $\text{st}(v)\subset\Lambda\cup\Lambda^\perp\subset\text{st}(w)$, for every $v\in\Lambda$ and $w\in\Lambda^\perp$, which again contradicts the assumption that $\Gamma$ is transvection-free.
Finally, since both $\Lambda$ and $\Lambda^\perp$ are not cliques we can find non-adjacent vertices $v_1,v_2\in\Lambda$ and $w_1,w_2\in\Lambda^\perp$. But then the vertices $v_1,w_1,v_2,w_2$ determine a square, which is a contradiction.
\end{enumerate}
\end{remark}

We record the following basic fact about collapsible subgraphs.

\begin{lemma}\label{lemma:collapsible-combinatorial}
    Let $\Lambda$ be a finite simple graph, and assume that every connected component of $\Lambda$ is strongly reduced. Let $\Delta\subset\Lambda$ be a collapsible subgraph with $|\Delta|\ge 2$. 

    Then $\Delta$ is a union of connected components of $\Lambda$.
\end{lemma}

\begin{proof}
Let $\Omega$ be a connected component of $\Delta$, and assume towards a contradiction that $\Omega\cap\Delta\neq\emptyset$ and $\Omega\cap\Delta\neq\Omega$. Then $\Omega\cap\Delta$ and $\Omega\setminus (\Delta\cap\Omega)$ are two non-empty subgraphs of the connected graph $\Omega$, so there must be an edge joining $\Omega\cap\Delta$ to $\Omega\setminus (\Delta\cap\Omega)$. In other words, we can find $w\in \Omega\cap\Delta$ such that $\lk(w)\cap (\Omega\setminus\Delta)\neq\emptyset$. We have \[\lk(w)\cap (\Omega\setminus (\Delta\cap\Omega))=\lk(w)\cap(\Lambda\setminus\Delta)=\Delta^{\perp},\] where the first equality comes from the fact that $\Omega$ is the connected component of $\Lambda$ containing $w$, and the second from the fact that $\Delta$ is a collapsible subgraph of $\Lambda$. It follows that $\Delta^{\perp}\neq\emptyset$, which forces $\Delta$ to be contained in $\Omega$. It follows that $\Delta$ is a proper collapsible subgraph of $\Omega$ with $|\Delta|\ge 2$, contradicting the fact that $\Omega$ is strongly reduced.     
\end{proof}

\subsection{Graph product von Neumann algebras}\label{graph_vN} Let $\Gamma$ be a finite simple graph and $(M_v,\tau_v)_{v\in\Gamma}$ be a collection of tracial von Neumann algebras. 
 In order to recall the definition of the graph product von Neumann algebra $M_\Gamma$  \cite{Ml04,CF17} we label (following \cite{BC24,BCC24}) alternating words in $(M_v)_{v\in\Gamma}$ using elements of the right-angled Coxeter group $\mathcal W_\Gamma$ associated to $\Gamma$. 
 
 Recall that $\mathcal W_\Gamma$   is the group generated by the vertex set of $\Gamma$, subject to the relations $v^2=e$, for every $v\in\Gamma$, and $vw=wv$, for every adjacent vertices $v,w\in\Gamma$. 
Equivalently, $\mathcal W_\Gamma$ is the graph product group over $\Gamma$ where each vertex group is equal to $\mathbb Z/2\mathbb Z$. For a full subgraph $\Lambda\subset\Gamma$, we consider the natural embedding $\mathcal W_\Lambda\subset\mathcal W_\Gamma$.
For ${\bf v}\in\mathcal W_\Gamma$, let $|{\bf v}|$ be the length of ${\bf v}$ with respect to the set of generators $\Gamma\subset\mathcal W_\Gamma$.  A  word  ${\bf v}=v_1\cdots v_n$, with $v_1,\dots,v_n\in\Gamma$, is said to be {\it reduced} if $|{\bf v}|=n$. We say that ${\bf v}\in \mathcal W_\Gamma$ {\it starts} (respectively, {\it ends}) with a letter $a\in\Gamma$ if $|a{\bf v}|<|{\bf v}|$ (respectively, $|{\bf v}a|<|{\bf v}|$). For $v\in\Gamma$, let $\mathring{M}_v=\{x\in M_v\mid\tau_v(x)=0\}$.

The \emph{graph product von Neumann algebra} $(M_\Gamma,\tau)=*_{v\in\Gamma}(M_v,\tau_v)$ is the unique (up to isomorphism) tracial von Neumann algebra with the following properties:
\begin{itemize}
\item[(a)] $M_\Gamma$ is generated by copies of $(M_v)_{v\in\Gamma}$ such that $\tau_{|M_v}=\tau_v$, for every $v\in\Gamma$;
\item[(b)] $M_v$ and $M_w$ commute, whenever $v,w\in\Gamma$ are adjacent vertices;
\item[(c)] $\tau(x_1\cdots x_n)=0$, for every $n\geq 1$ and every $x_1\in\mathring{M}_{v_1},\dots,x_n\in\mathring{M}_{v_n}$, where $v_1,\dots,v_n\in\Gamma$ are any vertices such that $v_1\cdots v_n$ is a reduced word in $\mathcal W_\Gamma$.
\end{itemize}

The reader is referred to \cite[Proposition~2.22]{CF17} for the equivalence of this definition with the other usual one given as an operator algebra on a graph product of Hilbert spaces.

\begin{remark} We record here some  properties of $M_\Gamma$ (see \cite[Remark~3.23 and Theorem~3.26]{CF17}):
\begin{enumerate}
\item Whenever $\Lambda\subset\Gamma$ is a full subgraph, we have a natural embedding $M_\Lambda\subset M_\Gamma$, given by identifying $M_\Lambda$ with the von Neumann subalgebra of $M_\Gamma$ generated by $\{M_v\}_{v\in\Lambda}$. 
\item If $v,w\in\Gamma$ are distinct vertices then the von Neumann algebra $M_v\vee M_w$ is canonically isomorphic to $M_v\overline{\otimes}M_w$, if $v$ and $w$ are adjacent, and to $M_v*M_w$, otherwise.

\item If $v\in\Gamma$, then $M_\Gamma=M_{\Gamma\setminus\{v\}}*_{M_{\text{lk}(v)}}M_{\text{st}(v)}$ and $M_{\text{st}(v)}=M_{\text{lk}(v)}\overline{\otimes}M_v.$ 
\item The graph product constructions for  groups and von Neumann algebras are compatible. For a collection $(G_v)_{v\in\Gamma}$ of discrete groups, let $G_\Gamma=*_{v,\Gamma}G_v$ the associated graph product group. If $M_v=\text{L}(G_v)$, for every $v\in\Gamma$, then we have a canonical isomorphism $\text{L}(G_{\Gamma})\cong M_\Gamma$.
\end{enumerate}
\end{remark}

Next, we recall a useful decomposition of  $\text{L}^2(M_\Gamma)$.
Let $\mathring{M}_{e}=\mathbb C1$, where $e\in \mathcal W_\Gamma$ is the identity element. For ${\bf v}\in\mathcal W_\Gamma$, let ${\bf v}=v_1\cdots v_n$ be a reduced decomposition, for $v_1,\dots,v_n\in\Gamma$. We define $$\mathring{M}_{\bf{v}}:=\mathring{M}_{v_1}\cdots\mathring{M}_{v_n}=\{x_1\cdots x_n\mid x_1\in \mathring{M}_{v_1},\dots,x_n\in\mathring{M}_{v_n}\}.$$ Then $\mathring{M}_{\bf v}$ does not depend on the choice of the reduced decomposition of ${\bf v}$. Indeed, by the normal form theorem for graph products \cite{Gr90}, 
any other reduced decomposition of ${\bf v}$ is obtained by 
repeating finitely many times the operation of swapping $v_i$ and $v_{i+1}$, for some $i$ with $v_i\in\text{lk}(v_{i+1})$, in which case $M_{v_i}$ and $M_{v_{i+1}}$ commute by (b). 
Moreover, $\mathring{M}_{\bf u}\perp\mathring{M}_{\bf v}$, for distinct ${\bf u},{\bf v}\in\mathcal W_\Gamma$ and $\bigoplus_{{\bf v}\in\mathcal W_\Gamma}\text{sp}(\mathring{M}_{\bf v})$ 
is $\|\cdot\|_2$-dense in $\text{L}^2(M_\Gamma)$, as a consequence of (c) and (a), respectively. In other words, if  ${e}_{\bf v}:\text{L}^2(M_\Gamma)\rightarrow \overline{\text{sp}}(\mathring{M}_{\bf v})$ denotes the orthogonal projection, for every ${\bf v}\in\mathcal W_\Gamma$, then $\sum_{{\bf v}\in\mathcal W_\Gamma}{e}_{\bf v}=1.$

The rest of this subsection is devoted to basic results about graph product von Neumann algebras. 

\begin{lemma}\label{intersections}
Assume the above setting. Let $n\geq 1$ and $\Gamma_0,\Gamma_1,\dots,\Gamma_n\subset\Gamma$ be full subgraphs. 
For  $S\subset M_\Gamma$, we denote by $\emph{sp}(S)$ the linear span of $S$ and by $\overline{\emph{sp}}(S)$ its $\|\cdot\|_2$-closure.
Then we have

\begin{enumerate}
\item $\emph{sp}(M_{v_1}\cdots M_{v_n})=\emph{sp}(\{\mathring{M}_{\bf v}\mid {\bf v}\in\mathcal W_{v_1}\cdots\mathcal W_{v_n}\})$, for every $v_1,\dots,v_n\in\Gamma$.
\item $\overline{\emph{sp}}(M_{\Gamma_1}\cdots M_{\Gamma_n})=\overline{\emph{sp}}(\{\mathring{M}_{\bf v}\mid {\bf v}\in\mathcal W_{\Gamma_1}\cdots\mathcal W_{\Gamma_n}\})$.
\item $\emph{E}_{M_{\Gamma_0}}(\overline{\emph{sp}}(M_{\Gamma_1}\cdots M_{\Gamma_n}))=\overline{\emph{sp}}(\{\mathring{M}_{\bf v}\mid {\bf v}\in\mathcal W_{\Gamma_0}\cap\mathcal W_{\Gamma_1}\cdots\mathcal W_{\Gamma_n}\})$.
\end{enumerate}
\end{lemma}

\begin{proof}
We will first prove by induction that for every $n\geq 1$ we have 
\begin{equation}\label{induct}\text{$\mathring{M}_{v_1}\cdots \mathring{M}_{v_n}\subset\text{sp}(\{\mathring{M}_{\bf v}\mid {\bf v}\in\mathcal W_{v_1}\cdots\mathcal W_{v_n}\})$, for every $v_1,\dots,v_n\in\Gamma$.}
\end{equation}
For $n=1$, this is obvious since $v_1\in\mathcal W_{v_1}$. Assume that \eqref{induct} holds for some $n\geq 1$ and let $v_1,\dots,v_n,v_{n+1}\in\Gamma$. If ${\bf v}=v_1\cdots v_{n+1}$ is a reduced word in $\mathcal W_\Gamma$, then $\mathring{M}_{v_1}\cdots\mathring{M}_{v_{n+1}}=\mathring{M}_{\bf v}$ by the definition of $\mathring{M}_{\bf v}$. Since ${\bf v}\in\mathcal W_{v_1}\cdots\mathcal W_{v_{n+1}}$, \eqref{induct} holds in this case. We can therefore assume that ${\bf v}$ is not a reduced word. This means that we can find $1\leq i<j\leq n+1$ such that $v_i=v_j$ and $v_k\in\text{lk}(v_i)$, for every $i<k<j$. Thus, $\mathring{M}_{v_i}=\mathring{M}_{v_j}$ commutes with $\mathring{M}_{v_k}$, for every $i<k<j$.  Using this fact, that $\mathring{M}_{v_i}\mathring{M}_{v_i}\subset M_{v_i}$ and $M_{v_i}=\mathring{M}_{v_i}+\mathbb C1$, we derive that \begin{equation}\label{mathring}\mathring{M}_{v_1}\cdots\mathring{M}_{v_{n+1}}\subset\mathring{M}_{v_1}\cdots\mathring{M}_{v_{i-1}}M_{v_i}\mathring{M}_{v_{i+1}}\cdots\mathring{M}_{v_{j-1}}\mathring{M}_{v_{j+1}}\cdots\mathring{M}_{v_n}\subset S_1+S_2,\end{equation}
with \[\begin{array}{l}S_1= \mathring{M}_{v_1}\cdots\mathring{M}_{v_{i+1}}\cdots\mathring{M}_{v_{j-1}}\mathring{M}_{v_{j+1}}\cdots\mathring{M}_{v_{n+1}},  \\ S_2=\mathring{M}_{v_1}\cdots\mathring{M}_{v_{i-1}}\mathring{M}_{v_{i+1}}\cdots\mathring{M}_{v_{j-1}}\mathring{M}_{v_{j+1}}\cdots\mathring{M}_{v_{n+1}}.\end{array}\]
By the induction hypothesis, we get that $S_1\subset\text{sp}(\{\mathring{M}_{\bf v}\mid {\bf v}\in\mathcal V_1\})$ and $S_2\subset\text{sp}(\{\mathring{M}_{\bf v}\mid {\bf v}\in\mathcal V_2\})$, where $\mathcal V_1=\mathcal W_{v_1}\cdots\mathcal W_{v_{j-1}}\mathcal W_{v_{j+1}}\cdots\mathcal W_{v_{n+1}}$ and $\mathcal V_2=\mathcal W_{v_1}\cdots\mathcal W_{v_{i-1}}\mathcal W_{v_{i+1}}\cdots\mathcal W_{v_{j-1}}\mathcal W_{v_{j+1}}\cdots\mathcal W_{v_{n+1}}.$
Since $\mathcal V_1,\mathcal V_2\subset \mathcal W_{v_1}\cdots\mathcal W_{v_{n+1}}$, we get that $S_1+S_2\subset \text{sp}(\{\mathring{M}_{\bf v}\mid {\bf v}\in\mathcal W_{v_1}\cdots\mathcal W_{v_{n+1}}\})$, which by \eqref{mathring} implies that  $\mathring{M}_{v_1}\cdots\mathring{M}_{v_{n+1}}\subset \text{sp}(\{\mathring{M}_{\bf v}\mid {\bf v}\in\mathcal W_{v_1}\cdots\mathcal W_{v_{n+1}}\})$. This finishes the proof of \eqref{induct}.

We are now ready to prove assertions (1)-(3).

(1) Let $v_1,\dots,v_n\in\Gamma$.
If ${\bf v}\in\mathcal W_{v_1}\cdots\mathcal W_{v_n}\setminus\{e\}$, then it follows by induction on $n$ that ${\bf v}$ admits a reduced decomposition ${\bf v}=v_{i_1}\cdots v_{i_k}$, for some $1\leq k\leq n$ and some $1\leq i_1<i_2<\dots<i_k\leq n$. Therefore, $\mathring{M}_{\bf v}=\mathring M_{v_{i_1}}\cdots\mathring{M}_{v_{i_k}}\subset M_{v_1}\cdots M_{v_n}$, which justifies the inclusion $\supset$ from (1). 

For the opposite inclusion, note that \begin{equation}\label{span}\text{sp}(M_{v_1}\cdots M_{v_n})=\mathbb C1+\text{sp}(\{\mathring{M}_{v_{i_1}}\cdots\mathring{M}_{v_{i_k}}\mid 1\leq k\leq n,1\leq i_1<\dots<i_k\leq n\}).\end{equation}
By \eqref{induct} we get that $\mathring{M}_{v_{i_1}}\cdots \mathring{M}_{v_{i_k}}\subset\text{sp}(\{\mathring{M}_{\bf v}\mid {\bf v}\in\mathcal W_{v_{i_1}}\cdots\mathcal W_{v_{i_k}}\})\subset\text{sp}(\{\mathring{M}_{\bf v}\mid {\bf v}\in\mathcal W_{v_1}\cdots\mathcal W_{v_n}\})$. In combination with \eqref{span}, the inclusion $\subset$ from (1) follows.

(2) Let ${\bf v}\in\mathcal W_{\Gamma_1}\cdots\mathcal W_{\Gamma_n}$. Write ${\bf v}=v_{1,1}\cdots v_{1,k_1}\cdots v_{n,1}\cdots v_{n,k_n}$, where $v_{i,1},\dots,v_{i,k_i}\in\Gamma_i$, for every $1\leq i\leq n$. The proof of (1) gives that $\mathring{M}_{\bf v}\subset M_{v_{1,1}}\cdots M_{v_{1,k_1}}\cdots M_{v_{n,1}}\cdots M_{v_{n,k_n}}$. Since $M_{v_{i,1}}\cdots M_{v_{i,k_i}}\subset M_{\Gamma_i}$, for every $1\leq i\leq n$, we get  that $\mathring{M}_{\bf v}\subset M_{\Gamma_1}\cdots M_{\Gamma_n}$. This implies the inclusion $\supset$ from (2). 

To justify the opposite inclusion, for $1\leq i\leq n$,  let $v_{i,1},\dots,v_{i,k_i}\in \Gamma_i$. 
By  (1), we have $$M_{v_{1,1}}\cdots M_{v_{1,k_1}}\cdots M_{v_{n,1}}\cdots M_{v_{n,k_n}}\subset \text{sp}(\{\mathring{M}_{\bf v}\mid {\bf v}\in\mathcal W_{v_{1,1}}\cdots\mathcal W_{v_{1,k_1}}\cdots\mathcal W_{v_{n,1}}\cdots\mathcal W_{v_{n,k_n}}\}).$$ Since $\mathcal W_{v_{i,1}}\cdots\mathcal W_{v_{i,k_i}}\subset\mathcal W_{\Gamma_i}$, for every $1\leq i\leq n$, we derive that $$M_{v_{1,1}}\cdots M_{v_{1,k_1}}\cdots M_{v_{n,1}}\cdots M_{v_{n,k_n}}\subset \text{sp}(\{{\mathring M}_{\bf v}\mid {\bf v}\in \mathcal W_{\Gamma_1}\cdots\mathcal W_{\Gamma_n}\}).$$

Since the span of $M_{v_{1,1}}\cdots M_{v_{1,k_1}}\cdots M_{v_{n,1}}\cdots M_{v_{n,k_n}}$, ranging over all $v_{i,1},\dots,v_{i,k_i}\in\Gamma_i$, for every $1\leq i\leq n$, is $\|\cdot\|_2$-dense in $\text{sp}(M_{\Gamma_1}\cdots M_{\Gamma_n})$, the inclusion $\subset$ from (2) follows.

(3) This assertion follows readily by using (2) and the fact that if ${\bf v}\in\mathcal W_\Gamma$, then $\text{E}_{M_{\Gamma_0}}(\mathring{M}_{\bf v})$ is equal to $\mathring{M}_{\bf v}$, if ${\bf v}\in\mathcal W_{\Gamma_0}$, and to $\{0\}$, if ${\bf v}\not\in\mathcal W_{\Gamma_0}$.
\end{proof}

\begin{lemma}\label{basic}
Assume the above setting and additionally that $M_v$ is diffuse, for every $v\in\Gamma$. Let $\Lambda, \Lambda'\subset\Gamma$ be full  subgraphs.
Then the following hold: 
\begin{enumerate}
\item  $M_\Lambda'\cap M_\Gamma=\mathcal Z(M_\Lambda)\overline{\otimes}M_{\Lambda^\perp}$ and $\mathcal Z(M_\Lambda)=\overline{\otimes}_{v\in C}\mathcal Z(M_v)$, where $C$ is  the maximal clique factor of $\Lambda$. 
Moreover, if $M_v$ is a II$_1$ factor, for every $v\in\Lambda$, then $M_\Lambda$ is a II$_1$ factor and $M_\Lambda'\cap M_\Gamma=M_{\Lambda^\perp}$. 
\item $M_\Lambda\prec_{M_\Gamma}M_{\Lambda'}$ if and only if $\Lambda\subset\Lambda'$.
\item If $\Lambda$ is not a clique, then $M_\Lambda$ has no amenable direct summand.
\end{enumerate}
\end{lemma}
\begin{proof}
For part (1), see \cite[Corollaries 3.28 and 3.29]{CF17}. Part (2) is a consequence of the following fact: if $v\in\Gamma$ and $M_v\prec_{M_\Gamma}M_{\Lambda'}$, then $v\in\Lambda'$. For a proof of this fact, see \cite[Proposition~5.9]{BCC24}. This fact can be proved directly as follows. Assume that $v\not\in\Lambda'$. Since $M_v$ is diffuse and $M_{\text{st}(v)}=M_{\text{lk}(v)}\overline{\otimes}M_v$,  we clearly have that $M_v\nprec_{M_{\text{st}(v)}}M_{\text{lk}(v)}$. Since $M_\Gamma=M_{\Gamma\setminus\{v\}}*_{M_{\text{lk}(v)}}M_{\text{st}(v)}$, \cite[Theorem 1.1]{IPP08} implies that $M_v\nprec_{M_\Gamma}M_{\Gamma\setminus\{v\}}$ and hence $M_v\nprec_{M_\Gamma}M_{\Lambda'}$.
For part (3), note that as $\Lambda$ is not a clique, it contains two vertices $w,w'$ which are not adjacent. Hence, $M_w*M_{w'}\subset M_\Lambda$. Since $M_w$ and $M_{w'}$ are diffuse, $M_w*M_{w'}$ is a nonamenable II$_1$ factor and thus $M_\Lambda$ has no amenable direct summand.
\end{proof}

Note that part (2) is also a consequence of a more general result (Proposition~\ref{intertwiners}) that we prove in the next section.

For further reference, we also recall the following key calculation of conditional expectations in graph product von Neumann algebras.
For a reduced word ${\bf v}=v_1\cdots v_n\in\mathcal W_\Gamma$,  we define, abusing notation, $\text{lk}({\bf v})=\bigcap_{i=1}^n\text{lk}(v_i).$

\begin{proposition}
\label{expectation}  
Assume the above setting. Let $\Gamma_1,\Gamma_2\subset\Gamma$ be full subgraphs,  ${\bf u},{\bf v}\in\mathcal W_\Gamma$, $a\in\mathring{M}_{\bf u}$ and $b\in\mathring{M}_{\bf v}$.
Consider reduced decompositions ${\bf u}={\bf u}_l{\bf u}_c{\bf u}_r$ and ${\bf v}={\bf v}_l{\bf v}_c{\bf v}_r$, where ${\bf u}_l,{\bf v}_l\in\mathcal W_{\Gamma_1}$, ${\bf u}_r,{\bf v}_r\in\mathcal W_{\Gamma_2}$, and ${\bf u}_c,{\bf v}_c$ do not start with a letter from $\Gamma_1$ and do not end with a letter from $\Gamma_2$.
Write $a=a_la_ca_r$, $b=b_lb_cb_r$, where $a_l\in\mathring{M}_{{\bf u}_l},a_c\in\mathring{M}_{{\bf u}_c},a_r\in\mathring{M}_{{\bf u}_r}$ and $b_l\in\mathring{M}_{{\bf v}_l},b_c\in\mathring{M}_{{\bf v}_c}, b_r\in\mathring{M}_{{\bf v}_r}$.

Then for every $x\in M_\Gamma$ we have that $$\emph{E}_{M_{\Gamma_2}}(a^*\emph{E}_{M_{\Gamma_1}}(x)b)=\tau(a_c^*b_c)a_r^*\emph{E}_{M_{\Gamma_1\cap\Gamma_2\cap\emph{lk}({\bf u}_c)}}(a_l^*xb_l)b_r.$$

 \end{proposition}
 
 Proposition \ref{expectation} was proved in \cite[Lemma 3.4]{BC24} when $M_v=\text{L}(\mathbb Z/2\mathbb Z)$ for every $v\in\Gamma$. It was established in general in  \cite[Lemma 3.17]{CdSHJKEN24}, where it is stated in an equivalent form. The above  formulation of Proposition \ref{expectation} is given in \cite[Proposition 5.1]{BCC24}.
 
 In the next section, we establish several intertwining results for graph product von Neumann algebras. We derive a first such result here as an immediate consequence of Proposition \ref{expectation}.
 
 \begin{lemma}\label{descend}
Assume the above setting.
Let $\Gamma_1,\Gamma_2\subset\Gamma$ be full subgraphs such that $\Gamma_2\subset\Gamma_1$. 
Let $P\subset pM_{\Gamma_1}p$ be a von Neumann subalgebra. Assume that $P\prec_{M_\Gamma}^sM_{\Gamma_2}$.

Then $P\prec_{M_{\Gamma_1}}^sM_{\Gamma_2}$.
\end{lemma}

\begin{proof}
Assume by contradiction that the conclusion fails. Then we can find a nonzero projection $p'\in P'\cap pM_{\Gamma_1}p$ such that $Pp'\nprec_{M_{\Gamma_1}}M_{\Gamma_2}$.  Hence we can find a net $(u_n)\subset\mathcal U(Pp')$ such that $\|\text{E}_{M_{\Gamma_2}}(a^*u_nb)\|_2\rightarrow 0$, for every $a,b\in M_{\Gamma_1}$.  We claim that $\|\text{E}_{M_{\Gamma_2}}(a^*u_nb)\|_2\rightarrow 0$, for every $a,b\in M_{\Gamma}$. Assuming the claim, we get that $Pp'\nprec_{M_\Gamma}M_{\Gamma_2}$, which contradicts the fact that  $P\prec_{M_\Gamma}^sM_{\Gamma_2}$.

In order to prove the claim, we may assume that $a\in\mathring{M}_{\bf u}$ and $b\in\mathring{M}_{\bf v}$, for some ${\bf u},{\bf v}\in\mathcal W_\Gamma$. Proposition \ref{expectation}   gives $a_l,b_l\in M_{\Gamma_1}$ and $C>0$ such that $\|\text{E}_{M_{\Gamma_2}}(a^*\text{E}_{M_{\Gamma_1}}(x)b)\|_2\leq C\|\text{E}_{M_{\Gamma_2}}(a_l^*xb_l)\|_2$, for every $x\in M_\Gamma$. Since $a_l,b_l\in M_{\Gamma_1}$ we get that $\|\text{E}_{M_{\Gamma_2}}(a_l^*u_nb_l)\|_2\rightarrow 0$. Since all $u_n$ belong to $Pp'$, which is contained in $ M_{\Gamma_1}$, we derive that  $\|\text{E}_{M_{\Gamma_2}}(a^*u_nb)\|_2=\|\text{E}_{M_{\Gamma_2}}(a^*\text{E}_{M_{\Gamma_1}}(u_n)b)\|_2\rightarrow 0$, which proves the claim.
\end{proof}

\section{Intertwining results for graph product von Neumann algebras and hulls}\label{intertwining_results}

This section is devoted to several intertwining results for graph product von Neumann algebras, and introduces the central notion of the \emph{hull} of a subalgebra in a graph product von Neumann algebra. Several of the results obtained in this section develop and expand results from \cite[Section~5]{BCC24} in a more systematic way. 
We first introduce some notation that we use throughout this section. 

\begin{notation}\label{setup}
Let $\Gamma$  be a finite simple graph and $(M_v,\tau_v)_{v\in\Gamma}$ a collection of tracial von Neumann algebras. Let $M_\Gamma=*_{v,\Gamma}M_v$ the associated graph product von Neumann algebra and $p\in M_\Gamma$ a projection. 
\end{notation}

 \subsection{A general intertwining result}
We start with the following crucial result from which we will deduce a number of corollaries.

 \begin{proposition}\label{intertwiners}
Assume the setting from Notation~\ref{setup}. Let $\Gamma_1,\Gamma_2\subset\Gamma$ be full subgraphs, and assume that $p\in M_{\Gamma_1}$. Let $P\subset pM_{\Gamma_1} p$ be a von Neumann subalgebra such that $P\nprec_{M_{\Gamma_1}}M_{\Gamma_1'}$, for every full subgraph $\Gamma_1'\subsetneq\Gamma_1$. 
\begin{enumerate}
\item If $P\prec_{M_\Gamma} M_{\Gamma_2}$, then $\Gamma_1\subset\Gamma_2$. 
\item Additionally, if $v\in pM_\Gamma\setminus\{0\}$ is such that $Pv\subset \sum_{i=1}^k v_iM_{\Gamma_2}$ for some $v_1,\dots,v_k\in M_\Gamma$, then $v$ belongs to the $\|\cdot\|_2$-closure of the linear span of $M_{\Gamma_1^\perp}M_{\Gamma_2}$.
\end{enumerate}
\end{proposition}

In order to motivate Proposition \ref{intertwiners}, we note that it is a von Neumann algebraic analogue of the following group-theoretic fact. Let $G_\Gamma=*_{v,\Gamma}G_v$ be the graph product group associated to a collection $(G_v)_{v\in\Gamma}$ of discrete groups. Let $H\subset G_{\Gamma_1}$ be a subgroup which is not contained in any subgroup of $G_{\Gamma_1}$ of the form $hG_{\Gamma_1'}h^{-1}$, for some proper full subgraph $\Gamma_1'\subsetneq\Gamma_1$ and $h\in G_{\Gamma_1}$. If $H$ is contained in $gG_{\Gamma_2}g^{-1}$, for some $g\in G_\Gamma$, then $\Gamma_1\subset\Gamma_2$ and $g$ belongs to $G_{\Gamma_1^\perp}G_{\Gamma_2}$. This group-theoretic fact can be deduced for instance from the results in \cite[Section~3]{AM}.

\begin{proof} 
In order to prove the conclusion, it suffices to show that 
if there exists  $v\in pM_\Gamma\setminus\{0\}$ such that $Pv\subset \sum_{i=1}^k v_iM_{\Gamma_2}$, for some $v_1,\dots,v_k\in M_\Gamma$, then $\Gamma_1\subset\Gamma_2$ and $v\in\overline{\text{sp}}(M_{\Gamma_1^\perp}M_{\Gamma_2})$. By using Proposition \ref{approx} we may assume that $v$ satisfies the stronger condition that  $(P)_1v\subset \sum_{i=1}^kv_i(M_{\Gamma_2})_1$. By assumption we have $P\nprec_{M_{\Gamma_1}}M_{\Gamma_1'}$, for every full subgraph $\Gamma_1'\subsetneq\Gamma_1$.  Therefore, Proposition~\ref{results}(5) gives a net $(u_n)\subset\mathcal U(P)$ such that for every $x,y\in pM_{\Gamma_1}$ and every full subgraph $\Gamma_1'\subsetneq\Gamma_1$, one has $\|\mathrm{E}_{M_{\Gamma_1'}}(x^*u_ny)\|_2\rightarrow 0$.

\begin{claim}\label{weakly} 
The following hold.
\begin{enumerate}
    \item If $\Gamma_1\not\subset\Gamma_2$, then $\|\emph{E}_{M_{\Gamma_2}}(a^*u_nb)\|_2\rightarrow 0$, for every $a,b\in M_\Gamma$.
    \item 
Let ${\bf u}\in\mathcal W_\Gamma$ and consider a reduced decomposition ${\bf u}={\bf u}_l{\bf u}_c{\bf u}_r$, where ${\bf u}_l\in\mathcal W_{\Gamma_1}$, ${\bf u}_r\in\mathcal W_{\Gamma_2}$, and ${\bf u}_c$ does not start with a letter from $\Gamma_1$ and does not end with a letter from $\Gamma_2$. 
If ${\bf u}_c\not\in\mathcal W_{\Gamma_1^\perp}$, then  $\|\emph{E}_{M_{\Gamma_2}}(a^*u_nb)\|_2\rightarrow 0$, for every $a\in \mathring{M}_{\bf u}$ and $b\in M_\Gamma$.
\end{enumerate}
\end{claim}

\begin{proof}[Proof of Claim \ref{weakly}] 
We prove both assertions at the same time. The linear span of $\{\mathring{M}_{\bf v} \mid {\bf v}\in\mathcal W_\Gamma\}$ is $\|\cdot\|_2$-dense in $M_\Gamma$, so we may assume that $b\in\mathring{M}_{\bf v}$, for some ${\bf v}\in\mathcal W_\Gamma$. When proving~(1), we may also assume that $a\in\mathring{M}_{\bf u}$ for some ${\bf u}\in\mathcal W_\Gamma$ (when proving (2) the element $\bf u$ is given). Consider reduced decompositions ${\bf u}={\bf u}_l{\bf u}_c{\bf u}_r$ and ${\bf v}={\bf v}_l{\bf v}_c{\bf v}_r$, where ${\bf u}_l,{\bf v}_l\in\mathcal W_{\Gamma_1}$, ${\bf u}_r,{\bf v}_r\in\mathcal W_{\Gamma_2}$, and ${\bf u}_c,{\bf v}_c$ do not start with a letter from $\Gamma_1$ and do not end with a letter from $\Gamma_2$. Assuming either that $\Gamma_1\not\subset\Gamma_2$ (for (1)) or that ${\bf u}_c\not\in\mathcal{W}_{\Gamma_1^{\perp}}$ (for (2)), our aim is to prove that $\|\mathrm{E}_{M_{\Gamma_2}}(a^*u_nb)\|_2\rightarrow 0$, for every $a\in \mathring{M}_{\bf u}$ and every $b\in\mathring{M}_{\bf v}$.

Write $a=a_la_ca_r$, $b=b_lb_cb_r$, where $a_l\in\mathring{M}_{{\bf u}_l},a_c\in\mathring{M}_{{\bf u}_c},a_r\in\mathring{M}_{{\bf u}_r}$ and $b_l\in\mathring{M}_{{\bf v}_l},b_c\in\mathring{M}_{{\bf v}_c}, b_r\in\mathring{M}_{{\bf v}_r}$. Let $\Gamma_1'=\Gamma_1\cap\Gamma_2\cap\text{lk}({\bf u}_c)$. Since $u_n\in M_{\Gamma_1}$, Proposition \ref{expectation} gives  \begin{equation}\label{expe}\text{$\text{E}_{M_{\Gamma_2}}(a^*u_nb)=\tau(a_c^*b_c)a_r^*\text{E}_{M_{\Gamma_1'}}(a_l^*u_nb_l)b_r$, for every $n$}.\end{equation} 

If $\Gamma_1\not\subset\Gamma_2$, then $\Gamma_1'\subsetneq\Gamma_1$. If ${\bf u}_c\not\in\mathcal W_{\Gamma_1^\perp}$, then $\Gamma_1\not\subset\text{lk}({\bf u}_c)$ and thus $\Gamma_1'\subsetneq\Gamma_1$. So in either case, we have $\Gamma_1'\subsetneq\Gamma_1$. Since $a_l,b_l\in M_{\Gamma_1}$, we get that $\|\text{E}_{M_{\Gamma_1'}}(a_l^*u_nb_l)\|_2=\|\text{E}_{M_{\Gamma_1'}}((pa_l)^*u_n(pb_l))\|_2\rightarrow 0$. Together with 
\eqref{expe} we deduce that $\|\text{E}_{M_{\Gamma_2}}(a^*u_nb)\|_2\rightarrow 0$, which finishes the proof of the claim.
\end{proof}
We are now ready to prove the conclusion. 
First, assume by contradiction that $\Gamma_1\not\subset\Gamma_2$.
 If $u\in\mathcal U(P)$, then  $uv=\sum_{i=1}^kv_ix_i$, for some  $x_1,\dots,x_k\in (M_{\Gamma_2})_1$.
 Since for every $1\leq i\leq k$ we have $\langle uv,v_ix_i\rangle=\tau(x_i^*v_i^*uv)=\tau(\text{E}_{M_{\Gamma_2}}(x_i^*v_i^*uv))=\tau(x_i^*\text{E}_{M_{\Gamma_2}}(v_i^*uv))$, 
 we deduce that
 \begin{equation}\label{norm_of_v}
\|v\|_2^2=\langle uv,uv\rangle=\sum_{i=1}^k\Re\langle uv,v_ix_i\rangle=\sum_{i=1}^k\Re\tau(x_i^*\text{E}_{M_{\Gamma_2}}(v_i^*uv))\leq\sum_{i=1}^k\|\text{E}_{M_{\Gamma_2}}(v_i^*uv)\|_2.
\end{equation}
As $\Gamma_1\not\subset\Gamma_2$, Claim \ref{weakly}(1) gives  $\|\text{E}_{M_{\Gamma_2}}(a^*u_nb)\|_2\rightarrow 0$, for every $a,b\in M_\Gamma$. So, $\|\text{E}_{M_{\Gamma_2}}(v_i^*u_nv)\|_2\rightarrow 0$, for every $1\leq i\leq k$. Since $v\not=0$ and $(u_n)\subset\mathcal U(P)$, this contradicts \eqref{norm_of_v} and proves that $\Gamma_1\subset\Gamma_2$.

Second, let $\mathcal H_0$ be the linear span of $\{\mathring{M}_{\bf u}\mid {\bf u}\in\mathcal W_\Gamma\setminus \mathcal W_{\Gamma_1^\perp}\mathcal W_{\Gamma_2}\}$, and put $\mathcal H=\overline{\mathcal H_0}^{\|\cdot\|_2}$. We claim that 
\begin{equation}\label{mixing_rel}\text{$\|\text{E}_{M_{\Gamma_2}}(a^*u_nb)\|_2\rightarrow 0$, for every $a\in \mathcal H_0$ and $b\in M_\Gamma$.}
\end{equation}
To prove this, we may assume that $a\in \mathring{M}_{\bf u}$, for some ${\bf u}\in\mathcal W_\Gamma\setminus \mathcal W_{\Gamma_1^\perp}\mathcal W_{\Gamma_2}$. Consider a reduced decomposition ${\bf u}={\bf u}_l{\bf u}_c{\bf u}_r$, where ${\bf u}_l\in\mathcal W_{\Gamma_1}$, ${\bf u}_r\in\mathcal W_{\Gamma_2}$ and ${\bf u}_c$ does not start with a letter from $\Gamma_1$ and does not end with a letter from $\Gamma_2$. Since ${\bf u}\not\in \mathcal W_{\Gamma_1^\perp}\mathcal W_{\Gamma_2}$ and $\Gamma_1\subset\Gamma_2$, we get that ${\bf u}_c\not\in\mathcal W_{\Gamma_1^\perp}$. 
Thus, \eqref{mixing_rel} follows from Claim \ref{weakly}(2).

Finally, \eqref{mixing_rel} gives that $\|\text{E}_{M_{\Gamma_2}}(a^*u_nb)\|_2\rightarrow 0$, for every $a\in p\mathcal H_0$ and $b\in M_\Gamma$. 
Since $\mathcal H$ is an $M_{\Gamma_1}$-$M_{\Gamma_2}$-bimodule,  $p\mathcal H=\overline{p\mathcal H_0}^{\|\cdot\|_2}$ is a $P$-$M_{\Gamma_2}$-bimodule, so Corollary \ref{space_of_intertwiners} implies that $v\in \text{L}^2(M_\Gamma)\ominus p\mathcal H$. Since $v\in pM_\Gamma$, we have $v\in [(1-p)\mathcal H]^\perp$, 
hence we conclude that $v\in\text{L}^2(M_\Gamma)\ominus\mathcal H$.
Noticing that $\text{L}^2(M_\Gamma)\ominus\mathcal H$ is equal to the $\|\cdot\|_2$-closure of the linear span of $\{\mathring{M}_{\bf u}\mid {\bf u}\in \mathcal W_{\Gamma_1^\perp}\mathcal W_{\Gamma_2}\}$ and thus to the $\|\cdot\|_2$-closure of the linear span of
$M_{\Gamma_1^\perp}M_{\Gamma_2}$ by Lemma \ref{intersections}(2), this finishes the proof.
\end{proof}

\subsection{The hull of a subalgebra}

The following definition, central to our work, is inspired by Zimmer's notion of the algebraic hull of a cocycle with values into an algebraic group (see \cite[Section~9.2]{Zi84}); 
see also \cite[Section~3.3]{HH22}, where a related notion of parabolic support was introduced for measured groupoids equipped with cocycles towards right-angled Artin groups.

\begin{definition}
Assume the setting from Notation \ref{setup}.
We say that $M_\Lambda$, for a full subgraph $\Lambda\subset\Gamma$, is the (von Neumann algebraic) {\it hull} of a von Neumann subalgebra $P\subset pM_\Gamma p$ if $P\prec_{M_\Gamma}^sM_{\Lambda}$ and $P\nprec_{M_{\Gamma}}M_{\Lambda'}$, for every full subgraph $\Lambda'\subsetneq\Lambda$.
\end{definition}

We now address the existence and uniqueness of the hull of a von Neumann subalgebra of $pM_\Gamma p$. The uniqueness of the hull is a consequence of Corollary~\ref{intersection} below. In general, a von Neumann subalgebra $P\subset pM_\Gamma p$ need not have a hull. However, Proposition~\ref{Zariski_hull} below shows that the hull exists provided that $P'\cap pM_\Gamma p$ is a factor. In general, there is a partition $\{p_\Lambda\}_\Lambda$ of $p$ in $\mathcal Z(P'\cap pM_\Gamma p)$ such that $Pp_\Lambda$ has hull $M_\Lambda$. 

We start with the following consequence of Proposition \ref{intertwiners}, which justifies the uniqueness of the hull.

\begin{corollary}\label{intersection}
Assume the setting from Notation \ref{setup}.
Let $P\subset pM_\Gamma p$ be a von Neumann subalgebra.  
Let $\Gamma_1,\Gamma_2\subset\Gamma$ be full subgraphs.

If $P\prec_{M_\Gamma}^sM_{\Gamma_1}$ and $P\prec_{M_\Gamma}^sM_{\Gamma_2}$, then $P\prec_{M_\Gamma}^sM_{\Gamma_1\cap\Gamma_2}$. 
\end{corollary}

Corollary \ref{intersection} generalizes \cite[Proposition 5.9]{BCC24} which implies the case $\Gamma_1=\{v\}$, for $v\in\Gamma$.

\begin{proof} 
We will prove the following claim.

\begin{claim}
    Let $Q\subset pM_\Gamma p$ be a von Neumann subalgebra, and let $\Lambda_1,\Lambda_2\subset\Gamma$ be two full subgraphs. Assume that $Q\prec^s_{M_\Gamma}M_{\Lambda_1}$ and $Q\prec^s_{M_\Gamma}M_{\Lambda_2}$, and that for every proper full subgraph $\Lambda'_1\subsetneq\Lambda_1$, one has $Q\not\prec_{M_\Gamma}M_{\Lambda'_1}$.
    Then $\Lambda_1\subset\Lambda_2$.
\end{claim}

Before proving the claim, let us first explain how the corollary follows. Let $P\subset pM_\Gamma p$ be as in the statement, and let $p'\in P'\cap pM_\Gamma p$ be a nonzero projection. Let $\Lambda_1\subset\Gamma_1$ be a minimal subgraph such that $Pp'\prec_{M_\Gamma}M_{\Lambda_1}$ (so that $Pp'\not\prec_{M_\Gamma}M_{\Lambda'_1}$ for every proper full subgraph $\Lambda'_1\subsetneq\Lambda_1$). 
Then we can find a nonzero projection $p''\in P'\cap pM_\Gamma p$ such that
$p''\leq p'$ and $Pp''\prec^s_{M_\Gamma}M_{\Lambda_1}$. 
We also have $Pp''\prec^s_{M_\Gamma}M_{\Gamma_2}$. Applying the claim to $Q=Pp''$ (with $\Lambda_2=\Gamma_2$), we deduce that $\Lambda_1\subset\Gamma_2$, so $\Lambda_1\subset\Gamma_1\cap\Gamma_2$. We thus have $Pp''\prec_{M_{\Gamma}}M_{\Gamma_1\cap\Gamma_2}$, in particular $Pp'\prec_{M_\Gamma}M_{\Gamma_1\cap\Gamma_2}$. Since this is true for any $p'\in P'\cap pM_\Gamma p$, we have proved that $P\prec^s_{M_\Gamma}M_{\Gamma_1\cap\Gamma_2}$, as desired.

We now prove the above claim. Since $Q\prec_{M_\Gamma}M_{\Lambda_2}$, we have projections $q_2\in Q,r_2\in M_{\Lambda_2}$, a nonzero partial isometry $v_2\in r_2M_\Gamma q_2$ and a $*$-homomorphism $\theta_2:q_2Qq_2\rightarrow r_2M_{\Lambda_2}r_2$ such that 
\begin{enumerate}
\item $\theta_2(y)v_2=v_2y$, for every $y\in q_2Qq_2$.  
\end{enumerate}
Since $v_2^*v_2\in (q_2Qq_2)'\cap q_2M_\Gamma q_2=(Q'\cap pM_\Gamma p)q_2$, we can find a projection $q'\in Q'\cap pM_\Gamma p$ such that $v_2^*v_2=q_2q'$. Let $r=v_2v_2^*$. Then $r\in\theta_2(q_2Qq_2)'\cap r_2M_{\Gamma}r_2$ and $r\leq r_2$.

Since $q_2Qq_2q'\prec_{M_\Gamma}M_{\Lambda_1}$ and $q_2Qq_2q'\nprec_{M_\Gamma}M_{\Lambda_1'}$, for every full subgraph $\Lambda_1'\subsetneq\Lambda_1$, by Proposition~\ref{cutting_down}(2) we can find projections $q_1\in q_2Qq_2$, $r_1\in M_{\Lambda_1}$, a nonzero partial isometry $v_1\in r_1M_\Gamma q_1q'$ and a $*$-homomorphism $\theta_1:q_1Qq_1\rightarrow r_1M_{\Lambda_1}r_1$ such that 
\begin{itemize}
\item[(2)] $\theta_1(x)v_1=v_1x$, for every $x\in q_1Qq_1$, and 
\item [(3)] $\theta_1(q_1Qq_1)\nprec_{M_{\Lambda_1}}M_{\Lambda_1'}$, for every proper full subgraph $\Lambda_1'\subsetneq\Lambda_1$.
\end{itemize}

Letting $v=v_1v_2^*$ and using (1) and (2), we get that 
\begin{equation}\text{$\theta_1(x)v=v\theta_2(x)$, for every $x\in q_1Qq_1$.}\end{equation}
Since $vv^*=v_1(v_2^*v_2)v_1^*=v_1(q_2q')v_1^*=v_1v_1^*\not=0$, we get that $v\not=0$.  Let $R:=\theta_1(q_1Qq_1)\subset r_1M_{\Lambda_1}r_1$. Then $Rv\subset vM_{\Lambda_2}$ and  $R\nprec_{M_{\Lambda_1}}M_{\Lambda_1'}$, for every full subgraph $\Lambda_1'\subsetneq\Lambda_1$,  by (3). Proposition~\ref{intertwiners} thus implies that $\Lambda_1\subset\Lambda_2$, as desired.
\end{proof}

As explained above, we now address the existence of the hull.

\begin{proposition}\label{Zariski_hull}
Assume the setting from Notation \ref{setup}.
Let $P\subset pM_\Gamma p$ be a von Neumann subalgebra. 
For every full subgraph $\Lambda\subset\Gamma$, let $p_\Lambda\in\mathcal N_{pM_\Gamma p}(P)'\cap pM_{\Gamma}p$ be the maximal projection such that $Pp_\Lambda\prec^s_{M_{\Gamma}}M_{\Lambda}$ and $Pp_\Lambda\nprec_{M_{\Gamma}}M_{\Lambda'}$, for every full subgraph $\Lambda'\subsetneq\Lambda$.

Then $\sum_{\text{$\Lambda$}} p_\Lambda=p$. 

Moreover, if $\Lambda\subset\Gamma$ is a full subgraph, then $Pp_\Lambda\nprec_{M_{\Gamma}}M_{\Lambda'}$, for every full subgraph $\Lambda'\subset\Gamma$ with  $\Lambda\not\subset\Lambda'$.
\end{proposition}

The moreover part of the statement ensures that if $P$ admits a hull, 
of the form $M_\Lambda$ for some full subgraph $\Lambda$ of $\Gamma$, 
and $P\prec_{M_\Gamma}M_{\Lambda'}$, for a full subgraph $\Lambda'\subset\Gamma$, then $\Lambda\subset\Lambda'$.

Before proving this result, we justify in the following remark that the projection $p_\Lambda$ is well-defined.

\begin{remark}\label{rk:maximal-projections}
Let $(M,\tau)$ be a tracial von Neumann algebra and $P\subset pMp, Q, R_1,\dots, R_k\subset qMq$ be von Neumann subalgebras. By Zorn's lemma we can find a maximal, with respect to inclusion, family of pairwise orthogonal projections $\{z_i\}_{i\in I}\subset\mathcal N_{pMp}(P)'\cap pMp$ such that  for every $i\in I$ we have $Pz_i\prec_M^sQ$ and $Pz_i\nprec_MR_j$, for every $1\leq j\leq k$. 
Put $z=\sum_{i\in I}z_i$.  Then $z\in\mathcal N_{pMp}(P)'\cap pMp$, $Pz\prec_M^sQ$ and $Pz\nprec_MR_j$, for every $1\leq j\leq k$. Using that  $\mathcal N_{pMp}(P)'\cap pMp$ is abelian as it is contained in $\mathcal Z(P'\cap pMp)$, it follows that $z$ is the maximal projection with these properties.

\end{remark}

\begin{proof}[Proof of Proposition~\ref{Zariski_hull}]
We first claim that if $\Lambda_1,\Lambda_2\subset\Gamma$ are distinct full subgraphs, then $p_{\Lambda_1}p_{\Lambda_2}=0$. Assume by contradiction that  $p'=p_{\Lambda_1}p_{\Lambda_2}\not=0$. Then $Pp'\prec^s_{M_{\Gamma}}M_{\Lambda_1}$ and 
$Pp'\prec^s_{M_{\Gamma}}M_{\Lambda_2}$. 
Corollary~\ref{intersection} implies that $Pp'\prec^s_{M_\Gamma}M_{\Lambda_1\cap\Lambda_2}$. Hence $Pp_{\Lambda_1}\prec_{M_\Gamma}M_{\Lambda_1\cap\Lambda_2}$ and $Pp_{\Lambda_2}\prec_{M_\Gamma}M_{\Lambda_1\cap\Lambda_2}$. The definitions of $p_{\Lambda_1}$ and $p_{\Lambda_2}$ imply that $\Lambda_1=\Lambda_1\cap\Lambda_2=\Lambda_2$,
which gives a contradiction and proves the claim.

Next, we claim that $\sum_{\text{$\Lambda$}} p_\Lambda=p$. Assume by contradiction that
  $q=p-\sum_\Lambda p_\Lambda\in \mathcal N_{pM_\Gamma p}(P)'\cap pM_{\Gamma}p$ is nonzero. 
Let $\Lambda\subset\Gamma$ be a full subgraph of minimal cardinality with $Pq\prec_{M_\Gamma}M_\Lambda$.  Proposition \ref{results}(3) gives a nonzero projection $r\in\mathcal N_{qM_\Gamma q}(Pq)'\cap qM_\Gamma q$ such that $Pr\prec_{M}^sM_\Lambda$. If $u\in\mathcal N_{pM_\Gamma p}(P)$, then $uq\in\mathcal N_{qM_\Gamma q}(Pq)$ and so $ur=(uq)r=r(uq)=r(qu)=ru$. This shows that $r\in\mathcal N_{pM_\Gamma p}(P)'\cap pM_\Gamma p$. As $Pr\nprec_{M_\Gamma}M_{\Lambda'}$, for every full subgraph $\Lambda'\subsetneq\Lambda$, and $r\leq p-p_\Lambda$, this contradicts the maximality of $p_\Lambda$. 

Finally, to prove the moreover assertion assume that $Pp_\Lambda\prec_{M_{\Gamma}}M_{\Lambda'}$, for a full subgraph $\Lambda'\subset\Gamma$. 
As above, Proposition \ref{results}(3)  gives a nonzero projection $s\in  \mathcal N_{pM_\Gamma p}(P)'\cap pM_{\Gamma}p$ with $s\leq p_\Lambda$ and $Ps\prec_{M_\Gamma}^sM_{\Lambda'}$.  Since $s\leq p_\Lambda$, we have $Ps\prec_{M_\Gamma}^sM_{\Lambda}$. Thus, Proposition \ref{intersection} gives that $Ps\prec_{M_\Gamma}^sM_{\Lambda\cap\Lambda'}$ and hence $Pp_\Lambda\prec_{M_\Gamma}M_{\Lambda\cap \Lambda'}$. The definition of $p_\Lambda$ implies that $\Lambda\cap\Lambda'=\Lambda$ and hence $\Lambda\subset\Lambda'$, as desired.
\end{proof}

\subsection{Normalizers}

We continue with consequences of Proposition \ref{intertwiners} concerning normalizers inside graph product von Neumann algebras. 
The first assertion of the following proposition is a direct consequence of Proposition~\ref{intertwiners} in the case where $\Gamma_1=\Lambda$ and $\Gamma_2=\Lambda\cup\Lambda^\perp$. Both assertions are also consequences of \cite[Proposition~5.8(2)]{BCC24} -- we omit the proof of the second, as we will prove a stronger version in Proposition~\ref{normalizer'} below.

\begin{proposition}[\!\!{\cite[Proposition 5.8]{BCC24}}]\label{normalizer}
Assume the setting from Notation \ref{setup}.
Let $P\subset pM_\Gamma p$ be a von Neumann subalgebra and $\Lambda\subset\Gamma$ a full subgraph. 
Then we have:
\begin{enumerate}
\item Assume that $p\in M_\Lambda$, $P\subset pM_\Lambda p$ and $P\nprec_{M_{\Lambda}}M_{\Lambda'}$, for every proper full subgraph $\Lambda'\subsetneq\Lambda$.  If $v\in pM_\Gamma$ satisfies $Pv\subset \sum_{i=1}^kv_iM_{\Lambda\cup\Lambda^\perp}$, for some $k\geq 1$ and $v_1,\dots,v_k\in M_\Gamma$, then $v\in M_{\Lambda\cup\Lambda^\perp}$. In particular,
 $(\mathcal Q\mathcal N_{pM_\Gamma p}(P))''\subset M_{\Lambda\cup\Lambda^\perp}$.
\item Assume that $P\prec_{M_{\Gamma}}^s M_\Lambda$ and that $P\nprec_{M_{\Gamma}}M_{\Lambda'}$, for every full subgraph $\Lambda'\subsetneq\Gamma$. Then $\mathcal N_{pM_\Gamma p}(P)''\prec_{M_\Gamma}^sM_{\Lambda\cup\Lambda^\perp}$. \qedhere
\end{enumerate}
\end{proposition}

For further use, we establish the following strengthening of Proposition \ref{normalizer}(2).
\begin{proposition}\label{normalizer'}
Assume the setting from Notation \ref{setup}.
Let $P\subset pM_\Gamma p$ be a von Neumann subalgebra, $\Lambda\subset\Gamma$ a full subgraph and $e\in M_{\Lambda\cup\Lambda^\perp}'\cap M_\Gamma$ a nonzero projection. Assume that $P\prec_{M_\Gamma}^sM_{\Lambda}e$ and $P\nprec_{M_{\Gamma}}M_{\Lambda'}$, for every proper full subgraph $\Lambda'\subsetneq\Lambda$.

Then $\mathcal N_{pM_\Gamma p}(P)''\prec^s_{M_\Gamma}M_{\Lambda\cup\Lambda^\perp}e$. \end{proposition}

\begin{proof} 
We adapt the proof of \cite[Proposition 5.8(2)]{BCC24}, following closely the proofs of \cite[Lemma 3.5]{Po06b} and \cite[Lemma 9.4]{Io15}.
Let $p'\in \mathcal N_{pM_\Gamma p}(P)'\cap pM_\Gamma p$ be a nonzero projection. Since $Pp'\prec_{M_\Gamma}M_\Lambda e$ and $P\nprec_{M_{\Gamma}}M_{\Lambda'}$, for every $\Lambda'\subsetneq\Lambda$, by Proposition \ref{cutting_down}(2), we can find projections $q\in P$, $r\in M_\Lambda$, a nonzero partial isometry $v\in reM_\Gamma qp'$ and a $*$-homomorphism $\theta:qPq\rightarrow rM_\Lambda r$ such that 
\begin{itemize}
\item $\theta(x)v=vx$, for every $x\in qPq$, and \item $\theta(qPq)\nprec_{M_\Lambda}M_{\Lambda'}$, for every proper full subgraph $\Lambda'\subsetneq\Lambda$.  
\end{itemize}

Let $u\in\mathcal N_{pM_\Gamma p}(P)$. 
Let $z\in\mathcal Z(P)$ be a projection such that $z=\sum_{i=1}^m w_iw_i^*$, for some partial isometries $w_i\in P$ satisfying $w_i^*w_i\leq q$, for every $1\leq i\leq m$. For $m\geq 1$, this condition is equivalent to having $z\leq m\text{E}_{\mathcal Z(P)}(q)$ and thus
 holds for $z={\bf 1}_{[\frac{1}{m},\infty)}(\text{E}_{\mathcal Z(P)}(q))$.
Then $$qzuqz(qPq)\subset qzuP=qzPu=qPzu\subset\sum_{i=1}^m(qPw_i)(w_i^*u)\subset \sum_{i=1}^m qPq(w_i^*u), $$ and similarly $(qPq)qzuqz\subset\sum_{i=1}^muw_i(qPq)$.
Thus, we have that $qzuqz\in \mathcal Q\mathcal N_{qM_\Gamma q}(qPq)$.

Next, if $x\in \mathcal Q\mathcal N_{qM_\Gamma q}(qPq)$, then $vxv^*\in \mathcal Q\mathcal N_{rM_\Gamma r}(\theta(qPq))$. 
Hence $v(qzuqz)v^*\in \mathcal Q\mathcal N_{rM_\Gamma r}(\theta(qPq))$. 
On the other hand, by Proposition \ref{normalizer}(1), we have that $\mathcal Q\mathcal N_{rM_\Gamma r}(\theta(qPq))\subset M_{\Lambda\cup\Lambda^\perp}$. Thus, we derive that $v(qzuqz)v^*\in M_{\Lambda\cup\Lambda^\perp}$. Since the projection $z$ can be taken to approximate arbitrarily well the central support of $q$ in $P$, we conclude that $vuv^*=v(quq)v^*\in M_{\Lambda\cup\Lambda^\perp}$. Since this holds for every $u\in\mathcal N_{pM_\Gamma p}(P)$, we deduce that  $t:=vv^*\in M_{\Lambda\cup\Lambda^\perp}$ and $v\mathcal N_{pM_\Gamma p}(P)''v^*\subset tM_{\Lambda\cup\Lambda^\perp}t$. Since $p_1:=v^*v\in (qPq)'\cap qM_{\Gamma}q$ and $p_1\leq p'$, we have $p_1\in \mathcal N_{pM_\Gamma p}(P)''p'$. 
Therefore, the map $\rho:p_1(\mathcal N_{pM_\Gamma p}(P)''p')p_1\to t M_{\Lambda\cup\Lambda^\perp}t$ sending $x$ to $vxv^*$
is a $*$-homomorphism which satisfies $\rho(x)v=vx$, for every $x\in p_1(\mathcal N_{pM_\Gamma p}(P)''p')p_1$, and we have $vv^*=t$ and $v^*v=p_1$. As $t\leq re\leq e$, we also have $t=te\in M_{\Lambda\cup\Lambda^\perp}e$.
Hence, $\mathcal N_{pM_\Gamma p}(P)''p'\prec_{M_\Gamma}M_{\Lambda\cup\Lambda^\perp}e$, which finishes the proof.
\end{proof}

\subsection{Additional consequences of Proposition \ref{intertwiners}}
In the proofs of Theorems \ref{arbitrary} and \ref{amenable}, given a von Neumann subalgebra $P\subset pM_\Gamma p$, we will need to analyze not only the full subgraphs $\Lambda\subset\Gamma$ such that $P\prec_{M_\Gamma}M_\Lambda$, but also the projections $s\in M_\Lambda'\cap M_\Gamma$ such that $P\prec_{M_\Gamma}M_{\Lambda} s$. We continue with two corollaries of Proposition \ref{intertwiners} dealing with such projections.

\begin{corollary}\label{right_corner}
Assume the setting from Notation \ref{setup}.
Let $P\subset pM_\Gamma p$ be a von Neumann subalgebra, $\Gamma_1,\Gamma_2\subset\Gamma$ full  subgraphs with $\Gamma_1\subset\Gamma_2$, and $e\in M_{\Gamma_2}'\cap M_\Gamma$ a projection. Assume that $P\prec_{M_\Gamma}^sM_{\Gamma_1}$ and $P\prec_{M_\Gamma}^sM_{\Gamma_2}e$.

Then $P\prec_{M_\Gamma}^sM_{\Gamma_1}e$.
\end{corollary}

The proof of Corollary~\ref{right_corner} elaborates on the proof of Corollary~\ref{intersection} that it extends.

\begin{proof}
We will establish the following claim.

\begin{claim}
    Let $Q\subset pM_\Gamma p$ be a von Neumann subalgebra, and let $\Lambda_1,\Lambda_2\subset\Gamma$ be two full subgraphs. Let $e\in M'_{\Lambda_2}\cap M_\Gamma$ be a projection. Assume that $Q\prec^s_{M_\Gamma}M_{\Lambda_1}$ and $Q\prec^s_{M_\Gamma}M_{\Lambda_2}e$, and that for every proper full subgraph $\Lambda'_1\subsetneq\Lambda_1$, one has $Q\not\prec_{M_\Gamma}M_{\Lambda'_1}$.

    Then $\Lambda_1\subset\Lambda_2$ and $Q\prec_{M_{\Gamma}}M_{\Lambda_1}e$.
\end{claim}

Before proving the claim, let us first explain how the corollary follows. Let $P\subset pM_\Gamma p$ be as in the statement, and let $p'\in P'\cap pM_\Gamma p$. Let $\Lambda_1\subset\Gamma_1$ be a minimal subgraph such that $Pp'\prec_{M_\Gamma}M_{\Lambda_1}$. Then we can find a non-zero projection $p''\in P'\cap pM_\Gamma p$ with $p''\leq p'$ such that $Pp''\prec^s_{M_\Gamma}M_{\Lambda_1}$. We also have $Pp''\prec^s_{M_\Gamma}M_{\Gamma_2}e$. By choice of $\Lambda_1$, for every proper full subgraph $\Lambda'_1\subsetneq\Lambda_1$, we have $Pp''\not\prec_{M_\Gamma}M_{\Lambda'_1}$.  Applying the claim to $Q=Pp''$ (with $\Lambda_2=\Gamma_2$), we deduce that $Pp''\prec_{M_{\Gamma}}M_{\Lambda_1}e$, in particular $Pp'\prec_{M_\Gamma}M_{\Lambda_1}e$. Since this is true for any $p'\in P'\cap pM_\Gamma p$, we have proved that $P\prec^s_{M_\Gamma}M_{\Lambda_1}e$, in particular $P\prec^s_{M_{\Gamma}}M_{\Gamma_1}e$, as desired.

We now prove the above claim.
Since $Q\prec_{M_\Gamma}M_{\Lambda_2}e$, Proposition \ref{cutting_down}(1) provides projections $q_2\in Q,r_2\in M_{\Lambda_2}$, a nonzero partial isometry $v_2\in r_2eM_\Gamma q_2$ and a $*$-homomorphism $\theta_2:q_2Qq_2\rightarrow r_2M_{\Lambda_2}r_2$ such that 
\begin{enumerate}
\item $\theta_2(y)v_2=v_2y$, for every $y\in q_2Qq_2$.  
\end{enumerate}
Since $v_2^*v_2\in (q_2Qq_2)'\cap q_2M_\Gamma q_2=(Q'\cap pM_\Gamma p)q_2$, we can find a projection $q'\in Q'\cap pM_\Gamma p$ such that $v_2^*v_2=q_2q'$. Let $r=v_2v_2^*$. Then $r\in\theta_2(q_2Qq_2)'\cap r_2M_{\Gamma}r_2$ and $r\leq r_2e$.

We have $q_2Qq_2q'\prec_{M_\Gamma}M_{\Lambda_1}$ and $q_2Qq_2q'\nprec_{M_\Gamma}M_{\Lambda_1'}$, for every full subgraph $\Lambda_1'\subsetneq\Lambda_1$. Therefore, Proposition~\ref{cutting_down}(2) yields projections $q_1\in q_2Qq_2$, $r_1\in M_{\Lambda_1}$, a nonzero partial isometry $v_1\in r_1M_\Gamma q_1q'$ and a $*$-homomorphism $\theta_1:q_1Qq_1\rightarrow r_1M_{\Lambda_1}r_1$ such that 
\begin{itemize}
\item[(2)] $\theta_1(x)v_1=v_1x$, for every $x\in q_1Qq_1$, and 
\item [(3)] $\theta_1(q_1Qq_1)\nprec_{M_{\Lambda_1}}M_{\Lambda_1'}$, for every full subgraph $\Lambda_1'\subsetneq\Lambda_1$.
\end{itemize}

Let $v=v_1v_2^*$. Using (1) and (2), we get that 
\begin{equation}\label{theta12}\text{$\theta_1(x)v=v\theta_2(x)$, for every $x\in q_1Qq_1$.}\end{equation}
Since $vv^*=v_1(v_2^*v_2)v_1^*=v_1(q_2q')v_1^*=v_1v_1^*\not=0$, we get that $v\not=0$.  Let $R:=\theta_1(q_1Qq_1)\subset r_1M_{\Lambda_1}r_1$. Then $Rv\subset vM_{\Lambda_2}$ and  $R\nprec_{M_{\Lambda_1}}M_{\Lambda_1'}$, for every full subgraph $\Lambda_1'\subsetneq\Lambda_1$,  by (3). Proposition ~\ref{intertwiners} thus implies that $\Lambda_1\subset\Lambda_2$, and that $v$ lies in the $\|\cdot\|_2$-closure of the linear span of $M_{\Lambda_1^\perp}M_{\Lambda_2}$. 
Since every element of $M_{\Lambda_1^\perp}$ is a sum of four unitaries, it follows that $v$ belongs to the $\|\cdot\|_2$-closure of the linear span of $\{uM_{\Lambda_2}\mid u\in\mathcal U(M_{\Lambda_1^\perp})\}$. 

For $u\in\mathcal U(M_{\Lambda_1^\perp})$, define $\xi_u=\text{E}_{M_{\Lambda_2}}(u^*v)r$. 
Let $x\in q_1Qq_1$. By using \eqref{theta12} and that $u\in M_{\Lambda_1}'\cap M_{\Gamma}$ we get that  $\theta_1(x)(u^*v)=(u^*v)\theta_2(x)$. Applying $\text{E}_{M_{\Lambda_2}}$ and using that $\Lambda_1\subset\Lambda_2$, hence $M_{\Lambda_1}\subset M_{\Lambda_2}$, gives that $\theta_1(x)\text{E}_{M_{\Lambda_2}}(u^*v)=\text{E}_{M_{\Lambda_2}}(u^*v)\theta_2(x)$.
Since $r\in\theta_2(q_2Qq_2)'\cap r_2M_{\Gamma}r_2$, we deduce that 
\begin{equation}\label{xi_u}\text{$\theta_1(x)\xi_u=\xi_u\theta_2(x)$, for every $x\in q_1Qq_1$ and $u\in\mathcal U(M_{\Lambda_1^\perp})$.}\end{equation}

We next claim that there exists $u\in\mathcal U(M_{\Lambda_1^\perp})$ such that $\xi_u\not=0$. 
Since $v\not=0$ and $v$ belongs to the $\|\cdot\|_2$-closure of the linear span of $\{uM_{\Lambda_2}\mid u\in\mathcal U(M_{\Lambda_1^\perp})\}$, we can find $u\in\mathcal U(M_{\Lambda_1^\perp})$ such that $\text{E}_{M_{\Lambda_2}}(u^*v)\not=0$.  Since $v^*v=v_2(v_1^*v_1)v_2^*\leq v_2v_2^*=r$, we get that $vr=v$ and so $rv^*=v^*$. Thus,  $$\tau(\xi_uv^*u)=\tau(\text{E}_{M_{\Lambda_2}}(u^*v)rv^*u)=
\tau(\text{E}_{M_{\Lambda_2}}(u^*v)v^*u)=\|\text{E}_{M_{\Lambda_2}}(u^*v)\|_2^2\not=0,$$ and therefore $\xi_u\not=0$.

Let $u\in\mathcal U(M_{\Lambda_1^\perp})$ such that $\xi_u\not=0$. 
Since $e\in M_{\Lambda_2}'\cap M_{\Gamma}$ and $r\leq r_2e\leq e$, we get that $e\xi_u=\xi_u$.
By using this fact,  \eqref{xi_u} and (1), we derive that 
\begin{equation}\label{final_inter}
\text{$(\theta_1(x)e)(\xi_uv_2)=(\theta_1(x)\xi_u)v_2=(\xi_u\theta_2(x))v_2=(\xi_uv_2)x$, for every $x\in q_1Qq_1$.}
\end{equation}
Let $w=\xi_uv_2$ and define a $*$-homomorphism $\rho:q_1Qq_1\rightarrow r_1M_{\Lambda_1}r_1e$ by $\rho(x)=\theta_1(x)e$. Then \eqref{final_inter} implies that $\rho(x)w=wx$, for every $x\in q_1Qq_1$. Thus, in order to conclude that $Q\prec_{M_\Gamma}M_{\Lambda_1}e$ and  finish the proof, it remains to argue that 
$\rho(q_1)wq_1\not=0$, i.e.\ that $\rho(q_1)w\not=0$. 

  Note that since $vv^*=v_1v_1^*\leq r_1$, we have $r_1v=v$.
Since $r_1\in M_{\Lambda_1}\subset M_{\Lambda_2}$, $u\in M_{\Lambda_1^\perp}$ and $r_1v=v$, we  get that $r_1\xi_u=r_1\text{E}_{M_{\Lambda_2}}(u^*v)r=\text{E}_{M_{\Lambda_2}}(r_1u^*v)r=\text{E}_{M_{\Lambda_2}}(u^*r_1v)r=\xi_u$.
By using this fact we get that $\rho(q_1)w=\theta_1(q_1)(e\xi_u)v_2=r_1\xi_uv_2=\xi_uv_2$. 
Finally, since $\xi_ur=\xi_u$ and $v_2v_2^*=r$, we get that $(\xi_uv_2)(\xi_uv_2)^*=\xi_u(v_2v_2^*)\xi_u^*=\xi_ur\xi_u^*=\xi_u\xi_u^*$. Since $\xi_u\not=0$, we derive that $\xi_uv_2\not=0$ and thus $\rho(q_1)w\not=0$, which finishes the proof.
\end{proof}

\begin{corollary}\label{center_corner}
Assume the setting from Notation \ref{setup}.
Let $P\subset pM_\Gamma p$ be a von Neumann subalgebra, $\Lambda\subset\Gamma$ be a full  subgraph and $z\in M_{\Lambda\cup\Lambda^\perp}'\cap M_\Gamma$ be a nonzero projection.
 Assume that $P\prec^s_{M_\Gamma}M_{\Lambda}z$ and $P\nprec_{M_{\Gamma}}M_{\Lambda'}$, for every full subgraph $\Lambda'\subsetneq\Lambda$.

Then $P\nprec_{M_\Gamma}M_{\Lambda}(1-z)$.
\end{corollary}

\begin{proof}
Assume by contradiction that $P\prec_{M_\Gamma}M_\Lambda (1-z)$.  By Proposition \ref{cutting_down}(1) we find projections $p_2\in P,q_2\in M_\Lambda$, a nonzero partial isometry $v_2\in q_2(1-z)M_\Gamma p_2$ and a $*$-homomorphism $\theta_2:p_2Pp_2\rightarrow q_2M_\Lambda q_2$ such that $\theta_2(x)v_2=v_2x$, for every $x\in p_2Pp_2$. 
Since $v_2^*v_2\in (p_2Pp_2)'\cap p_2M_\Gamma p_2=(P'\cap pM_\Gamma p)p_2$, we can find a projection $q'\in P'\cap pM_\Gamma p$ such that $v_2^*v_2=p_2q'$. 

As $P\prec^s_{M_\Gamma}M_{\Lambda}z$ and $P\nprec_{M_{\Gamma}}M_{\Lambda'}$, 
for every full subgraph $\Lambda'\subsetneq\Lambda$, we have $p_2Pp_2q'\prec_{M_\Gamma}M_{\Lambda}z$ and  $p_2Pp_2q'\nprec_{M_{\Gamma}}M_{\Lambda'}$, 
for every full subgraph $\Lambda'\subsetneq\Lambda$.
Proposition \ref{cutting_down}(2) gives projections $p_1\in p_2Pp_2$, $q_1\in M_{\Lambda}$, a nonzero partial isometry $v_1\in q_1zMp_1q'$ and a $*$-homomorphism $\theta_1:p_1Pp_1\rightarrow q_1M_\Lambda q_1$ satisfying $\theta_1(x)v_1=v_1x$, for every $x\in p_1Pp_1$, and $Q:=\theta_1(p_1Pp_1)\nprec_{M_{\Lambda}}M_{\Lambda'}$, for every full subgraph $\Lambda'\subsetneq\Lambda$.

Let $v=v_1v_2^*$. then $\theta_1(x)v=v\theta_2(x)$, for every $x\in p_1Pp_1$. Hence, we have $Qv\subset vM_\Lambda$. Since $vv^*=v_1(v_2^*v_2)v_1^*=v_1(p_2q')v_1^*=v_1v_1^*\not=0$, we get that $v\not=0$.
Since $Q\nprec_{M_{\Lambda}}M_{\Lambda'}$, for every full subgraph $\Lambda'\subsetneq\Lambda$, by applying Proposition \ref{normalizer}(1) we conclude that $v\in M_{\Lambda\cup\Lambda^\perp}$.
On the other hand, we have $vv^*=v_1v_1^*\leq z$ and $v^*v\leq v_2v_2^*\leq 1-z$. Since  $z\in M_{\Lambda\cup\Lambda^\perp}'\cap M_\Gamma$, $v$ and $z$ commute and thus $v=zv(1-z)=z(1-z)v=0$, which gives the desired contradiction.
\end{proof}

\subsection{A unitary conjugacy result}
We continue with a result improving \cite[Theorem 5.11]{BCC24}.

\begin{theorem}\label{uni_conj}
Assume the setting from Notation \ref{setup}.
Let $\Lambda\subset\Gamma$ be a full subgraph and $P\subset p(\mathbb M_n(\mathbb C)\otimes M_\Gamma)p$ be a von Neumann subalgebra, for some $n\geq 1$. Assume that $P\prec_{\mathbb M_n({\mathbb C})\otimes M_\Gamma}^s\mathbb M_n(\mathbb C)\otimes M_{\Lambda}$ and $P\nprec_{\mathbb M_n(\mathbb C)\otimes M_{\Gamma}}\mathbb M_n(\mathbb C)\otimes M_{\Lambda'}$, for every full subgraph $\Lambda'\subsetneq\Lambda$. Assume also that $M_{\Lambda\cup\Lambda^\perp}$ is a II$_1$ factor.

Then there exists a unitary $u\in \mathbb M_n(\mathbb C)\otimes M_\Gamma$ such that 
\begin{itemize}
\item $uPu^*\subset \mathbb M_n(\mathbb C)\otimes M_{\Lambda\cup\Lambda^\perp}$;
\item $uPu^*\prec^s_{\mathbb M_n(\mathbb C)\otimes M_{\Lambda\cup\Lambda^\perp}}\mathbb M_n(\mathbb C)\otimes M_\Lambda$;
\item if in addition $P$ and  $P'\cap p(\mathbb M_n(\mathbb C)\otimes M_\Gamma)p$ are factors, then there are a unitary $u'\in \mathbb M_n(\mathbb C)\otimes M_{\Lambda\cup\Lambda^\perp}$ and a decomposition $M_{\Lambda\cup\Lambda^\perp}=M_\Lambda^t\overline{\otimes}M_{\Lambda^\perp}^{1/t}$, for some $t>0$,  such that $u'uPu^*u'^*\subset \mathbb M_n(\mathbb C)\otimes M_\Lambda^t$. \qedhere
\end{itemize}
\end{theorem}

In the proof of Theorem \ref{uni_conj}, we will assume that $n=1$. The fact that one can reduce to this case is a consequence of the following remark.

\begin{remark}\label{amplify_results}
Let $(N,\tau)$ be any tracial von Neumann algebra. Then the results proved so far in this section apply to von Neumann subalgebras $P\subset p(M_\Gamma\overline{\otimes}N)p$, where in all of the statements  we replace $M_{\Gamma'}$ by $M_{\Gamma'}\overline{\otimes}N$, for every full subgraph $\Gamma'\subset\Gamma$. This is because we can realize  $M_\Gamma\overline{\otimes}N$ as the graph product von Neumann algebra $M_{\widetilde\Gamma}$, where $\widetilde\Gamma=\Gamma\circ\{v\}$, for a vertex $v$ with $M_v=N$.

In particular, the above results apply when $P\subset p(M_\Gamma\otimes\mathbb M_k(\mathbb C))p$, for some $k\geq 1$. Note that if  $\Gamma'\subset\Gamma$ is a full subgraph, then $P\prec_{M_\Gamma\otimes\mathbb M_k(\mathbb C)}M_{\Gamma'}\otimes\mathbb M_k(\mathbb C)$ if and only if $P\prec_{M_\Gamma\otimes\mathbb M_k(\mathbb C)}M_{\Gamma'}\otimes\mathbb C1$.
\end{remark}

To prove Theorem \ref{uni_conj}, we will use the following fact. The proof of this fact is standard and is left to the reader (see, e.g., the proof of \cite[Theorem 5.1]{IPP08}).

\begin{fact}\label{conjugacy_fact}
Let $(M,\tau)$ be a tracial von Neumann algebra, $P\subset pMp$ a von Neumann subalgebra and $Q\subset M$ be a von Neumann subalgebra which is a II$_1$ factor. 
Let $\mathcal F$ be the family of  projections $p'\in P'\cap pMp$ for which there is a unitary $u\in M$ such that $uPp'u^*\subset Q$.
Then the following hold:
\begin{enumerate}
    \item If $p_0\in P$ is a projection for which there is a unitary $v\in M$ such that $vp_0Pp_0v^*\subset Q$, then the central support of $p_0$ in $P$ belongs to $\mathcal F$.
     \item If $p'\in\mathcal F$, then the central support of $p'$ in $P'\cap pMp$ belongs to $\mathcal F$.
    \item If $\{p_i'\}_{i\in I}\subset\mathcal F$ is a family of pairwise orthogonal projections, then $\sum_{i\in I}p_i'\in\mathcal F$.
\end{enumerate}
\end{fact}

\begin{proof}[Proof of Theorem~\ref{uni_conj}] We assume that $n=1$, so that $P\subset pM_\Gamma p$.
Since $M_{\Lambda\cup\Lambda^\perp}$ is a II$_1$ factor, there is a maximal projection  $r\in\mathcal Z(P'\cap pM_\Gamma p)$ for which we can find a unitary $u\in M_\Gamma$ such that $uPru^*\subset M_{\Lambda\cup\Lambda^\perp}$. Indeed, using Fact \ref{conjugacy_fact}(3), we can take $r=\sum_{i\in I}r_i$, where $\{r_i\}_{i\in I}\subset\mathcal Z(P'\cap pM_\Gamma p)$ is a maximal, with respect to inclusion, family of pairwise orthogonal projections such that for every $i\in I$ there is a unitary $u_i\in M_\Gamma$ with $u_iPr_iu_i^*\subset M_{\Lambda\cup\Lambda^\perp}$. 

Let $r'=p-r\in\mathcal Z(P'\cap pM_\Gamma p)$ and assume by contradiction that $r'\not=0$. Then $Pr'\prec_{M_\Gamma}M_\Lambda$ and $Pr'\nprec_{M_\Gamma}M_{\Lambda'}$, for every full subgraph $\Lambda'\subsetneq\Lambda$. 

By Proposition \ref{results}(6), there are projections $p_1\in P, q\in M_{\Lambda}$, a nonzero partial isometry $v\in qM_\Gamma p_1r'$ and a $*$-homomorphism $\theta:p_1Pp_1r'\rightarrow qM_{\Lambda}q$ such that 
\begin{itemize}
\item $\theta(x)v=vx$, for every $x\in p_1Pp_1p'$, and \item $\theta(p_1Pp_1r')\nprec_{M_\Lambda}M_{\Lambda'}$, for every  full subgraph $\Lambda'\subsetneq\Lambda$.  
\end{itemize}
Since $v^*v\in (p_1Pp_1r')'\cap p_1r'M_\Gamma p_1r'=(P'\cap pM_\Gamma p)r'p_1$, we can find a projection $r''\in P'\cap pM_\Gamma p$ such that $r''\leq r'$ and $v^*v=p_1r''$. 
By Proposition \ref{normalizer}(1), we get that $\theta(p_1Pp_1r')'\cap qM_{\Gamma}q\subset qM_{\Lambda\cup\Lambda^\perp}q$ and thus $vv^*\in qM_{\Lambda\cup\Lambda^\perp}q$.
Hence we conclude that
\begin{equation}\label{corner}
vp_1Pp_1r''v^*=\theta(p_1Pp_1r')vv^*\subset M_{\Lambda\cup\Lambda^\perp}.
\end{equation}

Let $z'\in \mathcal Z(Pr'')=\mathcal Z(P)r''$ be the central support of $p_1r''$ in $Pr''$.  Then $z'\in P'\cap pM_\Gamma p$.  Let $z\in\mathcal Z(P'\cap pM_\Gamma p)$ be the central support of $z'$ in $P'\cap pM_\Gamma p$.
We have $p_1r''\leq r''\leq r'=p-r$. Since $z'=\vee_{u\in\mathcal{U}(Pr'')}up_1r''u^{\ast}$, and  $up_1r''u^*\leq p-r$ for every $u\in\mathcal{U}(Pr'')$, we have $z'\leq p-r$. Taking the central support in $P'\cap pM_\Gamma p$, we deduce that $z\leq p-r$.
Since $v^*v=p_1r''$ and $M_{\Lambda\cup\Lambda^\perp}$ is a factor, 
using Fact \ref{conjugacy_fact},  (1) and (2), we find a unitary $w\in M_\Gamma$ such that $wPzw^*\subset M_{\Lambda\cup\Lambda^\perp}$. 
By using again that $M_{\Lambda\cup\Lambda^\perp}$ is a factor and  Fact~\ref{conjugacy_fact}(3)  we find a unitary $\xi\in M_\Gamma$ such that $\xi P(r+z)\xi^*\subset M_{\Lambda\cup\Lambda^\perp}$. As $z\not=0$, this contradicts the maximality of $r$. Thus, $r'=0$, hence, $r=p$ and $uPu^*\subset qM_{\Lambda\cup\Lambda^\perp}q$, where $q=upu^*$.

Put $Q=uPu^*$.
Since $Q\subset qM_{\Lambda\cup\Lambda^\perp}q$ and $Q\prec_{M_\Gamma}^sM_\Lambda$,  Lemma \ref{descend} implies that
$Q\prec^s_{M_{\Lambda\cup\Lambda^\perp}}M_\Lambda$.

To prove the last assertion of the theorem, assume that $P$ and $P'\cap pM_\Gamma p$. 
Thus, $Q$ and $Q'\cap qM_\Gamma q$ are factors. Since $Q\nprec_{M_\Gamma}M_{\Lambda'}$, for every full subgraph $\Lambda'\subsetneq\Lambda$, by Proposition~\ref{normalizer}(1) we deduce that $Q'\cap qM_\Gamma q\subset M_{\Lambda\cup\Lambda^\perp}$. Thus, $Q'\cap qM_{\Lambda\cup\Lambda^\perp}q=Q'\cap qM_\Gamma q$ is a factor. 
Since $M_{\Lambda\cup\Lambda^\perp}=M_\Lambda\overline{\otimes}M_{\Lambda^\perp}$, Lemma~\ref{tensor_product}(1) gives the moreover assertion.
\end{proof}

\subsection{A criterion for collapsibility} We end this section with the following criterion for a subgraph to be collapsible 
(see Section~\ref{graph_notions} for the definition). This criterion will be needed in the proofs of Theorems \ref{amenable} and \ref{factors}.

\begin{proposition}\label{collapsible}
Let $\Gamma,\Lambda$ be finite simple graphs. Let $(M_v,\tau_v)_{v\in\Gamma}, (N_w,\tau_w)_{w\in\Lambda}$ be families of diffuse tracial von Neumann algebras such that the associated graph product von Neumann algebras $M_\Gamma=*_{v,\Gamma}M_v$ and $N_\Lambda=*_{w,\Lambda}N_w$ are II$_1$ factors. 
Let $v\in\Gamma$ and $C\subset\Lambda$ be a full subgraph. 
Let $\theta:M_\Gamma\rightarrow rPr$ be a $*$-isomorphism, where $P=\mathbb M_n(\mathbb C)\otimes N_\Lambda$, for some $n\geq 1$, and $r\in P$ is a projection.
Identify $M_\Gamma=rPr$, via $\theta$, and $N_\Lambda=sPs$, where $s=e_{1,1}\otimes 1$.

Assume that  $M_vp\prec_P^s N_{C\cup C^\perp}q$ and $N_Cq\prec_P^s M_vp$, for some nonzero projections $p\in \mathcal Z(M_v'\cap rPr)$ and $q\in \mathcal Z(N_C'\cap sPs)$.

Then $C$ is a collapsible subgraph of $\Lambda$. Moreover, $M_{\emph{st}(v)}p\prec_P^s N_{C\cup C^\perp}q$.
\end{proposition}

\begin{proof} 
We will treat the case where $n=1$ (so in particular $s=1$). The general case can be handled using Remark~\ref{amplify_results}; note that the argument from Remark~\ref{amplify_results}, working with the graph $\tilde{\Lambda}=\Lambda\circ\{\ast\}$, introduces a vertex group which is not diffuse, but the following proof can be checked to work in this setting.

We first note that $\mathcal Z(M_v'\cap rPr)=\mathcal Z(M_{\text{st}(v)}'\cap rPr)$. Indeed, since $\text{st}(v)^\perp=\emptyset$, Lemma~\ref{basic}(1) implies that $\mathcal Z(M_v'\cap rPr)=\mathcal Z(M_{\text{st}(v)}'\cap rPr)=\mathcal Z(M_v)\overline{\otimes}\mathcal Z(M_B)$, where $B$ is the maximal clique factor of $\text{lk}(v)$. In particular $p\in\mathcal Z(M_{\text{st}(v)}'\cap rPr)$.

Next, we claim that 
\begin{equation}\label{all_projections}\text{$N_Cq\prec_P M_vp'$, for every nonzero projection $p'\in\mathcal Z(M_{\text{st}(v)}'\cap rPr)p$.}\end{equation}
 Assume by contradiction that $N_Cq\nprec_P M_vp'$, for some nonzero projection $p'\in\mathcal Z(M_{\text{st}(v)}'\cap rPr)p$. Since $N_Cq\prec_P^sM_vp$, we get that $N_Cq\prec_P^sM_v(p-p')$. 
Since $N_C$ is diffuse, we have $N_Cq\nprec M_\emptyset=\mathbb C1$. 

We now observe that \begin{equation}\label{Ccup1} N_{C\cup C^\perp}q\prec_P^sM_{\text{st}(v)}(p-p') \end{equation} 

To see this, let $r_0\in N_C$ be any projection with $\tau(r_0q)\leq \tau(r)$. Since $P$ is a II$_1$ factor, we can find a partial isometry $\delta\in P$ with $\delta^*\delta=r_0q$ and $\delta\delta^*\leq r$, such that $\delta r_0 N_Cr_0q\delta^*\subset rPr=M_\Gamma$. 
Since $N_Cq\prec_P^sM_v(p-p')$, we get that $\delta r_0 N_Cr_0q\delta^*\prec_{M_\Gamma}^sM_v(p-p')$.
Since $p-p'\in M_{\text{st}(v)}'\cap rPr$, Proposition \ref{normalizer'} gives that $\delta r_0 N_{C\cup C^\perp}r_0q\delta^*\prec_{M_\Gamma}^sM_{\text{st}(v)}(p-p')$. Since this holds for any projection $r_0\in N_C$ satisfying $\tau(r_0q)\leq \tau(r)$, \eqref{Ccup1} follows.

Since $M_vp\prec_P^s N_{C\cup C^\perp}q$, 
equation~\eqref{Ccup1} combined with Proposition~\ref{results}(2) implies that $M_vp\prec_P^s M_{\text{st}(v)}(p-p')$. Since obviously $M_vp\prec_P^s M_v$, by using Corollary \ref{right_corner} and that $p-p'\in M_{\text{st}(v)}'\cap rPr$, we get that $M_vp\prec_P^s M_v(p-p')$. Thus, $M_vp'\prec_P^s M_v(r-p')$,
which contradicts Corollary \ref{center_corner} as $p'\in M_{\text{st}(v)}'\cap rPr$ is a nonzero projection. This finishes the proof of \eqref{all_projections}.

We continue by proving the moreover assertion of the proposition: \begin{equation}\label{Ccup}M_{\text{st}(v)}p\prec_P^s N_{C\cup C^\perp}q.\end{equation}

To this end, let $p_0\in\mathcal Z(M_{\text{st}(v)}'\cap rPr)p$ be a nonzero projection. 
Since $M_vp_0\prec_P^s N_{C\cup C^\perp}$, our existence result for hulls (Proposition \ref{Zariski_hull}) ensures that we can find a nonzero projection $p_1\in\mathcal Z(M_{\text{st}(v)}'\cap P)p_0$ and a full subgraph $D\subset\Gamma$ such that $M_vp_1\prec_P^s N_{D}$ and $M_vp_1\nprec_P N_{D'}$, for every full subgraph $D'\subsetneq D$.  Corollary~\ref{intersection} then ensures that $D\subset C\cup C^\perp$.

Since $M_vp_1\prec_P^s N_{C\cup C^\perp}q$, $q\in\mathcal Z(N_C'\cap P)\subset N_{C\cup C^\perp}'\cap P$, $M_vp_1\prec_P^s N_D$ and $D\subset C\cup C^\perp$, Corollary~\ref{right_corner} implies that $M_vp_1\prec_P^s N_Dq$.
As $N_Cq\prec_P M_vp_1$ by \eqref{all_projections}, Proposition \ref{results}(2) gives that $N_Cq\prec_P N_Dq$. Hence, $N_C\prec_P N_D$, and Lemma~\ref{basic}(2) gives that $C\subset D$, thus $D^\perp\subset C^\perp$. Since $D\subset C\cup C^\perp$, we get that $D\cup D^\perp\subset C\cup C^\perp$. As $q\in N_{C\cup C^\perp}'\cap P$, we get that $q\in N_{D\cup D^\perp}'\cap P$. 

Next, since $M_vp_1\prec_P^s N_Dq$ and $M_vp_1\nprec_P N_{D'}$, for every full subgraph $D'\subsetneq D$, Proposition \ref{normalizer'} implies that $M_{\text{st}(v)}p_1\prec_P^sN_{D\cup D^\perp}q$.
Using this fact and again that $D\cup D^\perp\subset C\cup C^\perp$, we get that $M_{\text{st}(v)}p_1\prec_P^sN_{C\cup C^\perp}q$. Since $p_1\leq p_0$, we get that $M_{\text{st}(v)}p_0\prec_PN_{C\cup C^\perp}q$. 
Since the latter holds for every nonzero projection $p_0\in\mathcal{Z}(M_{\st(v)}'\cap rPr)$, Remark~\ref{elementary_facts}(5) ensures that \eqref{Ccup} holds.

Finally, let $w\in C$. Since $N_Cq\prec_P^s M_vp$, we get that $N_wq\prec_P^s M_vp$.  By applying Proposition \ref{normalizer'} 
(arguing as in the proof of Equation~\eqref{Ccup1}) we derive that
$N_{\text{st}(w)}q\prec_P^sM_{\text{st}(v)}p$. In combination with \eqref{Ccup} and Proposition \ref{results}(2) it follows that $N_{\text{st}{(w)}}q\prec_P N_{C\cup C^\perp}q$. By Lemma~\ref{basic}(2) this gives that $\text{st}(w)\subset C\cup C^\perp$. Hence, we have that $\text{st}(w)\cap (\Gamma\setminus C)=C^\perp$, for every $w\in C$, which implies that $C$ is collapsible.
\end{proof}

\section{Commuting subalgebras of graph product von Neumann algebras}\label{sec:commuting}

\subsection{A structural result}
The main goal of this section is to prove the following result concerning the structure of commuting subalgebras of graph product von Neumann algebras.

\begin{theorem}\label{commute}
Let $\Gamma$  be a finite simple graph, $(M_v,\tau_v)_{v\in\Gamma}$ a family of tracial von Neumann algebras, and   $M_\Gamma=*_{v,\Gamma}M_v$ the associated graph product von Neumann algebra. Let $P_1,P_2\subset pM_\Gamma p$ be commuting von Neumann subalgebras.

Then there exist four pairwise disjoint, possibly empty, full subgraphs $\Gamma_i\subset\Gamma$, for $0\leq i\leq 3$, and a nonzero projection $q\in \mathcal Z((P_1\vee P_2)'\cap pM_\Gamma p)$ such that the following conditions hold:
\begin{enumerate}
\item[(a)] $\Gamma_0$ is a clique;
\item[(b)] $\Gamma_i$ has an empty maximal clique factor, for every $1\leq i\leq 3$;
\item[(c)] $\Gamma_i\subset\Gamma_j^\perp$, for every $0\leq i<j\leq 3$;
\item[(d)] $P_iq\prec_{M_\Gamma}^sM_{\Gamma_0\cup\Gamma_i\cup\Gamma_3}$, for every $i\in\{1,2\}$;
\item[(e)] $(P_1\vee P_2)q\prec_{M_\Gamma}^sM_{\Gamma_0\cup\Gamma_1\cup\Gamma_2\cup\Gamma_3}$;
\item[(f)] $(P_1\vee P_2)q$ is amenable relative to $M_{\Gamma_0\cup\Gamma_1\cup\Gamma_2}$ inside $M_\Gamma$.

\end{enumerate}

Moreover, if $P_1\vee P_2\prec^s_{M_{\Gamma}}M_\Lambda$, for a full subgraph $\Lambda\subset\Gamma$, then we may take $\Gamma_i\subset\Lambda$, for every $0\leq i\leq 3$.
\end{theorem}

In order to prove Theorem \ref{commute}, we use \cite[Theorem A]{Va13} (generalizing \cite[Theorem 1.6]{Io15}) to control commuting subalgebras inside graph product von Neumann algebras.  First, we note that \cite[Theorem A]{Va13} readily implies the following.

\begin{theorem}\label{relamen}
Let $(M_1,\tau_1),(M_2,\tau_2)$ be tracial von Neumann algebras with a common von Neumann subalgebra $(B,\tau)$ such that ${\tau_1}_{|B}={\tau_2}_{|B}=\tau$
and denote $M=M_1*_BM_2$. Let $P_1,P_2\subset pMp$ be two commuting von Neumann subalgebras. Then at least one of the following conditions holds:
\begin{enumerate}
\item $P_1\prec_MM_1$ or $P_1\prec_MM_2$.
\item $P_2\prec_MM_1$ or $P_2\prec_MM_2$.
\item $P_1\vee P_2$ is amenable relative to $B$ inside $M$.
\end{enumerate}
\end{theorem}

\begin{proof}
Assuming that (1) and (2) fail, we show that (3) holds. Since $P_1\nprec_MM_1$ and $P_1\nprec_MM_2$ by \cite[Corollary F.14]{BO08} we can find a diffuse abelian von Neumann subalgebra $A\subset P_1$ such that $A\nprec_MM_1$ and $A\nprec_MM_2$. Since $A$ is amenable, by \cite[Theorem A]{Va13} we get that $\mathcal N_{pMp}(A)''$ is amenable relative to $B$ inside $M$. Since $P_2\subset\mathcal N_{pMp}(A)''$, we deduce that $P_2$ is amenable relative to $B$ inside $M$. By applying \cite[Theorem A]{Va13} again and using that $P_2\nprec_MM_1$ and $P_2\nprec_MM_2$, we derive that $\mathcal N_{pMp}(P_2)''$ is amenable relative to $B$ inside $M$. Since $P_1\vee P_2\subset\mathcal N_{pMp}(P_2)''$, it follows that 
$P_1\vee P_2$ is amenable relative to $B$ inside $M$, so (3) holds, as claimed.
\end{proof}

To prove Theorem \ref{commute} we will also need the following result from \cite{BCC24}.

\begin{theorem}[\!\!{\cite[Theorem 5.3]{BCC24}}]\label{intersection2}
Let $\Gamma$  be a finite simple graph, $(M_v,\tau_v)_{v\in\Gamma}$ a family of tracial von Neumann algebras, and   $M_\Gamma=*_{v,\Gamma}M_v$ the associated graph product von Neumann algebra. Let $P\subset pM_\Gamma p$ be a von Neumann subalgebra and $\Gamma_1,\Gamma_2\subset\Gamma$  full subgraphs 
such that $P$ is amenable relative to $M_{\Gamma_1}$ and $M_{\Gamma_2}$ inside $M_\Gamma$.
Then $P$ is amenable relative to $M_{\Gamma_1\cap\Gamma_2}$ inside $M_\Gamma$.
\end{theorem}

Combining Theorems \ref{relamen} and \ref{intersection2} allows to derive the following corollary. 

\begin{corollary}\label{relamen2}
Let $\Gamma$  be a finite simple graph, $(M_v,\tau_v)_{v\in\Gamma}$ a family of tracial von Neumann algebras, and $M_\Gamma=*_{v,\Gamma}M_v$ the associated graph product von Neumann algebra. Let $P_1,P_2\subset pM_\Gamma p$ be commuting von Neumann subalgebras, and $\Gamma_1,\Gamma_2\subset\Gamma$ be full subgraphs. 
Assume that for every $i\in\{1,2\}$, we have  $P_i\prec_{M_\Gamma}^sM_{\Gamma_i}$ and $P_i\nprec_{M_\Gamma}M_{\Gamma_i'}$, for every full subgraph $\Gamma_i'\subsetneq\Gamma_i$.
Let $\Lambda\subset\Gamma$ be the full subgraph consisting of vertices $v\in\Gamma_1\cap\Gamma_2$ such that $\Gamma_1\not\subset\emph{st}(v)$ and $\Gamma_2\not\subset\emph{st}(v)$.

Then $P_1\vee P_2$ is amenable relative to $M_{\Lambda^\perp}$ inside $M_\Gamma$.
\end{corollary}

\begin{proof}
Let $v\in \Lambda$ and $i\in\{1,2\}$. Then $\Gamma_i\not\subset\Gamma\setminus\{v\}$ and $\Gamma_i\not\subset\text{st}(v)$. Since by assumption, $M_{\Gamma_i}$ 
is the hull of $P_i$, the moreover part of Proposition \ref{Zariski_hull} implies that $P_i\nprec_{M_\Gamma}M_{\Gamma\setminus\{v\}}$ and $P_i\nprec_{M_\Gamma}M_{\text{st}(v)}$. Since $M_\Gamma=M_{\Gamma\setminus\{v\}}*_{M_{\mathrm{lk}(v)}}M_{\text{st}(v)}$, Theorem \ref{relamen} implies that $P_1\vee P_2$ is amenable relative to $M_{\text{lk}(v)}$ inside $M_\Gamma$, for every $v\in \Lambda$. By applying Theorem \ref{intersection2} finitely many times we conclude that $P_1\vee P_2$ is amenable relative to $M_{\cap_{v\in \Lambda}\text{lk}(v)}=M_{\Lambda^\perp}$ inside $M_\Gamma$.
\end{proof}

We are now ready to prove Theorem \ref{commute}.

\begin{proof}[\bf Proof of Theorem \ref{commute}]
Put $\mathcal Z= \mathcal Z((P_1\vee P_2)'\cap pM_\Gamma p)$. Since $P_1\vee P_2\subset\mathcal N_{pM_\Gamma p}(P_1)''\cap\mathcal N _{pM_\Gamma p}(P_2)''$, we get that $\mathcal N_{pM_\Gamma p}(P_i)'\cap pM_\Gamma p\subset \mathcal Z$, for every $i\in\{1,2\}$. Proposition \ref{Zariski_hull}  implies that we can find a nonzero projection $q\in\mathcal Z$, and, for every $i\in\{1,2\}$, a 
full subgraph $\Lambda_i\subset\Gamma$ such that
\begin{equation}\label{P1P2}
\text{$P_iq\prec^s_{M_\Gamma}M_{\Lambda_i}$ and $P_iq\nprec_{M_\Gamma}M_{\Lambda'_i}$, for every full subgraph $\Lambda_i'\subset\Gamma$ such that $\Lambda_i\not\subset\Lambda_i'$.}\end{equation}
By applying Proposition \ref{normalizer}(2) we get that 
\begin{equation}\label{P1veeP2}\text{$(P_1\vee P_2)q\prec^s_{M_\Gamma}M_{\Lambda_i\cup\Lambda_i^\perp}$, for every $i\in\{1,2\}$.}
\end{equation}
In particular $P_1q\prec^s_{M_\Gamma}M_{\Lambda_2\cup\Lambda_2^\perp}$ and $P_2q\prec^s_{M_\Gamma}M_{\Lambda_1\cup\Lambda_1^\perp}$. By combining with \eqref{P1P2} we get that 
\begin{equation}\label{Lambda12}
\text{$\Lambda_1\subset\Lambda_2\cup\Lambda_2^\perp$ \;\; and \;\; $\Lambda_2\subset\Lambda_1\cup\Lambda_1^\perp$.}
\end{equation}
Let $\Gamma_3$ be the set of $v\in\Lambda_1\cap\Lambda_2$ such that $\Lambda_1\not\subset\text{st}(v)$ and $\Lambda_2\not\subset\text{st}(v)$.
Define $\Gamma_0=(\Lambda_1\cap\Lambda_2)\setminus\Gamma_3$, $\Gamma_1=\Lambda_1\setminus\Lambda_2$ and $\Gamma_2=\Lambda_1^\perp$.
Then the graphs $\Gamma_i$, $0\leq i\leq 3$, are pairwise disjoint by definition.

If $v\in\Gamma_0$, then $\Lambda_1\subset\text{st}(v)$ or $\Lambda_2\subset\text{st}(v)$, therefore $\Lambda_1\cap\Lambda_2\subset\text{st}(v)$. This shows that $\Gamma_0$ is contained in the maximal clique factor of $\Lambda_1\cap\Lambda_2$. Thus, $\Gamma_0$ is a clique, which proves (a), and $\Gamma_0\subset\Gamma_3^\perp$. 

If $i\in\{0,1,3\}$, then $\Gamma_i\subset\Lambda_1$ and since $\Gamma_2=\Lambda_1^\perp$ we get that $\Gamma_i\subset\Gamma_2^\perp$. 
Recall that $\Gamma_1=\Lambda_1\setminus\Lambda_2$, so by \eqref{Lambda12} we get $\Gamma_1\subset\Lambda_2^\perp$.
If $i\in\{0,3\}$, then since $\Gamma_i\subset\Lambda_2$, we deduce that $\Gamma_i\subset\Gamma_1^\perp$, which altogether proves (c).

Since $\Gamma_0\cup\Gamma_3=\Lambda_1\cap\Lambda_2$, we get that $\Gamma_0\cup\Gamma_1\cup\Gamma_3=\Lambda_1$ and
$\Gamma_0\cup\Gamma_2\cup\Gamma_3=(\Lambda_1\cap\Lambda_2)\cup\Lambda_1^\perp$. The latter equality and \eqref{Lambda12} imply that $\Lambda_2\subset\Gamma_0\cup\Gamma_2\cup\Gamma_3$. By \eqref{P1P2} we deduce that (d) holds.

Since we also have that $\Gamma_0\cup\Gamma_1\cup\Gamma_2\cup\Gamma_3=\Lambda_1\cup\Lambda_1^\perp$, (e) holds by \eqref{P1veeP2}.

Towards proving (f), we note first that \eqref{P1P2} together with Corollary \ref{relamen2} gives that $(P_1\vee P_2)q$ is amenable relative to $M_{\Gamma_3^\perp}$ inside $M_\Gamma$. By (e), Proposition \ref{results}(4) implies that $(P_1\vee P_2)q$ is amenable relative to $M_{\Gamma_0\cup\Gamma_1\cup\Gamma_2\cup\Gamma_3}$ inside $M_\Gamma$. Theorem \ref{intersection2} implies that $(P_1\vee P_2)q$ is amenable relative to $M_{\Gamma_3^\perp\cap(\Gamma_0\cup\Gamma_1\cup\Gamma_2\cup\Gamma_3)}=M_{\Gamma_0\cup\Gamma_1\cup\Gamma_2}$ inside $M_\Gamma$, which proves (f). Altogether, we proved (a) and (c)-(f).

Finally, to justify (b), let $C_i$ be the maximal clique factor of $\Gamma_i$, for every $1\leq i\leq 3$. Then it is clear that conditions (a) and (c)-(f) still hold if we replace the quadruple $(\Gamma_0,\Gamma_1,\Gamma_2,\Gamma_3)$ by $(\Gamma_0',\Gamma_1',\Gamma_2',\Gamma_3')=(\Gamma_0\cup C_1\cup C_2\cup C_3,\Gamma_1\setminus C_1,\Gamma_2\setminus C_2,\Gamma_3\setminus C_3)$. Additionally, since $\Gamma_i'$ has an empty maximal clique factor, for every $1\leq i\leq 3$,  (b) is also satisfied.

To prove the moreover assertion, assume that $P_1\vee P_2\prec^s_{M_{\Gamma}}M_\Lambda$, for a  full subgraph $\Lambda\subset\Gamma$. By Proposition \ref{results}(4) we get that $P_1\vee P_2$ is amenable relative to $M_{\Lambda}$ inside $M_\Gamma$. Using Corollary~\ref{intersection} and Theorem \ref{intersection2} it follows that conditions (a) and (c)-(f) hold after we replace $\Gamma_i$ by $\Gamma_i\cap\Lambda$, for every $0\leq i\leq 3$. 
Thus, we may find $\Gamma_i\subset\Lambda$, for $0\leq i\leq 3$, so that conditions (a) and (c)-(f) hold.
By repeating the previous paragraph we can also ensure that (b) is satisfied.
\end{proof}

\subsection{Corollaries of Theorem \ref{commute}}
As a consequence of Theorem \ref{commute}, we derive two useful results. 
The reader is referred to Section~\ref{graph_notions} for the notion of a join subgraph.

\begin{corollary}\label{join_embed}
Let $\Gamma$  be a finite simple graph, $(M_v,\tau_v)_{v\in\Gamma}$ a family of tracial von Neumann algebras, and  $M_\Gamma=*_{v,\Gamma}M_v$ the associated graph product von Neumann algebra. Let $P_1,P_2\subset p(M_\Gamma\otimes \mathbb{M}_n(\mathbb{C})) p$ be diffuse commuting von Neumann subalgebras, for some $n\ge 1$. Assume that $P_1\vee P_2$ has no amenable direct summand and $P_1\vee P_2\nprec_{M_\Gamma\otimes\mathbb{M}_n(\mathbb{C})} M_v\otimes\mathbb{M}_n(\mathbb{C})$, for every isolated vertex $v\in\Gamma$.

Then there exists a join full subgraph $\Lambda\subset\Gamma$ such that 
$P_1\vee P_2\prec_{M_\Gamma\otimes \mathbb{M}_n(\mathbb{C})}M_{\Lambda}\otimes\mathbb{M}_n(\mathbb{C})$.
\end{corollary}

\begin{proof}
As in Remark~\ref{amplify_results}, we let $\tilde\Gamma=\Gamma\circ\{v_0\}$, and view $M_\Gamma\otimes\mathbb{M}_n(\mathbb{C})$ as a graph product $M_{\tilde\Gamma}$ where the vertex algebra on $v_0$ is $\mathbb{M}_n(\mathbb{C})$.

Assume by contradiction that the conclusion fails. Theorem \ref{commute} provides pairwise disjoint full subgraphs $\tilde\Gamma_i\subset\tilde\Gamma$, $0\leq i\leq 3$, and a nonzero projection $q\in\mathcal Z((P_1\vee P_2)'\cap pM_{\tilde\Gamma} p))$ such that conditions (a)-(f) hold. For every $i\in\{0,1,2,3\}$, we let $\Gamma_i:=\tilde\Gamma_i\cap\Gamma$. Since the conclusion fails, conditions (c) and (e) imply that at most one of the graphs $\Gamma_i$, $0\leq i\leq 3$, is non-empty 
(in fact exactly one, using that $P_1,P_2$ are diffuse). If $\Gamma_0=\Gamma_1=\Gamma_2=\emptyset$, then (f) would imply that $(P_1\vee P_2)q$ is amenable, contradicting that $P_1\vee P_2$ has no amenable direct summand. If $i\in\{1,2\}$ and $\Gamma_0=\Gamma_i=\Gamma_3=\emptyset$, then (d) would imply that $P_iq\prec_{M_{\tilde\Gamma}}\mathbb{M}_n(\mathbb C)$, contradicting that $P_i$ is diffuse. Altogether, we must have that $\Gamma_1=\Gamma_2=\Gamma_3=\emptyset$. By (e) we get that $(P_1\vee P_2)q\prec_{M_{\tilde\Gamma}}^sM_{\Gamma_0}\otimes\mathbb{M}_n(\mathbb{C})$. Since $\Gamma_0$ is a clique by (a) and the conclusion is assumed to fail, we conclude that $\Gamma_0=\{v\}$, for some $v\in\Gamma$. 
Hence,  $P_1\vee P_2\prec_{M_{\tilde\Gamma}} M_v\otimes\mathbb{M}_n(\mathbb{C})$.
By the hypothesis $v\in\Gamma$ is not isolated, so there is $v'\in\text{lk}(v)$. But then $P_1\vee P_2\prec_{M_{\tilde\Gamma}}M_{\{v,v'\}}\otimes\mathbb{M}_n(\mathbb{C})$, which also gives a contradiction as $\{v,v'\}$ is a join graph.    
\end{proof}

\begin{corollary}\label{tensor_dec}
Let $\Gamma$  be a finite simple graph, $(M_v,\tau_v)_{v\in\Gamma}$ a family of diffuse tracial von Neumann algebras, and  $M_\Gamma=*_{v,\Gamma}M_v$ the associated graph product von Neumann algebra. Let $P_1,P_2\subset pM_\Gamma p$ be diffuse commuting von Neumann subalgebras. Put $P=P_1\vee P_2$. Let $\Lambda\subset \Gamma$  be a full subgraph, and $q\in\mathcal Z(M_\Lambda'\cap M_\Gamma)$ a nonzero projection.

 Assume that $P\prec_{M_\Gamma}^s M_\Lambda q$ and $M_\Lambda q\prec_{M_\Gamma}^s P$. Assume also that $Pp'\nprec_{M_\Gamma}P(p-p')$, for every nonzero projection $p'\in\mathcal Z(P'\cap pM_\Gamma p)$. 
Let $p_0\in\mathcal Z(P'\cap pM_\Gamma p)$ be a nonzero projection.

Then there exist a nonzero projection $p'\in\mathcal Z(P'\cap pM_\Gamma p)$ with $p'\leq p_0$,  and disjoint full subgraphs $\Lambda_1,\Lambda_2\subset\Lambda$ such that 
\begin{itemize}
\item $\Lambda\setminus (\Lambda_1\cup\Lambda_2)$ is the maximal clique factor $C$ 
of $\Lambda$, 
\item $\Lambda_1\subset\Lambda_2^\perp$, and \item $P_ip'\prec_{M_\Gamma}M_{C\cup\Lambda_{i}}q$, for every $i\in\{1,2\}$.
\end{itemize}
\end{corollary}

\begin{remark}
    The corollary remains true for subalgebras $P_1,P_2\subset p(M_\Gamma\otimes\mathbb{M}_n(\mathbb{C}))p$, for some $n\ge 1$. Again this can be derived using Remark~\ref{amplify_results}.
\end{remark}

\begin{proof}
Since $Pp_0\prec_{M_\Gamma} M_\Lambda q$, the moreover assertion of Theorem \ref{commute} provides pairwise disjoint full subgraphs $\Lambda_i\subset\Lambda$ (with $0\le i\le 3$), and a nonzero projection $p'\in\mathcal Z(P'\cap pM_\Gamma p)$ such that $p'\leq p_0$ and the following hold:
\begin{enumerate}
\item[(a)] $\Lambda_0$ is a clique;
\item[(b)] $\Lambda_i$ has an empty maximal clique factor, for every $1\leq i\leq 3$;
\item[(c)] $\Lambda_i\subset\Lambda_j^\perp$, for every $0\leq i<j\leq 3$;
\item[(d)] $P_ip'\prec_{M_\Gamma}^s M_{\Lambda_0\cup\Lambda_i\cup\Lambda_3}$, for every $i\in\{1,2\}$;
\item[(e)] $Pp'\prec_{M_\Gamma}^sM_{\Lambda_0\cup\Lambda_1\cup\Lambda_2\cup\Lambda_3}$;
\item[(f)] $Pp'$ is amenable relative to $M_{\Lambda_0\cup\Lambda_1\cup\Lambda_2}$ inside $M_\Gamma$.
\end{enumerate}
We first claim that $M_\Lambda q\prec_{M_\Gamma}Pp'$. Otherwise, since $M_\Lambda q\prec_{M_\Gamma}^sP$, we get $M_\Lambda q\prec_{M_\Gamma}^sP(p-p')$.
Since $Pp'\prec_{M_\Gamma}M_\Lambda q$, Proposition \ref{results}(2) implies that $Pp'\prec_{M_\Gamma}P(p-p')$, which is a contradiction.

Second, we claim that $\Lambda_0\cup\Lambda_1\cup\Lambda_2\cup\Lambda_3=\Lambda$.  If $v\in\Lambda$, then since $M_\Lambda q\prec_{M_\Gamma}Pp'$, we derive that $M_v\prec_{M_\Gamma}Pp'$. This fact, (e) and Proposition \ref{results}(2) together imply that $M_v\prec_{M_\Gamma}M_{\Lambda_0\cup\Lambda_1\cup\Lambda_2\cup\Lambda_3}$. By Lemma~\ref{basic}(2), we get that $v\in \Lambda_0\cup\Lambda_1\cup\Lambda_2\cup\Lambda_3$, which proves our claim.

Third, we claim that $\Lambda_0$ is the maximal clique factor $C$ of $\Lambda$. The previous claim together with (b) ensures that $C\subset\Lambda_0$, and together with (c), it ensures that $\Lambda\subset \Lambda_0\cup\Lambda_0^{\perp}$, so $\Lambda_0$ is a union of join factors of $\Lambda$. Since by (a) $\Lambda_0$ is a clique, it follows that $\Lambda_0\subset C$, which proves our claim.

Fourth, we claim that $\Lambda_3=\emptyset$.  
Assume by contradiction that $\Lambda_3\not=\emptyset$. Then (b) implies that $\Lambda_3$ is not a clique.
Since  $M_\Lambda q\prec_{M_\Gamma}Pp'$, by Proposition \ref{results}(3), we can find a nonzero projection $q'\in\mathcal Z(M_\Lambda'\cap M_\Gamma)$ such that $q'\leq q$ and $M_\Lambda q'\prec_{M_\Gamma}^s Pp'$. Thus, we have $M_\Lambda q'\prec_{M_\Gamma}^s Pp'\oplus\mathbb C(1-p')$, and hence that $M_\Lambda q'$ is amenable relative to $Pp'\oplus\mathbb C(1-p')$ inside $M_\Gamma$ by Proposition \ref{results}(4). By (f), we also get that $Pp'\oplus\mathbb C(1-p')$ is  amenable relative to $M_{\Lambda_0\cup\Lambda_1\cup\Lambda_2}$ inside $M_\Gamma$. 
The last two facts and \cite[Proposition 2.4(3)]{OP10} together imply that $M_\Lambda q'$ is amenable relative to $M_{\Lambda_0\cup\Lambda_1\cup\Lambda_2}$ inside $M_\Gamma$. In particular, $M_{\Lambda_3}q'$ is amenable relative to $M_{\Lambda_0\cup\Lambda_1\cup\Lambda_2}$ inside $M_\Gamma$. Since  $M_{\Lambda_3}q'$ is also amenable relative to $M_{\Lambda_3}$ inside $M_\Gamma$ and $\Lambda_3\cap (\Lambda_0\cup\Lambda_1\cup\Lambda_2)=\emptyset$, Theorem \ref{intersection2}  
implies that $M_{\Lambda_3}q'$ is amenable. 
By Lemma~\ref{basic}(3), this contradicts the fact that $\Lambda_3$ is not a clique.

Thus, we have that $\Lambda\setminus (\Lambda_1\cup\Lambda_2)=\Lambda_0$, which is the clique factor of $\Lambda$, and $\Lambda_1\subset\Lambda_2^\perp$. If $i\in\{1,2\}$, then (d) gives that $P_ip'\prec_{M_\Gamma}^sM_{\Lambda_0\cup\Lambda_i}$. Since $P\prec_{M_\Gamma}^sM_\Lambda q$, we also have that $P_ip'\prec_{M_\Gamma}^s M_\Lambda q$. By combining the last two facts and using Corollary \ref{right_corner} we get that $P_ip'\prec_{M_\Gamma}^s M_{\Lambda_0\cup\Lambda_i}q$, which finishes the proof.
\end{proof}

\section{Strong primeness: Proof of Theorem \ref{prime_factorization}}
In this section, we use Corollary \ref{tensor_dec} to prove Theorem \ref{prime_factorization}.
It was shown in \cite[Theorem B]{BCC24} that if $\Gamma$ is an irreducible graph with $|\Gamma|\geq 2$, then $M_\Gamma$ is a prime II$_1$ factor, whenever all the vertex algebras $(M_v)_{v\in\Gamma}$ are II$_1$ factors. Towards proving Theorem \ref{prime_factorization}, we first use Corollary \ref{tensor_dec} to strengthen this fact  by showing that $M_\Gamma$ is strongly prime in the following sense defined by Isono. 

\begin{definition}[\!\!{\cite[Section~5]{Is14}}] A II$_1$ factor $M$ is called {\it strongly prime} if for any II$_1$ factors $N,P_1,P_2$ with $M\overline{\otimes}N=P_1\overline{\otimes}P_2$, there exist $u\in\mathcal U(M\overline{\otimes}N)$ and $t>0$ such that after identifying $P_1\overline{\otimes}P_2=P_1^t\overline{\otimes}P_2^{1/t}$, we have that $uMu^*\subset P_1^t$ or $uMu^*\subset P_2^{1/t}$.
\end{definition}

\begin{theorem}\label{strongly_prime}
Let $\Gamma$  be an irreducible finite simple graph with $|\Gamma|\geq 2$, 
$(M_v,\tau_v)_{v\in\Gamma}$ be a family of diffuse tracial von Neumann algebras, and $M_\Gamma=*_{v,\Gamma}M_v$ the associated graph product von Neumann algebra.

Then $M_\Gamma$ is a strongly prime II$_1$ factor. 
\end{theorem}

\begin{proof} 
Since $\Gamma$ is irreducible, it follows from Lemma~\ref{basic}(1) that $M_\Gamma$ is a II$_{1}$ factor.

Let $N,P_1,P_2$ be II$_1$ factors such that $M_\Gamma\overline{\otimes}N=P$, where $P=P_1\overline{\otimes}P_2$.
Let $v$ be a vertex which does not belong to $\Gamma$ and define $\widetilde{\Gamma}=\Gamma\circ\{v\}$. Put $M_v=N$ and notice that $M_{\widetilde{\Gamma}}=M_{\Gamma}\overline{\otimes}N=P$. Since $\Gamma$ is irreducible and $|\Gamma|\geq 2$, the maximal clique factor of $\widetilde\Gamma$ is equal to $\{v\}$.

Note that since $M_\Gamma$ and $N$ are II$_1$ factors, we have $\mathcal{Z}(P)=\mathbb{C}1$. Therefore, by applying Corollary~\ref{tensor_dec} 
(with $\Lambda=\tilde{\Gamma}$) we can find a decomposition of $\Gamma=\widetilde\Gamma\setminus\{v\}=\Gamma_1\cup\Gamma_2$ such that $\Gamma_1\subset\Gamma_2^\perp$ and $P_i\prec_{P}^sM_{\{v\}\cup\Gamma_i}$, for every $i\in\{1,2\}$. 
Since $\Gamma$ is irreducible, we have $\Gamma_i=\emptyset$, for some $i\in\{1,2\}$. Let $j\in\{1,2\}\setminus\{i\}$. Since $P_i\prec_P M_v$, and since $M_\Gamma\subset M'_v\cap P$ and $P_i'\cap P=P_j$, by Proposition~\ref{results}(1) we get that $M_\Gamma\prec_P P_j$.  By Lemma~\ref{tensor_product}(1), 
we find $u\in\mathcal U(P)$ and $t>0$ such that after identifying $P=P_i^t\overline{\otimes}P_j^{1/t}$, we have $uM_\Gamma u^*\subset P_j^{1/t}$. This shows that $M_\Gamma$ is strongly prime. \end{proof}

As an immediate consequence of Theorem \ref{strongly_prime} we can now derive Theorem \ref{prime_factorization}.

\begin{proof}[\bf Proof of Theorem \ref{prime_factorization}]
Since $\Gamma$ has an empty maximal clique factor, we can decompose it as $\Gamma=\Gamma_1\circ\cdots\circ\Gamma_n$, where $\Gamma_i$ is an irreducible graph with 
$|\Gamma_i|\geq 2$, for every $1\leq i\leq n$. 
Since $\Gamma_i$ is an irreducible graph, Theorem \ref{strongly_prime} implies that $M_{\Gamma_i}$ is a strongly prime II$_1$ factor, for every $1\leq i\leq n$. 
Since $M_\Gamma=M_{\Gamma_1}\overline{\otimes}\cdots\overline{\otimes}M_{\Gamma_n}$, this fact gives the conclusion of Theorem \ref{prime_factorization} via standard arguments, see, e.g., \cite[Section 4]{Is20}. 

For instance, to justify (1), assume that $M_\Gamma=P_1\overline{\otimes}P_2$, where $P_1,P_2$ are II$_1$ factors. If $1\leq i\leq n$, then since  $M_{\Gamma_i}$ is strongly prime and $M_\Gamma=M_{\Gamma_i}\overline{\otimes}M_{\Gamma\setminus\Gamma_i}$, \cite[Lemma 4.1]{Is20} provides $j_i\in\{1,2\}$ such that $M_{\Gamma_i}\prec_{M_\Gamma}P_{j_i}$. For $j\in\{1,2\}$, put $I_j=\{1\leq i\leq n\mid j_i=j\}$. Then by using the last fact and applying \cite[Lemma 4.3]{Is20} repeatedly we get that $\overline{\otimes}_{i\in I_j}M_{\Gamma_i}\prec_{M_\Gamma}P_j$, for every $j\in\{1,2\}$. Since $I_1\cup I_2=\{1,\dots,n\}$ by taking relative commutants via Proposition \ref{results}(1) we also get that $P_j\prec_{M_\Gamma}\overline{\otimes}_{i\in I_j}M_{\Gamma_i}$, for every $j\in\{1,2\}$. Using Lemma \ref{tensor_product}(2), assertion (1) follows readily.
\end{proof}

We end this section by recording a result concerning strongly prime II$_1$ factors which will be needed in later sections.

\begin{lemma}\label{str_prime}
Let $M_0,N_0$ be II$_1$ factors and $M_1,\dots,M_m,N_1,\dots,N_n$ be strongly prime II$_1$ factors such that $P=\overline{\otimes}_{i=0}^mM_i=\overline{\otimes}_{j=0}^nN_j$. Suppose that  $M_i\nprec_P N_0$, for every $1\leq i\leq m$, and $N_j\nprec_PM_0$, for every $1\leq j\leq n$. Then there exist a decomposition $P=N_0^t\overline{\otimes}(\overline{\otimes}_{j=1}^nN_i)^{1/t}$, for some $t>0$, and a unitary $u\in P$ such that $uM_0u^*=N_0^t$.
\end{lemma}

\begin{proof} Put $Q=\overline{\otimes}_{i=1}^mM_i$ and $R=\overline{\otimes}_{j=1}^nN_j$.
For $1\le i\le m$, since $M_i\nprec_P N_0$ and $M_i$ is strongly prime, we deduce that $M_i\prec_P R$. Since this holds for every $1\leq i\leq m$, \cite[Lemma 4.3]{Is20} implies that $Q\prec_P R$. Similarly, we get that $R\prec_PQ$.  Lemma \ref{tensor_product} gives a decomposition $P=N_0^t\overline{\otimes}R^{1/t}$, for some $t>0$, and a unitary $u\in P$ such that $uQu^*=R^{1/t}$. By taking relative commutants, we get that $uM_0u^*=N_0^t$.
\end{proof}

\section{Classification/Rigidity theorems: Strategy and tools} \label{sec:tools}

\begin{notation}\label{notation} We start by introducing the following setting which we keep throughout the section:
\begin{enumerate}
\item $\Gamma$ and $\Lambda$ are nonempty finite simple graphs. 
\item $(M_v,\tau_v)_{v\in\Gamma}$ and $(N_w,\tau_w)_{w\in\Lambda}$ are families of diffuse tracial von Neumann algebras.
\item  $M_\Gamma=*_{v,\Gamma}M_v$ and $N_\Lambda=*_{w,\Lambda}N_w$ are the associated graph product von Neumann algebras.
\item Assume that $M_\Gamma$ and $N_\Lambda$ are II$_1$ factors.
\item Let $\theta:M_\Gamma\rightarrow pPp$ be a $*$-isomorphism, where $P=\mathbb M_n(\mathbb C)\otimes N_\Lambda$ for some $n\geq 1$, and $p\in P$ is a projection. We identify $M_\Gamma=pPp$, via $\theta$, and $N_\Lambda=qPq$, where $q=e_{1,1}\otimes 1\in P$.
\end{enumerate}
\end{notation}

\begin{remark}
    In view of Lemma~\ref{basic}(1), condition (4) is automatically satisfied if neither $\Gamma$ nor $\Lambda$ is contained in the star of a single vertex. This is satisfied in particular if $\Gamma$ and $\Lambda$ are transvection-free, which is the case in all statements from the introduction.
\end{remark}

We also introduce the following notation. 
\begin{notation}\label{nota:j}
For a finite simple graph $\Omega$, we denote by $\mathscr J(\Omega)$  the set of maximal join full subgraphs $\Sigma\subset\Omega$, and by $\mathscr C(\Omega)$ the set of complete subgraphs $C\subset\Omega$.
\end{notation}

\subsection{Locating von Neumann subalgebras of maximal join subgraphs}
In preparation for the proofs of Theorems \ref{arbitrary} to \ref{factors}, we establish that any $*$-isomorphism $\theta: M_\Gamma\rightarrow p(\mathbb M_n(\mathbb C)\otimes N_\Lambda)p$ must identify (up to ``piecewise" $\prec^s$) subalgebras corresponding to maximal join full subgraphs of $\Gamma$ with subalgebras corresponding to maximal join full subgraphs of $\Lambda$.

\begin{proposition}\label{structure_of_aut}
In the setting from Notation \ref{notation}, let $p_\Sigma\in\mathcal Z(M_\Sigma)$ and $q_\Delta\in\mathcal Z(N_\Delta)$ be the maximal projections such that $M_\Sigma p_\Sigma$ and $N_\Delta	q_\Delta$ have no amenable direct summands, for every $\Sigma\in\mathscr J(\Gamma)$ and $\Delta\in\mathscr J(\Lambda)$. 
Assume that $M_\Sigma p_\Sigma\nprec_P N_w$ and $N_\Delta q_\Delta\nprec_P M_v$, for every $\Sigma\in\mathscr J(\Gamma)$ and $\Delta\in\mathscr J(\Lambda)$ with $p_\Sigma\not=0$ and $q_\Delta\not=0$, and every isolated vertices $v\in\Gamma$ and $w\in\Lambda$. 

Then  we can find projections $p_{\Sigma,\Delta}\in\mathcal Z(M_\Sigma'\cap pPp)$ and $q_{\Delta,\Sigma}\in\mathcal Z(N_\Delta'\cap qPq)$,  for  every $\Sigma\in\mathscr J(\Gamma)$ and $\Delta\in\mathscr J(\Lambda)$, such that the following conditions hold: 
\begin{enumerate}
\item $\sum_{\Delta\in\mathscr J(\Lambda)}p_{\Sigma,\Delta}=p_\Sigma$, for every $\Sigma\in\mathscr J(\Gamma)$, and $\sum_{\Sigma\in\mathscr J(\Gamma)}q_{\Delta,\Sigma}=q_\Delta$, for every $\Delta\in\mathscr J(\Lambda)$;
\item $M_\Sigma p_{\Sigma,\Delta}\prec_P^sN_\Delta q_{\Delta,\Sigma}$ and $N_\Delta q_{\Delta,\Sigma}\prec_P^s M_\Sigma p_{\Sigma,\Delta}$, for every $\Sigma\in\mathscr J(\Gamma)$ and $\Delta\in\mathscr J(\Lambda)$. 

\end{enumerate}
\end{proposition}
\begin{remark}\label{a_few_remarks}
We make a few remarks on the statement. 
\begin{enumerate} 
\item The existence of the maximal projections $p_{\Sigma}$ and $q_{\Delta}$ is proven as in Remark~\ref{rk:maximal-projections}.
\item Condition (2) in particular implies that $p_{\Sigma,\Delta}=0$ if and only if $q_{\Delta,\Sigma}=0$. 
\item Under the extra assumption that no connected component of either  $\Gamma$ or $\Lambda$ is a complete subgraph, the assumptions from the proposition are automatically satisfied with $p_\Sigma=p$ and $q_\Delta=q$. Indeed in this case, a subgraph $\Sigma\in\mathscr{J}(\Gamma)$ cannot be complete, so Lemma~\ref{basic}(3) ensures that $M_\Sigma$ has no amenable direct summand (and likewise for $\Lambda$).
\item In the particular case where all vertex algebras are II$_{1}$ factors,  Lemma~\ref{basic}(1) implies that $\mathcal{Z}(M_\Sigma)=\mathbb{C}p$ and $\mathcal{Z}(M'_\Sigma\cap pPp)=\mathbb{C}p$. So for every $\Sigma\in\mathscr{J}(\Gamma)$, we have $p_\Sigma=p$, and there is a unique $\Delta$ such that $p_{\Sigma,\Delta}$ is non-zero (so $p_{\Sigma,\Delta}=p$). Likewise for every $\Delta\in\mathscr{J}(\Lambda)$, we have $q_{\Delta}=q$, and there is a unique $\Sigma$ such that $q_{\Delta,\Sigma}$ is non-zero (so $q_{\Delta,\Sigma}=q$). 
\end{enumerate}
\end{remark}
\begin{proof}
For $\Sigma\in\mathscr J(\Gamma)$ and $\Delta\in\mathscr J(\Lambda)$, let $p_{\Sigma,\Delta}\in\mathcal Z(M_\Sigma'\cap pPp)p_\Sigma$ and $q_{\Delta,\Sigma}\in\mathcal Z(N_\Delta'\cap qPq)q_\Delta$ be the maximal projections with $M_\Sigma p_{\Sigma,\Delta}\prec_P^sN_\Delta$ and $N_\Delta q_{\Delta,\Sigma}\prec_P^s M_\Sigma$, respectively 
(their existence is proved as in Remark~\ref{rk:maximal-projections}).
We claim that 
\begin{equation}\label{p_sigma}\text{$\bigvee_{\Delta\in\mathscr J(\Lambda)}p_{\Sigma,\Delta}=p_\Sigma$, for every $\Sigma\in\mathscr J(\Gamma)$}.
\end{equation}
Fix $\Sigma\in\mathscr J(\Gamma)$ and let $p_0=p_\Sigma-\bigvee_{\Delta\in\mathscr J(\Lambda)}p_{\Sigma,\Delta}\in\mathcal Z(M_\Sigma'\cap pPp)p_\Sigma$. 
Assume by contradiction that $p_0\not=0$.
Write $\Sigma=\Sigma_1\circ\Sigma_2$, where $\Sigma_1,\Sigma_2\subset\Sigma$ are nonempty full subgraphs.
Since $M_\Sigma=M_{\Sigma_1}\overline{\otimes}M_{\Sigma_2}$, $M_\Sigma p_0$
has no amenable direct summand and $M_\Sigma p_0\nprec_P N_w$, for every isolated vertex $w\in \Lambda$,
 Corollary~\ref{join_embed} implies that $M_\Sigma p_0\prec_PN_{\Delta}$ for some join full subgraph $\Delta\subset\Lambda$. After possibly replacing $\Delta$ with a larger subgraph, we may assume that $\Delta\in\mathscr J(\Lambda)$.
 By Proposition \ref{results}(3) we can find a nonzero projection $p'\in\mathcal Z(M_\Sigma'\cap M_\Gamma)p_0$ such that  $M_\Sigma p'\prec^s_P N_\Delta$. 
 Since $p'\leq  p_0\leq p_\Sigma-p_{\Sigma,\Delta}$, this contradicts the definition of $p_{\Sigma,\Delta}$ and proves \eqref{p_sigma}. Similarly, we get
\begin{equation}\label{q_delta}\text{$\bigvee_{\Sigma\in\mathscr J(\Gamma)}q_{\Delta,\Sigma}=q_\Delta$, for every $\Delta\in\mathscr J(\Lambda)$}.
\end{equation}

Next, let  $\Sigma\in\mathscr J(\Gamma)$ and $\Delta\in\mathscr J(\Lambda)$ with $p_{\Sigma,\Delta}\not=0$.  Assume that $M_\Sigma p_{\Sigma,\Delta}\not\prec_P^sN_\Delta q_{\Delta,\Sigma}$. Since $M_\Sigma p_{\Sigma,\Delta}\prec_P^sN_\Delta$,  it follows that
$M_\Sigma p_{\Sigma,\Delta}\prec_PN_\Delta(q-q_{\Delta,\Sigma})$, where we recall from Notation~\ref{notation} that $q$ is the unit of $N_\Lambda$. Since $M_\Sigma p_{\Sigma,\Delta}$ has no amenable direct summand, while $N_\Delta(q-q_\Delta)$ is amenable, we get that $M_\Sigma p_{\Sigma,\Delta}\nprec_PN_\Delta(q-q_\Delta)$. Combining the last two facts gives that $M_\Sigma p_{\Sigma,\Delta}\prec_PN_\Delta(q_\Delta-q_{\Delta,\Sigma})$. Since $q_\Delta-q_{\Delta,\Sigma}\leq \bigvee_{\Sigma'\in\mathscr J(\Gamma)\setminus\{\Sigma\}}q_{\Delta,\Sigma'}$ by \eqref{q_delta}, it follows that $M_\Sigma p_{\Sigma,\Delta}\prec_P N_{\Delta}q_{\Delta,\Sigma'}$, for some $\Sigma'\in\mathscr{J}(\Gamma)\setminus\{\Sigma\}$. By combining this fact with $N_{\Delta}q_{\Delta,\Sigma'}\prec_P^sM_{\Sigma'}$ and  
 Proposition \ref{results}(2), we get that $M_\Sigma p_{\Sigma,\Delta}\prec_P M_{\Sigma'}$.  By Lemma~\ref{basic}(2) we further get that $\Sigma\subset\Sigma'$.
 Since $\Sigma,\Sigma'\in\mathscr J(\Gamma)$, this  gives $\Sigma=\Sigma'$, a contradiction. 
 
 This proves that $M_\Sigma p_{\Sigma,\Delta}\prec_P^sN_\Delta q_{\Delta,\Sigma}$. Similarly,  $N_\Delta q_{\Delta,\Sigma}\prec_P^s M_\Sigma p_{\Sigma,\Delta}$, which proves (2).
 
 Finally, by \eqref{p_sigma} and symmetry, in order to prove (1), it suffices to show that $p_{\Sigma,\Delta}p_{\Sigma,\Delta'}=0$, for every $\Sigma\in\mathscr{J}(\Gamma)$ and $\Delta,\Delta'\in\mathscr{J}(\Lambda)$ with $\Delta\not=\Delta'$. Assume by contradiction that $p''=p_{\Sigma,\Delta}p_{\Sigma,\Delta'}\not=0$.
 Then $M_{\Sigma}p''\prec^s_P N_{\Delta}$ and $M_{\Sigma}p''\prec_P^s N_{\Delta'}$. By Corollary \ref{intersection} we get that $M_\Sigma p''\prec_P^s N_{\Delta\cap\Delta'}$.
  On the other hand, Lemma \ref{proof_of_3} below gives that 
   $N_\Delta q_{\Delta,\Sigma}\prec_P M_\Sigma p''$ and $N_{\Delta'} q_{\Delta',\Sigma}\prec_P M_\Sigma p''$. By using Proposition \ref{results}(2) we derive that $N_\Delta q_{\Delta,\Sigma}\prec_P N_{\Delta\cap\Delta'}$ and $N_{\Delta'} q_{\Delta',\Sigma}\prec_P N_{\Delta\cap\Delta'}$, 
   so in particular $N_\Delta\prec_P N_{\Delta\cap\Delta'}$ and $N_{\Delta'}\prec_P N_{\Delta\cap\Delta'}$. By Lemma~\ref{basic}(2) we deduce that $\Delta\subset\Delta\cap\Delta'$ and $\Delta'\subset\Delta\cap\Delta'$, which contradicts that $\Delta\not=\Delta'$ and thus proves (1). 
\end{proof}

\begin{lemma}\label{proof_of_3}
In the setting from Notation \ref{notation}, assume that $\Sigma\subset\Gamma, \Delta\subset\Lambda$ are full subgraphs and $p'\in\mathcal Z(M_\Sigma'\cap pPp), q'\in\mathcal Z(N_\Delta'\cap qPq)$ are nonzero projections such that $M_\Sigma p'\prec_P^s N_\Delta q'$ and $N_\Delta q'\prec_P^s M_\Sigma p'$. 

Then $M_\Sigma p'\prec_P N_\Delta q''$, for every nonzero projection $q''\in (N_\Delta'\cap qPq)q'$.
\end{lemma}

\begin{proof}
Assume by contradiction that there is a nonzero projection $q''\in (N_\Delta'\cap qPq)q'$ such that $M_\Sigma p'\nprec_P N_\Delta q''$. Let $q_0\in\mathcal Z(N_\Delta'\cap qPq)q'$ be the central support of $q''$ in $N_\Delta'\cap qPq$.

We claim that $M_\Sigma p'\nprec_P N_\Delta q_0$. Indeed, assume towards a contradiction that $M_\Sigma p'\prec_P N_\Delta q_0$. By \cite[Lemma~2.4(4)]{DHI19}, we can find a non-zero projection $z\in\mathcal Z(N_{\Delta}'\cap qPq)q_0$ such that $M_\Sigma p'\prec_P N_\Delta r$ for every non-zero projection $r\in (N_\Delta'\cap qPq)z$. Since $q_0=\bigvee_{u\in\mathcal U(N_\Delta'\cap qPq)}uq''u^*$, we can find $u\in\mathcal U(N_\Delta'\cap qPq)$ such that $z(uq''u^*)\neq 0$. This implies the existence of non-zero equivalent projections $z_1,q_1\in N_\Delta'\cap qPq$ such that $z_1\leq z$ and $q_1\leq uq''u^*$. By definition of $z$, we have $M_\Sigma p'\prec_P N_\Delta z_1$, hence $M_\Sigma p'\prec_P N_\Delta q_1$. Since $N_\Delta q_1\subset u(N_\Delta q'')u^*$, it follows that $M_\Sigma p'\prec_P N_\Delta q''$, a contradiction. This proves our claim.

Since $M_\Sigma p'\prec_P^s N_\Delta q'$, we get that $M_\Sigma p'\prec_P^s N_\Delta (q'-q_0)$. Since $q_0\leq q'$, we also have $N_\Delta q_0\prec_P^s M_\Sigma p'$. Using Proposition \ref{results}(2), we derive that $N_\Delta q_0\prec_P N_\Delta (q'-q_0)$. 
Notice also that $N_\Delta \not\prec_P N_{\Delta'}$ for any proper full subgraph $\Delta'\subsetneq\Delta$, in view of Lemma~\ref{basic}(2). Since $q',q_0\in\mathcal Z(N_\Delta'\cap qPq)$, this contradicts Corollary \ref{center_corner} and finishes the proof.
\end{proof}

\subsection{A combinatorial lemma.}\label{sec:combinatorial}

In preparation for the next subsection, we now establish a combinatorial lemma.
Given two vertices $v,v'\in \Gamma$, following \cite{CV09}, we write $v\le v'$ if $\lk(v)\subset\st(v')$. This is a partial order on the vertex set of $\Gamma$. We let $\sim$ be the associated equivalence relation (i.e.,\ $v\sim v'$ if $v\le v'$ and $v'\le v$), and we denote by $[v]$ the $\sim$-class of a vertex $v$. It turns out that every equivalence class of vertices is either a collapsible clique, or an edgeless collapsible  subgraph \cite[Lemma~2.3]{CV09}. We let $\mathcal{E}(\Gamma)$ be the graph whose vertices are $\sim$-equivalence classes of vertices of $\Gamma$, where two distinct classes $[v]$ and $[w]$ are joined by an edge if $v$ and $w$ are adjacent (this does not depend on the choice of representatives in the class). We denote by $\calp(\Gamma)$ (resp.\ $\calp(\Lambda)$) the set of all nonempty subsets of vertices of $\Gamma$ (resp.\ $\Lambda$).

\begin{lemma}\label{lemma:combinatorial2}
Let $\Gamma,\Lambda$ be clique-reduced finite simple graphs. Let $\alpha:\Gamma\to\calp(\Lambda)$ and $\beta:\Lambda\to\calp(\Gamma)$ be maps such that
    \begin{enumerate}
        \item if two distinct vertices $v,v'\in \Gamma$ are adjacent, then $\alpha(v')\subset\alpha(v)^{\perp}$;
        \item if two distinct vertices $w,w'\in \Lambda$ are adjacent, then $\beta(w')\subset\beta(w)^{\perp}$;
        \item for any vertices $v\in\Gamma$ and $w\in\Lambda$, we have $w\in\alpha(v)$ if and only if $v\in\beta(w)$.  \end{enumerate}
Then for every $v\in \Gamma$, the set $\alpha([v])$ is a $\sim$-class of vertices of $\Lambda$, and for every $w\in \Lambda$, the set $\beta([w])$ is a $\sim$-class of vertices of $\Gamma$. In addition the induced maps $\bar\alpha:\mathcal{E}(\Gamma)\to\mathcal{E}(\Lambda)$ and $\bar\beta:\mathcal{E}(\Lambda)\to\mathcal{E}(\Gamma)$ are graph isomorphisms, which are inverses of each other.         
\end{lemma}

\begin{proof}
    Let $v\in \Gamma$. We first prove that $\alpha([v])$ is contained in a single $\sim$-class. So assume towards a contradiction that $\alpha([v])$ contains two inequivalent vertices $w_1,w_2$. Then up to swapping $w_1$ and $w_2$, we can find $w'\in \Lambda$ that is adjacent to $w_1$ but not to $w_2$. It follows from (2) that $\beta(w')\subset\beta(w_1)^{\perp}$. Let $v_1,v_2\in [v]$ be such that $w_1\in\alpha(v_1)$ and $w_2\in\alpha(v_2)$. By (3) we have $v_1\in\beta(w_1)$, so $\beta(w')\subset\lk(v_1)$. Let $v'\in\beta(w')$. Then $v'\in\lk(v_1)$. Since $\Gamma$ is clique-reduced, the equivalence class $[v]$ corresponds to an edgeless subgraph, so $v'\in\lk(v_2)$. Then (1) implies that $\alpha(v')\subset\alpha(v_2)^\perp$, in particular $w'\in\lk(w_2)$, a contradiction. 

    By symmetry, for every $w\in \Lambda$, the set $\beta([w])$ is contained in a single $\sim$-class.

We now prove that $\alpha([v])$ is equal to a $\sim$-class. So let $w_1\in\alpha([v])$, and let $w_2\in \Lambda$ with $w_1\sim w_2$. We aim to prove that $w_2\in\alpha([v])$. Let $v_1\in [v]$ be such that $w_1\in\alpha(v_1)$, then (3) implies that $v_1\in\beta([w_1])$. Since by the above $\beta([w_1])$ is contained in a single $\sim$-class, it follows that $\beta(w_2)\subset [v]$. Let $v_2\in\beta(w_2)$, then by (3) we have $w_2\in\alpha(v_2)$, so $w_2\in\alpha([v])$, as desired.

By symmetry, for every $w\in \Lambda$, the set $\beta([w])$ is a single $\sim$-class. The induced maps $\bar\alpha:\mathcal{E}(\Gamma)\to\mathcal{E}(\Lambda)$ and $\bar\beta:\mathcal{E}(\Lambda)\to\mathcal{E}(\Gamma)$ preserve adjacency (by (1) and (2)) and are inverse to each other (by (3)), so the lemma follows. 
\end{proof}

\subsection{A conjugacy criterion} We next prove the following conjugacy criterion which will guarantee that the conclusions of Theorems~\ref{arbitrary},~\ref{amenable} and~\ref{factors} hold. In order to cover more cases, we state a general version that does not require the graphs $\Gamma$ and $\Lambda$ to be transvection-free, see Remark~\ref{rk:conjugacy-criterion} below. We insist that in the statement below, even if the vertices $v,w$ belong to the untransvectable subgraphs $\Gamma^u,\Lambda^u$, their stars $\st(v)$ and $\st(w)$ are always considered in the ambient graphs $\Gamma,\Lambda$.

\begin{proposition}\label{conjugacy_criterion}
In the setting from Notation \ref{notation}, assume that for every $v\in\Gamma^u$ and $w\in\Lambda^u$, there exist
projections $p_{v,w}\in\mathcal Z(M_{\emph{st}(v)}'\cap pPp)$ and $q_{w,v}\in\mathcal Z(N_{\emph{st}(w)}'\cap qPq)$, such that the following conditions hold: 
\begin{itemize}
\item[(a)] $\sum_{w\in\Lambda^u}p_{v,w}=p$, for every $v\in\Gamma^u$, and $\sum_{v\in\Gamma^u}q_{w,v}=q$, for every $w\in\Lambda^u$.
\item[(b)] $M_{v}p_{v,w}\prec_P^s N_{w}q_{w,v}$ and $N_{w}q_{w,v}\prec_P^s M_{v}p_{v,w}$, for every $v\in\Gamma^u$ and $w\in\Lambda^u$.
\end{itemize}
Then the following hold.
\begin{enumerate}
\item If $\Gamma^u$ and $\Lambda^u$ are clique-reduced, then there exists a graph isomorphism $\bar\alpha:\mathcal E(\Gamma^u)\rightarrow\mathcal E(\Lambda^u)$ such that for every $v\in\Gamma^u$ and $w\in\Lambda^u$, we have $M_{v}\prec_P^sN_{\bar\alpha([v])}$ and $N_{w}\prec_P^sM_{\bar\alpha^{-1}([w])}$ (where $[v]$ denotes the $\sim$-equivalence class of $v$ in $\Gamma^u$, and $[w]$ denotes the $\sim$-equivalence class of $w$ in $\Lambda^u$).
\item If all vertex algebras are II$_1$ factors, then there exists a graph isomorphism $\alpha:\Gamma^u\rightarrow\Lambda^u$ such that for every $v\in\Gamma^u$ and $w\in\Lambda^u$, we have $M_{v}\prec_P^sN_{\alpha(v)}$ and $N_{w}\prec_P^sM_{\alpha^{-1}(w)}$.
\end{enumerate}
\end{proposition}

\begin{remark}\label{rk:conjugacy-criterion}
The first statement applies in particular when $\Gamma$ and $\Lambda$ are transvection-free, with no additional assumptions on the vertex algebras other than they are diffuse tracial von Neumann algebras. In this case $\Gamma^u=\Gamma$ and $\Lambda^u=\Lambda$ are clique-reduced, and the conclusion is that there is a graph isomorphism $\alpha:\Gamma\to \Lambda$ such that for every $v\in\Gamma$, we have $M_v\prec^s_PN_{\alpha(v)}$ and $N_{\alpha(v)}\prec^s_PM_v$.
\end{remark}

Before proving Proposition~\ref{conjugacy_criterion}, let us start with a lemma that will be useful in the proof.

\begin{lemma}\label{rel_comm}
   In the setting from Notation~\ref{notation}, let $v\in\Gamma$ be an untransvectable vertex. Then \[\mathcal{Z}(M'_{\st(v)}\cap pPp)=\mathcal{Z}(M'_v\cap pPp)=\mathcal{Z}(M_v).\] 
\end{lemma}

\begin{proof}
 Since $\text{st}(v)^\perp=\emptyset$,  Lemma~\ref{basic}(1) implies that $$\mathcal Z(M_{\text{st}(v)}'\cap pPp)=\mathcal Z(M_v'\cap pPp)=\mathcal Z(M_v)\overline{\otimes}\mathcal Z(M_C),$$ where $C$ is the maximal clique factor of $\text{lk}(v)$. If $v'\in C$, then $\text{st}(v)\subset\text{st}(v')$ which contradicts the fact that $v\in\Gamma$ is untransvectable since $v'\not=v$. Thus,  $C=\emptyset$, which implies the lemma.    
\end{proof}

\begin{proof}[Proof of Proposition~\ref{conjugacy_criterion}] 
We first note that by combining (b) with Proposition \ref{normalizer'} it follows that
\begin{equation}\label{star}
\text{$M_{\text{st}(v)}p_{v,w}\prec_P^s N_{\text{st}(w)}q_{w,v}$ and $N_{\text{st}(w)}q_{w,v}\prec_P^s M_{\text{st}(v)}p_{v,w}$, for every $v\in\Gamma^u$ and $w\in\Lambda^u$.}
\end{equation}
We denote $E_v=\{w\in\Lambda^u\mid p_{v,w}\not=0\}$, for every $v\in\Gamma^u$, and $F_w=\{v\in\Gamma^u\mid q_{w,v}\not=0\}$, for every $w\in\Lambda^u$. Condition (b) guarantees that $w\in E_v$ if and only if $v\in F_w$, for every $v\in\Gamma^u$ and $w\in\Lambda^u$.

Next, we prove the following claim.

\begin{claim}\label{perp}  If $v,v'\in \Gamma^u$ are adjacent, then $E_{v'}\subset E_v^\perp$.
\end{claim}
\begin{proof}
Let $w\in E_v$. 
We first show that 
\begin{equation}\label{eq:mv}
    M_{v'}\prec^s_P N_{\st(w)}q_{w,v}.
\end{equation}
To this end, note that 
$M_{\text{st}(v)}p_{v,w}\prec_P^sN_{\text{st}(w)}q_{w,v}$, thus 
\begin{equation}\label{v'_to_st(w)}M_{v'}p_{v,w}\prec_P^sN_{\text{st}(w)}q_{w,v}.\end{equation}
 By Proposition \ref{results}(3) we find a projection $z\in\mathcal Z(M_{v'}'\cap pPp)$ with $p_{v,w}\leq z$ and $M_{v'}z\prec_P^sN_{\text{st}(w)}q_{w,v}$. By Lemma~\ref{rel_comm}, we have $\mathcal Z(M_{v'}'\cap pPp)=\mathcal Z(M_{v'})$, so
$z\in \mathcal Z(M_{v'})$. Lemma~\ref{rel_comm} also gives $\mathcal Z(M_{\text{st}(v)}'\cap pPp)=\mathcal Z(M_v)$, so 
 $p_{v,w}\in\mathcal Z(M_v)$. Moreover, if $x\in M_v$, then $\text{E}_{M_{v'}}(x)=\tau(x)p$, where we consider the trace $\tau$ on $M_\Gamma=pPp$ normalized so that $\tau(p)=1$.
  Using these facts and applying $\text{E}_{\mathcal Z(M_{v'})}$ to the inequality $p_{v,w}\leq z$ gives $$\tau(p_{v,w})p=\text{E}_{\mathcal Z(M_{v'})}(p_{v,w})\leq \text{E}_{\mathcal Z(M_{v'})}(z)=z.$$ Since $z\leq p$ and $p_{v,w}\not=0$, this forces $z=p$. Therefore, $M_{v'}\prec_P^s N_{\text{st}(w)}q_{w,v}$, which proves \eqref{eq:mv}.  

We next prove that 
\begin{equation}\label{E_v}E_{v'}\subset\text{st}(w).\end{equation}
 If $w'\in E_{v'}$, then $N_{w'}\prec_P M_{v'}$. Combining this fact with \eqref{eq:mv} and using Proposition \ref{results}(2), we get $N_{w'}\prec_PN_{\text{st}(w)}q_{w,v}$, and hence $N_{w'}\prec_PN_{\st(w)}$. By Lemma~\ref{basic}(2) we get that $w'\in\text{st}(w)$. This altogether proves \eqref{E_v}.

Finally, we show that \begin{equation}\label{Ev'} w\not\in E_{v'}.\end{equation}

 Assume by contradiction that $w\in E_{v'}$. Then we have  $p_{v',w}\not=0$ and $M_{v'}p_{v',w}\prec_P^sN_w$. Since  $M_{v'}\prec_P^s N_{\text{st}(w)}q_{w,v}$ by equation~\eqref{eq:mv}, we also have $M_{v'}p_{v',w}\prec_P^s N_{\text{st}(w)}q_{w,v}$. By combining these facts, Corollary \ref{right_corner} implies that $M_{v'}p_{v',w}\prec_P^sN_wq_{w,v}$. Since $N_wq_{w,v}\prec_P^sM_vp_{v,w}$, applying Proposition \ref{results}(2) gives $M_{v'}p_{v',w}\prec_PM_vp_{v,w}$. However, Lemma~\ref{basic}(2) implies that $v'=v$, which gives a contradiction. This proves \eqref{Ev'}. 
 
 By combining \eqref{E_v} and \eqref{Ev'}, we conclude that $E_{v'}\subset\text{lk}(w)$. Since this holds for every $w\in E_v$, we deduce that $E_{v'}\subset\cap_{w\in E_v}\text{lk}(w)=E_v^\perp$, which finishes the proof of the claim. \end{proof}

By symmetry, repeating the proof of Claim \ref{perp} implies that
\begin{equation}
\label{F_w}\text{$F_{w'}\subset F_w^\perp$, for every $w,w'\in\Lambda^u$ with $w'\in\text{lk}(w)$.}
\end{equation}

Let $\alpha:\Gamma^u\to\calp(\Lambda^u)$ and $\beta:\Lambda^u\to\calp(\Gamma^u)$ be given by $\alpha(v)=E_v$ and $\beta(w)=F_w$.  

To prove (1), we now assume that $\Gamma^u$ and $\Lambda^u$ are clique-reduced. It follows from the above that $\alpha$ and $\beta$ satisfy the assumptions from Lemma~\ref{lemma:combinatorial2} (applied to the graphs $\Gamma^u$ and $\Lambda^u$). It follows from that lemma that $\alpha$ induces an isomorphism $\bar\alpha:\mathcal{E}(\Gamma^u)\to\mathcal{E}(\Lambda^u)$. Let now $v\in\Gamma^u$. Then $E_v\subset\bar\alpha([v])$, where $[v]$ denotes the $\sim$-equivalence class of $v$ in $\Gamma^u$. For every $w\in\bar\alpha([v])$, we have $M_v p_{v,w}\prec^s_P N_w$, so $M_v p_{v,w}\prec^s_P N_{\bar\alpha([v])}$. Since $\sum_{w\in\bar\alpha([v])} p_{v,w}=p$, it follows that $M_v\prec^s_P N_{\bar\alpha([v])}$. The same argument shows that for every $w\in\Lambda^u$, one has $N_w\prec^s_P M_{\bar\alpha^{-1}([w])}$, which completes the proof of (1).

To prove (2), we now assume that all vertex algebras are II$_1$ factors. Then the sets $E_v$ and $F_w$ are singletons (see Remark \ref{a_few_remarks}(4)) and it immediately follows that $\alpha:\Gamma^u\rightarrow\Lambda^u$ and $\beta:\Lambda^u\rightarrow\Gamma^u$ are graph isomorphisms, which are inverse of each other. The conclusion follows.
 \end{proof}

\section{Arbitrary vertex algebras: Proof of Theorem \ref{arbitrary}} \label{sec:arbitrary}

We are now in position to prove the following result which generalizes Theorem~\ref{arbitrary}. 
The basic consequences of the double intertwining obtained in the conclusion will be clarified in Section~\ref{sec:intertwining}.

\begin{theorem}\label{arbitrary_gen}
In the setting from Notation \ref{notation}, assume additionally that $\Gamma$ and $\Lambda$ are transvection-free, square-free and not reduced to a vertex. 

Then the graphs $\Gamma$ and  $\Lambda$ are isomorphic, and there is a graph isomorphism $\alpha:\Gamma\rightarrow\Lambda$ such that $M_v\prec_P^sN_{\alpha(v)}$ and $N_{\alpha(v)}\prec_P^sM_v$, for every $v\in\Gamma$.
\end{theorem}

 In preparation for the proof of Theorem \ref{arbitrary_gen}, we prove the following lemma. 
 We recall first from Notation~\ref{nota:j} that for a graph $\Omega$, we denote by $\mathscr{J}(\Omega)$ the set of maximal join subgraphs of $\Omega$.

\begin{lemma}\label{classA}
Let $\Omega$ be a finite simple graph which is transvection-free, square-free and not reduced to a vertex. Then $\emph{st}(v)\in\mathscr J(\Omega)$, the maximal clique factor of $\emph{st}(v)$ is $\{v\}$, and $\emph{lk}(v)$ is an irreducible graph, for every $v\in\Omega$.
Moreover, every graph $\Sigma\in\mathscr J(\Omega)$ is equal to $\emph{st}(v)$, for some $v\in\Omega$.
\end{lemma}

\begin{proof}
Let $v\in\Omega$. Since $v$ is untransvectable, it is not isolated. Thus, $\text{st}(v)\not=\{v\}$ and $\text{st}(v)$ is a join graph. If $v'$ belongs to the maximal clique factor of $\text{st}(v)$, then 
$\text{st}(v)\subset\text{st}(v')$, which forces $v'=v$. Thus, the maximal clique factor of $\text{st}(v)$ is equal to $\{v\}$. 

Let $\Sigma\in\mathscr J(\Omega)$. Decompose $\Sigma=C\circ\Sigma_1\circ\cdots\circ\Sigma_n$, where $C$ is the maximal clique factor of $\Sigma$ and $\Sigma_1,\dots,\Sigma_n$ are irreducible graphs with at least two vertices. If $n\geq 2$, then since $\Sigma_i$  is not a clique, we can find vertices $v_i,v_i' \in\Sigma_i$ which are not adjacent, for every $i\in\{1,2\}$. Then the vertices $v_1,v_2,v_1',v_2'$ would form a square, which is a contradiction. This implies that  $n\leq 1$. Hence, $C\not=\emptyset$, since $\Sigma$ is a join graph. 
If $w\in C$, then $\Sigma\subset\text{st}(w)$ and the maximality of $\Sigma$ implies that $\Sigma=\text{st}(w)$.

Finally, let $v\in\Omega$. Since $\text{st}(v)\subset\Omega$ is a join full subgraph, we can find $\Sigma\in\mathscr J(\Omega)$ containing $\text{st}(v)$. By the previous paragraph, we have that $\Sigma=\text{st}(w)$, for some $w\in\Omega$. Since $\text{st}(v)\subset\text{st}(w)$, we get that $w=v$, thus $\text{st}(v)=\Sigma$ is a maximal join full subgraph. 
Since $\{v\}$ is the maximal clique factor of $\text{st}(v)$, the previous paragraph implies that $\text{lk}(v)$ is an irreducible graph.
This finishes the proof.
\end{proof}

\begin{proof}[\bf Proof of Theorem \ref{arbitrary_gen}]
By Lemma \ref{classA}, $\mathscr J(\Gamma)=\{\text{st}(v)\mid v\in\Gamma\}$ and $\mathscr J(\Lambda)=\{\text{st}(w)\mid w\in\Lambda\}$.
If $v\in\Gamma$ and $w\in\Lambda$, then $\text{st}(v)$ and $\text{st}(w)$ are not cliques by Lemma \ref{classA},  thus $M_{\text{st}(v)}$ and $N_{\text{st}(w)}$ have no amenable direct summands by Lemma~\ref{basic}(3). Moreover, since $\Gamma$ and $\Lambda$ are transvection-free and are not reduced to a vertex, they have no isolated vertices.

By applying Proposition \ref{structure_of_aut} 
(see also Remark~\ref{a_few_remarks}(3)) we 
find projections $p_{v,w}\in\mathcal Z(M_{\text{st}(v)}'\cap pPp)$ and $q_{w,v}\in\mathcal Z(N_{\text{st}(w)}'\cap qPq)$, for  every $v\in\Gamma$ and $w\in\Lambda$, such that the following conditions hold: 
\begin{enumerate}
\item $\sum_{w\in\Lambda}p_{v,w}=p$, for every $v\in\Gamma$, and $\sum_{v\in\Gamma}q_{w,v}=q$, for every $w\in\Lambda$;
\item $M_{\text{st}(v)}p_{v,w}\prec_P^s N_{\text{st}(w)}q_{w,v}$ and $N_{\text{st}(w)}q_{w,v}\prec_P^s M_{\text{st}(v)}p_{v,w}$, for every $v\in\Gamma$ and $w\in\Lambda$.
\end{enumerate}

Our next goal is to show the following claim.

\begin{claim}\label{vertices}
$M_vp_{v,w}\prec_P^s N_wq_{w,v}$ and $N_wq_{w,v}\prec_P^s M_vp_{v,w}$, for every $v\in\Gamma$ and $w\in\Lambda$.
\end{claim}

\begin{proof}
Let $v\in\Gamma,w\in\Lambda$ with $p_{v,w}\not=0$. By symmetry,  it suffices to argue that $M_vp_{v,w}\prec_P^s N_wq_{w,v}$. Let $p_0\in\mathcal Z(M_v'\cap pPp)p_{v,w}$ be a nonzero projection. Since $\Gamma$ is transvection-free, Lemma~\ref{rel_comm} implies that $\mathcal Z(M_{\text{st}(v)}'\cap pPp)=\mathcal Z(M_v)=\mathcal Z(M_v'\cap pPp)$. Thus, $p_0\in\mathcal Z(M_{\text{st}(v)}'\cap pPp)p_{v,w}$.

Corollary \ref{center_corner} and Lemma~\ref{basic}(2) give that $M_{\text{st}(v)}z\nprec_PM_{\text{st}(v)}(p-z)$, for every nonzero projection $z\in M_{\text{st}(v)\cup\text{st}(v)^\perp}'\cap pPp=M_{\text{st}(v)}'\cap pPp$. By Lemma \ref{classA}, the maximal clique factor of $\text{st}(w)$ is equal to $\{w\}$.
By using (2) and applying Corollary \ref{tensor_dec} to the decomposition $M_{\text{st}(v)}p_{v,w}=M_vp_{v,w}\vee M_{\text{lk}(v)}p_{v,w}$ we find a nonzero projection $p'\in\mathcal Z(M_{\text{st}(v)}'\cap pPp)p_0$ and a decomposition $\text{lk}(w)=\Lambda_1\circ\Lambda_2$ such that \begin{equation}\label{split}\text{$M_vp'\prec_P^s N_{\{w\}\cup\Lambda_1}q_{w,v}$ \;\;\; and \;\;\; $M_{\text{lk}(v)}p'\prec_P^s N_{\{w\}\cup\Lambda_2}q_{w,v}$.}\end{equation}

However, since $\text{lk}(w)$ is an irreducible graph according to Lemma \ref{classA}, $\Lambda_1$ or $\Lambda_2$ is empty. 
We claim that $\Lambda_1=\emptyset$. Assuming by contradiction that $\Lambda_1\not=\emptyset$, we get that
$\Lambda_2=\emptyset$ and therefore $M_{\text{lk}(v)}p'\prec_P^s N_wq_{w,v}$ by \eqref{split}. Let $v'\in\text{lk}(v)$. Then $M_{v'}p'\prec_P^s N_wq_{w,v}$. Proposition \ref{normalizer'} gives $M_{\text{st}(v')}p'\prec_P^s N_{\text{st}(w)}q_{w,v}$. Since (2) gives $N_{\text{st}(w)}q_{w,v}\prec_P^s M_{\text{st}(v)}p_{v,w}$, Proposition \ref{results}(2) implies that $M_{\text{st}(v')}p'\prec_P M_{\text{st}(v)}p_{v,w}$. By Lemma~\ref{basic}(2) we get $\text{st}(v')\subset \text{st}(v)$. Since $v'\not=v$, this contradicts the fact that $\Gamma$ is transvection-free, which proves our claim that $\Lambda_1=\emptyset$.

Since $\Lambda_1=\emptyset$, by Equation~\eqref{split} we get $M_vp'\prec_P^s N_wq_{w,v}$. In conclusion, for every nonzero projection $p_0\in\mathcal Z(M_v'\cap pPp)p_{v,w}$, we found a nonzero projection $p'\in\mathcal Z(M_v'\cap pPp)p_0$ such that $M_vp'\prec_P^s N_wq_{w,v}$.
This implies that $M_vp_{v,w}\prec_P^sN_wq_{w,v}$, which finishes the proof of the claim.  \end{proof}

By using Claim \ref{vertices} and applying Proposition \ref{conjugacy_criterion}(1) (with $\Gamma=\Gamma^u$ and $\Lambda=\Lambda^u$), the conclusion of Theorem \ref{arbitrary_gen} follows.
\end{proof}

\section{Amenable vertex algebras: Proof of Theorem \ref{amenable}}\label{sec:amenable}

In this section we prove Theorem~\ref{amenable}, dealing with the case where all vertex algebras are amenable. The consequences of this result to the $W^*$-classification of right-angled Artin groups, and to the classification of graph products of hyperfinite II$_1$ factors, will be recorded in Section~\ref{applications}.

\subsection{Proof of Theorem~\ref{amenable}}

Let $\Gamma$ be a finite simple graph. Recall that the untransvectable subgraph $\Gamma^u$ of $\Gamma$ is the full subgraph of $\Gamma$ spanned by its untransvectable vertices. We also recall the definition of the graph $\mathcal{E}(\Gamma^u)$ from Section~\ref{sec:combinatorial}. We finally recall that $\Gamma$ is clique-reduced if it does not contain any collapsible complete subgraph on at least $2$ vertices.

The goal of the present section is to prove the following result which generalizes Theorem \ref{amenable}.

\begin{theorem}\label{amenable_gen}
In the setting from Notation \ref{notation}, assume additionally that the graphs $\Gamma$ and $\Lambda$ are clique-reduced, and that
$(M_v,\tau_v)$ and $(N_w,\tau_w)$ are amenable, for every $v\in \Gamma$ and $w\in\Lambda$. 
\begin{enumerate}
\item If $\Gamma^u,\Lambda^u$ are clique-reduced, then the graphs $\mathcal{E}(\Gamma^u)$ and $\mathcal{E}(\Lambda^u)$ are isomorphic, and there exists a graph isomorphism $\bar\alpha:\mathcal{E}(\Gamma^u)\rightarrow\mathcal{E}(\Lambda^u)$ such that $M_v\prec_P^sN_{\bar\alpha([v])}$ and $N_{w}\prec_P^sM_{\bar\alpha^{-1}([w])}$, for every $v\in\Gamma^u$ and $w\in\Lambda^u$.
\item If $M_v$ and $N_w$ are II$_1$ factors, for every $v\in\Gamma$ and $w\in\Lambda$, then the graphs $\Gamma^u$ and $\Lambda^u$ are isomorphic, and there exists a graph isomorphism $\alpha:\Gamma^u\to\Lambda^u$ such that $M_v\prec_P^sN_{\alpha(v)}$ and $N_{\alpha(v)}\prec_P^s M_v$, for every $v\in \Gamma^u$.
\end{enumerate}
\end{theorem}

\begin{remark}
We make a few comments on the statement.
\begin{enumerate}
 \item   If $\Gamma$ is transvection-free, then $\Gamma^u=\Gamma$ and $\Gamma$ is clique-reduced, 
 and $\mathcal{E}(\Gamma)=\Gamma$. Therefore, if $\Gamma$ and $\Lambda$ are both transvection-free, then the first part of Theorem~\ref{amenable_gen} provides an isomorphism between $\Gamma$ and $\Lambda$ which satisfies the conclusion of Theorem~\ref{amenable}.

 \item   
 Our assumption that $\Gamma$ and $\Lambda$ are clique-reduced is mild: if $\Gamma$ contains a collapsible clique $C$, then we can collapse this clique to a single vertex, and get a new graph product structure for $M_\Gamma$ where the new vertex algebra is $\overline\otimes_{v\in C}M_v$, and therefore still amenable.
    
 \item   The following question remains open: in the setting of Notation~\ref{notation}, if $\Gamma$ and $\Lambda$ contain no collapsible subgraph on at least two vertices that is either complete or edgeless, and if all vertex algebras are amenable, must $\Gamma$ and $\Lambda$ be isomorphic? Theorem~\ref{amenable_gen}(1) shows, on the other hand, that if we further assume the untransvectable subgraphs $\Gamma^u,\Lambda^u$ to have no collapsible subgraph on at least two vertices that is either complete or edgeless, then we have an isomorphism between $\Gamma^u$ and $\Lambda^u$.
 \end{enumerate}
\end{remark}

The proof of Theorem \ref{amenable_gen} relies on a few auxiliary results.

\begin{proposition}\label{clique_factor}
In the setting from Notation \ref{notation}, assume additionally that $(M_v,\tau_v)$ and $(N_w,\tau_w)$ are amenable, for every $v\in \Gamma$ and $w\in\Lambda$. 
Let $\Sigma\in\mathscr J(\Gamma)$ and $\Delta\in\mathscr J(\Lambda)$. Let $p'\in \mathcal Z(M_\Sigma'\cap pPp)$ and $q'\in \mathcal Z(N_\Delta'\cap qPq)$ be nonzero projections such that $M_\Sigma p'\prec_P^s N_\Delta q'$ and $N_\Delta q'\prec_P^s M_\Sigma p'$. Let $\Sigma_0$ and $\Delta_0$ be the maximal clique factors of $\Sigma$ and $\Delta$, respectively.
Assume that $\Sigma_0\not=\emptyset$.

Then $\Delta_0\not=\emptyset$ and we have $M_{\Sigma_0}p'\prec_P^s N_{\Delta_0}q'$ and $N_{\Delta_0}q'\prec_P^s M_{\Sigma_0}p'$.

\end{proposition}

\begin{proof} 
If $\Sigma$ is a clique, then $\Sigma_0=\Sigma$ and $M_\Sigma$ is amenable. Since $N_\Delta q'\prec_P^s M_\Sigma p'$, we deduce that $N_\Delta q'$ is amenable. Hence, $\Delta$ is a clique by Lemma~\ref{basic}(3). Thus, $\Delta_0=\Delta\not=\emptyset$ and the conclusion holds.

We may therefore assume that $\Sigma$ is not a clique.
Let $\Sigma_1=\Sigma\setminus \Sigma_0\not=\emptyset$. 
Let $p_0\in\mathcal Z(M_\Sigma'\cap pPp)p'$ be a nonzero projection.
Corollary \ref{center_corner} and Lemma~\ref{basic}(2) give that $M_{\Sigma}z\nprec_PM_{\Sigma}(p-z)$, for every nonzero projection $z\in\mathcal Z(M_{\Sigma}'\cap pPp)$. Using that $M_\Sigma p'\prec_P^s N_\Delta q'$ and $N_\Delta q'\prec_P^s M_\Sigma p'$,  by applying Corollary \ref{tensor_dec} to the decomposition $M_\Sigma p'=M_{\Sigma_0}p'\vee M_{\Sigma_1}p'$ we find a nonzero projection $p''\in\mathcal Z(M_\Sigma'\cap pPp)p_0$ and a join decomposition $\Delta\setminus\Delta_0=\Delta_1\circ\Delta_2$ such that
\begin{equation}\label{D_12}\text{$M_{\Sigma_0}p''\prec^s_P N_{\Delta_0\cup\Delta_1}q'$\;\; and\;\; $M_{\Sigma_1}p''\prec^s_P N_{\Delta_0\cup\Delta_2}q'$.}\end{equation}

We claim that $\Delta_1=\emptyset$. Since the maximal clique factor of $\Sigma_1$ is empty, Lemma~\ref{basic}(1) gives $M_{\Sigma_1}'\cap pPp=M_{\Sigma_1^\perp}$.
Since $M_{\Sigma_1}p''\prec^s_P N_{\Delta_0\cup\Delta_2}q'$, we have
$M_{\Sigma_1}\prec_P N_{\Delta_0\cup\Delta_2}q'$. Since we also have $N_{\Delta_1}q'\subset (N_{\Delta_0\cup\Delta_2}q')'\cap q'Pq'$,   these facts and Proposition \ref{results}(1) imply that $N_{\Delta_1}q'\prec_P M_{\Sigma_1^\perp}$. By Proposition \ref{results}(3) we find a nonzero projection $q''\in\mathcal Z(N_{\Delta_1}'\cap qPq)q'$ such that $N_{\Delta_1}q''\prec_P^s M_{\Sigma_1^\perp}$. Since $\Sigma\subset\Sigma_1\cup\Sigma_1^{\perp}$, by maximality we have $\Sigma=\Sigma_1\cup\Sigma_1^{\perp}$, and therefore $\Sigma_1^\perp=\Sigma_0$. So we have $N_{\Delta_1}q''\prec_P^s M_{\Sigma_0}$. Since $\Sigma_0$ is a clique, $M_{\Sigma_0}$ is amenable and so $N_{\Delta_1}q''$ is amenable.
By Lemma~\ref{basic}(3), this forces $\Delta_1$ to be a clique. Since the maximal clique factor of $\Delta_1$ is trivial, we conclude that indeed $\Delta_1=\emptyset$.

Further, by \eqref{D_12}, we derive that $M_{\Sigma_0}p''\prec_P^s N_{\Delta_0}q'$. In particular, since $\Sigma_0\not=\emptyset$, we get that $\Delta_0\not=\emptyset$. Moreover, we showed that for every nonzero projection $p_0\in\mathcal Z(M_\Sigma'\cap pPp)p'$  there is a nonzero projection $p''\in\mathcal Z(M_\Sigma'\cap pPp)p_0$  such that $M_{\Sigma_0}p''\prec_P^s N_{\Delta_0}q'$. 

Finally, since $\Sigma\in\mathscr J(\Gamma)$ and $\Sigma_0\subset\Sigma$ is its maximal clique factor, we have $\Sigma^\perp=\emptyset$ and $\Sigma_0^\perp=\Sigma\setminus\Sigma_0$. Lemma~\ref{basic}(1) implies that $\mathcal Z(M_\Sigma'\cap pMp)=\overline{\otimes}_{v\in\Sigma_0}\mathcal Z(M_v)=\mathcal Z(M_{\Sigma_0}'\cap pPp)$.
The previous paragraph (combined with Remark~\ref{elementary_facts}(5)) thus implies that $M_{\Sigma_0}p'\prec_P^sN_{\Delta_0}q'$. By symmetry, we also get $N_{\Delta_0}q'\prec_P^s M_{\Sigma_0}p'$, which finishes the proof.
\end{proof}

We say that a complete subgraph $C\subset\Gamma$ is \emph{untransvectable} if there is no vertex $w\notin C$ such that $C^\perp\subset\st(w)$. We recall from Notation~\ref{nota:j} that  $\mathscr{C}(\Gamma)$ denotes the set of all complete subgraphs of $\Gamma$. We let $\mathscr{C}^u(\Gamma)$ be the set of all untransvectable complete subgraphs of $\Gamma$.

\begin{lemma}\label{lemma:untransvectable-clique}
Let $\Lambda$ be a finite simple graph, and let $\Delta\in\mathscr{J}(\Lambda)$ be a maximal join subgraph, and write $\Delta=C\circ\Delta_1\circ\dots\circ\Delta_n$, where $C\subset\Delta$ is its maximal clique factor, and the subgraphs $\Delta_i$ are irreducible with at least two vertices.    
\begin{enumerate}
    \item If  $C\neq\emptyset$ and $\Delta\not= C$ (i.e., $\Delta$ is not a clique), 
    then $C\in\mathscr{C}^u(\Lambda)$.
    \item If $C=\emptyset$, and if $D\subset\Delta_1$ is a complete subgraph such that $D\in\mathscr{C}^u(\Delta_1)$, then $D\in\mathscr{C}^u(\Lambda)$.
\end{enumerate}
\end{lemma}

\begin{proof}
    For (1), assume towards a contradiction that $C^\perp\subset\st(w)$ for some vertex $w\notin C$. Then $w$ cannot belong to one of the subgraphs $\Delta_i$ because $\Delta_i$ is not contained in $\st(w)$. So $w\notin \Delta$. Since $\Delta_1\circ\dots\circ\Delta_n=\Delta\setminus C\not=\emptyset$, it  follows that $(C\cup\{w\})\circ\Delta_1\circ\dots\circ\Delta_n$ is a join subgraph, contradicting the maximality of $\Delta$.

    For (2), assume towards a contradiction that $D^\perp\subset\st(w)$ for some $w\notin D$. Since $D\in\mathscr{C}^u(\Delta_1)$, we have $w\notin \Delta_1$. We also have $\Delta_2\cup\dots\cup\Delta_m\subset\st(w)$. As the subgraphs $\Delta_i$ are irreducible, we cannot have $w\in\Delta_i$, for every $2\leq i\leq n$. It follows that $(\Delta_1\cup\{w\})\circ\Delta_2\circ\dots\circ\Delta_m$ is a join subgraph, contradicting the maximality of $\Delta$. 
\end{proof}

\begin{proposition}\label{vertex_to_clique}
In the setting from Notation \ref{notation}, assume additionally that $(M_v,\tau_v)$ and $(N_w,\tau_w)$ are amenable, for every $v\in \Gamma$ and $w\in\Lambda$. 
Let  $v\in\Gamma^u$.  

Then we can find projections $r_{C}\in\mathcal Z(M_v'\cap pPp), s_C\in\mathcal Z(N_C'\cap qPq)$, for every $C\in\mathscr C^u(\Lambda)$, such that $M_vr_C\prec_P^s N_Cs_C$ and $N_Cs_C\prec_P^s M_vr_C$, for every $C\in\mathscr C^u(\Lambda)$, and $\sum_{C\in\mathscr C^u(\Lambda)}r_C=p$.
\end{proposition}

\begin{proof} We will prove the  conclusion by induction on $|\Gamma|$. 
If $|\Gamma|=1$, then $\Gamma=\{v\}$. Since $pPp=M_v$ is amenable we deduce that $P$ and thus $N_\Lambda$ has an amenable direct summand. Lemma~\ref{basic}(3) gives that $\Lambda$ is a clique. Since by Notation \ref{notation}, $M_\Gamma=M_v$ and $N_\Lambda$ are II$_1$ factors,  the statement holds trivially by letting $r_\Lambda=p$ and $r_{\Lambda'}=0$, for every full subgraph $\Lambda'\subsetneq\Lambda$. 
To prove the inductive step, suppose that the statement holds whenever $|\Gamma|\leq d$, for a fixed $d\geq 1$.
Assume that $|\Gamma|=d+1\geq 2$.

Since $v$ is untransvectable, $\text{lk}(v)\not=\emptyset$, for otherwise $\text{lk}(v)\subset\text{st}(v')$, for every $v'\in\Gamma\setminus\{v\}$. Thus, $\text{st}(v)\subset\Gamma$ is a join subgraph.
Let $\Sigma\subset\Gamma$ be a maximal join full subgraph which contains $\text{st}(v)$.  
If $v'\in\text{lk}(v)$, then $\text{lk}(v)\not\subset\text{st}(v')$,  and thus $\text{st}(v)$ is not a clique. 
Hence,  $\Sigma$ is not a clique and so $M_\Sigma$ has no amenable direct summand by Lemma~\ref{basic}(3). 
In particular,  $M_\Sigma\nprec_P N_w$, for every $w\in\Lambda$.

Moreover, if $\Sigma'\subset\Gamma$  is a full subgraph such that $M_{\Sigma'}$  is not amenable, then $\Sigma'$ is not a clique, hence $M_{\Sigma'}$ has no amenable direct summand and thus $M_{\Sigma'}\nprec_PN_w$ for every $w\in\Lambda$. Similarly, if $\Delta'\subset\Lambda$ is a full subgraph such that $N_{\Delta'}$ is not amenable, then $\Delta'$ is not a clique, hence $N_{\Delta'}$ has no amenable direct summand and thus $N_{\Delta'}\nprec_PM_{v'}$ for every $v'\in\Gamma$.

By applying Proposition \ref{structure_of_aut} we find projections $r_\Delta\in\mathcal Z(M_\Sigma'\cap pPp)$, $s_\Delta\in\mathcal Z(N_\Delta'\cap qPq)$ such that $M_\Sigma r_\Delta\prec_P^s N_\Delta s_\Delta$ and $N_\Delta s_\Delta\prec_P^s M_\Sigma r_\Delta$, for every $\Delta\in\mathscr J(\Lambda)$, and $\sum_{\Delta\in\mathscr J(\Lambda)}r_\Delta=p$.
Specifically, in the context of Proposition \ref{structure_of_aut}, we note that $p_\Sigma=p$ and denote $r_\Delta:=p_{\Sigma,\Delta}$ and $s_\Delta:=q_{\Delta,\Sigma}$. Note that $r_\Delta=0$ if and only if $s_\Delta=0$.

Let $\Sigma_0$ be the maximal clique factor of $\Sigma$. We continue by treating separately two cases.

{\bf Case 1.} $\Sigma_0\not=\emptyset$.

Let $\Delta\in\mathscr J(\Lambda)$ such that $r_\Delta\not=0$. 
Since $M_\Sigma$ has no amenable direct summand and $M_\Sigma r_\Delta\prec_P^s N_\Delta s_\Delta$, we get that $N_\Delta$ is not amenable, thus $\Delta$ is not a clique.
Let $C_\Delta\in\mathscr C(\Lambda)$ be the maximal clique factor of $\Delta$. Then $C_\Delta\in\mathscr{C}^u(\Lambda)$ by Lemma~\ref{lemma:untransvectable-clique}(1).
Since $M_\Sigma r_\Delta\prec_P^s N_\Delta s_\Delta$ and $N_\Delta s_\Delta\prec_P^s M_\Sigma r_\Delta$, applying  Proposition~\ref{clique_factor} gives that $C_\Delta\not=\emptyset$, 
$M_{\Sigma_0} r_\Delta\prec_P^s N_{C_\Delta} s_\Delta$ and $N_{C_\Delta} s_\Delta\prec_P^s M_{\Sigma_0} r_\Delta$. 
As $\Sigma\subset\Gamma$ is a maximal join full subgraph, we have $\Sigma^\perp=\emptyset$ and $\Sigma_0^\perp=\Sigma\setminus\Sigma_0$. By Lemma~\ref{basic}(1) we get \[\mathcal Z(M_\Sigma'\cap pPp)=\overline{\otimes}_{v'\in\Sigma_0}\mathcal Z(M_{v'})=\mathcal Z(M_{\Sigma_0}'\cap pPp).\] Similarly, $\mathcal Z(N_\Delta'\cap qPq)=\mathcal Z(N_{C_\Delta}'\cap qPq)$. 
Hence, 
$r_\Delta\in\mathcal Z(M_{\Sigma_0}'\cap pPp)$ and $s_\Delta\in\mathcal Z(N_{C_\Delta}'\cap qPq)$.

Finally, for every $v'\in \Sigma_0$, we have $\text{st}(v)\subset\Sigma\subset\text{st}(v')$. Since $v$ is untransvectable, we get that $v'=v$, and thus $\Sigma_0=\{v\}$.
Note that if $\Delta,\Delta'\in\mathscr{J}(\Lambda)$ satisfy $\Delta\not=\Delta'$, then $C_\Delta\not=C_{\Delta'}$. Indeed, if $C_\Delta=C_{\Delta'}$, then  $\Delta\cup\Delta'=C_\Delta\circ ((\Delta\cup\Delta')\setminus C_\Delta)$ would be a join subgraph of $\Lambda$ which strictly contains $\Delta$, contradicting that $\Delta\in\mathscr J(\Lambda)$.

If $C\in\mathscr{C}^u(\Lambda)$, let  $r_C=r_\Delta$, if $C=C_{\Delta}$, for some $\Delta\in\mathscr J(\Lambda)$ (such a $\Delta$ is unique by the previous paragraph), and let $r_C=0$, if $C$ is not equal to any $C_{\Delta}$.
Since $\sum_{\Delta\in\mathscr J(\Lambda)}r_\Delta=p$, we get that $\sum_{C\in\mathscr C^u(\Lambda)}r_C=p$, 
and the above analysis ensures that the conclusion of the proposition holds.

{\bf Case 2.} $\Sigma_0=\emptyset$.

Let $\Delta\in\mathscr J(\Lambda)$ with $r_\Delta\not=0$. Since $\Sigma_0=\emptyset$, $M_\Sigma r_\Delta\prec_P^s N_\Delta s_\Delta$ and $N_\Delta s_\Delta\prec_P^s M_\Sigma r_\Delta$, Proposition~\ref{clique_factor} implies that $\Delta$ has a trivial maximal clique factor. Since $\Sigma\subset\Gamma$ and $\Delta\subset\Lambda$ are maximal join full subgraphs, $\Sigma^\perp=\emptyset$ and $\Delta^\perp=\emptyset$. Lemma~\ref{basic}(1) implies that $M_\Sigma'\cap pPp=\mathbb Cp$ and $N_\Delta'\cap qPq=\mathbb Cq$. 
In particular, $M_\Sigma$ and $N_\Delta$  are II$_1$ factors.
Moreover, we derive that $r_\Delta=p$ and $s_\Delta=q$. Hence, $M_\Sigma\prec_P^s N_\Delta$ and $N_\Delta\prec_P^s M_\Sigma$. By combining Proposition \ref{results}(2)  
and Lemma~\ref{basic}(2), this gives that $M_\Sigma\nprec_PN_{\Delta'}$ and $N_{\Delta}\nprec_PM_{\Sigma'}$, for any full subgraphs $\Sigma'\subsetneq\Sigma$ and $\Delta'\subsetneq\Delta$. 

Let $P_0\subset P$ be a subfactor such that $p\in P_0$ and $pP_0p=M_\Sigma$. Then we can identify $P_0=M_\Sigma^t$, where $t=(\text{Tr}\otimes\tau)(p)^{-1}$ and $\text{Tr}:\mathbb M_n(\mathbb C)\rightarrow\mathbb C$ is the non-normalized trace.
By Theorem \ref{uni_conj} we get unitaries $a,b\in P$ such that $aP_0a^*\subset \mathbb M_n(\mathbb C)\otimes N_\Delta$ and $b(\mathbb M_n(\mathbb C)\otimes N_\Delta)b^* \subset P_0$. Thus,  $(ab)(\mathbb M_n(\mathbb C)\otimes N_\Delta)(ab)^*\subset \mathbb M_n(\mathbb C)\otimes N_\Delta$, which by Proposition \ref{normalizer} implies that $ab\in\mathbb M_n(\mathbb C)\otimes N_\Delta$. This gives that $(ab)(\mathbb M_n(\mathbb C)\otimes N_\Delta)(ab)^*=\mathbb M_n(\mathbb C)\otimes N_\Delta$, which implies that $aP_0a^*=\mathbb M_n(\mathbb C)\otimes N_\Delta$. 

Since $\Sigma$ and $\Delta$ are join graphs with trivial maximal clique factors, we can decompose them as $\Sigma=\Sigma_1\circ\cdots\circ \Sigma_k$ and $\Delta=\Delta_1\circ\cdots\circ\Delta_m$, where $\Sigma_1,\dots,\Sigma_k,\Delta_1,\dots,\Delta_m$ are irreducible graphs, not reduced to a vertex, for some $k,m\geq 2$. 
Thus, $M_\Sigma=\overline{\otimes}_{i=1}^kM_{\Sigma_i}$ and $N_\Delta=\overline{\otimes}_{j=1}^m N_{\Delta_j}$. 

Since $P_0=M_\Sigma^t$ and $aP_0a^*=\mathbb M_n(\mathbb C)\otimes N_\Delta$,  applying Theorem \ref{prime_factorization} gives that $k=m$ and there exist a unitary $\xi\in\mathbb M_n(\mathbb C)\otimes N_\Delta$ and decompositions $P_0=\overline{\otimes}_{i=1}^k M_{\Sigma_i}^{s_i}$ and $\mathbb M_n(\mathbb C)\otimes N_\Delta=\overline{\otimes}_{j=1}^kN_{\Delta_j}^{t_j}$, for some $s_1,\dots,s_k,t_1,\dots,t_k>0$ with $s_1\cdots s_k=t$ and $t_1\cdots t_k=n$, such that modulo a permutation of $\{1,\dots,k\}$ we have  
\begin{equation}\label{conjugacy}\text{$(\xi a)M_{\Sigma_i}^{s_i}(\xi a)^*=N_{\Delta_i}^{t_i}$, for every $1\leq i\leq k$.}\end{equation}

Next,  assume without loss of generality that $v\in \Sigma_1$. 
We claim that $v$ is untransvectable in $\Sigma_1$. Indeed, assume that $\text{lk}(v)\cap\Sigma_1\subset\text{st}(v')\cap\Sigma_1$, for some $v'\in\Sigma_1$. 
Since $\text{st}(v)\subset\Sigma$, we have $\text{lk}(v)=(\text{lk}(v)\cap\Sigma_1)\cup(\Sigma\setminus\Sigma_1)$. Since we also  have $(\text{st}(v')\cap\Sigma_1)\cup(\Sigma\setminus\Sigma_1)\subset\st(v')$, it follows that $\text{lk}(v)\subset\text{st}(v')$. Since $v$ is untransvectable in $\Gamma$, we get $v'=v$, which proves our claim.

Since $\Sigma_1\subsetneq\Sigma$, we have $\Sigma_1\subsetneq\Gamma$ and thus $|\Sigma_1|\leq d$. Let $\kappa=t_1/s_1>0$. By the induction assumption, the conclusion of the proposition holds for the $*$-isomorphism $M_{\Sigma_1}\cong N_{\Delta_1}^\kappa$ induced by the equality $(\xi a)M_{\Sigma_1}^{s_1}(\xi a)^*=N_{\Delta_1}^{t_1}$ given by \eqref{conjugacy}.
This implies that we can find projections $r_C\in\mathcal Z(M_v'\cap M_{\Sigma_1}), s_C\in\mathcal Z(N_C'\cap N_{\Delta_1})$ such that $M_vr_C\prec^s_P N_Cs_C$ and $N_Cs_C\prec_P^sM_vr_C$, for every $C\in\mathscr C^u(\Delta_1)$, and $\sum_{C\in\mathscr C^u(\Delta_1)}r_C=p$. Notice also that whenever $C\in\mathscr{C}^u(\Delta_1)$, we also have $C\in\mathscr{C}^u(\Lambda)$ by Lemma~\ref{lemma:untransvectable-clique}(2).

Finally, since  $v$ is untransvectable in $\Gamma$, we get that $\text{lk}(v)$ has an empty maximal clique factor. Lemma~\ref{basic}(1)  implies that $\mathcal Z(M_v'\cap M_\Gamma)=\mathcal Z(M_v)$. Since $v$ is untransvectable in $\Sigma_1$, we also have that  $\mathcal Z(M_v'\cap M_{\Sigma_1})=\mathcal Z(M_v)$. This allows us to derive that $r_C\in\mathcal Z(M_v'\cap M_\Gamma)$, for every $C\in\mathscr C^u(\Delta_1)$. In the same way, since the clique $C$ is untransvectable in $\Lambda$, the subgraph $C^\perp$ has an empty maximal clique factor, so $\mathcal Z(N_C'\cap N_\Lambda)=\bar\otimes_{w\in C}\mathcal Z(N_w)$. Likewise, since $C$ is untransvectable in $\Delta_1$ we have $\mathcal Z(N_C'\cap N_{\Delta_1})=\overline\otimes_{w\in C}\mathcal Z(N_w)$. So $s_C\in\mathcal{Z}(N_C'\cap N_\Lambda)$, for every $C\in\mathscr{C}^u(\Delta_1)$, which completes our proof.
\end{proof}

\begin{proof}[\bf Proof of Theorem \ref{amenable_gen}]
Let $v\in\Gamma$ be an untransvectable vertex. By applying Proposition \ref{vertex_to_clique}  we find projections  $r_{v,C}\in\mathcal Z(M_v'\cap pPp), s_{C,v}\in\mathcal Z(N_C'\cap qPq)$ such that $M_vr_{v,C}\prec_P^s N_Cs_{C,v}$ and $N_Cs_{C,v}\prec_P^s M_vr_{v,C}$, for every $C\in\mathscr C^u(\Lambda)$, and $\sum_{C\in\mathscr C^u(\Lambda)}r_{v,C}=p$. 

If $C\in\mathscr C^u(\Lambda)$ is such that $r_{v,C}\not=0$, then $C$ is collapsible by Proposition \ref{collapsible}. Since $\Lambda$ is clique-reduced, $C$ must consist of a single vertex. Since $C$ is untransvectable, this vertex is untransvectable. Thus, we have projections $r_{v,w}\in\mathcal Z(M_v'\cap pPp), s_{w,v}\in\mathcal Z(N_w'\cap qPq)$ such that $M_vr_{v,w}\prec_P^s N_ws_{w,v}$ and $N_ws_{w,v}\prec_P^s M_vr_{v,w}$, for every $w\in\Lambda^u$, and $\sum_{w\in\Lambda^u} r_{v,w}=p$.

Assume that we are in case (1), that is, $\Gamma^u$ and $\Lambda^u$ are clique-reduced. Let $v\in\Gamma^u$. In order to apply Proposition \ref{conjugacy_criterion} and derive the conclusion, it will be enough to argue that $\sum_{v\in\Gamma^u}s_{w,v}=q$, for every $w\in\Lambda^u$.
Towards this goal, we first note that by symmetry, for every $w\in\Lambda^u$, we find projections $r_{v,w}'\in\mathcal Z(M_v'\cap pPp), s_{w,v}'\in\mathcal Z(N_w'\cap qPq)$ such that $M_vr_{v,w}'\prec_P^s N_ws_{w,v}'$ and $N_ws_{w,v}'\prec_P^s M_vr_{v,w}'$, for every $v\in\Gamma^u$, and $\sum_{v\in\Gamma^u} s_{w,v}'=q$.

We claim that $r_{v,w}'\leq r_{v,w}$, for every $v\in\Gamma^u$ and $w\in\Lambda^u$. Let $v\in\Gamma^u$ and $w\in\Lambda^u$. Since $r_{v,w}'\leq p$ and $p=\sum_{w\in\Lambda^u}r_{v,w}$, in order to justify the claim it suffices to show that $r_{v,w}'r_{v,w'}=0$, for every $w'\in\Lambda^u\setminus\{w\}$. Assume by contradiction that $r:=r_{v,w}'r_{v,w'}\not=0$, for some $w'\in\Lambda^u\setminus\{w\}$. Since $r\leq r_{v,w}'$, Lemma \ref{proof_of_3} implies that $N_ws_{w,v}'\prec_P M_vr$. Since $r\leq r_{v,w'}$, we also have  $M_vr\prec_P^s N_{w'}s_{w',v}$, 
and $s_{w',v}\neq 0$ because $r_{v,w'}\neq 0$. Combining these facts and Proposition \ref{results}(2) gives $N_ws_{w,v}'\prec_P N_{w'}s_{w',v}$. By Lemma~\ref{basic}(2), this implies that $\{w\}\subset\{w'\}$, which contradicts the fact that $w'\not=w$, and proves the claim.

Similarly, we deduce that $s_{w,v}\leq s_{w,v}'$, for every $v\in\Gamma^u$ and $w\in\Lambda^u$.

Let $v\in\Gamma^u$ and $w\in\Lambda^u$.
Since $N_ws_{w,v}'\prec_P^s M_vr_{v,w}'$, $r_{v,w}'\leq r_{v,w}$ and thus  $M_vr_{v,w}'\prec_P^s M_vr_{v,w}$, and $M_vr_{v,w}\prec_P^sN_ws_{w,v}$, by applying Proposition \ref{results}(2) we deduce that $N_ws_{w,v}'\prec_P^sN_ws_{w,v}$. By Corollary \ref{center_corner} we derive that $s_{w,v}'\leq s_{w,v}$. Since $s_{w,v}\leq s_{w,v}'$, we conclude that $s_{w,v}=s_{w,v}'$. 
Therefore, $\sum_{v\in\Gamma^u}s_{w,v}=\sum_{v\in\Gamma^u}s_{w,v}'=q$, for every $w\in\Lambda^u$,
and
 Proposition \ref{conjugacy_criterion} gives the conclusion.

Now, assume that we are in case (2), that is, $M_v$ and $N_w$ are II$_1$ factors, for every $v\in\Gamma$ and $w\in\Lambda$. Then Lemma~\ref{basic}(1) implies that $\mathcal Z(M_v'\cap pPp)=\mathbb Cp$ and $\mathcal Z(N_w'\cap qPq)=\mathbb Cq$, and therefore $r_{v,w}\in\{0,p\}$ and $s_{w,v}\in\{0,q\}$, for every $v\in\Gamma^u$ and $w\in\Lambda^u$. So for every $v\in \Gamma^u$ there exists a (unique) vertex $\alpha(v)\in\Lambda^u$ such that $M_v\prec_P^sN_{\alpha(v)}$ and $N_{\alpha(v)}\prec_P^s M_v$. By symmetry, for every $w\in\Lambda^u$ there exists a unique vertex $\beta(w)\in\Gamma^u$ such that $N_w\prec_P^s M_{\beta(w)}$ and $M_{\beta(w)}\prec_P^s N_w$. It then follows from Lemma~\ref{basic}(2) that for every $v\in\Gamma^u$, we have $\beta\circ\alpha(v)=v$, and that for every $w\in\Lambda^u$, we have $\beta\circ\alpha(w)=w$. By Proposition~\ref{conjugacy_criterion}(2), the conclusion now follows in case (2).
\end{proof}

\section{II$_1$ factor vertex algebras: Proof of Theorem \ref{factors}}\label{sec:factors}

In this section, we prove the following theorem, which is the main part of Theorem~\ref{factors} (the moreover part of Theorem~\ref{factors} will be a consequence of Lemma \ref{interwining_to_isomorphism}(1) below).

\begin{theorem}\label{factors1}
    Let $\Gamma,\Lambda$ be finite simple graphs, and assume that every connected component of $\Gamma$ and $\Lambda$ is strongly reduced, transvection-free,
    and not reduced to one vertex. Let $(M_v)_{v\in\Gamma}$ and $(N_w)_{w\in\Lambda}$ be families of II$_1$ factors, and let $M_\Gamma=\ast_{v,\Gamma}M_v$ and $N_\Lambda=\ast_{w,\Lambda}N_w$. Let $t>0$.

    If $\theta:M_\Gamma\to N_\Lambda^t$ is any $*$-isomorphism,
    then  $\Gamma$ and $\Lambda$ are isomorphic and there exists a graph isomorphism $\alpha:\Gamma\to\Lambda$ such that $\theta(M_v)\prec_{N_\Lambda^t}N_{\alpha(v)}^t$ and $N_{\alpha(v)}^t\prec_{N_\Lambda ^t}^s\theta(M_v)$, for every $v\in\Gamma$. 
\end{theorem}

Here and after, given II$_1$ factors $R\subset Q$ and $t>0$, we consider the usual inclusion $R^t\subset Q^t$.

The proof of Theorem \ref{factors1} relies on the following technical result.

\begin{proposition}\label{induction_on_graphs}
Let $\Gamma,\Lambda$  be two finite simple graphs, and $(M_v)_{v\in\Gamma}$, $(N_w)_{w\in\Lambda}$ be families of II$_1$ factors. Let $M_\Gamma=*_{v,\Gamma}M_v$ and  $N_\Lambda=*_{w,\Lambda}N_w$ be the associated graph product II$_1$ factors, and $\theta:M_\Gamma\rightarrow N_\Lambda^t$ a $*$-isomorphism, for some $t>0$. Let $v\in\Gamma$ be a vertex.
Assume that the following three conditions hold:
\begin{enumerate}
\item $v$ is untransvectable.
\item There is no full subgraph $\Delta\subset\Lambda$ such that $|\Delta|\geq 2$, $N_\Delta^t\prec_{N_\Lambda^t} \theta(M_v)$ and $\theta(M_v)\prec_{N_\Lambda^t} N_{\Delta\cup\Delta^\perp}^t$. 
\item There are no vertex $w\in\Lambda$ and full subgraph $\Sigma\subset\Gamma$ such that $|\Sigma|\geq 2$, $\theta(M_{\Sigma})\prec_{N_\Lambda^t}N_w^t$ and $N_w^t\prec_{N_\Lambda^t}\theta(M_{\Sigma\cup\Sigma^\perp})$.
\end{enumerate}
Then there exists a vertex $w\in\Lambda$ such that $\theta(M_v)\prec_{N_\Lambda^t} N_w^t$.
\end{proposition}

Before proving Proposition \ref{induction_on_graphs}, we first show that it implies Theorem \ref{factors1}.

\begin{proof}[\bf Proof of Theorem \ref{factors1}] We note first that by Lemma~\ref{basic}(1) we have $\mathcal Z(M_\Sigma'\cap M_\Gamma)=\mathbb C1$ and $\mathcal Z(N_\Delta'\cap N_\Lambda)=\mathbb C1$, for any nonempty full subgraphs $\Sigma\subset\Gamma$ and $\Delta\subset\Lambda$. 

\begin{claim}\label{no_collapsible_subgraph}
There are no vertex $v\in\Gamma$ and full subgraph $\Delta\subset\Lambda$ such that $|\Delta|\geq 2$, $N_\Delta^t\prec_{N_\Lambda^t}\theta(M_v)$ and $\theta(M_v)\prec_{N_{\Lambda}^t}N_{\Delta\cup\Delta^\perp}^t$.
\end{claim}

\begin{proof}
Assume by contradiction that there exist $v\in\Gamma$ and a full subgraph $\Delta\subset\Lambda$ such that $|\Delta|\geq 2$, $N_\Delta^t\prec_{N_\Lambda^t}\theta(M_v)$ and $\theta(M_v)\prec_{N_{\Lambda}^t}N_{\Delta\cup\Delta^\perp}^t$.
Since $\mathcal Z(\theta(M_v)'\cap N_\Lambda^t)=\mathbb C1$ and $\mathcal Z((N_\Delta^t)'\cap N_\Lambda^t)=\mathbb C1$, Proposition \ref{results}(3) gives $N_\Delta^t\prec_{N_\Lambda^t}^s\theta(M_v)$ and $\theta(M_v)\prec^s_{N_\Lambda^t}N_{\Delta\cup\Delta^\perp}^t$. By  Proposition \ref{collapsible} we deduce that $\Delta\subset\Lambda$ is a collapsible subgraph and $\theta(M_{\text{st}(v)})\prec^s_{N_\Lambda^t}N_{\Delta\cup\Delta^\perp}^t$.

Since $\Delta$ is collapsible and the connected components of $\Lambda$ are strongly reduced, Lemma~\ref{lemma:collapsible-combinatorial} ensures that $\Delta$ is a union of connected components of $\Lambda$. 
In particular, $\Delta^\perp=\emptyset$. Then $N_\Delta^t\prec_{N_\Lambda^t}^s\theta(M_v)$ and $\theta(M_{\text{st}(v)})\prec^s_{N_\Lambda^t}N_{\Delta}^t$. By Proposition \ref{results}(2) we get that $\theta(M_{\text{st}(v)})\prec_{N_\Lambda^t}\theta(M_v)$, thus $M_{\text{st}(v)}\prec_{M_\Gamma}M_v$. However, by Lemma~\ref{basic}(2) this implies that $\text{st}(v)\subset \{v\}$. Hence, $\text{lk}(v)=\emptyset$, so $v$ is isolated, contradicting our assumption that no connected component of $\Gamma$ is reduced to one vertex.
\end{proof}

By symmetry, the analogue of Claim \ref{no_collapsible_subgraph} holds if we swap the roles of $\Gamma$ and $\Lambda$. Thus,  there are no vertex $w\in\Lambda$ and full subgraph $\Sigma\subset\Gamma$ such that $|\Sigma|\geq 2$, $\theta(M_\Sigma)\prec_{N_\Lambda^t}N_w^t$ and $N_w^t\prec_{N_\Lambda^t}\theta(M_{\Sigma\cup\Sigma^\perp})$.

Since by assumption all vertices of $\Gamma$ and $\Lambda$ are untransvectable, we can apply Proposition \ref{induction_on_graphs} and get maps $\alpha:\Gamma\rightarrow\Lambda$ and $\beta:\Lambda\rightarrow\Gamma$ such that $\theta(M_v)\prec_{N_\Lambda^t}N_{\alpha(v)}^t$ and $N_w^t\prec_{N_\Lambda^t}\theta(M_{\beta(w)})$, for every $v\in\Gamma$ and $w\in\Lambda$. Since $\mathcal Z(\theta(M_v)'\cap N_\Delta^t)=\mathcal Z((N_w^t)'\cap N_\Lambda^t)=\mathbb C1$, Proposition \ref{results}(3) further gives that
\begin{equation}\label{bijections}
\text{$\theta(M_v)\prec_{N_\Lambda^t}^sN_{\alpha(v)}^t$ and $N_w^t\prec_{N_\Lambda^t}^s\theta(M_{\beta(w)})$, for every $v\in\Gamma$ and $w\in\Lambda$.}
\end{equation}
Combining \eqref{bijections} with Proposition \ref{results}(2) gives that $\theta(M_v)\prec_{N_\Lambda^t}\theta(M_{\beta(\alpha(v))})$, hence $M_v\prec_{M_\Gamma}M_{\beta(\alpha(v))}$, for every $v\in\Gamma$. By Lemma~\ref{basic}(2) we deduce that $\beta(\alpha(v))=v$, for every $v\in\Gamma$. Similarly, we deduce that $\alpha(\beta(w))=w$, for every $w\in\Lambda$. Therefore, $\alpha:\Gamma\rightarrow\Lambda$ is a bijection and $\beta=\alpha^{-1}$.

The fact that $\alpha$ is a graph isomorphism is now a direct consequence of Proposition~\ref{conjugacy_criterion}(2), but let us also provide a short direct argument. Let $v\in\Gamma$ and $v'\in\text{lk}(v)$. Since $\theta(M_v)\prec_{N_\Lambda^t}^sN_{\alpha(v)}^t$, Proposition \ref{normalizer}(2) implies that $\theta(M_{\text{st}(v)})\prec_{N_\Lambda^t}^sN_{\text{st}(\alpha(v))}^t$. In particular, $\theta(M_{v'})\prec_{N_\Lambda^t}^sN_{\text{st}(\alpha(v))}^t$. Since $N_{\alpha(v')}^t\prec_{N_\Lambda^t}^s\theta(M_{v'})$, Proposition \ref{results}(2) gives that $N_{\alpha(v')}^t\prec_{N_\Lambda^t}^sN_{\text{st}(\alpha(v))}^t$. By Lemma~\ref{basic}(2) we deduce that $\alpha(v')\in\text{st}(\alpha(v))$. Since $\alpha(v')\not=\alpha(v)$, we get that $\alpha(v')\in\text{lk}(\alpha(v))$. Similarly, if $w\in\Lambda$ and $w'\in\text{lk}(w)$, then $\alpha^{-1}(w')\in\text{lk}(\alpha^{-1}(w)))$. This shows that $\alpha$ is a graph isomorphism, which finishes the proof.
\end{proof}

\begin{proof}[\bf Proof of Proposition \ref{induction_on_graphs}]
Denote $L=N_\Lambda^t$ and identify $M_\Gamma=L$, via $\theta$.  By Lemma~\ref{basic}(1) we have $\mathcal Z(M_\Sigma'\cap M_\Gamma)=\mathbb C1$ and $\mathcal Z(N_\Delta'\cap N_\Lambda)=\mathbb C1$, for any nonempty full subgraphs $\Sigma\subset\Gamma$ and $\Delta\subset\Lambda$. This implies that if $M_\Sigma\prec_LN_\Delta^t$, then $M_\Sigma\prec_L^sN_\Delta^t$, and if $N_\Delta^t\prec_LM_\Sigma$, then $N_\Delta^t\prec_L^sM_\Sigma$.

We will prove the conclusion by induction on $|\Gamma|$. If $|\Gamma|=1$, then  $\Gamma=\{v\}$ and $L=M_v$. Condition (2) of the hypothesis implies that $|\Lambda|=1$, and the conclusion follows trivially.
To prove the inductive step, suppose that the conclusion holds if $|\Gamma|\leq d$, for a fixed integer $d\geq 1$. 
We will prove that the conclusion holds if $|\Gamma|=d+1\geq 2$.

We claim that since $v$ is untransvectable, $\text{st}(v)\subset\Gamma$ is a join subgraph that is not a clique. 
Indeed,  $\text{lk}(v)\not=\emptyset$, since otherwise $\text{lk}(v)\subset\text{st}(v')$, for every $v'\in\Gamma\setminus\{v\}$.  Thus, $\text{st}(v)=\{v\}\circ\text{lk}(v)$ is a join graph.
Moreover, $\text{st}(v)$ is not a clique because otherwise $\text{lk}(v)\subset\text{st}(v)\subset\text{st}(v')$, for every $v'\in\text{lk}(v)$.

Let $\Sigma\subset\Gamma$ be a maximal join full subgraph containing $\text{st}(v)$, i.e., a join full subgraph which contains $\text{st}(v)$ and is not a proper subgraph of any join full subgraph of $\Gamma$. Then $\Sigma$ is not a clique, and thus $M_\Sigma$ is nonamenable by Lemma~\ref{basic}(3). 
If $M_\Sigma\prec_P N_w^t$, for some vertex $w\in\Lambda$, then $M_v\prec_P N_w^t$ and the conclusion holds.
We may therefore assume that $M_\Sigma\nprec_P N_w^t$, for every vertex $w\in\Lambda$. 

By Corollary \ref{join_embed}, we can find a join full subgraph $\Delta\subset\Lambda$ such that  $M_\Sigma\prec_L N_\Delta^t$ and thus $M_\Sigma\prec_L^s N_\Delta^t$. Moreover, we may assume that $\Delta\subset\Lambda$ is a maximal join full subgraph. 
 We next claim that $N_\Delta^t\nprec_L M_{v'}$, for every $v'\in\Gamma$. If $N_\Delta^t\prec_L M_{v'}$, for some $v'\in\Gamma$, then $N_\Delta^t\prec_L^sM_{v'}$. Since $M_\Sigma\prec_LN_\Delta^t$, Proposition \ref{results}(2) implies that $M_\Sigma\prec_L M_{v'}$. By Lemma~\ref{basic}(2) this entails that $\Sigma\subset\{v'\}$, which gives a contradiction and proves the claim. 

 Since $M_\Sigma\prec_L^s N_\Delta^t$ and $M_\Sigma$ is nonamenable,  $N_\Delta$ is  nonamenable.
 Since $N_\Delta$ is nonamenable, the claim from the previous paragraph and Corollary \ref{join_embed} together yield a join full subgraph $\Sigma'\subset\Gamma$ such that $N_\Delta^t\prec_L M_{\Sigma'}$. By reasoning as above, we get that $N_\Delta^t\prec_L^s M_{\Sigma'}$ and $\Sigma\subset\Sigma'$. The maximality of $\Sigma$ implies  that $\Sigma=\Sigma'$.
In conclusion, we found a maximal join full subgraph $\Delta\subset\Lambda$ with
\begin{equation}\label{bi_embed}\text{$M_\Sigma\prec_L^s N_\Delta^t$ \;\;\; and \;\;\; $N_\Delta^t\prec_L^s M_\Sigma$.}\end{equation}

We continue by showing that $M_\Sigma\nprec_L N_{\Delta'}^t$, for every full subgraph $\Delta'\subsetneq\Delta$. Indeed, if we have $M_\Sigma\prec_L N_{\Delta'}^t$, 
for a full subgraph $\Delta'\subsetneq\Delta$, then $M_\Sigma\prec_L^s N_{\Delta'}^t$. Since $N_\Delta^t\prec_L M_\Sigma$, Proposition~\ref{results}(2)
implies that $N_\Delta^t\prec_L N_{\Delta'}^t$, which by Lemma~\ref{basic}(2) gives that $\Delta\subset\Delta'$, a contradiction. Similarly, we get $N_\Delta^t\nprec_L M_{\Sigma'}$, for every full subgraph $\Sigma'\subsetneq\Sigma$.

Since $\Sigma^\perp=\Delta^\perp=\emptyset$, by combining \eqref{bi_embed}, the last paragraph and  Theorem \ref{uni_conj} we find unitaries $u,v\in L$ such that $uM_\Sigma u^*\subset N_\Delta^t$ and $vN_\Delta^t v^*\subset M_\Sigma$. Thus, $(vu)M_\Sigma(vu)^*\subset  vN_\Delta^t v^*\subset M_\Sigma$. By Proposition \ref{normalizer}(1) we get $vu\in M_\Sigma$ and hence $vuM_\Sigma (vu)^*=M_\Sigma$. Therefore, $vN_\Delta^t v^*=M_\Sigma$. Hence, if we denote $Q=N_\Delta^t$, then, after unitary conjugacy, we may assume that $M_\Sigma=Q$.

Decompose $\Sigma=C\circ\Sigma_1\circ\cdots\circ\Sigma_m$, where $C$ is the maximal clique factor of $\Sigma$ and $\Sigma_i$ is an irreducible graph with $|\Sigma_i|\geq 2$, for every $1\leq i\leq m$.
Similarly, decompose $\Delta=D\circ\Delta_1\circ\cdots\circ\Delta_n$, where $D$ is the maximal clique factor of $\Delta$ and $\Delta_j$ is an irreducible graph with $|\Delta_j|\geq 2$, for every $1\leq j\leq n$. By Theorem \ref{strongly_prime} we have that $M_{\Sigma_i}$ and $N_{\Delta_j}$ are strongly prime II$_1$ factors, for every $1\leq i\leq m$ and $1\leq j\leq n$.

We claim that   
\begin{equation}\label{SigmaiD}\text{$M_{\Sigma_i}\nprec_QN_D^t$, \;\;\; for every $1\leq i\leq m$.}\end{equation} This is clear if $D=\emptyset$, so we will assume otherwise. Assume by contradiction that $M_{\Sigma_i}\prec_Q N_D^t$, for some $1\leq i\leq m$. Since $Q=M_{\Sigma_i}\overline{\otimes}M_{\Sigma\setminus\Sigma_i}=N_D^t\overline{\otimes} N_{\Delta\setminus D}$, \cite[Proposition 12]{OP04} (see Lemma~\ref{tensor_product}(1)) gives a decomposition $Q=N_D^s\overline{\otimes}N_{\Delta\setminus D}^{t/s}$, for some $s>0$, such that after conjugacy with a unitary from $Q$ we have $M_{\Sigma_i}\subset N_D^s$. Moreover, if we put $R=M_{\Sigma_i}'\cap N_D^s$, then $N_D^s=M_{\Sigma_i}\overline{\otimes}R$ by \cite[Lemma~3.3]{Is20}. Since $N_D=\overline{\otimes}_{w\in D}N_w$, using the fact that $M_{\Sigma_i}$ is strongly prime and \cite[Lemma~4.1]{Is20}, we find $w\in D$ such that $M_{\Sigma_i}\prec_{N_D^s}N_w^s$. Thus, $M_{\Sigma_i}\prec_Q N_w^t$, contradicting condition (3) from the hypothesis since $|\Sigma_i|\geq 2$, and $\Sigma=\Sigma_i\cup\Sigma_i^\perp$ and thus $N_w^t\subset Q=M_{\Sigma_i\cup\Sigma_i^\perp}$. This proves \eqref{SigmaiD}.

We continue by treating separately two cases.

{\bf Case 1.} $C\not=\emptyset$.

If $v'\in C$, then $\text{st}(v)\subset\Sigma\subset\text{st}(v')$. 
As $v$ is untransvectable, we get $v'=v$ and thus $C=\{v\}$. 

Since $M_v\subset Q=N_\Delta^t=(N_{\Delta_j\cup\Delta_j^\perp})^t$ and $|\Delta_j|\geq 2$, condition (2) from the hypothesis implies that $N_{\Delta_j}^t\nprec_Q M_v$, for every $1\leq j\leq n$.
Using this fact, that $Q=M_v\overline{\otimes}(\overline{\otimes}_{i=1}^mM_{\Sigma_i})=N_D^t\overline{\otimes}(\overline{\otimes}_{j=1}^n N_{\Delta_j})$ and  claim \eqref{SigmaiD},   Lemma 
\ref{str_prime}  implies that we can find a decomposition $Q=N_D^s\overline{\otimes}(\overline{\otimes}_{j=1}^n N_{\Delta_j})^{t/s}$, for some $s>0$, and a unitary $u\in Q$ such that $uM_vu^*=N_D^s$. In particular, $N_D^t\prec_Q M_v$. Since $M_v\subset Q=N_{D\cup D^\perp}^t$, condition (2) from the hypothesis implies that $|D|=1$. Thus, $D=\{w\}$, for some $w\in \Lambda$, which implies that $M_v\prec_P N_w^t$. This finishes the proof in Case 1.

{\bf Case 2.} $C=\emptyset$.

Then $\Sigma=\Sigma_1\circ\cdots\circ\Sigma_m$ and
 we may assume that $v\in\Sigma_1$. We claim that $v$ is untransvectable in $\Sigma_1$. Indeed, assume that $\text{lk}(v)\cap\Sigma_1\subset\text{st}(v')\cap\Sigma_1$, for $v'\in\Sigma_1$. 
Since $\text{st}(v)\subset\Sigma$, we have that $\text{lk}(v)=(\text{lk}(v)\cap\Sigma_1)\cup(\Sigma\setminus\Sigma_1)$. Since we also  have $\text{st}(v')\supset (\text{st}(v')\cap\Sigma_1)\cup(\Sigma\setminus\Sigma_1)$, it follows that $\text{lk}(v)\subset\text{st}(v')$. Since $v$ is untransvectable in $\Gamma$, we get $v'=v$, which proves our claim.

We now claim that $D=\emptyset$. Indeed, $Q=N_D^t\overline\otimes (\overline\otimes_{j=1}^n N_{\Delta_j})$, so Theorem~\ref{prime_factorization}(1) ensures that there exist $s>0$ and $I\subset\{1,\dots,m\}$ such that $N_D^s$ is unitarily conjugate to $\overline\otimes_{i\in I}M_{\Sigma_i}$ inside $Q$. In particular, $M_{\Sigma_i}\prec_Q N_D^s$, for every $i\in I$, which contradicts \eqref{SigmaiD}. This contradiction proves that $D=\emptyset$.

Since  $M_{\Sigma_1}$ is strongly prime and $Q=M_{\Sigma_1}\overline{\otimes} M_{\Sigma\setminus\Sigma_1}=(\overline{\otimes}_{j=1}^nN_{\Delta_j})^t$, we deduce that $M_{\Sigma_1}\prec_Q N_{\Delta_j}^t$, for some $1\leq j\leq n$ (see \cite[Lemma 4.1]{Is20}).

Since $N_{\Delta_j}$ is prime, by applying Lemma \ref{tensor_product} we can find a decomposition $Q=N_{\Delta_j}^s\overline{\otimes}N_{\Delta\setminus\Delta_j}^{t/s}$, for some $s>0$, and a unitary $\eta\in Q$ such that $\eta M_{\Sigma_1}\eta^*=N_{\Delta_j}^s$.
Therefore, we have a $*$-isomorphism $\theta_0:M_{\Sigma_1}\rightarrow N_{\Delta_j}^s$ given by $\theta_0(x)=\eta x\eta^*$. 
Since $\Sigma_1\subsetneq\Sigma\subset\Gamma$, we get $|\Sigma_1|\leq |\Gamma|-1=d$. 

We now observe that conditions (2) and (3) from the hypothesis are inherited by $\theta_0$ from $\theta$. To see that $\theta_0$ satisfies condition (2), assume by contradiction that there is a full subgraph $F\subset\Delta_j$ such that $|F|\geq 2$, $N_F^s\prec_{N_{\Delta_j}^s} \theta_0(M_v)$ and $\theta_0(M_v)\prec_{N_{\Delta_j}^s} (N_{F\cup (F^\perp\cap\Delta_j)})^s$. 
Since $F\cup (F^\perp\cap\Delta_j)\subset F\cup F^\perp$ and $\theta_0(x)=\eta x\eta^*$ for $x\in M_{\Sigma_1}$, we get that $N_F^t\prec_{L} M_v$ and $M_v\prec_{L} (N_{F\cup F^{\perp}})^t$, which contradicts condition (2) for $\theta$. Similarly, $\theta_0$ also satisfies condition (3).

Since $v$ is untransvectable in $\Sigma_1$, we can apply the induction hypothesis to $\theta_0$ to find a vertex $w\in\Delta_j$ such that $\theta_0(M_v)\prec_{N_{\Delta_j}}N_w^s$. We have $M_v\prec_Q N_w^t$, which finishes the proof of Case 2.
\end{proof}

\section{From strong intertwining to stable isomorphism}\label{sec:intertwining}

Assume the setting of Notation \ref{notation}. Theorems \ref{arbitrary_gen}, \ref{amenable_gen} and~\ref{factors1} provide instances when there exists a graph isomorphism $\alpha:\Gamma\rightarrow\Lambda$ such that $M_v\prec_P^sN_{\alpha(v)}$ and $N_{\alpha(v)}\prec_P^sM_v$, for every $v\in\Gamma$. The following result clarifies the relationship between the algebras $M_v$  and $N_{\alpha(v)}$ entailed by the last condition.

\begin{lemma}\label{interwining_to_isomorphism} In the setting from Notation \ref{notation}, assume that $v\in\Gamma$ and $w\in\Lambda$ are such that $M_v\prec_P^sN_w$ and $N_w\prec_P^sM_v$. Then the following hold:

\begin{enumerate}
\item    Assume that $M_{\emph{st}(v)}$, $N_{\emph{st}(w)}$ are II$_1$ factors, and  the projection $p\in P$ from Notation \ref{notation} belongs to $\mathbb M_n(\mathbb C)\otimes N_{\emph{st}(w)}$. Then there exist a unitary  $u\in pPp$ and a decomposition $p(\mathbb M_n(\mathbb C)\otimes N_{\emph{st}(w)})p=N_w^s\overline{\otimes}N_{\emph{lk}(w)}^t$, for $s,t>0$, such that $uM_{\emph{st}(v)}u^*=p(\mathbb M_n(\mathbb C)\otimes N_{\emph{st}(w)})p$ and $uM_vu^*=N_w^s$. In particular, $M_v$ and $N_w$ are stably isomorphic II$_1$ factors.
\item In general, we can find nonzero projections 
$p\in M_v,q\in N_w$ and a $*$-homomorphism $\psi:qN_wq\rightarrow pM_vp$ such that $pM_vp\prec_{pM_vp}\psi(qN_wq)$. Moreover, there exist nonzero projections $s\in M_v$ and $t\in N_w$ and a unital embedding $tN_wt\subset sM_vs$ which has finite index.
\end{enumerate}

\end{lemma}

\begin{proof}
    (1) Since $M_v\prec_P^sN_w$ and $N_w\prec_P^sM_v$, Proposition \ref{normalizer} implies that $M_{\text{st}(v)}\prec_P^sN_{\text{st}(w)}$ and $N_{\text{st}(w)}\prec_P^sM_{\text{st}(v)}$. 
    Since $M_{\text{st}(v)}, N_{\text{st}(w)}$ are II$_1$ factors and $\text{st}(v)^\perp=\text{st}(w)^\perp=\emptyset$, Lemma \ref{basic}(2) implies that $M_{\text{st}(v)}'\cap pPp=\mathbb Cp$ and $N_{\text{st}(w)}'\cap qPq=\mathbb Cq$.
    These facts together with Proposition \ref{results}, (2) and (3), imply that $M_{\text{st}(v)}\nprec_PN_{\Lambda_0}$ and $N_{\text{st}(w)}\nprec_P M_{\Gamma_0}$, for every full subgraphs $\Gamma_0\subsetneq\text{st}(v)$ and $\Lambda_0\subsetneq\text{st}(w)$. By using again that $\text{st}(v)^\perp=\text{st}(w)^\perp=\emptyset$, and $M_{\text{st}(v)}$, $N_{\text{st}(w)}$ are II$_1$ factors,  Theorem \ref{uni_conj} gives unitaries $\xi,\eta\in pPp$ such that $$\text{$\text{Ad}(\xi)(M_{\text{st}(v)})\subset p(\mathbb M_n(\mathbb C)\otimes N_{\text{st}(w)})p$ \;\;\;and\;\;\; $\text{Ad}(\eta)(p(\mathbb M_n(\mathbb C)\otimes N_{\text{st}(w)})p)\subset M_{\text{st}(v)}$.}$$
    
    Thus, $\text{Ad}(\eta\xi)(M_{\text{st}(v)})\subset \text{Ad}(\eta)(p(\mathbb M_n(\mathbb C)\otimes N_{\text{st}(w)})p)\subset M_{\text{st}(v)}$. Proposition \ref{normalizer}(1) gives $\eta\xi\in M_{\text{st}(v)}$, hence $\text{Ad}(\eta\xi)(M_{\text{st}(v)})=M_{\text{st}(v)}$. This implies that $\text{Ad}(\xi)(M_{\text{st}(v)})=p(\mathbb M_n(\mathbb C)\otimes N_{\text{st}(w)})p$.

    Since $M_v\prec_P^sN_w$ and $N_w\prec_P^sM_v$, Lemma \ref{descend} implies that $\text{Ad}(\xi)(M_v)\prec^s_{\mathbb M_n(\mathbb C)\otimes N_{\text{st}(w)}}N_w$ and $N_w\prec^s_{\mathbb M_n(\mathbb C)\otimes N_{\text{st}(w)}}\text{Ad}(\xi)(M_v)$. 
    Since we have that $M_{\text{st}(v)}=M_v\overline{\otimes}M_{\text{lk}(v)}$, $N_{\text{st}(w)}=N_w\overline{\otimes}N_{\text{lk}(w)}$, and $M_v,M_{\text{lk}(v)},N_w,N_{\text{lk}(w)}$ are II$_1$ factors, the conclusion follows from Lemma \ref{tensor_product}.

    (2) Since $N_w\prec_PM_v$, we can find projections $r\in N_w,q\in M_v$ a nonzero partial isometry $\zeta_2\in qPr$
    and a $*$-homomorphism $\theta_2:rN_wr\rightarrow qM_vq$ such that $\theta_2(y)\zeta_2=\zeta_2y$, for every $y\in rN_wr$.
    Since $\zeta_2^*\zeta_2\in (rN_wr)'\cap rN_\Lambda r=(N_w'\cap N_\Lambda)r$, we can write $\zeta_2^*\zeta_2=rr'$, for a projection $r'\in N_w'\cap N_\Lambda$.

    We claim that $M_v\prec_P rN_wrr'$. Otherwise, if $M_v\nprec_P rN_wrr'=(rr')(N_wr')(rr')$, Remark \ref{elementary_facts}(6) would imply that $M_v\nprec_PN_ws$, where $s\in\mathcal Z(N_wr')=\mathcal Z(N_w)r'$ is the central support of $rr'$. On the other hand, since $M_v\prec_P^sN_w$, $N_w\prec_P^sM_v$, and $s\in N_w'\cap N_\Lambda$, Lemma \ref{proof_of_3} gives that $M_v\prec_PN_ws$. This gives a contradiction, and finishes the proof of the claim.

    Next, since $M_v\prec_P rN_wrr'$, by Proposition \ref{cutting_down}(1) we can find projections $e\in M_v,f\in rN_wr$, a partial isometry $\zeta_1\in fr'Pe$ and a $*$-homomorphism $\theta_1: eM_ve\rightarrow fN_wf$ such that $\theta_1(x)\zeta_1=\zeta_1x$, for every $x\in eM_ve$.

Let $q_0=\theta_2(f)=\theta_2(\theta_1(e))$,
 $\rho:eM_ve\rightarrow q_0M_vq_0$  the $*$-homomorphism given by $\rho(x)=\theta_2(\theta_1(x))$ and $\zeta=\zeta_2\zeta_1\in qPe$. Then $\rho(x)\zeta=\zeta x$, for every $x\in eM_ve$.
 Since $\zeta_2^*\zeta_2=rr'$ and $\zeta_1\in rr'P$ is nonzero we deduce that $\zeta^*\zeta=\zeta_1^*(\zeta_2^*\zeta_2)\zeta_1=\zeta_1^*\zeta_1\not=0$, and therefore $\zeta\not=0$.
 Moreover, $q_0\zeta=\rho(e)\zeta=\zeta e=\zeta$, hence $\zeta\in q_0Pe$.
 Since $\zeta (eM_ve)\subset (q_0M_vq_0)\zeta$, by Proposition \ref{normalizer}(1) we also deduce that $\zeta\in M_{\text{st}(v)}$ and thus $\zeta\in q_0M_{\text{st}(v)}e=q_0M_ve\overline{\otimes}M_{\text{lk}(v)}$.

 Note that for every $\delta\in M_{\text{lk}(v)}$ we have $\rho(x)\text{E}_{M_v}(\zeta (1\otimes\delta))=\text{E}_{M_v}(\zeta(1\otimes\delta))x$, for every $x\in eM_ve.$
 Consequently, we can find a nonzero $\gamma\in q_0M_ve$ such that $\rho(x)\gamma=\gamma x$, for every $x\in eM_ve$. This in particular implies that $q_0M_vq_0\prec_{q_0M_vq_0}\rho(eM_ve)$. Since $\rho(eM_ve)\subset\theta_2(fN_wf)$, we deduce that $q_0M_vq_0\prec_{q_0M_vq_0}\theta_2(fN_wf)$. This proves the main assertion of part (2).
 
 Since $\theta_2(fN_wf)\subset q_0M_vq_0$, Remark \ref{finite_index} provides a nonzero projection
 $s\in \theta_2(fN_wf)'\cap q_0M_vq_0$ such that the inclusion 
 $\theta_2(fN_wf)s\subset sM_vs$ has finite index.
  Since $\theta_2(fN_wf)s$ is isomorphic to $tN_wt$, for some projection $t\in N_w$, this finishes the proof.
\end{proof}

\section{From stable isomorphism to isomorphism: Proof of Theorem \ref{factors'}}

In this section we complete the proof of Theorem \ref{factors'}. 

\begin{theorem}\label{factors2}
Let $\Gamma$ and $\Lambda$  be two finite simple graphs. Assume  that $\Gamma$ has no isolated vertices, no leaves and girth at least $5$.
Let $(M_v)_{v\in\Gamma}$, $(N_w)_{w\in\Lambda}$ be families of II$_1$ factors. Denote by $M_\Gamma=*_{v,\Gamma}M_v$ and  $N_\Lambda=*_{w,\Lambda}N_w$ the associated graph product II$_1$ factors.
Let $\theta:M_\Gamma^t\rightarrow N_\Lambda$ be a $*$-isomorphism, for some $t>0$, and let $\alpha:\Gamma\to \Lambda$ be a graph isomorphism such that $\theta(M_v^t)\prec^s_{N_{\Lambda}}N_{\alpha(v)}$ and $N_{\alpha(v)}\prec_{N_\Lambda}^s\theta(M_v^t)$, for every vertex $v\in\Gamma$.

Then $t=1$ and for every $v\in\Gamma$  there exists a unitary $u_v\in N_\Lambda$ such that $\theta(M_v)=u_vN_{\alpha(v)}u_v^*$.

Moreover, if in addition $\Gamma$ has no separating star, then there exists a unitary $u\in N_\Lambda$ such that $\theta(M_v)=uN_{\alpha(v)}u^*$, for every $v\in\Gamma$.

\end{theorem}

Before proving Theorem~\ref{factors2}, we first explain how Theorem~\ref{factors'} from the introduction follows.

\begin{proof}[\bf Proof of Theorem~\ref{factors'}]
    Let $\Gamma,\Lambda$ be finite simple graphs of girth at least $5$ with no vertices of valence $0$ or $1$, and let $M_\Gamma$ and $N_\Lambda$ be graph products over $\Gamma,\Lambda$, with II$_1$ factor vertex algebras. Let $t>0$, and let $\theta:M_\Gamma\to N_\Lambda^t$ be an isomorphism. Remark~\ref{basic_graph_facts} ensures that $\Gamma,\Lambda$ are transvection-free and strongly reduced. Therefore,  Theorem~\ref{factors} guarantees the existence of a graph isomorphism $\alpha:\Gamma\to\Lambda$ such that $\theta(M_v)\prec^s_{N_\Lambda^t}N_{\alpha(v)}^t$ and $N_{\alpha(v)}^t\prec^s_{N_\Lambda^t}\theta(M_v)$, for every $v\in\Gamma$. 
    Let $s=1/t>0$ and $\theta_0:M_\Gamma^s\rightarrow N_\Lambda$ the $*$-isomorphism obtained by amplifying $\theta$. Then we have $\theta_0(M_v^s)\prec_{N_\Lambda}^sN_{\alpha(v)}$
 and $N_{\alpha(v)}\prec_{N_\Lambda}^s\theta_0(M_v^s)$, for every $v\in\Gamma$.    
    The conclusion now follows readily from Theorem~\ref{factors2}.
\end{proof}

The rest of the section is devoted to the proof of Theorem~\ref{factors2}.

\begin{proof}[\bf Proof of Theorem \ref{factors2}]
By symmetry, we may assume that $t\leq 1$. Put $P=M_\Gamma$. We identify $P^t=pPp$, for a projection $p\in P$ with $\tau(p)=t$, and $pPp=N_\Lambda$, via $\theta$. For integers $\ell\geq k\geq 1$, we consider the usual non-unital embedding $\mathbb M_k(\mathbb C)\subset\mathbb M_{\ell}(\mathbb C)$. We denote by $\text{Tr}$  the non-normalized trace on $\mathbb M_k(\mathbb C)$ and by $e_{i,j}\in\mathbb M_k(\mathbb C), 1\leq i,j\leq k,$ the elementary matrices, so that $\text{Tr}(e_{i,i})=1$, for every $1\leq i\leq k$.

The proof is divided between several claims. We start with the following claim.

\begin{claim}\label{delta_v}
For every $v\in\Gamma$, we can find integers $k_v,\ell_v\geq 1$ with $1\in\{k_v,\ell_v\}$, projections $p_v\in \mathbb M_{k_v}(\mathbb C)\otimes M_v$ and $q_v\in \mathbb M_{\ell_v}(\mathbb C)\otimes M_{\emph{lk}(v)}$, and a partial  isometry $\delta_v\in \mathbb M_{k_v\ell_v}(\mathbb C)\otimes P$ such that, identifying 
$\mathbb M_{k_v\ell_v}(\mathbb C)\otimes M_{\emph{st}(v)}$ with $(\mathbb M_{k_v}\otimes M_v)\overline{\otimes}(\mathbb 
M_{\ell_v}(\mathbb C)\otimes M_{\emph{lk}(v)})$ in the natural way, we have
\begin{equation}\label{delta_vv}
\begin{array}{lr}
\delta_v\delta_v^*=e_{1,1}\otimes p, \\
\delta_v^*\delta_v=p_v\otimes q_v,\\ \delta_v(p_v(\mathbb M_{k_v}(\mathbb C)\otimes M_v)p_v\otimes q_v)\delta_v^*=e_{1,1}\otimes N_{\alpha(v)}.
\end{array}
\end{equation} 
Moreover, we may take $k_v,\ell_v$ as the smallest integers with $(\emph{Tr}\otimes\tau)(p_v)\leq k_v$ and $(\emph{Tr}\otimes\tau)(q_v)\leq \ell_v$. 
\end{claim}

\begin{proof}[Proof of Claim \ref{delta_v}]
Let $v\in \Gamma$ and $e\in M_v$ be a projection with $\tau(e)=\tau(p)$.
Since $M_v\prec_P^sN_{\alpha(v)}$ and $N_{\alpha(v)}\prec_P^sM_v$, Proposition \ref{normalizer'} implies that $M_{\text{st}(v)}\prec_P^sN_{\text{st}(\alpha(v))}$ and $N_{\text{st}(\alpha(v))}\prec_P^sM_{\text{st}(v)}$. By Lemma~\ref{basic} we get that $M_{\text{st}(v)}'\cap P=\mathbb C$ and $N_{\text{st}(\alpha(v))}'\cap pPp=\mathbb Cp$. 

We observe that for every proper subgraph $\Lambda'\subsetneq\st(\alpha(v))$, we have $M_{\st(v)}\not\prec_P N_{\Lambda'}$. Otherwise, as a consequence of Lemma~\ref{basic}(1) and our assumption that all vertex algebras are factors, $N_{\Lambda'}$ is a factor. Proposition~\ref{results}(3) would then imply that $M_{\st(v)}\prec_P^s N_{\Lambda'}$, and Proposition~\ref{results}(2) would further give that $N_{\st(\alpha(v))}\prec_P^s N_{\Lambda'}$, contradicting Lemma~\ref{basic}(2). Likewise, for every proper subgraph $\Gamma'\subsetneq\st(v)$, we have $N_{\st(\alpha(v))}\not\prec_P M_{\Gamma'}$.

Next, we apply Theorem~\ref{uni_conj}. We have $eM_{\st(v)}e\subset ePe$ and $N_{\st(\alpha(v))}\subset pPp$. Since $\tau(p)=\tau(e)$ and $P$ is a II$_1$ factor, we can find a partial isometry $\delta\in P$ such that $\delta\delta^*=p$ and $\delta^*\delta=e$. Thus, we have $\delta eM_{\st(v)}e\delta^*\subset pPp=N_\Lambda$. By applying Theorem~\ref{uni_conj} to this inclusion we get a unitary $u\in\mathcal{U}(N_\Lambda)$ such that $u(\delta eM_{\st(v)}e\delta^*)u^*\subset N_{\st(\alpha(v))}$. Thus, the partial isometry $\eta=u\delta\in P$ satisfies $\eta\eta^*=p$, $\eta^*\eta=e$, and \[\mathrm{Ad}(\eta)(eM_{\text{st}(v)}e)\subset N_{\text{st}(\alpha(v))}.\]
Similarly, we also find a partial isometry $\xi\in P$ such that $\xi\xi^*=e, \xi^*\xi=p$, and \[\text{Ad}(\xi)(N_{\text{st}(\alpha(v))})\subset eM_{\text{st}(v)}e.\]
Thus, we derive that \begin{equation}\label{etaxi}\text{Ad}(\eta\xi)(N_{\text{st}(\alpha(v))})\subset \text{Ad}(\eta)(eM_{\text{st}(v)}e)\subset N_{\text{st}(\alpha(v))}.\end{equation}
By using \eqref{etaxi}, Proposition \ref{normalizer} gives that $\eta\xi\in\mathcal U(N_{\text{st}(\alpha(v))})$. Hence,  $\text{Ad}(\eta\xi)(N_{\text{st}(\alpha(v))})=N_{\text{st}(\alpha(v))}$ and  $\text{Ad}(\eta)(eM_{\text{st}(v)}e)=N_{\text{st}(\alpha(v))}$. Further, since  $M_v\prec_P^sN_{\alpha(v)}$ and $N_{\alpha(v)}\prec_P^sM_v$,  Lemma \ref{descend} implies that $\text{Ad}(\eta)(eM_ve)\prec^s_{N_{\text{st}(\alpha(v))}}N_{\alpha(v)}$ and $N_{\alpha(v)}\prec^s_{N_{\text{st}(\alpha(v))}}\text{Ad}(\eta)(eM_ve)$.

Since $N_{\text{st}(\alpha(v))}=N_{\alpha(v)}\overline{\otimes}N_{\text{lk}(\alpha(v))}$ and $eM_{\text{st}(v)}e=eM_ve\overline{\otimes}M_{\text{lk}(v)}$, 
Lemma \ref{tensor_product} provides a unitary $\zeta\in N_{\text{st}(\alpha(v))}$ and a decomposition $eM_{\text{st}(v)}e=M_{v}^{t_1}\overline{\otimes}M_{\text{lk}(v)}^{t_2}$, for some $t_1,t_2>0$ with $t_1t_2=t$, such that $\text{Ad}(\zeta\eta)(M_v^{t_1})=N_{\alpha(v)}$. 
Let $k,\ell$ be any positive integers such that $k\geq t_1$ and $\ell\geq t_2$. Let $p_0\in\mathbb M_k(\mathbb C)\otimes M_v$ and  $q_0\in\mathbb M_\ell(\mathbb C)\otimes M_{\text{lk}(v)}$ be projections such that $(\text{Tr}\otimes\tau)(p_0)=t_1$ and $(\text{Tr}\otimes\tau)(q_0)=t_2$. If we identify $(\mathbb M_k(\mathbb C)\otimes M_v)\overline{\otimes}(\mathbb M_\ell(\mathbb C)\otimes M_{\text{lk}(v)})=\mathbb M_{k\ell}(\mathbb C)\otimes M_{\text{st}(v)}$, then $(\text{Tr}\otimes\tau)(p_0\otimes q_0)=t_1t_2=t$. Therefore, $M_{\text{st}(v)}^t\cong eM_{\text{st}(v)}e$ can be identified with $(p_0(\mathbb M_k(\mathbb C)\otimes M_v)p_0)\overline{\otimes}(q_0(\mathbb M_\ell(\mathbb C)\otimes M_{\text{lk}(v)})q_0)$ in such a way that we have $M_v^{t_1}=p_0(\mathbb M_k(\mathbb C)\otimes M_v)p_0$ and $M_{\text{lk}(v)}^{t_2}=q_0(\mathbb M_\ell(\mathbb C)\otimes M_{\text{lk}(v)})q_0$. Thus, the claim in particular holds for the smallest integers $k_v,\ell_v$ such that  $k_v\geq t_1$ and $\ell_v\geq t_2$. Since $t_1t_2=t\leq 1$, we have $t_1\leq 1$ or $t_2\leq 1$ and hence $1\in\{k_v,\ell_v\}$.
\end{proof}

From now on we fix $k_v,\ell_v,p_v,q_v,\delta_v$ given by Claim~\ref{delta_v}, for every vertex $v\in\Gamma$. We define
$$\text{$A=\{v\in\Gamma\mid (\text{Tr}\otimes\tau)(p_v)\leq 1\}$\;\;\; and \;\;\; $B=\{v\in\Gamma\mid(\text{Tr}\otimes\tau)(q_v)\leq 1\}$.}$$
Since by Claim~\ref{delta_v} we have \[(\text{Tr}\otimes\tau)(p_v)(\text{Tr}\otimes\tau)(q_v)=(\mathrm{Tr}\otimes\tau)(\delta_v^*\delta_v)=(\mathrm{Tr}\otimes\tau)(\delta_v\delta_v^*)=t\leq 1,\] it follows that $A\cup B=\Gamma$.
The moreover assertion of Claim \ref{delta_v} 
also gives that $v\in A$ if and only if $k_v=1$ (in which case $p_v\in M_v$), and that $v\in B$ if and only if  $\ell_v=1$ (in which case $q_v\in M_{\text{lk}(v)}$).

\begin{claim}\label{p_vq_v'} If $v\in\Gamma$ and $v'\in\emph{lk}(v)$, then $(\emph{Tr}\otimes\tau)(p_v)=(\emph{Tr}\otimes\tau)(q_{v'})$ and  $(\emph{Tr}\otimes\tau)(q_v)=(\emph{Tr}\otimes\tau)(p_{v'})$.
In particular, $v\in A$ if and only if $v'\in B$, and $(k_v,\ell_v)=(\ell_{v'},k_{v'})$.
\end{claim}

\begin{proof}[Proof of Claim \ref{p_vq_v'}] Let $v\in\Gamma$ and $v'\in\text{lk}(v)$. Since $M_{\text{st}(v)}\prec_P^sN_{\text{st}(\alpha(v))}$ and  $M_{\text{st}(v')}\prec_P^sN_{\text{st}(\alpha(v'))}$, we get that  $M_{\{v,v'\}}\prec_P^sN_{\text{st}(\alpha(v))}$ and  $M_{\{v,v'\}}\prec_P^sN_{\text{st}(\alpha(v'))}$.
By applying Corollary \ref{intersection} we derive that $M_{\{v,v'\}}\prec_P^s N_{\text{st}(\alpha(v))\cap\text{st}(\alpha(v'))}$.
Since $\Lambda$ is triangle-free, we get that $\text{st}(\alpha(v))\cap\text{st}(\alpha(v'))=\{\alpha(v),\alpha(v')\}$ and thus $M_{\{v,v'\}}\prec_P^s N_{\{\alpha(v),\alpha(v')\}}$. Similarly, we get that $N_{\{\alpha(v),\alpha(v')\}}\prec_P^s M_{\{v,v'\}}$. 

We also observe that $M_{\{v,v'\}}\not\prec_P N_{\Lambda'}$, for every proper full subgraph $\Lambda'\subsetneq\{\alpha(v),\alpha(v')\}$, and that $N_{\{\alpha(v),\alpha(v')\}}\not\prec_P M_{\Gamma'}$, for every proper full subgraph $\Gamma'\subsetneq\{v,v'\}$, see the argument in the second paragraph of the proof of Claim~\ref{delta_v}.

Since $\Gamma$ and $\Lambda$ are triangle-free, we also derive that $\{v,v'\}^\perp=\{\alpha(v),\alpha(v')\}^\perp=\emptyset$. By Lemma~\ref{basic}(1) we get that $M_{\{v,v'\}}'\cap P=\mathbb C$ and $N_{\{\alpha(v),\alpha(v')\}}'\cap pPp=\mathbb Cp$. Let $e\in M_v$ be a projection with $\tau(e)=\tau(p)$.
By reasoning as in the proof of Claim \ref{delta_v} we find a partial isometry $\rho\in P$ such that $\rho^*\rho=e$, $\rho\rho^*=p$ and $\text{Ad}(\rho)(eM_{\{v,v'\}}e)=N_{\{\alpha(v),\alpha(v')\}}$. Since $M_v\prec_P^s N_{\alpha(v)}$ and $N_{\alpha(v)}\prec_P^sM_v$, Lemma \ref{descend} implies that $\text{Ad}(\rho)(eM_ve)\prec^s_{N_{\{\alpha(v),\alpha(v')\}}}N_{\alpha(v)}$ and $N_{\alpha(v)}\prec^s_{N_{\{\alpha(v),\alpha(v')\}}}\text{Ad}(\rho)(eM_ve)$. 

Since $M_{\{v,v'\}}=M_v\overline{\otimes}M_{v'}$ and $N_{\{\alpha(v),\alpha(v')\}}=N_{\alpha(v)}\overline{\otimes}N_{\alpha(v')}$, by applying Lemma \ref{tensor_product} we can find a unitary $\gamma\in N_{\{\alpha(v),\alpha(v')\}}$  and a decomposition $eM_{\{v,v'\}}e=M_v^{s_1}\overline{\otimes}M_{v'}^{s_2}$, for some $s_1,s_2>0$ with $s_1s_2=t$,  such that $\text{Ad}(\gamma\rho)(M_v^{s_1})=N_{\alpha(v)}$ and $\text{Ad}(\gamma\rho)(M_{v'}^{s_2})=N_{\alpha(v')}$.
 
 Let  $k\geq\max\{s_1,k_v,k_{v'}\}$ and $\ell\geq\max\{s_2,\ell_v,\ell_{v'} \}$ be integers, chosen with $k=\ell$. Note that even though $k=\ell$, it will be useful to keep two different letters in order to better distinguish the embeddings of $\mathbb{M}_k(\mathbb{C})$ and $\mathbb{M}_{\ell}(\mathbb{C})$ in $\mathbb{M}_{k\ell}(\mathbb{C})$ through the identification $\mathbb{M}_k(\mathbb{C})\otimes\mathbb{M}_{\ell}(\mathbb{C})\cong\mathbb{M}_{k\ell}(\mathbb{C})$. Since $k\geq s_1$ and $\ell\geq s_2$, there are projections $r\in \mathbb M_k(\mathbb C)\otimes M_v$, $r'\in\mathbb M_\ell(\mathbb C)\otimes M_{v'}$ such that $(\text{Tr}\otimes\tau)(r)=s_1$ and $(\text{Tr}\otimes\tau)(r')=s_2$. By using the previous paragraph and arguing as in the end of the proof of Claim \ref{delta_v} we derive the existence of a unitary $w\in\mathbb M_{k\ell}(\mathbb C)\otimes P$ such that 
 \begin{equation}\label{w}\begin{split} & w(r(\mathbb M_k(\mathbb C)\otimes M_v)r\otimes r')w^*=e_{1,1}\otimes N_{\alpha(v)} \;\;\; \text{and}\\& w(r\otimes r'(\mathbb M_\ell(\mathbb C)\otimes M_{v'})r')w^*=e_{1,1}\otimes N_{\alpha(v')}.\end{split}\end{equation}

Using that $k\geq k_v$ and $\ell\geq\ell_v$, we consider the non-unital embedding $\mathbb M_{k_v\ell_v}(\mathbb C)\subset\mathbb M_{k\ell}(\mathbb C)$ induced from the non-unital embeddings $\mathbb M_{k_v}(\mathbb C)\subset\mathbb M_{k}(\mathbb C)$ and $\mathbb M_{\ell_v}(\mathbb C)\subset\mathbb M_\ell(\mathbb C)$. Similarly, we consider the embedding $\mathbb M_{k_{v'}\ell_{v'}}(\mathbb C)\subset\mathbb M_{k\ell}(\mathbb C)$. Altogether, this allows us to view $\delta_v,\delta_{v'}\in\mathbb M_{k\ell}(\mathbb C)\otimes P$.
 Let $\omega=w^*\delta_v,\omega'=w^*\delta_{v'}\in \mathbb M_{k\ell}(\mathbb C)\otimes P$. Claim \ref{delta_v} and Equation \eqref{w} together imply that  $\omega,\omega'$ are partial isometries such that $\omega^*\omega=p_v\otimes q_v,\omega\omega^*=r\otimes r', \omega'^*\omega'=p_{v'}\otimes q_{v'}, \omega'\omega'^*=r\otimes r'$, 
\begin{equation}\label{omega}\omega(p_v(\mathbb M_k(\mathbb C)\otimes M_v)p_v\otimes q_v)\omega^*=r(\mathbb M_k(\mathbb C)\otimes M_v)r\otimes r'\end{equation} and \begin{equation}\label{omega'}\omega'(p_{v'}(\mathbb M_k(\mathbb C)\otimes M_{v'})p_{v'}\otimes q_{v'})\omega'^*=r\otimes r'(\mathbb M_\ell(\mathbb C)\otimes M_{v'})r'.\end{equation}
By using Proposition \ref{normalizer}(1), we get that $\omega\in \mathbb M_{k\ell}(\mathbb C)\otimes M_{\text{st}(v)}$ and $\omega'\in\mathbb M_{k\ell}(\mathbb C)\otimes M_{\text{st}(v')}$. 
Since $\omega\in \mathbb M_{k\ell}(\mathbb C)\otimes M_{\text{st}(v)}$, by combining \eqref{omega} and Lemma \ref{conjugacy_in_tensors}, it follows that  
\begin{equation}\label{traces1}\text{$(\text{Tr}\otimes\tau)(p_v)=(\text{Tr}\otimes\tau)(r)$\;\;\; and \;\;\; $(\text{Tr}\otimes\tau)(q_v)=(\text{Tr}\otimes\tau)(r')$.}\end{equation}
We next claim that
\begin{equation}\label{traces2}\text{($\text{Tr}\otimes\tau)(p_{v'})=(\text{Tr}\otimes\tau)(r')$\;\;\; and \;\;\; $(\text{Tr}\otimes\tau)(q_{v'})=(\text{Tr}\otimes\tau)(r)$}.\end{equation}

Since $k=\ell$, there is  $y\in\mathcal U(\mathbb M_{k\ell}(\mathbb C))$ such that $\text{Ad}(y)(\mathbb M_k(\mathbb C)\otimes 1)=1\otimes\mathbb M_\ell(\mathbb C)$.
By taking relative commutants, we also have $\text{Ad}(y)(1\otimes\mathbb M_\ell(\mathbb C))=\mathbb M_k(\mathbb C)\otimes 1$. Put $z=y\otimes 1\in\mathcal U(\mathbb M_{k\ell}(\mathbb C)\otimes P)$.  Since $r\in \mathbb M_k(\mathbb C)\otimes M_v$ and  $r'\in\mathbb M_\ell(\mathbb C)\otimes M_{v'}$, we get that $s=\text{Ad}(z)(r)\in\mathbb M_\ell(\mathbb C)\otimes M_v$ and $s'=\text{Ad}(z)(r')\in\mathbb M_k(\mathbb C)\otimes M_{v'}$.
Moreover, $\text{Ad}(z)(r\otimes r'(\mathbb M_\ell(\mathbb C)\otimes M_{v'})r')=s\otimes s'(\mathbb M_k(\mathbb C)\otimes M_{v'})s'$, which in combination with \eqref{omega'} gives that 
$$(z\omega')(p_{v'}(\mathbb M_k(\mathbb C)\otimes M_{v'})p_{v'}\otimes q_{v'})(z\omega')^*=s\otimes s'(\mathbb M_k(\mathbb C)\otimes M_{v'})s'.$$
Since $z\omega'\in \mathbb M_{k\ell}(\mathbb C)\otimes M_{\text{st}(v')}$, we can apply Lemma \ref{conjugacy_in_tensors} to conclude that 
$$\text{$(\text{Tr}\otimes\tau)(p_{v'})=(\text{Tr}\otimes\tau)(s')=(\text{Tr}\otimes\tau)(r')$\;\;\; and \;\;\; $(\text{Tr}\otimes\tau)(q_{v'})=(\text{Tr}\otimes \tau)(s)=(\text{Tr}\otimes\tau)(r)$}.$$
This proves \eqref{traces2} which together with \eqref{traces1} implies the claim. \end{proof}

 Our next goal is to prove that for adjacent vertices $v,v'\in\Gamma$, the partial isometry $\delta_{v'}^*\delta_v$ admits a special decomposition given by the following claim.

\begin{claim}\label{deco} Let $v\in\Gamma$, $v'\in\emph{lk}(v)$, and let $m\geq \max\{k_v,\ell_v,k_{v'},\ell_{v'}\}$ be an integer. 
Then there exist partial isometries $\alpha_{v',v}, \alpha_{v,v'}\in\mathbb M_m(\mathbb C)\otimes P$  such that   
\begin{enumerate}
\item $\delta_{v'}^*\delta_v=\alpha_{v',v}^*\alpha_{v,v'}$.
\item $\alpha_{v',v}\in  1\otimes M_{\emph{lk}(v')}$ and $\alpha_{v,v'}\in \mathbb M_m(\mathbb C)\otimes M_{\emph{lk}(v)}$ if $v\in A$.

\item $\alpha_{v',v}\in \mathbb M_m(\mathbb C)\otimes M_{\emph{lk}(v')}$ and $\alpha_{v,v'}\in 1\otimes M_{\emph{lk}(v)}$ if $v\in B$.

\item $\alpha_{v',v}\alpha_{v',v}^*=p_v$ and $\alpha_{v',v}^*\alpha_{v',v}=q_{v'}$.
\item $\alpha_{v,v'}\alpha_{v,v'}^*=p_{v'}$ and $\alpha_{v,v'}^*\alpha_{v,v'}=q_v$.
\end{enumerate}
\end{claim}

As $1\in\{k_v,\ell_v\}$ and $1\in\{k_{v'},\ell_{v'}\}$, we have $m\geq\max\{k_v\ell_v,k_{v'}\ell_{v'}\}$. This allows in Claim \eqref{deco} to view $\delta_v,\delta_{v'}\in\mathbb M_m(\mathbb C)\otimes P$. Also, Claim \eqref{delta_v} gives that $\delta_v\delta_v^*=e_{1,1}\otimes p=\delta_{v'}\delta_{v'}^*$ and that $\delta_{v'}^*\delta_v\in\mathbb M_m(\mathbb C)\otimes P$ is a partial isometry with left and right supports  $p_{v'}\otimes q_{v'}$ and $p_v\otimes q_v$, respectively.

\begin{proof}[Proof of Claim \ref{deco}] Let $v\in\Gamma$ and $v'\in\text{lk}(v)$.
If $v\in B$, then $v'\in A$ by Claim \ref{p_vq_v'}.
Thus, in order to prove the claim, it suffices to show that if $v\in A$, then there are partial isometries $\alpha_{v',v}, \alpha_{v,v'}\in\mathbb M_m(\mathbb C)\otimes P$  satisfying   (1), (2), 
 (4) and (5). 

Since $v\in A$, Claim \ref{p_vq_v'} gives that $v'\in B$. Hence  $k_v=\ell_{v'}=1$. Since $m\geq\ell_v$ and $m\geq k_{v'}$, by Claim \ref{delta_v} we can view $\delta_v,\delta_{v'}\in\mathbb M_m(\mathbb C)\otimes P$. Moreover, recall that $\delta_v^*\delta_v=p_v\otimes q_v$ and $\delta_{v'}^*\delta_{v'}=p_{v'}\otimes q_{v'}$, where  $p_v\in M_v, q_v\in\mathbb M_m(\mathbb C)\otimes M_{\text{lk}(v)}, p_{v'}\in\mathbb M_m(\mathbb C)\otimes M_{v'}$, $q_{v'}\in M_{\text{lk}(v')}$ and we identify $\mathbb M_m(\mathbb C)\otimes M_{\text{st}(v)}=M_v\overline{\otimes}(\mathbb M_m(\mathbb C)\otimes M_{\text{lk}(v)})$ and $\mathbb M_m(\mathbb C)\otimes M_{\text{st}(v')}=(\mathbb M_m(\mathbb C)\otimes M_{v'})\overline{\otimes}M_{\text{lk}(v')}$.

Denote
$Q=p_v M_vp_v\otimes q_v$, $R=p_{v'}(\mathbb M_m(\mathbb C)\otimes M_{v'})p_{v'}\otimes q_{v'}$ and $S=p_{v'}\otimes q_{v'}M_{\text{lk}(v')}q_{v'}$. 
By using Claim \ref{delta_v} we deduce that
 \begin{equation}\label{Q_v}\begin{split}[\text{Ad}(\delta_{v'}^*\delta_v)(Q),R]&=\text{Ad}(\delta_{v'}^*)([\text{Ad}(\delta_v)(Q),\text{Ad}(\delta_{v'})(R)])\\&=\text{Ad}(\delta_{v'}^*)\big([e_{1,1}\otimes N_{\alpha(v)},e_{1,1}\otimes N_{\alpha(v')}]\big)=\{0\}.\end{split}\end{equation}
Since $R'\cap (p_{v'}\otimes q_{v'})(\mathbb M_{m}(\mathbb C)\otimes P)(p_{v'}\otimes q_{v'})=S$ by Lemma~\ref{basic}(1), using \eqref{Q_v}  we derive that \begin{equation}\text{Ad}(\delta_{v'}^*\delta_v)(Q)\subset S.\end{equation}
Write $\text{Ad}(\delta_{v'}^*\delta_v)(Q)=p_{v'}\otimes T$, for a von Neumann subalgebra $T\subset q_{v'}M_{\text{lk}(v')}q_{v'}$.
We claim that
\begin{equation}\label{T}\text{$T\prec_{M_{\text{lk}(v')}} M_v$ \;\;\; and\;\;\; $T'\cap  q_{v'}M_{\text{lk}(v')}q_{v'}=\mathbb Cq_{v'}$.}
\end{equation}

Assume by contradiction that $T\nprec_{M_{\text{lk}(v')}} M_v$. Then
 Theorem \ref{intertwine} provides a net $(u_i)\subset\mathcal U(T)$ such that $\|\text{E}_{M_v}(a^*u_ib)\|_2\rightarrow 0$, for every $a,b\in M_{\text{lk}(v')}$. The proof of Lemma \ref{descend}  implies that $\|\text{E}_{M_v}(a^*u_ib)\|_2\rightarrow 0$, for every $a,b\in P$. 
Recalling that $Q=p_v M_vp_v\otimes q_v$, it follows that 
$\|\text{E}_Q(\text{Ad}(\delta_v^*\delta_{v'})(p_{v'}\otimes u_i))\|_2\rightarrow 0$.
 This contradicts the fact that $\text{Ad}(\delta_v^*\delta_{v'})(p_{v'}\otimes u_i)\in Q$ and $\|\text{Ad}(\delta_v^*\delta_{v'})(p_{v'}\otimes u_i)\|_2=\|p_{v'}\otimes u_i\|_2=\|p_{v'}\|_2>0$, for every $i$, proving the first assertion of \eqref{T}.

To prove the second assertion from \eqref{T},  note that Claim \ref{delta_v} gives that $\text{Ad}(\delta_v)(Q)=e_{1,1}\otimes N_{\alpha(v)}$ and $\text{Ad}(\delta_{v'})(R)=e_{1,1}\otimes N_{\alpha(v')}$.  By using Lemma~\ref{basic}(1) to take relative commutants in the latter equality, we get that $\text{Ad}(\delta_{v'})(S)=e_{1,1}\otimes N_{\text{lk}(\alpha(v'))}$.
Since $\text{lk}(\alpha(v))\cap\text{lk}(\alpha(v'))=\emptyset$, Lemma~\ref{basic}(1) further implies that $N_{\alpha(v)}'\cap N_{\text{lk}(\alpha(v'))}=\mathbb Cp$.
By combining the last two facts we get that 
 \begin{equation}\begin{split}\text{Ad}(\delta_{v'}^*\delta_v)(Q)'\cap S&=\text{Ad}(\delta_{v'}^*)(\text{Ad}(\delta_v)(Q)'\cap \text{Ad}(\delta_{v'})(S))
\\&=\text{Ad}(\delta_{v'}^*)(e_{1,1}\otimes (N_{\alpha(v)}'\cap N_{\text{lk}(\alpha(v'))}))\\&=\mathbb C\text{Ad}(\delta_{v'}^*)(e_{1,1}\otimes p)=\mathbb C(p_{v'}\otimes q_{v'}).\end{split}\end{equation} 
Since $\text{Ad}(\delta_{v'}^*\delta_v)(Q)'\cap S=p_{v'}\otimes (T'\cap q_{v'}M_{\text{lk}(v')}q_{v'})$, it follows that $T'\cap  q_{v'}M_{\text{lk}(v')}q_{v'}=\mathbb Cq_{v'}$, which finishes the proof of \eqref{T}. 

Since $M_v$ is a factor, $Q$ and $T$ are also factors. By using \eqref{T} and that $\text{lk}(v)\cap\text{lk}(v')=\emptyset$, Theorem \ref{uni_conj} implies the existence of a partial isometry $\lambda\in M_{\text{lk}(v')}$ such that  $\lambda^*\lambda=q_{v'}$ and $\text{Ad}(\lambda)(T)\subset M_v$. Since  $\tau(q_{v'})=\tau(p_v)$ and $M_{\text{lk}(v')}$ is a factor, we may assume that $\lambda\lambda^*=p_v$ and hence $\text{Ad}(\lambda)(T)\subset p_vM_vp_v$. Denote $z=(p_{v'}\otimes \lambda)\delta_{v'}^*\delta_v\in\mathbb M_m(\mathbb C)\otimes P$. Then
$z$ is a partial isometry such that $zz^*=p_{v'}\otimes p_v, z^*z=p_v\otimes q_v$, and using that $\text{Ad}(\delta_{v'}^*\delta_v)(Q)=p_{v'}\otimes T$ implies that
\begin{equation}\label{tensor_inclusion}\text{Ad}(z)(Q)=\text{Ad}(p_{v'}\otimes\lambda)(p_{v'}\otimes T)=p_{v'}\otimes\text{Ad}(\lambda)(T)\subset p_{v'}\otimes p_vM_vp_v.\end{equation}
Since $Q=p_vM_vp_v\otimes q_v$, combining \eqref{tensor_inclusion} and  Proposition \ref{normalizer}(1) implies that $z\in \mathbb M_m(\mathbb C){\otimes}M_{\text{st}(v)}$. Moreover, note that $\mathbb M_m(\mathbb C){\otimes}M_{\text{st}(v)}=(\mathbb M_m(\mathbb C)\otimes M_{\text{lk}(v)})\overline{\otimes}M_v$, $p_v\in M_v$, $q_v,p_{v'}\in\mathbb M_m(\mathbb C)\otimes M_{\text{lk}(v)}$, $zz^*=p_{v'}\otimes p_v$ and $z^*z=q_v\otimes p_v$. By applying Lemma \ref{conjugacy_in_tensors} we find partial isometries $\sigma\in M_v$ and $\beta\in \mathbb M_m(\mathbb C)\otimes M_{\text{lk}(v)}$ such that $\sigma\sigma^*=\sigma^*\sigma=p_v$, $\beta\beta^*=p_{v'},\beta^*\beta=q_v$ and $z=\beta\otimes\sigma$.

Since $z=(p_{v'}\otimes \lambda)\delta_{v'}^*\delta_v$, we get that $$\delta_{v'}^*\delta_v=(p_{v'}\otimes\lambda^*)(\beta\otimes\sigma)=(1\otimes\lambda^*)(p_{v'}\beta\otimes 1)(1\otimes \sigma)=(1\otimes\lambda^*\sigma)(\beta\otimes 1).$$ Since $\lambda\in M_{\text{lk}(v')}$, $\sigma\in M_v$, $\lambda^*\lambda=q_{v'}$ and $\lambda\lambda^*=\sigma\sigma^*=\sigma^*\sigma=p_v$, we get that
$\alpha=\lambda^*\sigma\in M_{\text{lk}(v')}$ is a partial isometry with $\alpha\alpha^*=q_{v'}$ and $\alpha^*\alpha=p_v$. This finishes the proof of Claim \ref{deco}.
\end{proof}

We are ready to prove that $t=1$. We will achieve this by analyzing cycles of minimal length in $\Gamma$.

\begin{claim}\label{t=1} We have $t=1$. 

Moreover, let $\Gamma_0\subset\Gamma$ be a connected component of $\Gamma$ and
$v_0,v_1,\dots,v_{n-1},v_n=v_0$ be a cycle of minimal length in $\Gamma_0$.
Then $(\emph{Tr}\otimes\tau)(p_{v_i})=(\emph{Tr}\otimes\tau)(q_{v_i})=1$, thus $k_{v_i}=\ell_{v_i}=1$ and $p_{v_i}=q_{v_i}=1$, for every $0\leq i\leq n-1$. Additionally,  $\alpha_{v_i,v_{i-1}}\alpha_{v_i,v_{i+1}}^*\in\mathcal U(M_{v_{i-1}})\mathcal U(M_{v_{i+1}})$, for every $0\leq i\leq n-1$.

\end{claim}

\begin{proof}[Proof of Claim \ref{t=1}] 
We will prove the moreover assertion, and along the way show that $t=1$.
Put $m=\max\{k_{v_i},\ell_{v_i}\mid 0\leq i\leq n-1\}$. 
By Claim \ref{p_vq_v'},  $v_0\in B$ or $v_1\in B$. Thus, after possibly relabeling the cycle, we may assume that $v_0\in B$.
We treat two cases, depending on whether $n$ is even or odd.
The case when $n$ is even is the more challenging of the two cases because, unlike in the case when $n$ is odd, we cannot easily show that $k_{v_i}$ and $\ell_{v_i}$ are equal to $1$. As such, when $n$ is even, we will need to carefully keep track of various matrix amplifications.

{\bf Case 1.}  $n$ is even.

Since $v_0\in B$, Claim \ref{p_vq_v'} implies by induction that $v_i\in B$, if $i$ is even, and $v_i\in A$, if $i$ is odd. 
We identify $\mathbb M_m(\mathbb C)\otimes M_{\text{st}(v_i)}$ with $(\mathbb M_m(\mathbb C)\otimes M_{v_i})\overline{\otimes}M_{\text{lk}(v_i)}$, if $i$ is even, and $M_{v_i}\overline{\otimes}(\mathbb M_m(\mathbb C)\otimes M_{\text{lk}(v_i)})$, if $i$ is odd. 
We use these identifications  to view $\mathbb M_m(\mathbb C)\otimes M_{v_i}$ and $M_{\text{lk}(v_i)}$, if $i$ is even, and respectively, $M_{v_i}$ and $\mathbb M_m(\mathbb C)\otimes M_{\text{lk}(v_i)}$, if $i$ is odd, as unital subalgebras of $\mathbb M_m(\mathbb C)\otimes P$.
These identifications give rise to unital embeddings of the form $Q\otimes R\subset\mathbb M_m(\mathbb C)\otimes P$. In order to avoid confusion over which tensor product is being used we will write $xy$ instead of $x\otimes y$, for $x\in Q$ and $y\in R$.

By applying Claim \ref{deco}, for every $0\leq i\leq n-1$, we  find partial isometries $\alpha_{i-1,i}:=\alpha_{v_{i-1},v_{i}}$ and $\alpha_{i,i-1}:=\alpha_{v_{i},v_{i-1}}$ 
(with indices considered modulo $n$, i.e.,\ $-1$ is identified to $n-1$) such that the following conditions hold:
\begin{enumerate}
\item $\delta_{v_{i-1}}^*\delta_{v_{i}}=\alpha_{{i-1},i}^*\alpha_{i,{i-1}}$.
\item $\alpha_{i-1,i}\in M_{\text{lk}(v_{i-1})}$ and $\alpha_{i,i-1}\in\mathbb M_m(\mathbb C)\otimes M_{\text{lk}(v_{i})}$, if $i$ is odd.
\item $\alpha_{{i-1},i}\in \mathbb M_m(\mathbb C)\otimes M_{\text{lk}(v_{i-1})}$ and $\alpha_{i,{i-1}}\in M_{\text{lk}(v_{i})}$, if $i$ is even.
\item $\alpha_{{i-1},i}\alpha_{{i-1},i}^*=p_{v_i}$ and $\alpha_{{i-1},i}^*\alpha_{{i-1},i}=q_{v_{i-1}}$.
\item $\alpha_{i,{i-1}}\alpha_{i,{i-1}}^*=p_{v_{i-1}}$ and $\alpha_{i,{i-1}}^*\alpha_{i,{i-1}}=q_{v_{i}}$.
\end{enumerate}
With our convention, we have $\alpha_{n,i}=\alpha_{0,i}$ and $\alpha_{i,n}=\alpha_{i,0}$, whenever $i\in\{1,n-1\}$.

Further, by combining (1) with \eqref{delta_vv}, we get that \begin{equation}\label{product}\begin{split}\alpha_{0,1}^*\alpha_{1,0}\alpha_{1,2}^*\alpha_{2,1}\cdots\alpha_{{n-1},0}^*\alpha_{0,{n-1}}&=(\delta_{v_0}^*\delta_{v_1})(\delta_{v_1}^*\delta_{v_2})\dots (\delta_{v_{n-1}}^*\delta_{v_n})\\&=\delta_{v_0}^*(\delta_{v_1}\delta_{v_1}^*)\dots(\delta_{v_{n-1}}\delta_{v_{n-1}}^*)\delta_{v_n}\\&=\delta_{v_0}^*\delta_{v_0}=p_{v_0} q_{v_0}.\end{split}\end{equation}
By combining (4), (5) and \eqref{product} we get that \begin{equation}\label{product2}\begin{split}&\alpha_{0,1}(p_{v_0}q_{v_0})\alpha_{0,{n-1}}^*\\&=
(\alpha_{0,1}\alpha_{0,1}^*)(\alpha_{1,0}\cdots\alpha_{{n-1},0}^*)(\alpha_{0,{n-1}}\alpha_{0,{n-1}}^*)\\&=p_{v_1}(\alpha_{1,0}\alpha_{1,2}^*\alpha_{2,1}\cdots \alpha_{{n-2},{n-1}}^*\alpha_{{n-1},{n-2}}\alpha_{{n-1},0}^*)p_{v_{n-1}}\\&=(p_{v_1}\alpha_{1,0}\alpha_{1,2}^*\alpha_{2,1})(\alpha_{2,3}^*\alpha_{3,2}\cdots \alpha_{{n-3},{n-2}}^*\alpha_{{n-2},{n-3}})(\alpha_{{n-2},{n-1}}^*\alpha_{{n-1},{n-2}}\alpha_{{n-1},n}^*p_{v_{n-1}})
\end{split}\end{equation}

Since $\alpha_{0,1}\in M_{\text{lk}(v_0)}q_{v_0}$ and $\alpha_{0,n-1}^*\in q_{v_0}M_{\text{lk}(v_0)}$ $n$ (as $n$ is even) 
by (2)-(5), and $p_{v_0}\in \mathbb M_m(\mathbb C)\otimes M_{v_0}$, we get $\alpha_{0,1}(p_{v_0}q_{v_0})\alpha_{0,n-1}^*=p_{v_0}(\alpha_{0,1}q_{v_0}\alpha_{0,n-1}^*)=p_{v_0}(\alpha_{0,1}\alpha_{0,n-1}^*)$.
Since $p_{v_1}\in M_{v_1}, \alpha_{1,0},\alpha_{1,2}^*\in\mathbb M_m(\mathbb C)\otimes M_{\text{lk}(v_1)}$ and $\alpha_{2,1}\in p_{v_1}M_{\text{lk}(v_2)}$, we also get $p_{v_1}\alpha_{1,0}\alpha_{1,2}^*\alpha_{2,1}=\alpha_{1,0}\alpha_{1,2}^* p_{v_1}\alpha_{2,1}=\alpha_{1,0}\alpha_{1,2}^*\alpha_{2,1}$. Similarly, we deduce that $\alpha_{{n-2},{n-1}}^*\alpha_{{n-1},{n-2}}\alpha_{{n-1},n}^*p_{v_{n-1}}=\alpha_{{n-2},{n-1}}^*\alpha_{{n-1},{n-2}}\alpha_{{n-1},n}^*$. By combining these three facts, \eqref{product2} rewrites as \begin{equation}\label{product3} p_{v_0}(\alpha_{0,1}\alpha_{0,n-1}^*) =\alpha_{1,0}\alpha_{1,2}^*\alpha_{2,1}\alpha_{2,3}^*\alpha_{3,2}\cdots \alpha_{{n-3},{n-2}}^*\alpha_{{n-2},{n-3}}\alpha_{{n-2},{n-1}}^*\alpha_{{n-1},{n-2}}\alpha_{{n-1},n}^*.\end{equation}

Next, we claim that \begin{equation}\label{Winclusion} \mathcal W_{\text{lk}(v_0)}\cap\left(\mathcal W_{\text{lk}(v_1)}\cdots\mathcal W_{\text{lk}(v_{n-1})}\right)\subset\mathcal W_{v_1}\mathcal W_{v_{n-1}}. \end{equation}
To justify \eqref{Winclusion}, we note that since the cycle $\{v_0,v_1,\dots,v_{n-1},v_n=v_0\}$ has minimal length in its connected component, and $\Gamma$ has no leaves, no isolated vertices and girth at least $5$, we get that $n\geq 5$ and $\text{lk}(v_0)\cap (\cup_{i=1}^{n-1}\text{lk}(v_i))=\{v_1,v_{n-1}\}$. From this we deduce that
\begin{equation}\label{inclusion} \mathcal W_{\text{lk}(v_0)}\cap\left(\mathcal W_{\text{lk}(v_1)}\cdots\mathcal W_{\text{lk}(v_{n-1})}\right) \subset\mathcal W_{\text{lk}(v_0)}\cap\mathcal W_{\cup_{i=1}^{n-1}\text{lk}(v_i)}=\mathcal W_{\text{lk}(v_0)\cap(\cup_{i=1}^{n-1}\text{lk}(v_i))}=\mathcal W_{\{v_1,v_{n-1}\}},\end{equation}
where the first equality is a classical fact about intersections of parabolic subgroups in Coxeter groups \cite{Tit74,Sol76}. Let $g\in\mathcal W_{\{v_1,v_{n-1}\}}\setminus\mathcal W_{v_1}\mathcal W_{v_{n-1}}$. Then $v_{n-1}$ appears at least once before $v_1$ in the reduced form $g=g_1\cdots g_d$ of $g$, i.e.,  there exist $1\leq i<j\leq d$ such that $g_i=v_{n-1}$ and $g_j=v_1$. On the other hand, we have that $v_1\in\text{lk}(v_2)$, $v_1\not\in\text{lk}(v_b)$, for every $b\in\{1,\dots,n-1\}\setminus\{2\}$, $v_{n-1}\in \text{lk}(v_{n-2})$, and $v_{n-1}\not\in\text{lk}(v_c)$, for every $c\in\{1,\dots,n-1\}\setminus\{n-2\}$. This implies that $g\not\in \mathcal W_{\text{lk}(v_1)}\cdots\mathcal W_{\text{lk}(v_{n-1})}$. Together with \eqref{inclusion}, we get  $\mathcal W_{\text{lk}(v_0)}\cap\left(\mathcal W_{\text{lk}(v_1)}\cdots\mathcal W_{\text{lk}(v_{n-1})}\right)\subset\mathcal W_{v_1}\mathcal W_{v_{n-1}}$, which proves \eqref{Winclusion}

By combining \eqref{Winclusion} with Lemma \ref{intersections}, we get that $$\text{E}_{M_{\text{lk}(v_0)}}(M_{\text{lk}(v_1)}\cdots M_{\text{lk}(v_{n-1})})\subset\overline{\text{sp}}(\{\mathring{M}_{\bf v}\mid v\in\mathcal W_{v_1}\mathcal W_{v_{n-1}}\})\subset\overline{\text{sp}}(M_{v_1}M_{v_{n-1}}).$$ Since $\alpha_{i,i-1}\alpha_{i,i+1}^*$ belongs to $M_{\text{lk}(v_i)}$ or $\mathbb M_m(\mathbb C)\otimes M_{\text{lk}(v_i)}$, for every $1\leq i\leq n-1$, we deduce that 
\begin{equation}\label{project1}
\text{$\text{E}_{M_{\text{lk}(v_0)}}(\alpha_{1,0}\alpha_{1,2}^*\alpha_{2,1}\cdots \alpha_{{n-2},{n-1}}^*\alpha_{{n-1},{n-2}}\alpha_{{n-1},n}^*)\in\overline{\text{sp}}(M_{v_1}M_{v_{n-1}})$.}
\end{equation}
On other hand, since $p_{v_0}\in \mathbb M_m(\mathbb C)\otimes M_{v_0}$ and $\alpha_{0,1}\alpha_{0,n-1}^*\in M_{\text{lk}(v_0)}$, we get that 
\begin{equation}\label{project2}
\text{$\text{E}_{M_{\text{lk}(v_0)}}(p_{v_0}(\alpha_{0,1}\alpha_{0,n-1}^*))=(\text{Tr}\otimes\tau)(p_{v_0})\alpha_{0,1}\alpha_{0,n-1}^*.$}
\end{equation}

By combining  \eqref{product3}, \eqref{project1} and \eqref{project2}, we conclude that $\zeta:=\alpha_{0,1}\alpha_{0,n-1}^*$ belongs to $ \overline{\text{sp}}(M_{v_1}M_{v_{n-1}})$. Moreover, $\zeta$ is a partial isometry such that $\zeta\zeta^*=p_{v_1}\in M_{v_1}$ and $\zeta^*\zeta=p_{v_{n-1}}\in M_{v_{n-1}}$. Since $\Gamma$ is triangle-free, $v_1$ and $v_{n-1}$ are not adjacent, hence $M_{\{v_1,v_{n-1}\}}=M_{v_1}*M_{v_{n-1}}$. Applying Lemma~\ref{partial_isometry} we deduce that $\zeta$ is a unitary, $\zeta\in\mathcal U(M_{v_1})\mathcal U(M_{v_{n-1}})$, thus $\alpha_{0,n-1}\alpha_{0,1}^*=\zeta^*\in\mathcal U(M_{v_{n-1}})\mathcal U(M_{v_1})$.

Since $\zeta$ is a unitary, $p_{v_1}=1$. As $v_1\in A$,  we have $(\text{Tr}\otimes\tau)(p_{v_1})\leq (\text{Tr}\otimes \tau)(q_{v_1})$. Since $p_{v_1}=1$ and $(\text{Tr}\otimes\tau)(p_{v_1})(\text{Tr}\otimes \tau)(q_{v_1})=t\leq 1$, we conclude that $t=1$ and $(\text{Tr}\otimes\tau)(p_{v_1})=(\text{Tr}\otimes\tau)(q_{v_1})=1$. By using Claim \ref{p_vq_v'} and induction, we get $(\text{Tr}\otimes\tau)(p_{v_i})=(\text{Tr}\otimes\tau)(q_{v_i})=1$, for every $0\leq i\leq n-1$. Thus, $k_{v_i}=\ell_{v_i}=1$ and  $p_{v_i}=q_{v_i}=1$, for every $0\leq i\leq n-1$.  In particular, $v_i\in B$ and repeating the argument above gives that $\alpha_{i,{i-1}}\alpha_{i,{i+1}}^*\in\mathcal U(M_{v_{i-1}})\mathcal U(M_{v_{i+1}})$, for every $0\leq i\leq n-1$. This proves the claim in the case $n$ is even.

{\bf Case 2.} $n$ is odd.

Since $n$ is odd,   Claim \ref{p_vq_v'} and induction imply that $(\text{Tr}\otimes\tau)(p_{v_0})=(\text{Tr}\otimes\tau)(p_{v_{n-1}})=(\text{Tr}\otimes\tau)(q_{v_0})$.
Thus,
 $(\text{Tr}\otimes\tau)(p_{v_0})=(\text{Tr}\otimes \tau)(q_{v_0})=\sqrt{t}$. Similarly, we get $(\text{Tr}\otimes\tau)(p_{v_i})=(\text{Tr}\otimes \tau)(q_{v_i})=\sqrt{t}\leq 1$. Therefore, $k_{v_i}=\ell_{v_i}=1$ and $v_i\in A\cap B$, for every $0\leq i\leq n-1$. 
 By   Claim \ref{deco} for every $0\leq i\leq n-1$ we can find partial isometries $\alpha_{i-1,i}:=\alpha_{v_{i-1},v_{i}}\in M_{\text{lk}(v_{i-1})}$ and $\alpha_{i,i-1}:=\alpha_{v_{i},v_{i-1}}\in M_{\text{lk}(v_i)}$ such that (1) and (4)-(5) hold. The conclusion now follows by repeating the argument used in Case 1.
\end{proof}

Next, we put together the main findings of this proof so far in the following claim.

\begin{claim}\label{u_v}
For every $v\in\Gamma$,  we have $k_v=\ell_v=1$ and $p_v=q_v=1$. Thus, $\delta_v\in P$ is a unitary such that  $\delta_vM_v\delta_v^*=N_{\alpha(v)}$. Moreover, for every $v\in\Gamma$ and $v'\in\emph{lk}(v)$, $\delta_{v'}^*\delta_v=\alpha_{v',v}^*\alpha_{v,v'}$, for unitaries $\alpha_{v',v}\in M_{\emph{lk}(v')}$ and $\alpha_{v,v'}\in M_{\emph{lk}(v)}$.
\end{claim}

\begin{proof}[Proof of Claim \ref{u_v}]
Let  $\Gamma_0$ be the connected component of $v$ in $\Gamma$. Let $v_0,v_1,\dots,v_{n-1},v_n=v_0$ be a cycle of minimal length in $\Gamma_0$. Claim \ref{t=1} gives that  $(\text{Tr}\otimes\tau)(p_{v_0})=(\text{Tr}\otimes\tau)(q_{v_0})=1$. We have $v_0\in\Gamma_0$, $\Gamma_0$ is connected,  and by Claim~\ref{p_vq_v'} the set $\{w\in\Gamma\mid (\text{Tr}\otimes\tau)(p_w)=(\text{Tr}\otimes\tau)(q_w)=1\}$ has the property that if $w$ belong to this set, then any neighbour of $w$ belongs to the set. We thus get $(\text{Tr}\otimes\tau)(p_{v})=(\text{Tr}\otimes\tau)(q_{v})=1$.  Claims \ref{delta_v} and \ref{deco} now give the conclusion.
\end{proof}

Claims \ref{t=1} and \ref{u_v} imply the main assertion of Theorem \ref{factors2}. To prove the moreover assertion, assume additionally that $\Gamma$ has no separating star (in particular, $\Gamma$ is connected). By Claim \ref{u_v}, $\delta_v M_v\delta_v^*=N_{\alpha(v)}$, for every $v\in\Gamma$. The moreover assertion amounts to showing that we can take the unitary $\delta_v\in P$ to be independent of $v$.
We will first show that this is the case along minimal cycles, see Claim \ref{almost_the_end}.

Towards this goal, we fix for the rest of the proof a cycle 
$v_0,v_1,\dots, v_{n-1},v_n=v_0$  of minimal length in $\Gamma$. 
 By applying Claim \ref{t=1},  we have $\alpha_{v_i,v_{i-1}}\alpha_{v_i,v_{i+1}}^*\in\mathcal U(M_{v_{i-1}})\mathcal U(M_{v_{i+1}})$, for every $0\leq i\leq n-1$.
Hence, there exist $\alpha_{v_i,v_{i-1},v_{i+1}}\in\mathcal U(M_{v_{i-1}})$ and $\alpha_{v_i,v_{i+1},v_{i-1}}\in\mathcal U(M_{v_{i+1}})$ such that \begin{equation}\label{alpha1}\text{$\alpha_{v_i,v_{i-1}}\alpha_{v_i,v_{i+1}}^*=\alpha_{v_i,v_{i-1},v_{i+1}}\alpha_{v_i,v_{i+1},v_{i-1}}^*$, \;\;\; for every $0\leq i\leq n-1$.}\end{equation}

\begin{claim}\label{invariance}
 $\alpha_{v_{i-1},v_i,v_{i-2}}^*\alpha_{v_{i+1},v_i,v_{i+2}}\in\mathbb C1$, for every $0\leq i\leq n-1$.
\end{claim}

\begin{proof}[Proof of Claim \ref{invariance}] 
Using that $\delta_{v_i}^*\delta_{v_{i+1}}=\alpha_{v_i,v_{i+1}}^*\alpha_{v_{i+1},v_i}$, for every $0\leq i\leq n-1$, we get that 
$$1=(\delta_{v_0}^*\delta_{v_1})(\delta_{v_1}^*\delta_{v_2})\cdots(\delta_{v_{n-1}}^*\delta_{v_n})
=(\alpha_{v_0,v_1}^*\alpha_{v_1,v_0})(\alpha_{v_1,v_2}^*\alpha_{v_2,v_1})\cdots (\alpha_{v_{n-1},v_0}^*\alpha_{v_0,v_{n-1}}),
$$
and thus
\begin{equation}\label{alpha2}
(\alpha_{v_0,v_{n-1}}\alpha_{v_0,v_1}^*)(\alpha_{v_1,v_0}\alpha_{v_1,v_2}^*)\cdots (\alpha_{v_{n-1},v_{n-2}}\alpha_{v_{n-1},v_0}^*)=1.
\end{equation}

By combining \eqref{alpha1} and \eqref{alpha2} we further get that 
$$
(\alpha_{v_0,v_{n-1},v_1}\alpha_{v_0,v_1,v_{n-1}}^*)(\alpha_{v_1,v_0,v_2}\alpha_{v_1,v_2,v_0}^*)\cdots (\alpha_{v_{n-1},v_{n-2},v_0}\alpha_{v_{n-1},v_0,v_{n-2}}^*)=1.
$$
Since $\alpha_{v_{n-1},v_0,v_{n-2}}\in M_{v_0}$ and  $\alpha_{v_0,v_{n-1},v_1}\in M_{v_{n-1}}$ commute, by applying $\text{Ad}(\alpha_{v_{n-1},v_0,v_{n-2}}^*\alpha_{v_0,v_{n-1},v_1}^*)$ to the last identity we derive that 
$$\alpha_{v_{n-1},v_0,v_{n-2}}^*\alpha_{v_0,v_1,v_{n-1}}^*(\alpha_{v_1,v_0,v_2}\alpha_{v_1,v_2,v_0}^*)\cdots (\alpha_{v_{n-2},v_{n-3},v_{n-1}}\alpha_{v_{n-2},v_{n-1},v_{n-3}}^*)\alpha_{v_{n-1},v_{n-2},v_0}\alpha_{v_0,v_{n-1},v_1}=1.$$

Since $\alpha_{v,v',v''}\in M_{v'}$, and $M_v$ commutes with $M_{v'}$ whenever $v,v'\in\Gamma$ are adjacent, the last displayed identity rewrites as
\begin{equation}\label{alpha4}
(\alpha_{v_{n-1},v_0,v_{n-2}}^*\alpha_{v_1,v_0,v_2})(\alpha_{v_0,v_1,v_{n-1}}^*\alpha_{v_2,v_1,v_3})\cdots (\alpha_{v_{n-2},v_{n-1},v_{n-3}}^*\alpha_{v_0,v_{n-1},v_1})=1.
\end{equation}
Thus, if we denote $\rho_i=\alpha_{v_{i-1},v_i,v_{i-2}}^*\alpha_{v_{i+1},v_i,v_{i+2}}$, for every $0\leq i\leq n-1$, then $\rho_0\rho_1\cdots\rho_{n-1}=1$. Since $\rho_i\in M_{v_i}$, we deduce that $\rho_i\in\mathbb C1$, for every $0\leq i\leq n-1$, which proves the claim.
\end{proof}

Next, we prove the following claim.

\begin{claim}\label{almost_the_end}
There exists a unitary $\mu\in P$ such that $\mu M_{v_i}\mu^*=N_{\alpha(v_i)}$, for every $0\leq i\leq n-1$.
\end{claim}

{\it Proof of Claim \ref{almost_the_end}.}
For $0\leq i\leq n-1$, put $\beta_i:=\alpha_{v_{i-1},v_i,v_{i-2}}\in\mathcal U(M_{v_i})$. Then Claim \ref{invariance} gives 
\begin{equation}\label{alpha6}
\text{$\mathbb{C}\beta_i=\mathbb C\alpha_{v_{i-1},v_i,v_{i-2}}=\mathbb C\alpha_{v_{i+1},v_i,v_{i+2}}$,\;\;\; for every $0\leq i\leq n-1$.}
\end{equation}

Combining \eqref{alpha1} and \eqref{alpha6} 
gives that
$\mathbb C\alpha_{v_i,v_{i-1}}\alpha_{v_i,v_{i+1}}^*=\mathbb C\alpha_{v_i,v_{i-1},v_{i+1}}\alpha_{v_i,v_{i+1},v_{i-1}}^*=\mathbb C\beta_{i-1}\beta_{i+1}^*$ and therefore
\begin{equation}\label{alpha7}
    \text{$\mathbb C\beta_{i-1}^*\alpha_{v_i,v_{i-1}}=\mathbb C\beta_{i+1}^*\alpha_{v_i,v_{i+1}}$, \;\;\; for every $0\leq i\leq n-1.$}
\end{equation}

For $0\leq i\leq n-1$, put $\gamma_i=\beta_{i-1}^*\alpha_{v_i,v_{i-1}}$. Since $\beta_{i-1}\in\mathcal U(M_{v_{i-1}})$ and $\alpha_{v_i,v_{i-1}}\in\mathcal U(M_{\text{lk}(v_i)})$, we get that $\gamma_i\in\mathcal U(M_{\text{lk}(v_i)})$. Moreover, 
using the definition of $\gamma_i$ for the first equality, and Equation~\eqref{alpha7} for the second, we get 
\begin{equation}\label{alpha8}
    \text{$\mathbb C\alpha_{v_i,v_{i-1}}=\mathbb C \beta_{i-1}\gamma_i$ \;\;\; and \;\;\; $\mathbb C\alpha_{v_i,v_{i+1}}=\mathbb C\beta_{i+1}\gamma_i$, \;\;\; for every $0\leq i\leq n-1$.}
\end{equation}

By using Claim \ref{deco},
and that $\beta_{i+1}$ and $\beta_i$ commute we get that for  $0\leq i\leq n-1$ we have

\begin{equation}
\begin{split}
\mathbb C\delta_{v_i}^*\delta_{v_{i+1}}=\mathbb C\alpha_{v_i,v_{i+1}}^*\alpha_{v_{i+1},v_i}&=\mathbb C(\beta_{i+1}\gamma_i)^*(\beta_i\gamma_{i+1})\\&=\mathbb C(\gamma_i^*\beta_{i+1}^*\beta_i\gamma_{i+1})=\mathbb C((\gamma_i^*\beta_i)(\beta_{i+1}^*\gamma_{i+1}))\\&=\mathbb C((\beta_i^*\gamma_i)^*(\beta_{i+1}^*\gamma_{i+1})).
\end{split}
\end{equation}

Hence, $\mathbb C\delta_{v_i}(\beta_i^*\gamma_i)^*=\mathbb C\delta_{v_{i+1}}(\beta_{i+1}^*\gamma_{i+1})^*$, for every $0\leq i\leq n-1$. Let $\mu=\delta_{v_0}(\beta_0^*\gamma_0)^*\in\mathcal U(P)$. Then for every $0\leq i\leq n-1$, we have $\mathbb C\delta_{v_i}(\beta_i^*\gamma_i)^*=\mathbb C\mu$ and hence $\delta_{v_i}\in\mathbb C\mu\beta_i^*\gamma_i$. Since $\beta_i\in\mathcal U(M_{v_i})$ and $\gamma_i\in\mathcal U(M_{\text{lk}(v_i)})$, by Claim \ref{delta_v} we deduce that 
$$\text{$N_{\alpha(v_i)}=\text{Ad}(\delta_{v_i})(M_{v_i})=\text{Ad}(\mu\beta_i^*\gamma_i)(M_{v_i})=\text{Ad}(\mu)(M_{v_i})$,\;\;\;for every $0\leq i\leq n-1$.}$$

This finishes the proof of the claim. 
\end{proof}

In order to finish the proof, it will suffice to prove the following claim.

\begin{claim}
 There exists a unitary $u\in P$ such that $uM_vu^*=N_{\alpha(v)}$, for every $v\in\Gamma$.  
\end{claim}

\begin{proof}
 Since $\mu M_{v_i}\mu^*=N_{\alpha(v_i)}$ by Claim \ref{almost_the_end}, after replacing $\theta$ by $\text{Ad}(\mu)\circ\theta$, we have $M_{v_i}=N_{\alpha(v_i)}$, for every $0\leq i\leq n-1$. Thus, 
 we may assume from now on that the unitaries $(\delta_v)_{v\in\Gamma}\subset P$ provided by Claim \ref{delta_v} satisfy $\delta_{v_i}=1$, for every $0\leq i\leq n-1$, in addition to satisfying Claim \ref{u_v}. 
 
We next use  an idea from the proof of \cite[Theorem 6]{Se89} to show that
\begin{equation}\label{support}\text{$\delta_v\in\mathcal U(M_{\text{st}(v)})$,\;\;\; for every $v\in\Gamma$}.
\end{equation} Let $v\in\Gamma$ and $v'\in\Gamma\setminus\text{st}(v)$.
We claim that there exists $0\leq i\leq n-1$ such that $v_i\not\in\text{st}(v')$. Otherwise, if all $v_i$ belong to $\st(v')$, then since $n\geq 4$, we would find $0\leq j< n-1$ such that $v_j,v_{j+1}\in\text{lk}(v')$. Then $v', v_j$ and $v_{j+1}$ form a triangle, a contradiction. 

Since $v_i,v\in\Gamma\setminus\text{st}(v')$ and $\Gamma$ contains no separating stars,  $v_i$ and $v$ are connected in $\Gamma\setminus\text{st}(v')$. Let $w_0,w_1,\dots,w_k$ be a path in $\Gamma\setminus\text{st}(v')$ such that $w_0=v_i$ and $w_k=v$. 
Then $\delta_{w_0}=1$, $\delta_{w_k}=\delta_v$, thus
\begin{equation}
   \begin{split} \delta_v=\delta_{w_0}^*\delta_{w_k}&=(\delta_{w_0}^*\delta_{w_1})\cdots (\delta_{w_{k-1}}^*\delta_{w_k})\\&=(\alpha_{w_0,w_1}^*\alpha_{w_1,w_0})(\alpha_{w_1,w_2}^*\alpha_{w_2,w_1})\cdots(\alpha_{w_{k-1},w_k}^*\alpha_{w_k,w_{k-1}})\\&=\alpha_{w_0,w_1}^*(\alpha_{w_1,w_0}\alpha_{w_1,w_2}^*)\cdots (\alpha_{w_{k-1},w_{k-2}}\alpha_{w_{k-1},w_k}^*)\alpha_{w_k,w_{k-1}}
   \end{split}
\end{equation}
Since $\alpha_{w,w'}\in M_{\text{lk}(w)}$, for every $w,w'\in\Gamma$ with $w'\in\text{lk}(w)$, we get that $\delta_v\in M_{\text{lk}(w_0)}M_{\text{lk}(w_1)}\cdots M_{\text{lk}(w_k)}$. Hence, $\delta_v\in M_{\cup_{j=0}^k\text{lk}(w_j)}$. Since $w_j\not\in\text{st}(v')$, we get that $v'\not\in\text{lk}(w_j)$, for every $0\leq j\leq k$. Thus, $\cup_{j=0}^k\text{lk}(w_j)\subset\Gamma\setminus\{v'\}$ and $\delta_v\in M_{\Gamma\setminus\{v'\}}$. Since this holds for every $v'\in\Gamma\setminus\text{st}(v)$, we deduce that $$\delta_v\in\bigcap_{v'\in\Gamma\setminus\text{st}(v)}M_{\Gamma\setminus\{v'\}}=M_{\text{st}(v)},$$ which proves \eqref{support}.

We continue by proving that
\begin{equation}\label{ecu2}
    \text{$\alpha_{v,v'}\alpha_{v,v''}^*\in\mathcal U(M_{v'})\mathcal U(M_{v''}),$ \;\;\;for every $v\in\Gamma$ and $v',v''\in\text{lk}(v)$.}
\end{equation}

Let $v\in\Gamma$ and $v',v''\in\text{lk}(v)$. By Claim \ref{deco} we get $\alpha_{v',v}^*\alpha_{v,v'}=\delta_{v'}^*\delta_v$, thus $\delta_v\alpha_{v,v'}^*=\delta_{v'}\alpha_{v',v}^*$. 
On the other hand, since $\alpha_{v,v'}\in M_{\text{lk}(v)}$ and $\alpha_{v',v}\in M_{\text{lk}(v')}$, by using \eqref{support} we get $\delta_v\alpha_{v,v'}^*\in M_{\text{st}(v)}$ and $\delta_{v'}\alpha_{v',v}^*\in M_{\text{st}(v')}$. Since $\Gamma$ contains no triangles, we also have $\text{st}(v)\cap \text{st}(v')=\{v,v'\}$, hence $M_{\text{st}(v)}\cap M_{\text{st}(v')}=M_{\{v,v'\}}.$ Altogether, we deduce that
$z':=\delta_v\alpha_{v,v'}^*\in M_{\{v,v'\}}$.
Similarly, we have 
$z'':=\delta_v\alpha_{v,v''}^*\in M_{\{v,v''\}}$.
These facts together give that
$$\alpha_{v,v'}\alpha_{v,v''}^*=z'^*z''\in M_{\{v,v'\}}M_{\{v,v''\}}\subset M_{\{v,v',v''\}}.$$
Since  $\alpha_{v,v'}\alpha_{v,v''}^*\in M_{\text{lk}(v)}$ and $\{v,v',v''\}\cap\text{lk}(v)=\{v',v''\}$, we conclude that
$$\alpha_{v,v'}\alpha_{v,v''}^*=z'^*z''\in M_{\{v',v''\}}.$$
Hence, the unitaries $z',z''\in M_{\{v,v',v''\}}=(M_{v'}*M_{v''})\overline{\otimes}M_v$ satisfy $z'\in M_{v'}\overline{\otimes}M_v,z''\in M_{v''}\overline{\otimes}M_v$ and $z'^*z''\in (M_{v'}*M_{v''})\overline{\otimes}\mathbb C1$. We  can thus apply Corollary \ref{splitting} to derive that $z'^*z''\in\mathcal U(M_{v'})\mathcal U(M_{v''})$ and thus $\alpha_{v,v'}\alpha_{v,v''}^*\in\mathcal U(M_{v'})\mathcal U(M_{v''})$. This finishes the proof of \eqref{ecu2}.

Next, \eqref{ecu2} implies that for every $v\in\Gamma$ and $v',v''\in\text{lk}(v)$, there exist $\alpha_{v,v',v''}\in\mathcal U(M_{v'})$ and $\alpha_{v,v'',v'}\in\mathcal U(M_{v''})$ such that \begin{equation}\label{ecu4}
\text{$\alpha_{v,v'}\alpha_{v,v''}^*=\alpha_{v,v',v''}\alpha_{v,v'',v'}^*$, for every $v\in\Gamma$ and $v',v''\in\lk(v)$}.
\end{equation} 

Fix $v\in\Gamma$.
We claim that 
\begin{equation}\label{ecu3}
    \text{$\alpha_{v',v,w'}^*\alpha_{v'',v,w''}\in\mathbb C1$, for every $v',v''\in\text{lk}(v),w'\in\text{lk}(v')\setminus\{v\}$ and $w''\in\text{lk}(v'')\setminus\{v\}$.}
\end{equation}

Since $\Gamma$ is triangle-free we have that $w',w''\not\in\text{st}(v)$. Since $\Gamma$ has no separating star, $\Gamma\setminus\text{st}(v)$ is connected and therefore $w'$ and $w''$ are connected by a path in $\Gamma\setminus\text{st}(v)$. In particular, there is a sequence
 $t_0,t_1,\cdots,t_{l-1},t_l=t_0$ in $\Gamma$ in which any two consecutive terms are adjacent such that $t_{l-2}=w',t_{l-1}=v',t_0=v,t_1=v''$, $t_2=w''$ and $t_i\not=v$, for every $1\leq i\leq k-1$. By \eqref{ecu4} we have that 
$\alpha_{t_i,t_{i-1}}\alpha_{t_i,t_{i+1}}^*=\alpha_{t_i,t_{i-1},t_{i+1}}\alpha_{t_i,t_{i+1},t_{i-1}}^*$, for every $0\leq i\leq l-1$. By repeating the calculation from the proof of Claim \ref{invariance} we get that \begin{equation}\label{ecu5}
(\alpha_{t_{l-1},t_0,t_{l-2}}^*\alpha_{t_1,t_0,t_2})(\alpha_{t_0,t_1,t_{l-1}}^*\alpha_{t_2,t_1,t_3})\cdots (\alpha_{t_{l-2},t_{l-1},t_{l-3}}^*\alpha_{t_0,t_{l-1},t_1})=1.
\end{equation}
Thus, if we denote $\zeta_i=\alpha_{t_{i-1},t_i,t_{i-2}}^*\alpha_{t_{i+1},t_i,t_{i+2}}$, for every $0\leq i\leq l-1$, then $\zeta_0\zeta_1\cdots\zeta_{l-1}=1$. 
Since $\zeta_i\in M_{t_i}$, for every $0\leq i\leq l-1$ and $t_0\not\in\{t_1,\cdots,t_{l-1}\}$, we deduce that $\zeta_0\in\mathbb C1$. Hence, we have $\alpha_{v',v,w'}^*\alpha_{v'',v,w''}=\alpha_{t_{l-1},t_0,t_{l-2}}^*\alpha_{t_{1},t_0,t_{2}}=\zeta_0\in\mathbb C1$, 
which proves \eqref{ecu3}.

Next, \eqref{ecu3} implies that for every $v\in\Gamma$, we can find $\alpha_v\in \mathcal U(M_v)$ such that 
\begin{equation}\label{ecu6}\text{$\alpha_{v',v,w'}\in\mathbb C\alpha_v$, for every $v'\in\text{lk}(v)$ and $w'\in\text{lk}(v')\setminus\{v\}$.}\end{equation}

If $v\in\Gamma$, then combining \eqref{ecu6} and \eqref{ecu4} gives that for every  $v',v''\in\text{lk}(v)$ with $v'\not=v''$ we have
$$\text{$\alpha_{v,v'}\alpha_{v,v''}^*=\alpha_{v,v',v''}\alpha_{v,v'',v'}^*\in\mathbb C\alpha_{v'}\alpha_{v''}^*$,}$$
hence $\mathbb C\alpha_{v'}^*\alpha_{v,v'}=\mathbb C\alpha_{v''}^*\alpha_{v,v''}$. Let $\sigma_v:=\alpha_{v''}^*\alpha_{v,v''}$, for some $v''\in\text{lk}(v)$. 
Then $\mathbb C\alpha_{v'}^*\alpha_{v,v'}=\mathbb C\sigma_v$, thus 
\begin{equation}\label{ecu7}
\text{$\alpha_{v,v'}\in\mathbb C\alpha_{v'}\sigma_v$, for every $v\in\Gamma$ and $v'\in\text{lk}(v)$.}
\end{equation}
Moreover, since $\alpha_{v''}\in \mathcal U(M_{v''})\subset \mathcal U(M_{\text{lk}(v)})$ and $\alpha_{v,v''}\in \mathcal U(M_{\text{lk}(v)})$, we get that $\sigma_v\in\mathcal U(M_{\text{lk}(v)})$.

By combining Claim \ref{deco} and \eqref{ecu7}, we deduce that for any two adjacent vertices $v,v'\in\Gamma$ we have
\begin{equation}\label{alpha8}\begin{split}
\mathbb C\delta_{v'}^*\delta_v=\mathbb C\alpha_{v',v}^*\alpha_{v,v'}&=\mathbb C (\alpha_v\sigma_{v'})^*(\alpha_{v'}\sigma_v)\\&=\mathbb C(\sigma_{v'}^*\alpha_v^*\alpha_{v'}\sigma_v)=\mathbb C(\sigma_{v'}^*\alpha_{v'}\alpha_{v}^*\sigma_v)\\&=\mathbb C((\alpha_{v'}^*\sigma_{v'})^*(\alpha_v^*\sigma_v)). \end{split}
\end{equation}
Therefore, $\mathbb C\delta_v(\alpha_v^*\sigma_v)^*=\mathbb C\delta_{v'}(\alpha_{v'}^*\sigma_{v'})^*$, for every adjacent vertices $v,v'\in\Gamma$. Since $\Gamma$ is connected, we conclude that 
$\mathbb C\delta_v(\alpha_v^*\sigma_v)^*=\mathbb C\delta_{v'}(\alpha_{v'}^*\sigma_{v'})^*$, for every  $v,v'\in\Gamma$. Put $u=\delta_{v'}(\alpha_{v'}^*\sigma_{v'})\in\mathcal U(P)$, for some $v'\in\Gamma$. Then $\mathbb C\delta_v(\alpha_v^*\sigma_v)^*=\mathbb Cu$ and thus $\delta_v\in \mathbb Cu\alpha_v^*\sigma_v$, for every $v\in\Gamma$.

Finally, since $\alpha_v\in M_v$ and $\sigma_v\in\mathcal U(M_{\text{lk}(v)})$, using Claim \ref{delta_v} we deduce that 
$$N_{\alpha(v)}=\text{Ad}(\delta_v)(M_v)=\text{Ad}(u\alpha_v^*\sigma_v)(M_v)=\text{Ad}(u)(M_v),$$
for every $v\in\Gamma$. 
This finishes the proof of the claim and of the theorem. 
\end{proof}

\section{Applications: classification results and calculations of symmetries}\label{applications}

In this section, we use our main results to derive Corollaries \ref{RAAGs}-\ref{symmetry_groups} stated in the introduction.

\subsection{$W^*$-classification of right-angled Artin groups}

As a consequence of Theorem \ref{amenable}, we derive the following statement.

\newtheorem*{cor:raags}{Corollary~\ref{RAAGs}}

\begin{cor:raags}
Let $\Gamma$ and $\Lambda$ be  two finite simple graphs which are transvection-free. 

Then the right-angled Artin groups $A_\Gamma$ and $A_\Lambda$ are W$^*$-equivalent if and only if they are isomorphic.
\end{cor:raags}

\begin{proof}
    Clearly if $A_\Gamma$ and $A_\Lambda$ are isomorphic, then they are W$^*$-equivalent. Conversely, if $A_\Gamma$ and $A_\Lambda$ are W$^*$-equivalent, then $\text{L}(A_\Gamma)\cong \text{L}(A_\Lambda)$.
Hence, if we put $\mathcal A=\text{L}(\mathbb Z)$, then $$*_{v,\Gamma}\mathcal A\cong \text{L}(A_\Gamma)\cong\text{L}(A_\Lambda)\cong *_{w,\Lambda}\mathcal A.$$ Since $\mathcal A$ is a diffuse abelian (thus, amenable) tracial von Neumann algebra,  Theorem \ref{amenable} implies that the graphs $\Gamma$ and $\Lambda$ are isomorphic.
\end{proof}

The results established in Section~\ref{sec:amenable} also allow to cover some right-angled Artin groups beyond the transvection-free case. Here is an example.

\begin{example}\label{ex:raag}
    If we only assume that $\Gamma,\Lambda$ and their untransvectable subgraphs $\Gamma^u$ and $\Lambda^u$ are clique-reduced, and $\text{L}(A_\Gamma)\cong \text{L}(A_\Lambda)$, then Theorem~\ref{amenable_gen} still ensures that $\mathcal{E}(\Gamma^u)$ and $\mathcal{E}(\Lambda^u)$ are isomorphic. This provides classification results for right-angled Artin groups beyond the transvection-free case. For instance, for $n\ge 1$, if $P_n$ denotes the path with $n$ edges and $n+1$ vertices, then $P_n^u$ is the path with $n-1$ vertices (the two leaves are the only transvectable vertices). Noticing that $\mathcal E(P_n^u)\cong P_n^u$, it follows that for $m,n\ge 1$ we have $\text{L}(A_{P_m})\cong \text{L}(A_{P_n})$ if and only if $m=n$. We mention that the measure equivalence (or orbit equivalence) classification of these right-angled Artin groups is still open. We also mention that  for $m\neq n$, $A_{P_m}$ and $A_{P_n}$ are not commensurable unless $\{m,n\}=\{3,4\}$ by \cite{CRKZ21}, and all right-angled Artin groups $A_{P_n}$ with $n\ge 3$ are quasi-isometric by \cite{BN08}).
\end{example}

\subsection{Classification of graph products of amenable II$_1$ factors}

Our next corollary follows from our main theorem in the case of II$_1$ factor vertex algebras, Theorem \ref{amenable_gen}. The following statement provides a generalization of Corollary~\ref{hyperfinite} from the introduction (Corollary~\ref{hyperfinite} specializes Corollary~\ref{hyperfinite2} to the case where the graphs $\Gamma$ and $\Lambda$ are transvection-free). 

\begin{corollary}\label{hyperfinite2}
    Let $\Gamma$ and $\Lambda$ be two 
    clique-reduced finite simple graphs. For every $v\in\Gamma$ and $w\in\Lambda$, let $M_v=N_w=R$, where $R$ denotes the hyperfinite II$_1$ factor.

    If $M_\Gamma$ and $N_\Lambda$ are isomorphic, then the untransvectable subgraphs $\Gamma^u$ and $\Lambda^u$ are isomorphic.
\end{corollary}

\begin{proof}
This is a direct consequence of Theorem~\ref{amenable_gen}(2).    
\end{proof}

\begin{remark}
Our assumption that the graphs are clique-reduced is mild: if $\Gamma$ contains a collapsible clique $C$, then by collapsing $C$ to a vertex, we get a new graph product structure for $M_\Gamma$, where the new vertex algebra is $\overline\otimes_{v\in C}R\cong R$, where $R$ is the hyperfinite II$_1$ factor.
\end{remark}

\begin{remark}
Corollary~\ref{hyperfinite2} recovers \cite[Theorem~4.10]{CC25} in the case when $\Gamma$ and $\Lambda$ are finite. The latter result asserts that if $\Gamma$, $\Lambda$ are H-rigid and $M_\Gamma\cong N_\Lambda$, then $\text{Int}(\Gamma)\cong\text{Int}(\Lambda)$.
Recall from \cite[Definition~4.1]{CC25} that a nonempty set $\Gamma_0\subset\Gamma$ is called an internal set
if $\Gamma_0^\perp$ is not a clique; a vertex $v\in\Gamma$ is called an internal vertex if $\{v\}$ is an internal set. When all the internal sets of $\Gamma$ are vertices, the internal graph $\mathrm{Int}(\Gamma)$ is the full subgraph of $\Gamma$ whose vertices are the internal vertices of $\Gamma$. By \cite[Definition~4.6]{CC25}, $\Gamma$ is called H-rigid if all of its internal sets are vertices, and every connected component of every full subgraph $\Gamma_0\subset\Gamma$ with $\Gamma_0^\perp\not=\emptyset$ is a clique. 

To justify the above claim, we first observe that we do not lose any generality by assuming that $\Gamma$ and $\Lambda$ are clique-reduced. Indeed, if $\Gamma$ is H-rigid and contains a collapsible clique $C$ on at least two vertices, then the definition of H-rigid graphs ensures that $C^{\perp}$ is a clique. In particular the link in $\Gamma$ of every vertex $w$ of $C$ is a clique, so $w$ is not internal, and not untransvectable. In addition, if $\bar\Gamma$ is the graph obtained by collapsing $C$ to a vertex $v$, then $v$ is neither internal nor untransevctable. It follows that $\mathrm{Int}(\Gamma)=\mathrm{Int}(\bar\Gamma)$ and $\Gamma^u=\bar\Gamma^u$.

To recover \cite[Theorem~4.10]{CC25}, it is thus enough to prove that if $\Gamma$ is H-rigid and not reduced to a vertex, then the internal graph $\mathrm{Int}(\Gamma)$ coincides with our untransvectable subgraph $\Gamma^u$. To prove this, note if $v\in\Gamma$ is untransvectable, then $\text{lk}(v)$ is not a clique, which means that $v\in\text{Int}(\Gamma)$. Conversely, if $v$ is transvectable, say $\lk(v)\subset\st(w)$ for some vertex $w\in\Gamma\setminus\{v\}$, then the definition of H-rigid graphs implies that $\{v,w\}^\perp$  is a clique. Since $\lk(v)$ is equal to either $\{v,w\}^\perp$ (if $v$ and $w$ are non-adjacent) or to $\{w\}\cup \{v,w\}^\perp$ (if $v$ and $w$ are adjacent) it follows that $\lk(v)$ is a complete graph, so $v\notin\mathrm{Int}(\Gamma)$. 

Corollary \ref{hyperfinite2} applies to new classes of graphs, not covered by \cite[Theorem~4.10]{CC25}. Indeed, the latter result excludes the existence of full subgraphs $\Gamma_0\subset\Gamma$ such that $|\Gamma_0|\geq 2$ and $\Gamma_0^\perp$ is not a clique, which Corollary \ref{hyperfinite2} allows.
\end{remark}

\subsection{$W^*$-classification of graph products of ICC groups}

\newtheorem*{cor:icc}{Corollary~\ref{ICC}}

\begin{cor:icc}\label{ICC2}
Let $\Gamma,\Lambda$ be two finite simple graphs of girth at least $5$, which contain no vertices of valence $0$ or $1$. Let $(G_v)_{v\in\Gamma}$ and $(H_w)_{w\in\Lambda}$ be families of ICC groups. Let  $G_\Gamma=*_{v,\Gamma}G_v$ and $H_\Lambda=*_{w,\Lambda}H_w$ be the associated graph product groups.
Then the following conditions are equivalent:
\begin{enumerate}
    \item  $G_\Gamma$ and $H_\Lambda$  are W$^*$-equivalent.
    \item  $G_\Gamma$ and $H_\Lambda$  are stably W$^*$-equivalent.
    \item There exists a graph isomorphism $\alpha:\Gamma\rightarrow\Lambda$ such that $G_v$ is W$^*$-equivalent to $H_{\alpha(v)}$, for every $v\in\Gamma$.
\end{enumerate} 
\end{cor:icc}

\begin{proof}
We clearly have (1) $\Rightarrow$ (2) and (3) $\Rightarrow$ (1). To prove (2) $\Rightarrow$ (3), assume that 
 $G_\Gamma$ and $H_\Lambda$ are stably W$^*$-equivalent, i.e., there is a $*$-isomorphism $\theta:\text{L}(G_\Gamma)\rightarrow\text{L}(H_\Lambda)^t$, for some $t>0$. Since $\text{L}(G_\Gamma)=*_{v,\Gamma}\text{L}(G_v)$, $\text{L}(H_\Lambda)=*_{w,\Lambda}\text{L}(H_w)$, and $\text{L}(G_v),\text{L}(H_w)$ are II$_1$ factors, for every $v\in\Gamma,w\in\Lambda$, Theorem~\ref{factors'} implies that $t=1$ and there exist a graph isomorphism $\alpha:\Gamma\rightarrow\Lambda$ and for every $v\in\Gamma$ a unitary $u_v\in \text{L}(H_\Lambda)$ such that $\theta(\text{L}(G_v))=u_v\text{L}(H_{\alpha(v)})u_v^*$. Hence, $G_v$ is W$^*$-equivalent to $H_{\alpha(v)}$, for every $v\in\Gamma$, which proves that (3) holds.
\end{proof}

\subsection{Computations of symmetry groups}

\newtheorem*{cor:symmetry}{Corollary~\ref{symmetry_groups}}

\begin{cor:symmetry}\label{cor:symmetry_groups}

Let $\Gamma$ be a finite simple graph of girth at least $5$, which contains no vertices of valence $0$ or $1$. Let $(M_v)_{v\in\Gamma}$ be a family of II$_1$ factors.
Then $\mathcal F(M_\Gamma)=\{1\}$.

Moreover, if $\Gamma$ has no separating stars, then  \[\mathrm{Out}(M_\Gamma)\cong \left(\bigoplus_{v\in\Gamma}\mathrm{Aut}(M_v)\right)\rtimes\mathrm{Aut}(\Gamma;M_\Gamma).\]

\end{cor:symmetry}

The reader is referred to the introduction, in particular Equation~\eqref{aut_loc}, for the semi-direct product structure.

\begin{proof}
   Applying Theorem \ref{factors'} to $N_\Lambda=M_\Gamma$ implies that $\mathcal F(M_\Gamma)=\{1\}$. 
   
   Assume additionally that $\Gamma$ has no separating stars. Then the moreover assertion of Theorem \ref{factors'} implies that $\text{Aut}(M_\Gamma)=\text{Aut}_{\text{loc}}(M_\Gamma)\text{Inn}(M_\Gamma)$. Hence we have 
    \begin{equation}\label{isomorphism}\text{Out}(M_\Gamma)=\text{Aut}_{\text{loc}}(M_\Gamma)\text{Inn}(M_\Gamma)/\text{Inn}(M_\Gamma)\cong\text{Aut}_{\text{loc}}(M_\Gamma)/(\text{Inn}(M_\Gamma)\cap\text{Aut}_{\text{loc}}(M_\Gamma)).\end{equation}
We claim that $\text{Inn}(M_\Gamma)\cap\text{Aut}_{\text{loc}}(M_\Gamma)=\{\text{Id}\}$.
Let $\theta\in \text{Inn}(M_\Gamma)\cap\text{Aut}_{\text{loc}}(M_\Gamma)$. Then $\theta=\text{Ad}(u)$, for some $u\in\mathcal U(M_\Gamma)$ and there exists $\sigma\in\text{Aut}(\Gamma)$ such that $\theta(M_v)=M_{\sigma(v)}$, for every $v\in\Gamma$. If $v\in\Gamma$, then $M_{\sigma(v)}=\theta(M_v)=uM_vu^*$ and Lemma~\ref{basic}(2) implies that $\sigma(v)=v$. Hence, $uM_vu^*=M_v$, and therefore $u\in M_{\text{st}(v)}$, by Proposition \ref{normalizer}(1). In conclusion $u\in\cap_{v\in\Gamma}M_{\text{st}(v)}=M_{\cap_{v\in\Gamma}\text{st}(v)}$. 
Let us show that $\cap_{v\in\Gamma}\text{st}(v)=\emptyset$. Indeed, if $w\in \cap_{v\in\Gamma}\text{st}(v)$, then $\Gamma=\text{st}(w)$.
Let $w'\in\text{lk}(w)$ and $w''\in\text{lk}(w')\setminus\{w\}$. Since $w''\in\text{st}(w)$, we get that $w''\in\text{lk}(w)$, and thus $\{w,w',w''\}$ forms a triangle, which is a contradiction. Hence, we have shown that $\cap_{v\in\Gamma}\text{st}(v)=\emptyset$ and therefore $u\in\mathbb C1$. This shows that $\theta=\text{Ad}(u)=\text{Id}$, which finishes the proof of the claim.

   The claim, \eqref{isomorphism} and \eqref{aut_loc} give that $\text{Out}(M_\Gamma)\cong\mathrm{Aut}_{\mathrm{loc}}(M_\Gamma)\cong(\bigoplus_{v\in\Gamma}\mathrm{Aut}(M_v))\rtimes\mathrm{Aut}(\Gamma;M_\Gamma)$.
\end{proof}

\section{Obstructions to rigidity results}
Theorems \ref{arbitrary}-\ref{factors'} provide classes of graph product II$_1$ factors  $M_\Gamma$ which satisfy the following rigidity statement: for any automorphism $\theta:M_\Gamma\rightarrow M_\Gamma$, there exists an automorphism $\alpha$ of $\Gamma$ such that $\theta(M_v)\prec_{M_\Gamma}^sM_{\alpha(v)}$, for every $v\in\Gamma$. In all of these results, the graph $\Gamma$ is assumed to be transvection-free. The purpose of this section is to show that, conversely, the presence of transvectable vertices is in general an obstruction to such a rigidity statement. We start with the following result.

\begin{proposition}\label{obstruction1}
Let $\Gamma$ be a finite simple graph, $(M_v,\tau_v)_{v\in\Gamma}$  a family of tracial von Neumann algebras, and $M_\Gamma=*_{v,\Gamma}M_v$  the associated graph product von Neumann algebra.
Assume that $v,v'\in\Gamma$ are such that $\emph{lk}(v)\subset\emph{st}(v')$ and the following condition holds ($\star$)  there is  $\rho\in\emph{Aut}(M_{\{v,v'\}})$ satisfying $\rho(M_v)\nprec_{M_{\{v,v'\}}}M_v$ and ${\rho}(M_{v'})=M_{v'}$.

Then there exists  $\theta\in\emph{Aut}(M_\Gamma)$ such that $\theta(M_v)\nprec_{M_\Gamma}M_{v''}$, for every $v''\in\Gamma$.
\end{proposition}

\begin{proof}
We first show the existence of an automorphism $\theta$ of  $M_\Gamma$ with $\theta_{|M_{\{v,v'\}}}=\rho$ and $\theta_{|M_w}=\text{Id}_{M_w}$,  for every $w\in\Gamma\setminus\{v,v'\}$. Let $n=|\Gamma|$ and enumerate $\Gamma=\{v_1=v,v_2=v',\dots,v_n\}$.
For $k\geq 2$, let $\Gamma_k$ be the full subgraph of $\Gamma$ with vertex set $\{v_1,v_2,\dots,v_k\}$. We claim that for every $2\leq k\leq n$, there is an automorphism $\theta_{k}$ of $M_{\Gamma_k}$
such that ${\theta_k}_{|M_{\{v,v'\}}}=\rho$ and ${\theta_k}_{|M_w}=\text{Id}_{M_w}$, for every $w\in\Gamma_k\setminus\{v,v'\}$. Once this claim is proven, $\theta:=\theta_{n}$ is the desired automorphism of $M_\Gamma$.

We prove the claim by induction on $k$. For $k=2$, the claim holds for $\theta_2=\rho$.
Assume that we have constructed $\theta_k\in\text{Aut}(M_{\Gamma_k})$, for some $2\leq k<n$.
In order to construct $\theta_{k+1}\in\text{Aut}(M_{\Gamma_{k+1}})$, note that $M_{\Gamma_{k+1}}=M_{\Gamma_k}*_{B_k}(B_k\overline{\otimes}M_{v_{k+1}})$, where $B_k=M_{\Gamma_{k}\cap\text{lk}(v_{k+1})}$.

If $v\not\in\text{lk}(v_{k+1})$, then 
$\Gamma_k\cap \text{lk}(v_{k+1})\subset\Gamma_k\setminus\{v\}$.
Since $\theta_k(M_{w})=M_{w}$, for every $w\in\Gamma_k\setminus\{v\}$, and $B_k$ is generated by $\{M_w\mid w\in\Gamma_k\cap\text{lk}(v_{k+1})\}$, we get that
${\theta_k}(B_k)=B_k$. If $v\in\text{lk}(v_{k+1})$, then since $\text{lk}(v)\subset\text{st}(v')$ and $v'\not=v_{k+1}$, we get that $v'\in\text{lk}(v_{k+1})$.
Thus, $B_k$ is generated by $M_{\{v,v'\}}$ and $\{M_w\mid w\in (\Gamma_k\cap\text{lk}(v_{k+1}))\setminus\{v,v'\}\}$, which again gives that $\theta_k(B_k)=B_k$.
In either case,  $\theta_k(B_k)=B_k$. This implies that there is $\theta_{k+1}\in\text{Aut}(M_{\Gamma_{k+1}})$ such that ${\theta_{k+1}}_{|M_{\Gamma_k}}=\theta_k$ and ${\theta_{k+1}}_{|M_{v_{k+1}}}=\text{Id}_{M_{v_{k+1}}}$, which proves the induction step and the claim.

We finish the proof by showing that $\theta$ has the desired property. Assume by contradiction that $\rho(M_v)=\theta(M_v)\prec_{M_\Gamma}M_{v''}$, for some $v''\in\Gamma$. Since $M_v$ is diffuse and $\rho(M_v)\subset M_{\{v,v'\}}$, applying Corollary \ref{intersection} implies that $v''\in\{v,v'\}$. Using again that $\rho(M_v)\subset M_{\{v,v'\}}$ and  Lemma \ref{descend}, we derive that $\rho(M_v)\prec_{M_{\{v,v'\}}}M_{v''}$. If $v''=v$, this contradicts our assumption that $\rho(M_v)\nprec_{M_{\{v,v'\}}}M_v$. If $v''=v'$, then $\rho(M_v)\prec_{M_{v,v'}}M_{v'}=\rho(M_{v'})$, hence $M_v\prec_{M_{v,v'}}M_{v'}$. By Lemma~\ref{basic}(2), this again gives a contradiction.
\end{proof}

\begin{remark}
In order to apply Proposition \ref{obstruction1} we need to exhibit examples of pairs of vertices $v,v'\in\Gamma$ with $\text{lk}(v)\subset\text{st}(v')$ which satisfy condition $(\star)$. Before proving two general results in this direction (Proposition \ref{local_auts} and Corollary \ref{obstruction2}), let us point out that $(\star)$ holds if $M_v=M_{v'}=\text{L}(\mathbb Z)$. In this case, we have $M_{\{v,v'\}}=\text{L}(A)$, where $A$ is a group generated by elements $v,v'$ such that $M_v=\text{L}(\langle v\rangle)$ and $M_{v'}=\text{L}(\langle v'\rangle)$. Moreover, we have $A\cong \mathbb Z^2$, if $v'\in\text{lk}(v)$, and $A\cong\mathbb F_2$, if $v'\not\in\text{lk}(v)$. Thus, we can define $\rho_0\in\text{Aut}(A)$ by letting $\rho(v)=vv'$ and $\rho_0(v')=v'$. Then the automorphism $\rho$ of $M_{\{v,v'\}}=\text{L}(A)$ associated to $\rho_0$ (i.e., given by $\rho(u_g)=u_{\rho_0(g)}$, for every $g\in A$) satisfies $(\star)$. 

If $v$ and $v'$ are adjacent, the above construction works more generally if $M_v=\text{L}(G_v)$, $M_{v'}=\text{L}(G_{v'})$ and there is a homomorphism with infinite image $\theta:G_v\to Z(G_{v'})$. In this case, the map which sends $g\in G_v$ to $g\theta(g)$ and is the identity on $G_{v'}$, gives an automorphism of $G_v\times G_{v'}$ and thus of $M_{\{v,v'\}}=\text{L}(G_v\times G_{v'})$ which satisfies $(\star)$.

Assume now that $v$ and $v'$ are not adjacent, and that $M_v=\text{L}(G_v)$ and $M_{v'}=\text{L}(G_{v'})$, where $G_v$ splits as a free product $A\ast B$, and $G_{v'}$ contains an infinite order element, $h$. Then the map which is identity on $A$ and on $G_{v'}$, and is the conjugation by $h$ on $B$ gives an automorphism of $G_v\ast G_{v'}$ and thus of $M_{\{v,v'\}}=\text{L}(G_v*G_{v'})$ which satisfies $(\star)$.

In this section, we will extend these constructions to von Neumann algebras that do not necessarily come from groups.
\end{remark}

\begin{proposition}\label{local_auts}
    Let $(M_1,\tau_1)$ and $(M_2,\tau_2)$ be tracial von Neumann algebras. Then the following hold:

    \begin{enumerate}
        \item Assume that there exists a $*$-homomorphism $\sigma:M_1\rightarrow M_1\overline{\otimes}\mathcal Z(M_2)$ such that 
        \begin{enumerate}
        \item $\emph{E}_{\mathcal Z(M_2)}(\sigma(x))=\tau(x)1$, for every $x\in M_1$,
       \item     $\sigma(M_1)\nprec_{M_1\overline{\otimes}\mathcal Z(M_2)}M_1$, and
       \item $\sigma(M_1)$ and $\mathcal Z(M_2)$ generate $M_1\overline{\otimes}\mathcal Z(M_2)$.
        \end{enumerate}
Then there exists $\rho\in\emph{Aut}(M_1\overline{\otimes}M_2)$ such that $\rho(M_1)\nprec_{M_1\overline{\otimes}M_2}M_1$ and $\rho(M_2)=M_2$. 
        \item Assume that one of the following two conditions holds: (i) $M_1=\emph{L}(\mathbb Z)$ and $M_2$ is diffuse, or (ii)
        $M_1=P*Q$, for diffuse von Neumann algebras $P,Q$, and $M_2$ admits a trace zero unitary. 
        Then there exists $\rho\in\emph{Aut}(M_1*M_2)$ such that $\rho(M_1)\nprec_{M_1*M_2}M_1$ and $\rho(M_2)=M_2$.
    \end{enumerate}
\end{proposition}

\begin{proof}
    (1) By condition (a), for every $x\in M_1$ and $y\in M_2$ we have $$\tau(\sigma(x)y)=\tau(\sigma(x)\text{E}_{M_1\overline{\otimes}\mathcal Z(M_2)}(y))=\tau(\sigma(x)\text{E}_{\mathcal Z(M_2)}(y))=\tau(\text{E}_{\mathcal Z(M_2)}(\sigma(x))y)=\tau(x)\tau(y).$$
    Thus, we can define a trace-preserving $*$-homomorphism $\rho:M_1\overline{\otimes}M_2\rightarrow M_1\overline{\otimes}M_2$ by $\rho(x)=\sigma(x)$, for every $x\in M_1$, and $\rho(y)=y$, for every $y\in M_2$. 
    Since $\rho$ is trace-preserving, it is injective. Condition (c) ensures that $\rho$ is surjective, and thus an automorphism of $M_1\overline{\otimes}M_2$. By using condition (b), we find a net $(u_n)\subset\mathcal U(M_1)$ such that $\|\text{E}_{M_1}(a\sigma(u_n)b)\|_2\rightarrow 0$, for every $a,b\in M_1\overline{\otimes}\mathcal Z(M_2)$. 

We claim that $\|\text{E}_{M_1}(a\sigma(u_n)b)\|_2\rightarrow 0$, for every $a,b\in M_1\overline{\otimes}M_2$. To prove this claim, we may assume that $a=a_1\otimes a_2$ and $b=b_1\otimes b_2$, where $a_1,a_2\in M_1$ and $b_1,b_2\in M_2$. Let $\zeta\in M_1\overline{\otimes}\mathcal Z(M_2)$ of the form $\zeta=\zeta_1\otimes\zeta_2$, where $\zeta_1\in M_1$ and $\zeta_2\in\mathcal Z(M_2)$. Then, letting $c=\text{E}_{\mathcal Z(M_2)}(b_2a_2)\in\mathcal Z(M_2)$, we have $\tau(a_2\zeta_2b_2)=\tau(\zeta_2b_2a_2)=\tau(\zeta_2c)$ and thus
$$\text{E}_{M_1}(a\zeta b)
=\tau(a_2\zeta_2b_2)a_1\zeta_1 b_1=\tau(\zeta_2c)a_1\zeta_1b_1=\text{E}_{M_1}((a_1\otimes 1)\zeta(b_1\otimes c)).$$
From this we deduce by linearity that $\text{E}_{M_1}(a\zeta b)=\text{E}_{M_1}((a_1\otimes 1)\zeta(b_1\otimes c))$, for every $\zeta\in M_1\overline{\otimes}\mathcal Z(M_2)$.
Since $\sigma(u_n)\in M_1\overline{\otimes}\mathcal Z(M_2)$, for every $n$, and $a_1\otimes 1,b_1\otimes c\in M_1\overline{\otimes}\mathcal Z(M_2)$, we conclude that $$\|\text{E}_{M_1}(a\sigma(u_n)b)\|_2=\|\text{E}_{M_1}((a_1\otimes 1)\sigma(u_n)(b_1\otimes c))\|_2\rightarrow 0.$$

This proves the claim, which 
 gives that $\sigma(M_1)\nprec_{M_1\overline{\otimes}M_2}M_1$ and further that $\rho(M_1)\nprec_{M_1\overline{\otimes}M_2}M_1$.

    (2) First, assume that (i) holds. Let $v\in M_1$ be a Haar unitary which generates $M_1$. Since $M_2$ is diffuse, we can also find a Haar unitary $w\in M_2$. Then $vw$ is a Haar unitary which is freely independent from $M_2$ and thus we can define an automorphism $\rho$ of $M_1*M_2$ by letting $\rho(v)=vw$ and $\rho(y)=y$, for every $y\in M_2$. Then it is easy to check that $\|\text{E}_{M_1}(a\rho(v^n)b)\|_2=\|\text{E}_{M_1}(a(vw)^nb)\|_2\rightarrow 0$, for every $a,b\in M_1*M_2$.
This implies that $\rho(M_1)\nprec_{M_1*M_2}M_1.$

    Second, assume that (ii) holds. Let $u\in\mathcal U(M_2)$ with $\tau(u)=0$. Since $M_1*M_2=P*Q*M_2$, we can define an automorphism $\rho$ of $M_1*M_2$ by letting $\rho(x)=uxu^*$, for every $x\in P$, $\rho(y)=y$, for every $y\in Q*M_2$. Then clearly $\rho(M_2)=M_2$, so it remains to prove that $\rho(M_1)\nprec_{M_1*M_2}M_1$. 
    Assume by contradiction that $\rho(M_1)\prec_{M_1*M_2}M_1$. 
    
    Then we can find a nonzero $v\in M_1*M_2$ and $w_1,\dots,w_k\in M_1*M_2$ such that $\rho(M_1)v\subset\sum_{i=1}^kw_iM_1$ (see \cite[Proposition C.1]{Va07}). 
    In particular, $\rho(P)v\subset\sum_{i=1}^kw_iM_1$ and $\rho(Q)v\subset\sum_{i=1}^kw_iM_1$. Since $\rho(P)=uPu^*$ and $\rho(Q)=Q$, we get that $P(u^*v)\subset \sum_{i=1}^k(u^*w_i)M_1$ and $Qv\subset\sum_{i=1}^kw_iM_1$.
    Since $P,Q\subset M_1$ are diffuse, by applying \cite[Theorem 1.1]{IPP08} we deduce that $u^*v\in M_1$ and $v\in M_1$.
    Since $u\in M_2$ we have $\text{E}_{M_1}(u)=\tau(u)1=0$, and therefore
 $u^*v=\text{E}_{M_1}(u^*v)=\text{E}_{M_1}(u^*)v=0$. This implies that $v=0$ and gives a contradiction.    
\end{proof}

\begin{corollary}\label{obstruction2}
Let $\Gamma$ be a finite simple graph, $(M_v,\tau_v)_{v\in\Gamma}$  a family of tracial von Neumann algebras, and  $M_\Gamma=*_{v,\Gamma}M_v$  the associated graph product von Neumann algebra.
Let $v,v'\in\Gamma$ such that $\emph{lk}(v)\subset\emph{st}(v')$.
Assume that we are in one of the following two situations:
\begin{enumerate}
\item $v'\in\emph{lk}(v)$, $M_v=B\rtimes G$, for a trace preserving action $G\curvearrowright (B,\tau)$ of a discrete group $G$ admitting an infinite countable abelian quotient, and $\mathcal Z(M_{v'})$ is diffuse.
\item $v'\not\in\emph{lk}(v)$, and either (i) $M_v=\emph{L}(\mathbb Z)$ and $M_{v'}$ is diffuse, or (ii) $M_v=P*Q$, for some diffuse tracial von Neumann algebras $P,Q$, and $M_{v'}$ admits a trace zero unitary.     
\end{enumerate}

Then there exists  $\theta\in\emph{Aut}(M_\Gamma)$ such that $\theta(M_v)\nprec_{M_\Gamma}M_{v''}$, for every $v''\in\Gamma$.
\end{corollary}

Note that Corollary \ref{obstruction2} in particular applies if $M_v=\text{L}(\mathbb Z)$ and $\mathcal Z(M_{v'})$ is diffuse.

\begin{proof}
    Assume that (1) holds. Let $H$ be an infinite countable abelian quotient of $G$ and $\pi:G\rightarrow H$ be an onto homomorphism. Then $\text{L}(H)\cong\text{L}(\mathbb Z)$ and since $\mathcal Z(M_{v'})$
    is diffuse, we can view $\text{L}(H)$ as a unital subalgebra of $\mathcal Z(M_{v'})$. Using that $M_v=B\rtimes G$, we define a $*$-homomorphism  $\sigma:M_{v}\rightarrow M_{v}\overline{\otimes}\text{L}(H)\subset M_v\overline{\otimes}\mathcal Z(M_{v'})$ 
by letting $\sigma(bu_g)=bu_g\otimes u_{\pi(g)}$, for every $b\in B$ and $g\in G$. It is easy to see that $\sigma(M_v)$ and $\mathcal Z(M_{v'})$ generate $M_v\overline{\otimes}M_{v'}$, and $\text{E}_{\mathcal Z(M_{v'})}(\sigma(x))=\tau(x)1$, for every $x\in M_v$. Since $\pi(G)=H$ is infinite, we can find a sequence $(g_n)\subset G$ such that $\pi(g_n)\rightarrow\infty$ in $H$.
Then  $\|\text{E}_{M_v}(a\sigma(u_{g_n})b)\|_2\rightarrow 0$, for every $a,b\in M_v\overline{\otimes}\mathcal Z(M_{v'})$, hence
 $\sigma(M_v)\nprec_{M_v\overline{\otimes}\mathcal Z(M_{v'})}M_v$.
 Since $v'\in \text{lk}(v)$, we have that $M_{\{v,v'\}}=M_v\overline{\otimes}M_{v'}$. Thus, combining part (1) of Proposition \ref{local_auts} with Proposition \ref{obstruction1} gives the conclusion in this case.

 If (2) holds, then since $v'\not\in\text{lk}(v)$, we have that $M_{\{v,v'\}}=M_v*M_{v'}$. Then combining part (2) of Proposition \ref{local_auts} with Proposition \ref{obstruction1} finishes the proof in this case.
\end{proof}

\small

\begin{flushleft}
Camille Horbez\\ 
Universit\'e Paris-Saclay, CNRS,  Laboratoire de math\'ematiques d'Orsay, 91405, Orsay, France \\
\emph{e-mail:~}\texttt{camille.horbez@universite-paris-saclay.fr}\\[4mm]
\end{flushleft}

\begin{flushleft}
Adrian Ioana \\
AP\&M 5210, Department of Mathematics, UCSD, 9500 Gilman Drive, La Jolla CA 92093\\
\emph{e-mail:~}\texttt{aioana@ucsd.edu}
\end{flushleft}

\vfill

\noindent This work is openly licensed via Creative Commons CC-BY 4.0\\
\href{https://creativecommons.org/licenses/by/4.0/}{https://creativecommons.org/licenses/by/4.0/}.

\end{document}